\documentclass[12pt]{iopart}

\usepackage{iopams}
\usepackage{amsthm}
\usepackage{comment}
\usepackage{xcolor}
\usepackage{graphicx,color}
\usepackage[colorlinks,pdfdisplaydoctitle]{hyperref}
\usepackage{subfigure}
\newtheorem{theorem}{Theorem}[section]
\newtheorem{lemma}[theorem]{Lemma}
\newtheorem{proposition}[theorem]{Proposition}

\newtheorem{hypothesis}[theorem]{Hypothesis}
\newtheorem{remark}[theorem]{Remark}
\newtheorem{algorithm}[theorem]{Algorithm}

\definecolor{darkred}{rgb}{0.6,0.1,0.1}
\definecolor{darkgreen}{rgb}{0.1,0.6,0.1}
\definecolor{darkblue}{rgb}{0.1,0.1,0.6}

\newcommand{\bD}{{\overline{D}}}
\newcommand{\rb}{\|_{{\mathcal B}}}
\newcommand{\cB}{\mathcal B}
\newcommand{\cG}{\mathcal G}
\newcommand{\bbR}{\mathbb{R}}

\newcommand{\bbQ}{\mathbb{Q}}

\newcommand{\cN}{\mathcal{N}}
\newcommand{\Real}{\mathbb{R}}
\newcommand{\ve}{\varepsilon}
\newcommand{\la}{\langle}
\newcommand{\ra}{\rangle}

\begin{document}
\title{Well-Posed Bayesian Geometric Inverse Problems Arising in Subsurface Flow}
\author{Marco A. Iglesias$^\dag$, Kui Lin$^\ast$ \footnote[0]{$^\ast$ Corresponding author.} and Andrew M. Stuart$\ddag$}
\address{$^\dag$ School of Mathematical Sciences, University of Nottingham, Nottingham NG7 2RD, UK\\
$^\ast$ School of Mathematical Sciences, Fudan University, Shanghai, 200433, China\\
$\ddag$ Mathematics Institute, University of Warwick, Coventry CV4 7AL, UK}
\ead{link10@fudan.edu.cn}

\begin{abstract}
In this paper, we consider the inverse problem of
determining the permeability of the subsurface from hydraulic
head measurements, within the framework of a
steady Darcy model of groundwater flow. We study geometrically defined
prior permeability fields, which admit layered, fault and channel structures, in order
to mimic realistic subsurface features; within each layer we adopt either constant
or continuous function representation of the permeability. This prior model
leads to a parameter identification problem for a finite number of unknown parameters
determining the geometry, together with either a finite number of permeability
values (in the constant case) or a finite number of fields (in the continuous
function case).
We adopt a Bayesian framework showing existence and well-posedness of
the posterior distribution. We also introduce novel Markov Chain-Monte Carlo (MCMC) methods, which
exploit the different character of the geometric and permeability
parameters, and build on recent advances in function space MCMC. These algorithms
provide rigorous estimates of the permeability, as well as the uncertainty
associated with it, and only require forward model evaluations. No adjoint
solvers are required and hence the methodology is applicable to black-box
forward models.
We then use these methods to explore
the posterior and to illustrate the methodology with numerical experiments.
\end{abstract}

%Uncomment for PACS numbers title message
%\pacs{00.00, 20.00, 42.10}
% Keywords required only for MST, PB, PMB, PM, JOA, JOB?
%\vspace{2pc}
%\noindent{\it Keywords}: Article preparation, IOP journals
% Uncomment for Submitted to journal title message
%\submitto{\JPA}
% Comment out if separate title page not required

\submitto{\IP}
\maketitle
%\tableofcontents[]
\section{Introduction}\label{se:sec1}

Quantification of uncertainty in the geologic properties of the subsurface is essential for optimal management and decision-making in subsurface flow applications such as nuclear waste disposal, geologic storage of CO$_2$ and enhanced-oil-recovery. The environmental impact of these applications, for example, cannot be properly assessed without quantifying the uncertainty in geologic properties of the formation. However, typical uncertainty of \textit{a priori} geologic information often results in large uncertainty in the flow predictions. This uncertainty is, in turn, detrimental to optimally managing the process under consideration. A common strategy to reduce uncertainty and improve our ability to make decisions is to incorporate (or assimilate) data, typically corrupted by noise, that arise from the flow response to the given geologic scenario. Because of the
prior uncertainty in geological information, and because of the noise in the
data, this assimilation of data into subsurface flow models is naturally
framed in a Bayesian fashion: the aim is to characterize the \textit{posterior} (conditional) probability of geologic properties given the observed data \cite{OliverReview} . The Bayesian framework is therefore a statistical approach for solving the inverse problem of identifying unknown geologic parameters given noisy data from the flow model. However, in contrast to deterministic approaches, the Bayesian framework provides a quantification of the uncertainty via the posterior; this in turn allows an assessment of the uncertainty arising in the corresponding model predictions. The goal of this paper is to study a class of
inverse problems, arising in subsurface flow, in which both the geometry
and the physical characteristics of the formation are to be inferred, and
to demonstrate how the full power of the Bayesian methodology can be
applied to such problems.
We study existence and well-posedness of the posterior distribution,
describe Markov Chain-Monte Carlo (MCMC)
methods tailored to the specifics of the problems of interest,
and show numerical results based on this methodology. We work within
the context of a steady single-phase Darcy flow model, but the
mathematical approach, and resulting algorithms, may be employed
within more complex subsurface flow models.

\subsection{Literature Review: Computational and Mathematical Setting}

The first paper to highlight the power of the Bayesian approach
to regularization of inverse problems is \cite{Fra70}, where linear inverse
problems, including heat kernel inversion, were discussed. The theory
of Bayesian inversion for linear problems was then further developed in
\cite{Man84, LPS89} whilst the book \cite{KS07b} demonstrated the
potential of the approach for a range of complex inverse problems,
linear and nonlinear, arising in applications. Alongside these
developments was the work of Lasanen which lays the mathematical
foundations of Bayesian inversion for a general class of nonlinear
inverse problems \cite{La02,las12,las12b}. The papers
\cite{marzouk2009stochastic, Mondal} demonstrate approximation results
for the posterior distribution, employing Kullback-Leibler divergence
and total variation metric respectively, for finite dimensional
problems. The paper \cite{AMS10}
demonstrated how the infinite dimensional perspective on Bayesian
inversion leads to a well-posedness and approximation theory, and
in particular to the central role played by the Hellinger
metric in these results.
The papers \cite{CDRS08,CDS09,DS11} demonstrated application
of these theories to various problems arising in fluid mechanics
and in subsurface flow. Most of this work stemming from
\cite{AMS10} concerns the use of Gaussian random field priors and
is hence limited to the reconstruction of continuous fields or to fields with discontinuous properties where the positions of
the discontinuities are known and can be built in explicitly, through the
prior mean, or through construction of the inverse covariance operators
as in \cite{kaipio1999inverse}. The article \cite{las09} introduced Besov priors in order to
allow for the solution of linear inverse problems arising in imaging
where edges and discontinuities are important and this work
was subsequently partially generalized to nonlinear inverse problems arising
in subsurface flow \cite{DHS12}. However none of this work is well-adapted
to the geometrical discontinuous structures observed in subsurface formations,
where layers, faults and channels may arise. In this paper we address
this by formulating well-posed nonlinear Bayesian inverse problems
allowing for such structures.

The computational approach highlighted in \cite{KS07b} is based
primarily on the ``discretize first then apply algorithm'' approach
with, for example, MCMC methods used
as the algorithm. The work highlighted in \cite{David} shows the
power of an approach based on ``apply algorithm then discretize'', leading
to new MCMC methods which have the advantage of having mixing rates
which are mesh-independent \cite{HSV12,Vollmer};
the importance of mesh-independence is
also highlighted in the work \cite{las09}. However the work overviewed
in \cite{David} is again mainly aimed at problems with Gaussian
priors. In this paper we build on this work and develop MCMC
methods which use Metropolis-within-Gibbs methodologies to
separate geometric and physical parameters within the overall MCMC
iteration; furthermore these MCMC methods require only solution
of the forward problem and no linearizations of the forward operator, and are hence suitable
in situations where no adjoint solvers are available and only  black-box forward simulation software is
provided. We use the resulting MCMC methodology to solve some hard
geometric inverse problems arising in subsurface modelling.

\subsection{Literature Review: Subsurface Applications}

While standard approaches for data assimilation in subsurface flow models are mainly based on the Bayesian formalism, most of those approaches apply the Bayesian framework to the resulting finite-dimensional approximation of the model under consideration \cite{OliverReview}. However, recent work \cite{Us} has shown the potential detrimental effect of directly applying standard MCMC methods
to approximate finite-dimensional posteriors which arise from discretization of PDE based Bayesian inverse problems. For standard subsurface flow models, the forward (parameter-to-output) map is nonlinear, and so even if the prior distribution is Gaussian, the posterior is in general non-Gaussian. Therefore, the full characterization of the posterior can only be accomplished by sampling methods such as Markov Chain Monte Carlo (MCMC). On the other hand, unknown geologic properties are in general described by functions that appear as coefficients of the PDE flow model.  Then, the Bayesian posterior of these discretized functions is defined on a very high-dimensional space (e.g $10^{5}\sim 10^{6}$) and sampling with standard MCMC methods \cite{EmeRey,Oliver2} becomes computationally prohibitive. Some standard MCMC approaches avoid the aforementioned issue by parametrizing the unknown geologic properties in terms a small number of parameters (e.g. $10\sim 20$) \cite{Efendiev1,Mondal}. Some others, however, consider the full parameterization of the unknown (i.e. as a discretized function) but are only capable of characterizing posteriors from one-dimensional problems on a very coarse grids \cite{EmeRey,Oliver2}. While the aforementioned strategies offer a significant insight into to the solution of Bayesian inverse problems in subsurface models, there remains substantial opportunity for the improvement and development of Bayesian data assimilation techniques capable of describing the
posterior distributions accurately and efficiently, using the mesh-independent
MCMC schemes overviewed in \cite{David} and applied to subsurface applications
in \cite{Us}. In particular we aim to do so in this paper in the context
of geometrically defined models of the geologic properties.

The petrophysical characterization of complex reservoirs involves the identification of the geologic facies of the formation. For those complex geologies, prior knowledge may include uncertain information of the interface between the geologic facies as well as the spatial structure of each of the rock types. In addition, prior knowledge of complex reservoirs may also include information of potential faults. Moreover, if the depositional environment is known a priori, then geometrical shapes that characterize the petrophsyical properties of the formation may constitute an integral part of the prior information. Whenever the aforementioned information is part of the prior knowledge, the conditioning or assimilation of data should accommodate the geologic data provided a priori. This can be accomplished with the proper parameterization of the geologic properties so that different facies are honored. In \cite{Landa} for example, a channelized structure was parameterized with a small number of unknowns and a deterministic history matching (data assimilation) was conducted on a two-phase flow model. More sophisticated parameterization of geologic facies involves the level-set approach for history matching used by \cite{Dorn,Iglesias4} in a deterministic framework. Recently, in \cite{EfendievL} the level-set approach is combined with the Bayesian framework to provide a characterization of the posterior. This Bayesian application is constructed on the finite-dimensional approximation of the flow-model and is therefore subject to the computational limitations previously indicated, namely mesh-dependent convergence properties. While the work of \cite{EfendievL} provides an upscaling to provide computational feasibility, the computations reported involved a limited number of samples, potentially insufficient for the proper characterization of the Bayesian posterior.

There are also several facies estimation approaches based on
{\em ad hoc} Gaussian approximations of the posterior. For example, in \cite{Liu2005147} a pluri-Gaussian model of the geologic facies was used with an ensemble Kalman filter (EnKF) approach to generate an ensemble of realizations. In \cite{Bi} a randomized likelihood method was used to estimate the parameters characterized with channels. In \cite{EnKF_level} a level-set approach was used with EnKF to generate an ensemble of updated facies. Although the aforementioned implementations are computationally feasible and may recover the truth within credible intervals (Bayesian confidence intervals), the methods may provide uncontrolled approximations of the posterior. Even for simple Gaussian priors, in \cite{Us} numerical evidence has been provided of the poor characterization that ensemble methods may produce when compared to a fully resolved posterior.

It is also worth mentioning the recent work of \cite{CW} where the Bayesian framework was used for facies identification in petroleum reservoirs. This work considers a parametrization of the geologic facies in terms of piecewise constant permeabilities on a multiphase-flow model. The paper demonstrates the need to
properly sample multi-modal posterior distributions for which the
standard ensemble-based methods will perform poorly.

\subsection{Contribution of This Work}

We develop a mathematical and numerical framework for Bayesian inversion to identify geometric and physical parameters of the permeability in a steady Darcy flow model. The geometric parameters aim at characterizing the location and shape of regions where discontinuities in the permeability arise due to the presence of different geologic facies. The physical parameters represent the spatial (usually continuous) variability of the permeability within each of the geologic facies that constitute the formation. We make three primary contributions: (i) we demonstrate the
existence of a well-posed posterior distribution on the geometric
and physical parameters including both piecewise constant
(finite dimensional) and piecewise continuous (infinite dimensional)
representations of the physical parameters; (ii) we describe
appropriate MCMC methods
which respect both the geometry and the possibly infinite dimensional
nature of the physical parameters and which require only forward flow solvers and not
the adjoint; (iii) we exhibit numerical studies
of the resulting posterior distributions.

Clearly piecewise continuous fields will be able to represent more
detailed features within the subsurface than piecewise constant ones.
On the other hand we expect that piecewise constant fields will lead to
simpler Bayesian inference, and in particular to speed-up of the
Markov chains. There is hence a trade-off between accuracy and
efficiency within geometric models of this type. The decision about
which model to use will depend on the details of the problem at hand
and in particular the quantities of interest under the posterior
distribution. For this reason we study both piecewise continuous
and piecewise constant fields.

Continuity of the forward mapping from unknown parameters to data, which
is a key ingredient in establishing the existence of the posterior
distribution \cite{AMS10},
is not straightforward within classic elliptic theories based on $L^{\infty}$
permeabilities, because small changes in the geometry do not induce small
changes in $L^{\infty}$. Nonetheless, one can prove that the forward mapping is
continuous with respect to the unknown parameters, which allows us to show
that the posterior distribution is well-defined. Furthermore, well-posedness
of the inverse problem is established, namely continuity of the posterior
distribution in the Hellinger and total variation metrics, with respect to small
changes in the data. In the piecewise constant case, for log-normal and uniform priors on the values of the
permeability, the data to posterior mapping is Lipschitz in these metrics, whilst
for exponential priors it is H\"{o}lder with an exponent less than $1/2$
(resp. $1$) in the Hellinger (resp. total variation) metrics; problems
for which the dependence is  H\"{o}lder but not Lipschitz have not been
seen in the examples considered to date, such as those in \cite{AMS10},
and hence this dependence is interesting in its own right and may also
have implications for the rate of convergence of numerical approximations. In the case of log-normal permeability field, the posterior is Lipschitz continuous with respect to data in both of these metrics. A novel
Metropolis-within-Gibbs method is introduced in which proposals are made which are
prior-reversible, leading to an accept-reject determined purely by the likelihood
(or model-data mismatch) hence having clear physical interpretation, and
in which the geometric and physical parameters are updated alternately,
within the Gibbs framework.
Finally some numerical examples are presented, for both multiple layers and fault
models, demonstrating the feasibility of the methodology.
We develop a rigorous application of the Bayesian framework for the estimation of geologic facies parametrized with small number of parameters, together
with a finite or infinite dimensional set of physical parameters within
each of the facies. Regarding the geometry
we consider a general class of problems that includes stratified reservoirs with a potential fault, similar to the model used in \cite{CW}. In addition, we consider a simple channelized structure parameterized with small number of parameters, similar to the one described in \cite{Landa}.

The content of this paper is organized as follows. In \Sref{sec:forward}, we provide a simplified description of the the forward model, with piecewise continuous permeabilities, and prove the continuity of the forward and observation map
with respect to the geometric and physical parameters. \Sref{sec:ips} is devoted to the resulting Bayesian inverse problem from the geometric and physical parameters. The prior model is built both for the geometry and the values of permeability. We then show that the posterior measure is well defined and prove
well-posedness results with respect to the data under this prior modeling. In \Sref{se:MCMC}, we introduce the novel Metropolis-within-Gibbs MCMC method to probe the posterior distribution. Some numerical results are shown in \Sref{se:Numer} to illustrate the effectiveness of the proposed methods.

%%%%%%%%%%%%%%%%%%%%%%%%%%%%%%%%%%%%%%%%%%%%%%%%%%%%%%%%%%%%%%%%%%%%%%%%

\section{Forward Model}
\label{sec:forward}

In this section we introduce the subsurface flow model that we employ
for application of the Bayesian framework. In subsection \ref{ssec:fp}
we describe the underlying Darcy flow PDE model, in subsection
\ref{ssec:geo} we introduce the family of geometrically defined
permeabilities employed within the PDE and in subsection \ref{ssec:pd}
we describe the observation model and the permeability to data map.

\subsection{Darcy Flow Model}
\label{ssec:fp}
We are interested in characterizing the geologic properties of an aquifer whose physical domain is denoted by $D$. We assume that $D$ is a bounded open subset of
$\Real^2$ with Lipschitz boundary $\partial D$. We define the Hilbert spaces
$H:=(L^2(D),\la\cdot,\cdot\ra,\|\cdot\|)$, $V:=(H^1_0(D),\la\nabla\cdot,\nabla\cdot\ra,\|\nabla\cdot\|)$, where $H^1_0(D)$ is the usual Sobolev space with zero trace.
Let $V^*$ be as the dual space of $V$. We denote by $X$ the
subset of strictly positive $L^{\infty}$ functions on $D$ that
$X:=\{L^{\infty}(D;\Real)|\rm{ess}\inf_{x \in D} f(x)>0\}$.
We consider steady-state single-phase Darcy-flow described by,
\begin{eqnarray}
    -\nabla\cdot(\kappa\nabla p)& =f,\quad x \in D,\nonumber\\
    \phantom{-\nabla\cdot(\kappa(\nabla}p & =0,\quad x \in \partial D,
\label{eq:elliptic}
\end{eqnarray}
where $p$ denotes the hydraulic head and $\kappa$ the permeability
(proportional to hydraulic conductivity) tensor. For simplicity, the
permeability tensor is assumed to be isotropic and hence represented
as a scalar field. The right hand side $f$ accounts for groundwater recharge. For simplicity we consider Dirichlet boundary conditions where the hydraulic head is prescribed.

The forward Darcy flow problem is, given $\kappa\in X$, to find a weak solution $p\in V$ of (\ref{eq:elliptic}) for any $f\in V^*$. This forward problem
is well-posed by the Lax-Milgram Lemma: if $\kappa_{min}=\rm{ess}\inf_{x \in D} \kappa(x)>0$ then there exists a unique weak solution $p\in V$ satisfying
\begin{equation}
\|p\|_V \leq \|f\|_{V^*}/\kappa_{min},
\label{eq:lax-mil}
\end{equation}
which enables us to define a forward map $G:X \rightarrow V,$ by
\begin{equation}
\label{eq:G}
G(\kappa)=p.
\end{equation}
We concentrate on cases where $\kappa(x)$ is a piecewise function defined
by a geometrical parameterization designed to represent layers and faults or channels. We now describe how we do this.

\subsection{Permeability Model}
\label{ssec:geo}

We are interested in permeability functions $\kappa(x)$ which are either piecewise constant or
piecewise continuous function on different subdomains $D_i$ of $D$, each of which represents
a layer or a facies. Thus we write
\begin{equation}
\kappa(x)=\sum_{i=1}^n \kappa_i \chi_{D_i}(x),
\label{eq:k_piecewise}
\end{equation}
where $\{D_i\}_{i=1}^n$ are open subsets of $D$, moreover $D_i\cap D_j=\varnothing, \forall i\neq j$ and $\cup_{i=1}^n\overline{D_i}=\overline{D}$. Choices of the $D_i$
will be specified for two different geometric models in what follows and we use these models throughout
the paper for both our analysis and our numerical experiments. They are illustrated in Figure \ref{Figure1}(a) and Figure \ref{Figure1}(b).
To completely specify the models we need to parameterize
the geometry $\{D_i\}_{i=1}^n$,
and then describe the variability of the permeability within
each subdomain $D_i.$ We now describe how these are both done.

\begin{figure}[!ht]
\centering
\subfigure[Layer Model with Fault]
{
\includegraphics[width=.3\textwidth]{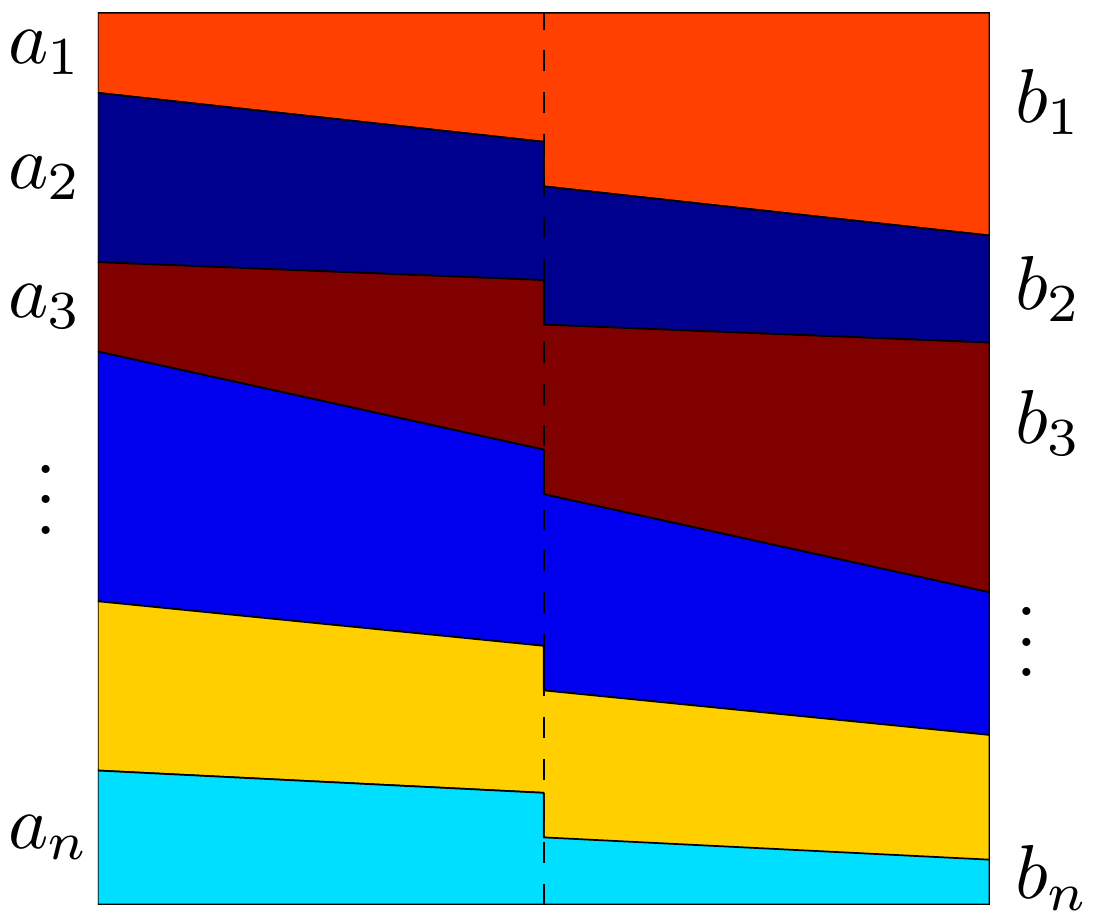}
%\caption{Fault Model}
\label{fig:test2_fault}
}
\subfigure[Channelized Model]
{
\includegraphics[width=.25\textwidth]{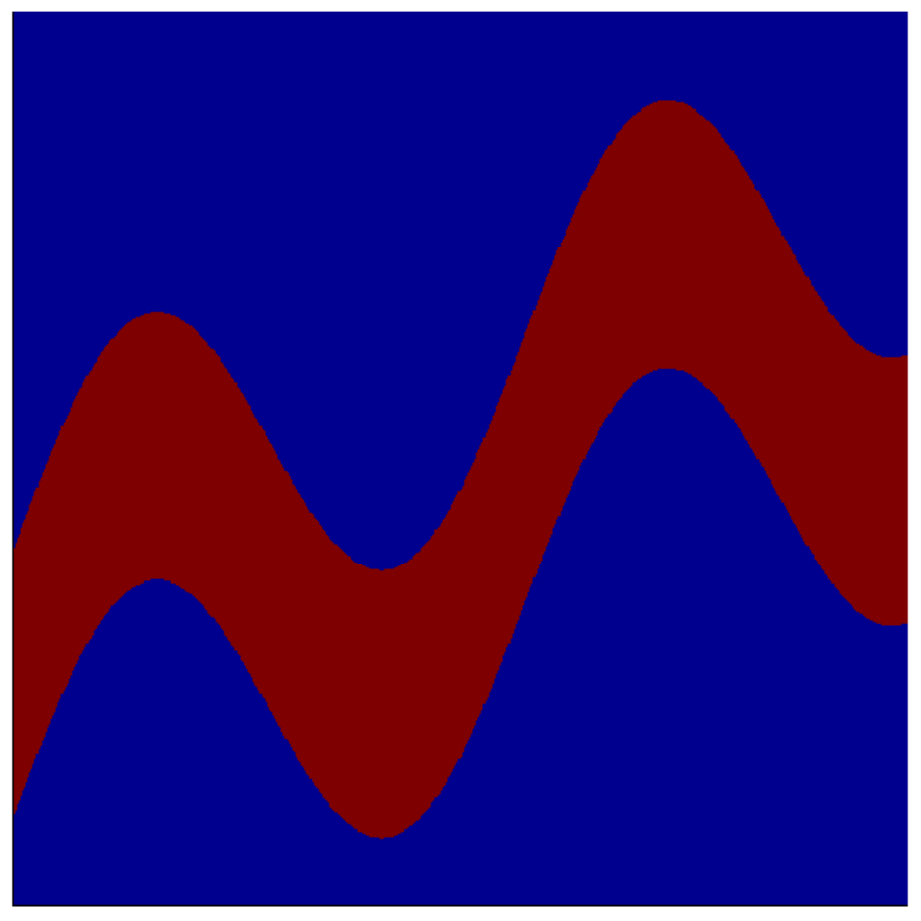}
%\caption{channelize Model}
\label{fig:test3_channel}
}
\subfigure[Layer Model]
{
\includegraphics[width=.3\textwidth]{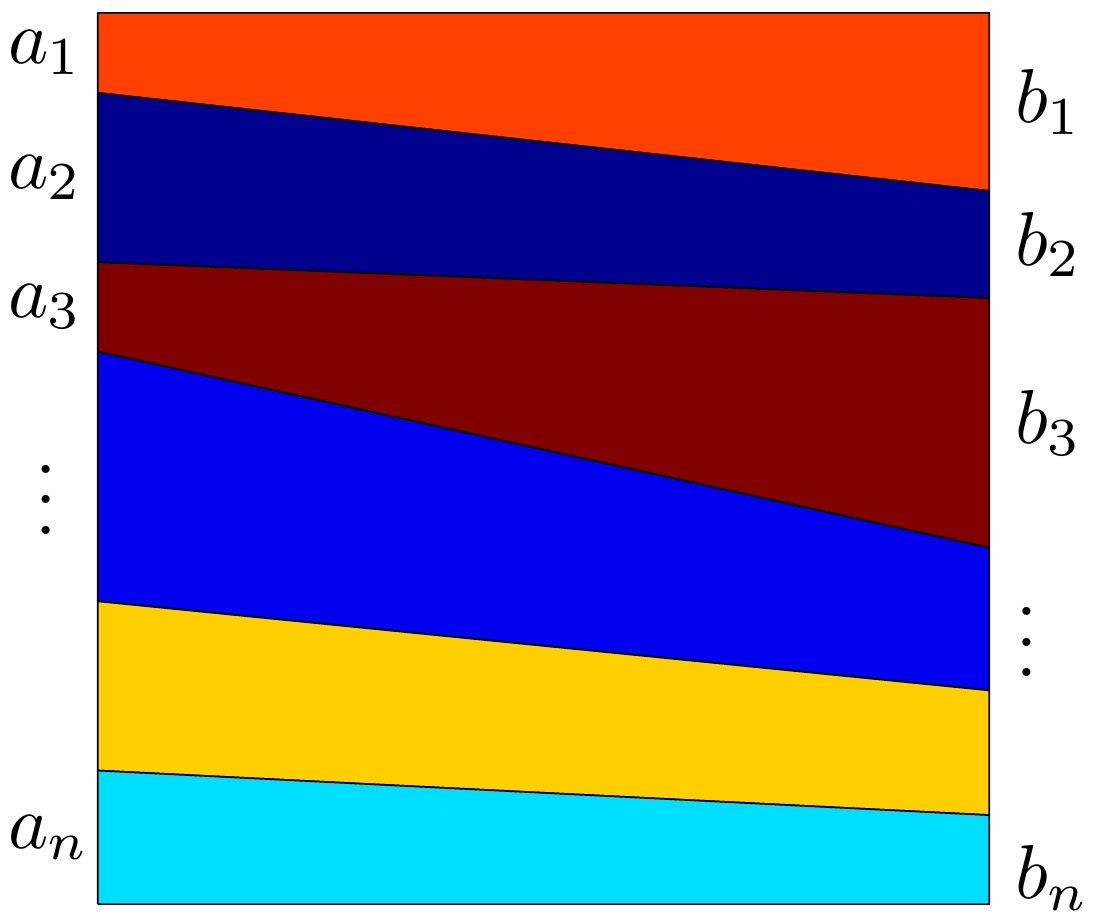}
%\caption{Layer Model}
\label{fig:test1_layer}
}
\caption{Test Models}

\label{Figure1}
\end{figure}

\subsubsection{Geometric Parameterization}

In the layer model of Figure \ref{fig:test2_fault},  we assume that the interfaces between the layers are straight lines. The slope of the interface is determined by the thickness of each layer on the two sides.
Hence, we use the layer thicknesses on the two sides to parameterize the geometry of the piecewise permeability function $\kappa(x)$. Furthermore, to describe potential faults, we introduce an additional parameter which determines the displacement of a vertical fault, whose horizontal location is specified. In the channel model, shown in Figure \ref{fig:test3_channel}, we simply assume the channel to be sinusoidal as in \cite{Landa}; the geometry can then be parameterized by the
amplitude, frequency and width together with parameters defining
the intersection of the channel with the boundaries.
All of these models are then parameterized by a finite set of real numbers
as we now detail, assuming that $D=(0,1)^2$ and letting $(x,y) \in D$ denote
the horizontal and vertical respectively.
%Figure \ref{fig:test1_layer}
\begin{itemize}

\item\textbf{Test Model $1$ (Layer Model with Fault):}
Given a fixed number $n$ of layers, the geometry in \Fref{fig:test2_fault} is determined by $\{a_i\}_{i=0}^n$ and $\{b_i\}_{i=0}^n$ and the {\em slip} $c$ that describes the (signed) height of
the vertical fault at a specified location (for simplicity fixed at $x=\frac12$). The geometry for the case of three layers is displayed in Figure \ref{Figure1A}. All the layers on the left hand side slip down$(c>0)$ or up $(c<0)$ with the same displacement $c$. Moreover because of the
constraint that the layer widths are positive and sum to one we can reduce
to finding the $2n-1$ parameters $a=(a_1,\cdots,a_{n-1})$, $b=(b_1,\cdots,b_{n-1})$,
each in
\begin{equation*}
A:=\{\textbf{x}\in \Real^{n-1}|\sum_{i=1}^{n-1}x_i\leq 1, x_i \geq 0\} \subset \Real^{n-1},
\end{equation*}
and $c\in C:=[-c^\star,c^\star]$. For this case we then define the geometric parameter $u_{g}$ and the space of geometric parameters $U_{g}$ by
\begin{equation*}
u_{g}= (a,b,c),\qquad U_{g}= A^2\times C
\end{equation*}
This geometric model thus has $2n-1$ parameters and $n$ domains $D_i$. Note that a particular case of this model is the layered model shown in Figure \ref{fig:test1_layer} where we take $c=0$ as a known parameter.

\begin{figure}[htbp]
\begin{center}
\subfigure[Two-layer configuration]{\includegraphics[scale=0.44]{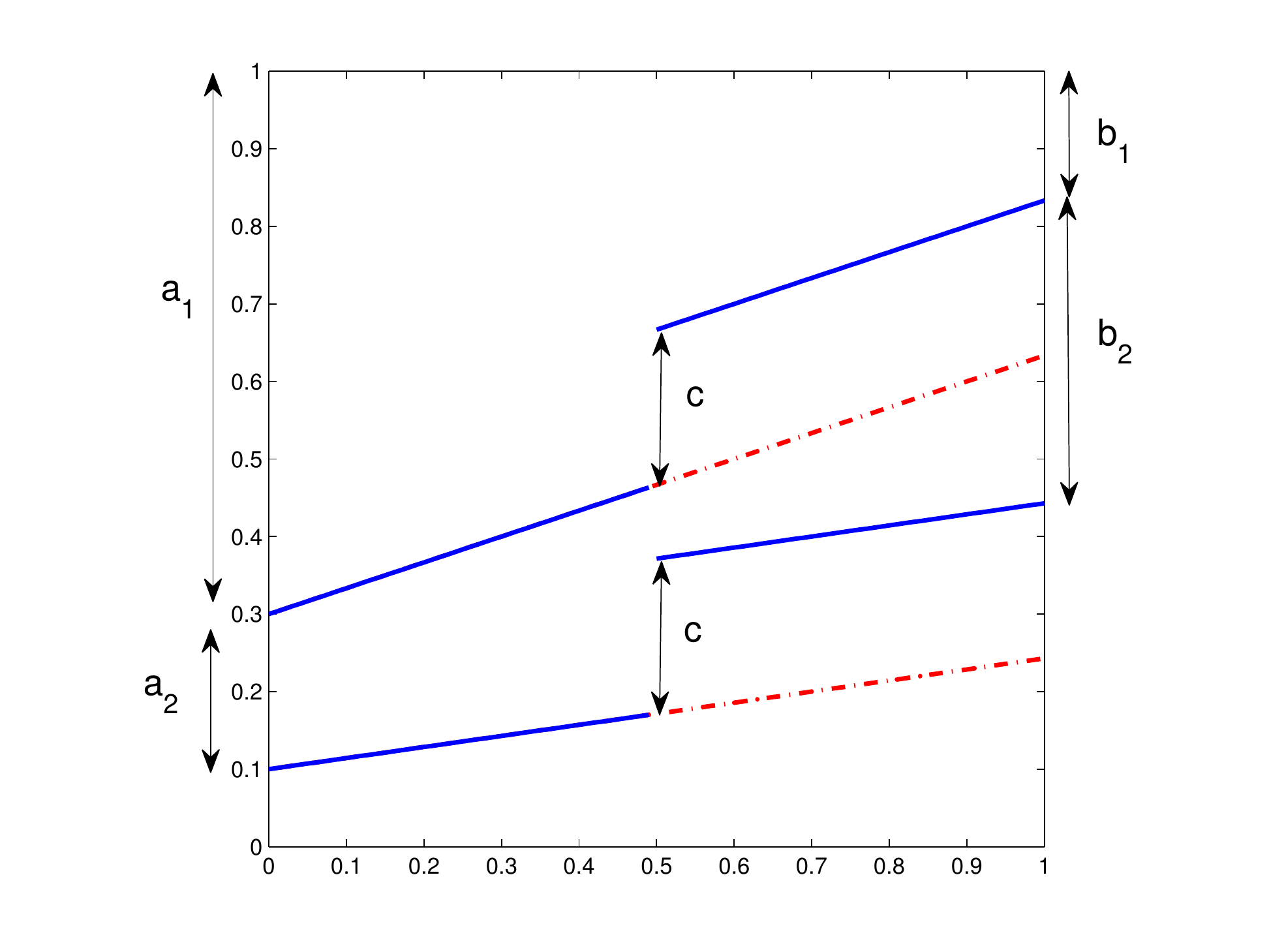} \label{Figure1A}}
\subfigure[Channel configuration]{
\includegraphics[width=6.5cm,height=6.5cm]{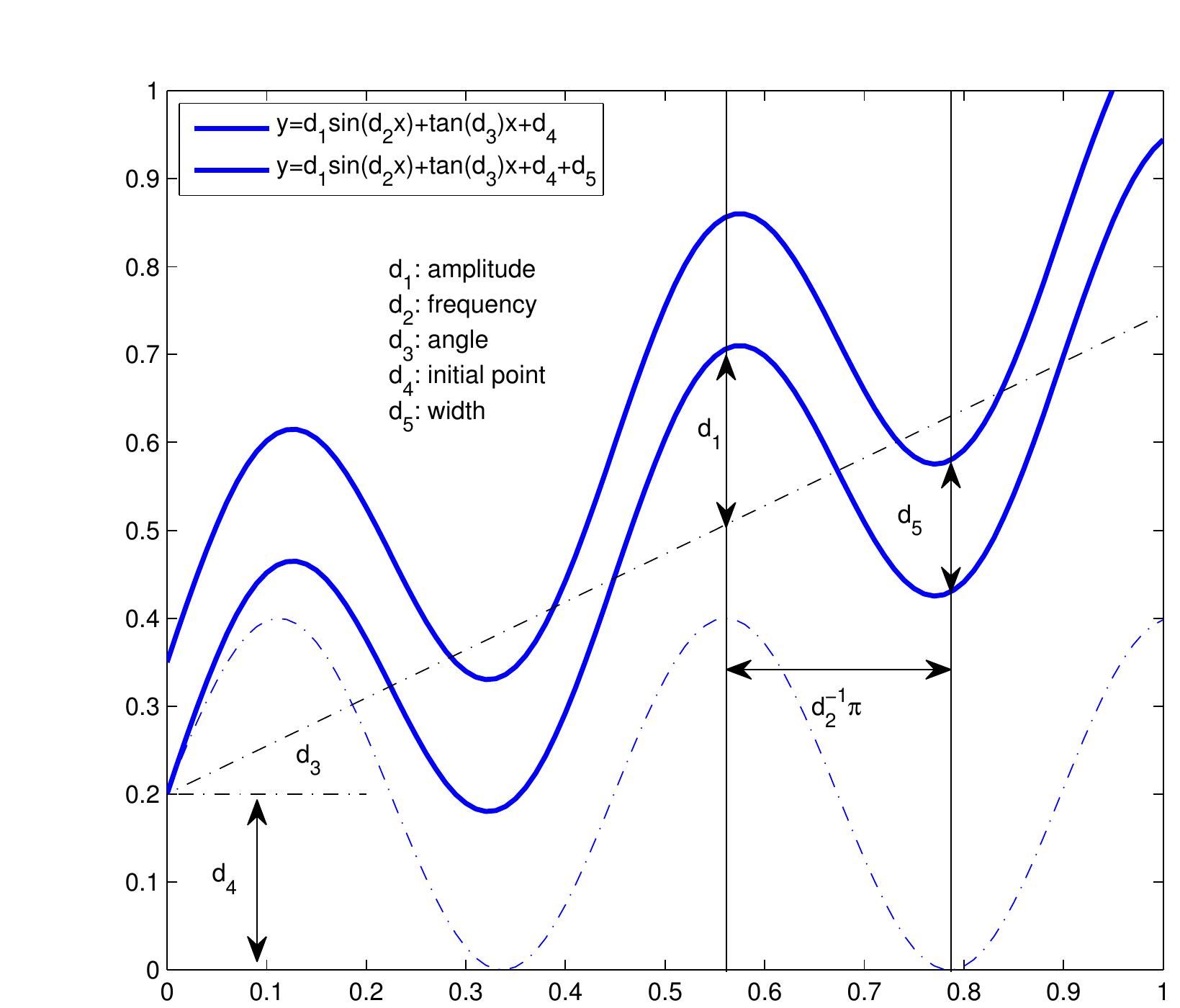}\label{Figure1B}}
\caption{Geometry of Test Models}

\label{FigureX}

\end{center}
\end{figure}

\item\textbf{Test Model $2$ (Channelized Model):}
Test model $2$ is shown in \Fref{Figure1B}. We assume that the lower
boundary of the  channel is described by the
sinusoid
$$y=d_1\sin(d_2x)+\tan(d_3)x+d_4$$
and we employ width parameter $d_5$. In this case, we characterize the geometric with
\begin{equation*}
u_{g}= (d_1,\cdots,d_5) \in \bbR^5,\qquad U_{g}= \prod_{i=1}^5[{d_i^-},d_i^+].
\end{equation*}
This geometric model thus has $5$ parameters and two
domains $D_1$ and $D_2$ denoting the interior and exterior of the channel
respectively, which gives $n=2$; note that $D_2$ contains disjoint components.

\end{itemize}

We now introduce an hypothesis that will be useful when we discuss the continuity of the forward map. Let $u_{g}^{\ve}$ represent
perturbations of $u_{g}$ and let the $D_i^{\ve}$ be the corresponding induced perturbations of the domains $D_i$. Thus $\{D_i^{\ve}\}_{i=1}^n$ is also a set of open subsets of $D$ such that $D_i^{\ve}\cap D_j^{\ve}=\varnothing$, for all  $i\neq j$ and $\cup_{i=1}^n\overline{D_i^{\ve}}=\overline{D}$.
\begin{hypothesis}
\label{hy:geo}
For all $i\neq j$, the Lebesgue measure of $D_j^{\ve}\cap D_i$, denoted as $|D_j^{\ve}\cap D_i|$ goes to zero, if $u_{g}^\ve\to u_{g}$.
\end{hypothesis}
It is clear that this hypothesis holds true for all the test models.

\subsubsection{Permeability Values}
%\textbf{Piecewise constant:}\\

Throughout the paper, we use $u_{\kappa}$ to denote the unknown parameters describing $\kappa_{i}$ in (\ref{eq:k_piecewise}) and $U_{\kappa}$ as the admissible set for the parameters. We will consider two parameterizations of the functions $\kappa_i$ appearing
in (\ref{eq:k_piecewise}): as constants and as continuous functions.
\begin{itemize}
\item Piecewise constant: Each $\kappa_i$ is a positive constant. We then have the following choices:
\begin{eqnarray*}
 u_{\kappa}=(\kappa_1,\cdots,\kappa_{n}), \qquad U_{\kappa}= (0,\infty)^n
\end{eqnarray*}
\item Continuous: We consider each $\kappa_i$ to be defined on the whole of $D$. We work with $\log \kappa$ as the exponential of this will always give a {\em positive} continuous function as required for the existence of
solution to (\ref{eq:elliptic}). This induces the following choices for $u_{\kappa}$ and $U_{\kappa}$:
\begin{eqnarray*}
u_{\kappa}=(\log \kappa_{1},\dots, \log \kappa_{n}) , \qquad U_{\kappa}= C(\bD;\bbR^n)
\end{eqnarray*}
\end{itemize}

% Then, for this case $U=U_{g}\times U_{\kappa}\subset \Real^{3n-1}$.
 % \item Channelized Model: the unknown parameters are $u_{g}=d$, $u_{\kappa}=(\kappa_1,\cdots,\kappa_{n})$, and the parameter %spaces $U_{g}=B$ and $ \subset \Real^{n+5}$. $U:=(0,\infty)^n\times B \subset \Real^{n+5}$.
%  \item Layer Model: the unknown parameter $u=(\log \kappa_{1},\dots, \log \kappa_{n}, a ,b,c)$ and the parameter space $U:=C(\bD;\bbR^n)\times A^2 \times [-c^\star,c^\star]\subset  C(\bD;\bbR^n) \times \Real^{2n-1}$;
%  \item Channelized Model: the unknown parameter $u=(\log \kappa_1,\cdots,\log \kappa_{n},d)$ and the parameter space %$U:=C(\bD;\bbR^n)\times B \subset  C(\bD;\bbR^n) \times \Real^{5}.$
%\end{itemize}

We may consider Test Models $1$ or $2$ with either constant or continuous $\kappa_i$, leading to four different models of the permeability.
In all four cases we define the unknown parameter and associated space by
\begin{eqnarray*}
 u= (u_{g},u_{\kappa}), \qquad U= U_{g}\times U_{\kappa}.
\end{eqnarray*}
For the cases of defined above, $U$ is a subset of
a separable Banach space $\bigl(\cB,\|\cdot\rb \bigr)$.

\subsection{Observation Model}
\label{ssec:pd}
Given the parameterizations described at the end of the previous section,
we define the function $F:U\rightarrow X$ as an abstract map from parameter space to
the space of the permeabilities, by
\begin{equation}
\label{eq:F}
F(u)=\kappa.
\end{equation}
Let $L$ denote a bounded linear observation operator on $V$, comprising a finite number of linear functionals $l_j\in V^*, j=1,\cdots, J$ that $L(p)=(l_1(p),\cdots,l_{J}(p))^T$. The measurements are some noisy observations from
\begin{equation*}
y_j=l_j(p)+\eta_j,
\end{equation*}
where $\eta_j$ represents the noise polluting the observation.

Let $y:=(y_1,\cdots,y_J)^T\in Y,$ where $Y:=\Real^J$, equipped
with Euclidean norm $|\cdot|$ and similarly for $\eta$; then, given
(\ref{eq:G},\ref{eq:F}), we define the observation operator
$\cG: U\rightarrow Y$ by $\cG=L\circ G\circ F$. We then have
\begin{equation}
y=\cG(u)+\eta.
\label{eq:ips}
\end{equation}
In remaining sections of the paper we study the inverse problem of using $y$ to
determine the unknown parameter $u \in U$.
A key foundational result underlying our
analysis is the the continuity of $\cG$.

\begin{theorem}
\label{t:c}
For all four of the permeability parameterizations the mapping
$\cG: U \subset \cB \to Y$ is continuous.
\end{theorem}
\begin{proof}
See the Appendix.
\end{proof}

\section{Bayesian Inverse Problems}
\label{sec:ips}
The inverse problem of interest here is to estimate the parameter $u$ from $y$ given by
(\ref{eq:ips}).
We adopt a Bayesian approach. The pair $(u,y) \in \cB \times \Real^J$ is modeled as
a random variable: we put a prior probability measure $\mu_0$ for $u$ viewed
as an element of the separable Banach space $\cB$, and define the random variable
$y|u$ by assuming that $\eta \sim N(0,\Gamma)$ independently of $u$.
The Bayesian solution to the inverse problem is then the posterior measure $\mu^y$ on
the random variable $u|y.$ We thereby define the probabilistic information about
the unknown $u$ in terms of the measurements $y$, forward model and prior information.
The paper  \cite{AMS10} and lecture notes \cite{AMS} describe a mathematical
framework for Bayes' Theorem in this infinite dimensional setting. We define
the priors and likelihood and then use this mathematical framework to establish
the existence and well-posedness of the posterior distribution; here
well-posedness refers to continuity of the posterior measure $\mu^y$, in the
Hellinger or total variation metrics, with respect to the data $y$.

\subsection{Prior Modeling}
\label{ssec:prior}
The unknown parameter $u$ is viewed as an element of the separable Banach space
$\cB$ defined for each of the four permeability models. Under the prior
we assume that the physical parameter $u_{\kappa}$ is independent of the geometric parameter $u_{g}$. Therefore we can build up a prior measure $\mu_0$ by
defining the geometric prior distribution density $\pi^G_0$ and
the permeability (or log permeability) prior measure $\mu_0^i$ respectively.

\subsubsection{Geometric Prior}

In the layer model with a fault, the geometry variables $a, b$ satisfy the geometric constraint
$a, b \in A$. In addition, we consider the case where there is no preference that the thickness of a certain layer is larger than another one. In other words, we consider the case where $a$ and $b$ are i.i.d random vector drawn from uniform distribution with a density $\pi_0^{A,g}(x)$ such that
\begin{equation}\label{pi0}
  \pi_0^{A,g}(x)=\left\{
  \begin{array}{cc}
  \frac{1}{|A|} &  x \in A,\\
    0 & x \notin A.
  \end{array}
  \right.
\end{equation}
In addition, the slip parameter $c$ is drawn uniformly from $C= [-c^\star,c^\star]$, and independently of $(a,b)$. Then, the prior density for this geometric model is
\begin{equation*}
  \pi_0^{G}(u_{g})= \pi_0^{A,g}(a)\pi_0^{A,g}(b)\pi_0^{C,g}(c)
\end{equation*}
For the Channel Model we assume the geometric parameter $u_{g}$ is drawn uniformly from the admissible set $U_{g}$. Therefore, we consider the prior
\begin{equation*}
  \pi_0^{G}(u_{g})= \Pi_{i=1}^{5}\pi_0^{[d_{i}^{-},d_{i}^{+}]}(d_{i})
\end{equation*}

\subsubsection{Permeability Prior}

We first discuss the case where the $\kappa_i$ are constant.
Under the prior we assume that they are all independent and
that each component $\kappa_i$ is drawn from a measure $\mu_0^i$
with Lebesgue density $\pi_0^i$, $i=1,\cdots,n$,  defined by one of the
following three cases:
\begin{itemize}
  \item Lognormal: $\log \kappa_i$ has the Gaussian distribution $N(m_i,\sigma_i^2)$.
  \item Uniform: $\kappa_i$ has the uniform distribution in $[\kappa^{i,-},\kappa^{i,+}]$, where $\kappa^{i,-}>0$.
  \item Exponential: $\kappa_i$ has the exponential distribution with parameter $\lambda_i$.
\end{itemize}
In the case of variable $\kappa_i$ we will assume that each $\log \kappa_i$
is independent of the others and is distributed according to a random field
prior $\mu_0^i=N(m_i,C_i)$ where the mean and covariance are chosen so that
$\mu_0^i\Bigl(C(\bD;\bbR)\Bigr)=1$; that is, so that draws from $\mu_0^i$
give rise to continuous functions almost surely.

\subsubsection{The Prior}
Combining the foregoing we obtain, in the case of piecewise constant permeabilities,
the following Lebesgue density of the prior for the Layer Model with fault:
\begin{equation}\label{prioreq1}
  \pi_0(u)=\pi_0^{A,g}(a)\pi_0^{A,g}(b)\pi_0^{C,g}(c)\prod_{i=1}^{n}\pi_0^i(\kappa_i).
\end{equation}
This may be viewed as the Lebesgue density of a probability measure $\mu_0$ on $U$;
here $U$ is finite dimensional. In the piecewise function case
we have a prior measure $\mu_0$ on the infinite dimensional space $U$
and it is given by
\begin{equation}\label{prioreq2}
\mu_0(du)=\pi_0^{A,g}(a)da \otimes \pi_0^{A,g}(b)db\otimes \pi_0^{C,g}(c)dc\otimes \prod_{i=1}^n\mu_0^i(d \alpha_i),
\end{equation}
where $\alpha_i=\log \kappa_i.$ Thus in both cases we have constructed a measure $\mu_0$ in the measure
space $\cB$ equipped with the Borel $\sigma-$algebra. Furthermore
the measure is constructed so that $\mu_0(U)=1.$
By similar arguments, we may construct the prior measure for
the Channelized Models with the same properties. We omit the details for brevity.

\subsection{Likelihood}
We assume the noise $\eta$ in (\ref{eq:ips}) is independent of $u$, and
drawn from the Gaussian distribution $N(0,\Gamma)$ on $Y$, with $\Gamma$ a self-adjoint
positive matrix. Thus $y|u \sim N(\cG(u),\Gamma)$.
We define the model-data misfit function $\Phi(u;y):U\times Y\rightarrow \Real$ by
\begin{equation}
    \Phi(u;y)=\frac{1}{2}|y-\cG(u)|^2_\Gamma,
    \label{eq:potential}
\end{equation} where $|\cdot|_\Gamma=|\Gamma^{-\frac{1}{2}}\cdot|$.
The negative log likelihood is given, up to a constant independent of
$(u,y)$, by $\Phi(u;y).$

\subsection{Posterior Distribution}

\begin{comment}
Now we recall two results from \cite{AMS} concerning well-definedness of the posterior distribution.
\begin{lemma}[\cite{AMS}]
\label{lem:meas}
  Let $(Z,C)$ be a measurable space, and assume that $f\in C(Z;\Real)$ and that $\nu(Z)=1$ for some probability measure $\nu$ on $Z$. Then $f$ is a $\nu$ measurable function.
\end{lemma}
\end{comment}
We now show that the posterior distribution is well-defined by applying
the basic theory in \cite{AMS10, AMS}.
Let $\bbQ_0$ be the Gaussian distribution $N(0,\Gamma)$. Define $\nu_0$ as a probability measure on $U \times Y$ by $$\nu_0(du,dy)=\mu_0(du)\otimes\bbQ_0(dy).$$
The following Proposition~\ref{thm:bayes} is a infinite dimensional version of Bayes Theorem, which implies the existence of a posterior distribution.
\begin{proposition}[Bayes Theorem \cite{AMS}]
\label{thm:bayes}
Assume that $\Phi:U\times Y\rightarrow$ is $\nu_0$ measurable and that, for $y~\bbQ_0$ -a.s.,
\begin{equation*}
Z=\int_U \exp(-\Phi(u;y))\mu_0(du)>0.
\end{equation*}
Then the conditional distribution of $u|y$ exists and is denoted $\mu^y$. Furthermore $\mu^y\ll\mu_0$ and for $y~\bbQ_0$ -a.s.,
\begin{equation*}
\label{eq:qqq}
\frac{d\mu^y}{d\mu_0} =\frac{1}{Z} \exp\left(-\Phi(u;y)\right).
\end{equation*}
\end{proposition}
We establish this Bayes theorem for our specific problem to show the well-definedness of the posterior distribution $\mu^y$. The key ingredient is the continuity of the forward and observation map $\cG(u)$ (and hence $\Phi(\cdot;y)$) on a full $\mu_0$ measure set $U$;
this may be used to establish the required measurability. We state two theorems,
one concerning the case of piecewise constant permeability and the other
concerning the case of variable permeability within each subdomain $D_i.$

\begin{theorem}[{\bf Piecewise Constant}]
\label{thm:existence_constant_case}
%Consider the cases where $\kappa(x)$ is a piecewise constant function.
Let $\Phi(u;y)$ be the model-data misfit function in (\ref{eq:potential}), $\mu_0$ be the prior distribution in (\ref{prioreq1}), then the posterior distribution $\mu^y$ is well-defined. Moreover, $\mu^y\ll\mu_0$ with a Radon-Nikodym derivative
\begin{equation}
\frac{d\mu^y}{d\mu_0} =\frac{1}{Z} \exp\left(-\Phi(u;y)\right),
\label{eq:post}
\end{equation}where
\begin{equation*}
Z=\int_U \exp(-\Phi(u;y))\mu_0(du)>0.
%\label{eq:Z}
\end{equation*}
\end{theorem}
\begin{proof}
Since $\cG(u)$ is continuous on $U$ by Theorem \ref{t:c}, $\Phi(u;y)$ is continuous on $U\times Y$. We also have $\nu_0(U\times Y)=\mu_0(U)\bbQ_0(Y)=1$. Since $\Phi: U \times Y \to {\mathbb R}$ is continuous it is a measurable function (Borel) between the spaces $U \times Y$ and ${\mathbb R}$
equipped with their respective Borel $\sigma-$algebras. Note that, by using (\ref{eq:lax-mil}), we have
\begin{eqnarray*}
|\Phi(u;y)|&= \frac{1}{2}|y-\cG(u)|_{\Gamma}^2 \leq |y|_{\Gamma}^2+|\cG(u)|_{\Gamma}^2\leq |y|_{\Gamma}^2+\left(\frac{C}{\min\{\kappa_i\}}\right)^2.
\end{eqnarray*}
Thus, for sufficiently small \footnote{This is needed only in the case of
uniform priors; for exponential and log-normal any positive $\ve$ will do.} positive $\ve>0$,
\begin{eqnarray}
\label{eq:Zpositive}
\fl Z =\int_U\exp(-\Phi(u;y))\mu_0(du) \geq \int_{\cap_{i}\{\kappa_i>\ve\}}\exp\left(-|y|_{\Gamma}^2-\left(\frac{C}{\ve}\right)^2\right)\mu_0(du)\nonumber\\
\fl                                    =\exp\left(-|y|_{\Gamma}^2-\left(\frac{C}{\ve}\right)^2\right)\mu_0(\cap_{i}\{\kappa_i>\ve\})
                                   =\exp\left(-|y|_{\Gamma}^2-\left(\frac{C}{\ve}\right)^2\right)\Pi_{i=1}^n\mu_0^i(\{\kappa_i>\ve\})
                                  >0.\nonumber
\end{eqnarray}
Therefore, by Proposition~\ref{thm:bayes} we obtain the desired result.
\end{proof}

\begin{remark}
In the piecewise constant case $u$ is in a finite dimensional space. Then
the posterior has a density $\pi^y$ with respect to Lebesgue measure
and we can rewrite (\ref{eq:post}) as

  \begin{equation}
    \label{eq:post_density}
    \pi^y(u)=\frac{1}{Z}\exp(-\Phi(u;y))\pi_0(u),
  \end{equation} which is exactly the usual Bayes rule.
\end{remark}
\begin{theorem}[{\bf Piecewise Continuous}]\label{PC}
%Assume that $\kappa(x)$ is a piecewise continuous function.
Let $\Phi(u;y)$ be the model-data misfit function in (\ref{eq:potential})
and  $\mu_0$ be the prior measure in (\ref{prioreq2}),
corresponding to Gaussian $\log\kappa_i$ in each domain.
Then the posterior distribution $\mu^y$ is well-defined. Moreover, $\mu^y\ll\mu_0$ with a Radon-Nikodym derivative
\begin{equation}
\label{eq:post_function}
\frac{d\mu^y}{d\mu_0} =\frac{1}{Z} \exp\left(-\Phi(u;y)\right),
\end{equation}where
\begin{equation*}
Z=\int_U \exp(-\Phi(u;y))\mu_0(du)>0.
\end{equation*}
\end{theorem}
\begin{proof}
Recall that the measure given by
(\ref{prioreq2}) is constructed so that $\mu_0(U)=1$. By the same argument as in the proof of Theorem \ref{thm:existence_constant_case},
we deduce that $\Phi(u;y)$ is $\nu_0$ measurable. Again using (\ref{eq:lax-mil}), we have
\begin{equation*}
|\Phi(u;y)|
  \leq |y|_{\Gamma}^2+\left(\frac{C}{\kappa_{\min}}\right)^2,
\end{equation*}
where $\kappa_{\min}=\min\limits_{i,x}\kappa_i(x)$. Since $\kappa_{\min}\ge \min\limits_i \exp(-\|\alpha_i(x)\|_{\infty})$, we have
\begin{equation}
\frac{1}{\kappa_{\min}} \leq \max\limits_i \exp(\|\alpha_i(x)\|_{\infty})
\label{eq:upper}
\end{equation}
Thus we have, using that Gaussian measure on a separable Banach space
charges all balls with positive measure,
\begin{eqnarray}
\label{eq:Zpositive_fun}
\fl Z  =\int_U\exp(-\Phi(u;y))\mu_0(du) \geq \int_U\exp\Bigl(-|y|_{\Gamma}^2-\left(\frac{C}{\kappa_{\min}}\right)^2\Bigr)\mu_0(du)\nonumber\\
\fl    \geq \int_{\cap_{i}\{\|\alpha_i\|_{L^{\infty}}\leq 1\}}\exp\left(-|y|_{\Gamma}^2-C^2 \max\limits_i \exp(2\|\alpha_i(x)\|_{\infty})\right)\mu_0(du)\nonumber\\
\fl    \geq \int_{\cap_{i}\{\|\alpha_i\|_{L^{\infty}}\leq 1\}}\exp\left(-|y|_{\Gamma}^2-C^2\exp(2)\right)\mu_0(du)
   =\exp\left(-|y|_{\Gamma}^2-C^2\exp(2)\right)\mu_0(\cap_{i}\{\|\alpha_i\|_{L^{\infty}}\leq 1\})\nonumber\\
  \fl  =\exp\left(-|y|_{\Gamma}^2-C^2\exp(2)\right)\Pi_{i=1}^n\mu_0^i(\{\|\alpha_i\|_{L^{\infty}}\leq 1\})   >0.\nonumber
\end{eqnarray}
Thus we have the desired result.
\end{proof}

\subsection{Well-Posedness}
Now we study the continuity property of the posterior measure $\mu^y$ with respect to $y$. We recall definitions of the total variation metric $d_{TV}$ and Hellinger metric $d_{Hell}$ on measures, and then study Lipschitz and H\"{o}lder continuity of the
posterior measure $\mu^y$, with respect to the data $y$, in these metrics.

Let $\mu$ and $\mu'$ be two measures, and choose a common reference
measure with respect to which both are absolutely continuous (the average
of the two measures for example). Then the Hellinger distance is defined by
  \begin{equation*}
  d_{{Hell}}(\mu,\mu')=\left(\frac{1}{2} \int_U\left(\sqrt{\frac{d\mu}{d\nu}}-\sqrt{\frac{d\mu'}{d\nu}}\right)^2d\nu\right)^{\frac{1}{2}}
\end{equation*}
and the Total variation distance $d_{TV}$ is defined by
  \begin{equation*}
  d_{{TV}}(\mu,\mu')=\frac{1}{2} \int_U\left|\frac{d\mu}{d\nu}-\frac{d\mu'}{d\nu}\right|d\nu.
\end{equation*}
Furthermore, the Hellinger and total variation distance are related as follows:
\begin{equation}
  \label{eq:tv_hell}
  \frac{1}{\sqrt{2}}d_{TV}(\mu,\mu')\leq d_{{Hell}}(\mu,\mu') \leq d_{{TV}}(\mu,\mu')^{\frac{1}{2}}.
\end{equation}
The Hellinger metric is stronger than total variation as, for square integrable
functions, the Hellinger metric defines the right order of magnitude of
perturbations to expectations caused by perturbation of the measure; the
total variation metric does this only for bounded functions.
See, for example, \cite{AMS10}, section 6.7.

The nature of the continuity result that we can prove depends, in general,
on the metric used and on the assumptions made about the prior model that
we use for the permeability $\kappa$. This is illustrated for piecewise
constant priors in the following theorem.

\begin{theorem}[{\bf Piecewise Constant}]
\label{thm:wellposed}
Assume that $\kappa(x)$ is a piecewise constant function correspond to parameter $u \in
U \subset \cB$.
Let $\mu_0$ be a prior distribution on $u$ satisfying $\mu_0(U)=1$. The resulting posterior distribution $\mu(\textrm{resp.}~ \mu')\ll\mu_0$ with Radon-Nikodym derivative is given by (\ref{eq:post}), for each $y(\textrm{resp.}~ y')\in Y$.
Then, for each of the test models we have:
\begin{enumerate}
\item for lognormal and uniform priors on the permeabilities, then for every fixed $r>0$, there is a $C_i=C_i(r), i=1,2$, such that, for all $y, y'$ with $\max\{|y|_{\Gamma},|y'|_{\Gamma}\}<r$,
\begin{eqnarray*}
  d_{{TV}}(\mu,\mu') \leq C_1 \left| y - y' \right|_{\Gamma},\\
  d_{{Hell}}(\mu,\mu') \leq C_2 \left| y - y' \right|_{\Gamma};
\end{eqnarray*}
\item for exponential priors on the permeabilities, then for every fixed $r>0$ and $\iota \in (0,1)$ there is $K_i=K_i(r,\iota)$ such that for all $y,y'$,with $\max\{|y|_{\Gamma},|y'|_{\Gamma}\}<r$,
\begin{eqnarray}
 \label{eq:tv} d_{{TV}}(\mu,\mu') \leq K_1 \left| y - y' \right|_{\Gamma}^{\iota},\\
 \label{eq:hell} d_{{Hell}}(\mu,\mu') \leq K_2 \left| y - y' \right|_{\Gamma}^{\frac{\iota}{2}}.
\end{eqnarray}
\end{enumerate}
\end{theorem}

\begin{remark}
For (ii), we only need to prove (\ref{eq:tv}), since (\ref{eq:hell})
then follows from (\ref{eq:tv_hell}). We can also deduce (\ref{eq:hell}) directly by a similar argument as the proof of (\ref{eq:tv}), but this does not improve the
H\"older exponent. Mathematically it would be possible to use uniform priors on the
permeabilities whose support extends to include the origin; in
terms of well-posedness this would result in a degradation
of the H\"older exponent as we see in Theorem \ref{thm:wellposed}(ii) for exponential
priors.
\end{remark}
\begin{proof} First note that
\begin{eqnarray}
\fl   |\Phi(u;y)-\Phi(u,y')| =   \left|\frac{1}{2}|y-\cG(u)|_{\Gamma}^2-\frac{1}{2}|y'-\cG(u)|_{\Gamma}^2\right| =  \frac{1}{2} \left| \langle y-y' , y+y'-2\mathcal{G}(u) \rangle_{\Gamma} \right|\nonumber\\
\label{eq:diff_phi}
                        \leq  \left(r+|\cG(u)|_{\Gamma}\right)|y-y'|_{\Gamma}.
    \end{eqnarray}
We denote the Lipschitz constant $$M(r,u):=r+|\cG(u)|_{\Gamma}.$$
Throughout the rest of this proof, the constant $C$ may depend on $r$ and changes from occurrence to occurrence. Let $Z$ and $Z'$ denote the normalization constant for $\mu^y$ and $\mu^{y'}$, so that
\begin{eqnarray*}
Z =\int_U\exp(-\Phi(u;y))\mu_0(du)\qquad \textrm{and} \qquad
Z'  =\int_U\exp(-\Phi(u;y'))\mu_0(du).
\end{eqnarray*}
By (\ref{eq:lax-mil}) $Z$ has a positive lower bound as we now show:
\begin{eqnarray}
\fl Z  \geq \int_{\cap_i\{\kappa_i>\ve\}}\exp\left(-|y|_{\Gamma}^2-\left(\frac{C}{\ve}\right)^2\right)\mu_0(du)\geq \exp\left(-r^2-\left(\frac{C}{\ve}\right)^2\right)\mu_0(\cap_{i}\{\kappa_i>\ve\}) >0,\nonumber
\end{eqnarray}
when $|y|_{\Gamma}<r$ and $\ve>0$, which should be sufficiently small in the case of uniform prior. We have an analogous lower bound for $Z'$.

The proof of (i) is an application of Theorem 4.3 in \cite{AMS}:
if the prior measure for $\kappa$ is uniform distribution with $\kappa_{min}>0$, or
a lognormal distribution, then $M^2(r,u)\in L^1_{\mu_0}$, which leads to Lipschitz
continuity.

Proof of (ii): We just need to prove (\ref{eq:tv}). If we substitute $\mu$ and $\mu'$ into $d_{TV}(\mu,\mu')$, we have
\begin{eqnarray*}
 d_{{TV}}(\mu,\mu')\leq I_1+I_2.
\end{eqnarray*}
where
\begin{eqnarray*}
  & I_1=\frac{1}{2Z}\int_U \left|\exp\left(-\Phi(u,y)\right)-\exp\left(-\Phi(u,y')\right)\right|d\mu_0\\
  & I_2=\frac{1}{2}|Z^{-1}-(Z')^{-1}|\int_U\exp\left(-\Phi(u,y')\right)d\mu_0.
\end{eqnarray*}
Note that
\begin{eqnarray*}
\fl |Z^{-1}-(Z')^{-1}| &= \frac{|Z-Z'|}{ZZ'}\leq  \frac{1}{ZZ'} \int_U \left|\exp\left(-\Phi(u,y)\right)-\exp\left(-\Phi(u,y')\right)\right|d\mu_0 = \frac{2}{Z'}I_1.
\end{eqnarray*}
Thus, by using the  positive bound on $Z'$ from below, we have
\begin{equation*}
I_2\leq C I_1.
\end{equation*}
 Therefore, we just need to estimate $I_1$. For any $\iota\in(0,1)$,
\begin{eqnarray*}
\fl       2ZI_1 = \int_U \left|\exp\left(-\Phi(u,y)\right)-\exp\left(-\Phi(u,y')\right)\right|\mu_0(du)\nonumber\\[0.7em]
\fl        = \int_{\{|\Phi(u,y)-\Phi(u,y')| \leq 1\}}|\exp\left(-\Phi(u,y)\right)-\exp\left(-\Phi(u,y')\right)|\mu_0(du)\nonumber\\[0.7em]
\fl + \int_{\{|\Phi(u,y)-\Phi(u,y')| > 1\}}|\exp\left(-\Phi(u,y)\right)-\exp\left(-\Phi(u,y')\right)|\mu_0(du)\nonumber\\[0.7em]
\fl        \leq \int_{\{|\Phi(u,y)-\Phi(u,y')| \leq 1\}}|\Phi(u,y)-\Phi(u,y')|^{\iota}\mu_0(du)+2\mu_0\left(\{|\Phi(u,y)-\Phi(u,y')| > 1\}\right).\nonumber
\label{eq:constantZ1}
\end{eqnarray*}
By (\ref{eq:phi_integral}) and Lemma \ref{lem:markov}, we have
\begin{eqnarray*}
\fl 2ZI_1  \leq \int_U M^{\iota}(r,u)d\mu_0~|y-y'|_{\Gamma}^{\iota}+2\int_U M^{\iota}(r,u)d\mu_0 ~|y-y'|_{\Gamma}^{\iota}= 3\int_U M^{\iota}(r,u)d\mu_0 ~|y-y'|_{\Gamma}^{\iota}.
\end{eqnarray*}
By using the  positive lower bound on $Z$, for any $\iota \in (0,1)$, we have
$$I_1\leq C_1|y-y'|_{\Gamma}^{\iota}$$
Thus
\begin{eqnarray*}
d_{{TV}}(\mu,\mu') \leq I_1 + I_2 \leqslant C_1 \left| y - y' \right|_{\Gamma}^{\iota} + C_2 \left| y - y' \right|_{\Gamma}^{\iota} \leqslant K_1(r) \left| y - y' \right|_{\Gamma}^{\iota}
\end{eqnarray*}
therefore, we obtain
\begin{equation*}
d_{{TV}}(\mu,\mu') \leq K_1 \left| y - y' \right|_{\Gamma}^{\iota}, \quad 0 < \iota < 1,
\end{equation*}
which completes the proof.
\end{proof}

As for the piecewise function case, we have a similar well-posedness result, which shows Lipschitz continuity with respect to data.
\begin{theorem}[{\bf Piecewise Continuous}]\label{thm:wellposed_function}
Assume that $\kappa(x)$ is a piecewise continuous function corresponding to parameter
$u \in U \subset \cB$ and that the prior $\mu_0$ is Gaussian on the variable $\log\kappa_i$ and satisfies $\mu_0(U)=1.$ The
resulting posterior distribution is $\mu(resp.~ \mu')\ll\mu_0$ with Radon-Nikodym derivative given by (\ref{eq:post_function}), for each $y(resp. ~y')\in Y$.
Then, for each of the test models we have that for every fixed $r>0$, there is a $C_i=C_i(r), i=1,2$, such that, for all $y, y'$ with $\max\{|y|_{\Gamma},|y'|_{\Gamma}\}<r$,
\begin{eqnarray*}
  d_{{TV}}(\mu,\mu') \leq C_1 \left| y - y' \right|_{\Gamma},\qquad   d_{{Hell}}(\mu,\mu') \leq C_2 \left| y - y' \right|_{\Gamma}.
\end{eqnarray*}
\end{theorem}

\begin{proof}
We just need to prove the Hellinger distance case as the TV distance case then
follows from (\ref{eq:tv_hell}). The remainder of the proof is an application of
the methods used Theorem 4.3 in \cite{AMS} with $M_1=0$ and
$$M_2=r+C\max_i\exp(\|\alpha_i(x)\|_{\infty}),$$
after noting that the normalization constant $Z>0.$
    By (\ref{eq:upper}), we obtain
    \begin{eqnarray*}
     \int_U M_2^2d\mu_0& \leq&  2r^2\mu_0(U)+2C^2\int_U \max_i \exp(2\|\alpha_i(x)\|_{\infty})d\mu_0\\
                          &\leq& 2r^2\mu_0(U)+2C^2\sum_{i=1}^n\int_U \exp(2\|\alpha_i(x)\|_{\infty})d\mu_0
    \end{eqnarray*}
Since for any $\epsilon>0$, there exist a $C(\epsilon)$ that $\exp(x)\leq C(\epsilon)\exp(\epsilon x^2)$ we obtain that for each $i$, $$\int_U \exp(2\|\alpha_i(x)\|_{\infty})d\mu_0\leq C(\epsilon)\int_U \exp(\epsilon\|\alpha_i(x)\|^2_{\infty})d\mu_0.$$
From the Fernique Theorem we deduce that,
as long as $\epsilon$ is small enough, $$\int_U M^2(r,u)d\mu_0<\infty.$$
  \end{proof}

\section{MCMC Algorithm}
\label{se:MCMC}
We have demonstrated the existence and well-posedness of the posterior distribution, which is the Bayesian solution to the inverse problem (\ref{eq:ips}). We now demonstrate
numerical methods to extract information from this posterior distribution; one way to do this, which we focus on in this paper, is to generate samples distributed according to the posterior distribution. These can
be used to approximate expectations with respect to the posterior and
hence to make predictions about, and quantify uncertainty in, the
permeability.  In this section we construct a class of MCMC methods to
generate such samples, using  a Gibbs splitting to separate geometric
and physical parameters, and using Metropolis (within Gibbs) proposals
which exploit the structure of the prior. We thereby generate
samples from the posterior given by (\ref{eq:post}) or (\ref{eq:post_function}) in Theorems \ref{thm:existence_constant_case} and \ref{PC}, respectively.
The resulting accept-reject parts of the algorithm depend only
on differences in the log-likelihood at the current and proposed
states, or alternatively on differences of the misfit
function $\Phi$ given by (\ref{eq:potential}).
Similar methods have been identified as beneficial for Gaussian priors
in \cite{neal,David} for example and here we extend to the non-Gaussian prior measure on the permeability values
(uniform distribution and exponential distribution) and on the geometry.

\subsection{Prior Reversible Proposals}
\label{ssec:rev_pro}

We start by demonstrating how prior reversible proposals lead to
a simple accept-reject mechanism, depending only on differences
in the model-data misfit function, within the context of the
Metropolis-Hasting algorithm applied to Bayesian inverse problems.
When the target distribution has density $\pi$ defined on $\Real^d$
this algorithm proceeds by proposing to move from current state
$u$ to proposed state $v$, drawn from Markov kernel $q(u,v)$,
and accepting the move with probability
\begin{equation}
a(u,v) = \frac{\pi(v) q(v,u)}{\pi(u) q(u,v)} \wedge 1.
\label{MH}
\end{equation}
Crucial to this algorithm is the design of the proposal $q(u,v)$;
algorithm efficiency will increase if we use a proposal which leads
to low integrated correlation in the resulting Markov chain.

When we have piecewise constant permeabilities then the posterior
density $\pi^y$ is given by (\ref{eq:post_density})
with $\Phi$ defined in (\ref{eq:potential}). If we chose a prior-reversible proposal
density which then satisfies
\begin{equation}\label{detbal}
\pi_0(u) q(u,v) = \pi_0(v) q(v,u) \quad u, v \in \bbR^d.
\end{equation}
then the acceptance probability from (\ref{MH}) becomes
 \begin{eqnarray*}
\fl a(u,v) &= \frac{\pi^y(v) q(v,u)}{\pi^y(u) q(u,v)} \wedge 1= \frac{\pi_0(v)\exp(-\Phi(v;y))q(v,u)}{\pi_0(u)\exp(-\Phi(u;y)) q(u,v)} \wedge 1
\end{eqnarray*}
so that
\begin{equation}
a(u,v)= \exp(\Phi(u;y)-\Phi(v;y)) \wedge 1.  \label{eq:ap}
\end{equation}
Thus, if the proposed state $v$ corresponds to a lower value of the model-data
misfit function $\Phi(v;y)$ than the current state $\Phi(u;y)$, it will accept the proposal definitely, otherwise it will accept with probability $a(u,v)<1.$
Hence, the accept-reject expressions depend purely on the model-data mismatch function $\Phi(u;y)$, having clear physical interpretation.

\subsection{Prior Reversible Proposal on $\Real^d$}
\label{ssec:rev_pro}

We now construct prior reversible proposals for the finite dimensional case when the permeability is piecewise constant. In this case $u_{\kappa}$ lives in
a subset $U$ of $\bbR^n$ where $n$ is the number of parameters to be
determined. The prior distribution $\pi_0$ constrains the parameters to the
admissible set $U$,  that is $\pi_0(x)=0$ if $x\notin U$. Given current
state $v$ we then construct a prior reversible proposal kernel $q$
as follows: we let
\begin{equation}
v = \left\{ \begin{array}{ll} w & w \in U \\ u & w \notin U \end{array} \right. ,
\label{eq:gen_v}
\end{equation}
where $w$ is drawn from kernel $p(u,w)$ that satisfies
  \begin{equation}
\pi_0(u) p(u,w) = \pi_0(w) p(w,u)\quad \forall u, w \in U.
\label{eq:rev_p}
\end{equation}
Note that the previous expression is constrained to $U$, rather than $\mathbb{R}^{d}$ as in (\ref{detbal}). Note, however, that we have not assumed that $p(u,v)$ is a Markov kernel on $U \times U$; it may generate draws outside $U$.
 Therefore, in (\ref{eq:gen_v}) we just simply accept or reject the proposal $w$ based on whether $w$ is in $U$ or not.  The reversibility with respect to the prior of the proposal  $q(u,v)$ given by (\ref{eq:gen_v}) follows from:
\begin{proposition}
\label{thm:mh}
Consider the Markov kernel $q(u,v)$ on $U \times U$ defined via (\ref{eq:gen_v})-(\ref{eq:rev_p}). Then
%..." and the displayed line should be "\forall u,v \in U
%If $p(u,w)$ satisfies (\ref{eq:rev_p}) and $v$ is defined by (\ref{eq:gen_v}), then
  \begin{equation*}
\pi_0(u) q(u,v) = \pi_0(v) q(v,u)\quad \forall u, v \in U.
\end{equation*}
\end{proposition}
\begin{proof}
See Appendix.
\end{proof}
For a uniform prior $\pi_0(u)$, $p(u,w)$ may be any symmetric function that satisfies $p(u,w)=p(w,u)$. For example, we can use a Gaussian local move such that $p(u,w) \propto \exp \left( - \frac{1}{2\varepsilon^2} |w-u|^2 \right)$ or we
can also propose $w$ by a local move drawn uniformly from an $\epsilon$
ball around $u$.
An analogous discussion also applies when the geometric parameter is uniform (as before) but the permeability values are exponential. For the exponential permeability prior $\exp(\lambda)$, with support $[0,\infty)$, we may choose $w=u-\lambda\delta+\sqrt{2\delta}\xi$ with $\xi\sim \cN(0,1)$. For all these examples, it is easy to check that the proposal density $p(u,w)$  satisfies (\ref{eq:rev_p}). Therefore, by Proposition \ref{thm:mh} we obtain that the proposal (\ref{eq:gen_v}) is reversible with respect to the prior. Thus, the acceptance probability that results is given by (\ref{eq:ap}).

\subsection{Prior Reversible Proposal on $C({\overline D};\bbR^n)$}
\label{ssec:rev_pro}

For the case of continuous permeabilities, we recall that $u_{\kappa,i}=\log \kappa_{i}$ has prior Gaussian measure $\mu_0^{\kappa}(du_{\kappa})= \cN(m,C)$ with covariance and mean chosen so that it charges the space $C({\overline D};\bbR^n)$ with full measure. In this case a prior reversible proposal is given by
 \begin{equation*}
% \label{eq:pcn}
 v_{\kappa}=m+\sqrt{1-\beta^2}(u_{\kappa}-m)+\beta \xi, \quad \xi\sim\cN(0,C), ~\beta\in[0,1].
 \end{equation*}
This is the pCN-MCMC method introduced in \cite{beskos2008mcmc}
(see the overview in \cite{David}) and the acceptance
probability is again given by (\ref{eq:ap}) in this infinite dimensional
context.

\subsection{Metropolis-within-Gibbs: Separating Geometric and Physical Parameters}

In practice the geometric and physical parameters have very different
effects on the model-data misfit and efficiency can be improved
by changing them separately. Using the
independence between the geometric parameter $u_{g}$ and the
physical parameters (permeabilities) $u_{\kappa}$ this can be
obtained by employing the
following Metropolis-within-Gibbs algorithm.

\begin{algorithm}[Metropolis-within-Gibbs]\label{MwG}{~}

Initialize $u^{(0)}=(u_{\kappa}^{(0)},u_{g}^{(0)})\in U$. \\
For $k=0,\dots$
\begin{enumerate}
  \item[(1)] Propose $v_{g}$ from $q_{g}(u_{g}^{(k)},v_{g})$
	\subitem(1.1) Draw $w\in \bbR^d$ from $p_{g}(u^{(k)}_{g},w)$.
	\subitem(1.2) Let $v_{g} \in U_{g}$ defined by
  \begin{equation*}
v_{g} = \left\{ \begin{array}{ll} w & w \in U \\ u_{g}^{(k)} & w \notin U \end{array} \right. .
\end{equation*}

  \item[(2)] Accept or reject $v=(u_{\kappa}^{(k)},v_{g})$:
\begin{equation*}
(u_{\kappa}^{(k)},u_{g}^{(k+1)}) =
\left\{
\begin{array}{cc}
 v & \mathrm{with~probability}~a(u^{(k)},v)  \\
 u^{(k)} & \mathrm{otherwise}
\end{array} \right. . \nonumber
\end{equation*}
where $u^{(k)}\equiv (u_{\kappa}^{(k)},u_{g}^{(k)})$ and
with $a(u,v)$ defined by (\ref{eq:ap}).
  \item[(3)] Propose $v_{\kappa}$ from $q_{\kappa}(u_{\kappa}^{(k)},v_{\kappa})$:\\
In the constant permeabilities case, propose $v_{\kappa}$ from $q_{\kappa}(u_{\kappa}^{(k)},v_{\kappa})$ as follows
	\subitem(3.1) Draw $w\in \bbR^d$ from $p_{\kappa}(u_{\kappa}^{(k)},w)$.
	\subitem(3.2) Let $v_{\kappa} \in U_{\kappa}$ defined by
  \begin{equation*}
v_{\kappa} = \left\{ \begin{array}{ll} w & w \in U_{\kappa} \\ u_{\kappa}^{(k)} & w \notin U_{\kappa} \end{array} \right. .
\end{equation*}
In the continuous permeabilities case we propose $v_{\kappa}$ according to
\begin{equation*}
 \label{eq:pcn2}
 v_{\kappa}=m+\sqrt{1-\beta^2}(u_{\kappa}^{(k)}-m)+\beta \xi, \quad \xi\sim\cN(0,C), ~\beta\in[0,1],\\
 \end{equation*}
  \item[(4)] Accept or reject $v=(v_{\kappa},u_{g}^{(k+1)})$:
\begin{equation*}
(u_{\kappa}^{(k+1)},u_{g}^{(k+1)}) =
\left\{
\begin{array}{cc}
 v & \mathrm{with~probability}~a(\hat{u}^{(k)},v)  \\
 \hat{u}^{(k)} & \mathrm{otherwise}
\end{array} \right. . \nonumber
\end{equation*}
where $\hat{u}^{(k)}\equiv (u_{\kappa}^{(k)},u_{g}^{(k+1)})$ and
with $a(u,v)$ defined by (\ref{eq:ap}).
\end{enumerate}
\end{algorithm}

Furthermore, in the experiments which follow we sometimes find it
advantageous to split the geometric parameters into different groupings,
and apply the Metropolis-within-Gibbs idea to these separate groupings;
again the independence of the parameters under the prior allows this
to be done in a straightforward fashion, generalizing the preceding
Algorithm \ref{MwG}. In particular, by using
the independence of the geometric parameters under the prior, together
with prior-reversibility of the proposals used, it again follows that
the accept-reject criterion for each Metropolis-within-Gibbs step is
given by (\ref{eq:ap}).
In addition, we generate multiple parallel MCMC chains to sample the posterior in the subsequent experiments. For some of those chains we often find that a low probability mode is explored for a very large number of iterations. In order to accelerate the convergence of these chains, within the Metropolis step of the aforementioned algorithm, we implement proposals where the local moves described earlier are replaced, with probability 1/2, by independent samples from the prior.

\section{Numerical Experiments}
\label{se:Numer}
We present some numerical examples to demonstrate the feasibility of our methodology. Specifically, the performance is illustrated by application to
three examples derived from our Test Models of \Sref{sec:forward}. For these examples, the forward model consists of the elliptic equation on a domain $D=[0,1]\times [0,1]$ discretized with the  finite difference method on a mesh of size $50\times 50$. In addition, $J$ measurement functionals $\ell_j(p)$ are defined by
\begin{equation}\label{meas_func}
\ell_j(p)=\frac{1}{2\pi\sigma^2}\int_D\exp\left(-\frac{|x-x_j|^2}{2\sigma^2}\right)p(x)dx, \qquad j\in\{1,\cdots,J\}
\end{equation}
which can be understood as a smooth approximation of pressure evaluations at certain locations $x_j$'s in the domain $D$.
This is because the kernel under the integrand
approaches a Dirac measure at $x_j$ as $\sigma \to 0.$
Note that $\ell_j\in V^{*}$, the dual space of $V$, as required for the analysis of \Sref{sec:forward}.
In all of the
numerical experiments reported $\sigma$ takes the value $0.01.$

For the subsequent experiments we generate synthetic data by first computing $p^{\dagger}$, the solution to the elliptic equation with a ``true'' permeability $\kappa^{\dagger}(x)$ associated to the true parameters $u^{\dagger}$. Then, synthetic data is defined by $y_j=\ell_j(p^{\dagger})+\eta_j$, where $\eta_j$ are i.i.d. Gaussian noise from $\cN(0,\gamma^2 )$. Our choices of $u^{\dagger}$ and $\gamma$ are described below. In order to avoid the inverse crime, the synthetic data is computed from the true permeability defined on a domain discretized on $100\times 100$ cells while the Bayesian inversions are performed on a $50\times 50$ grid. It is important to remark that the effect of the observational noise, at the scale we introduce
it, is sufficient to induce significant inversion challenges even in this
perfect model scenario. Furthermore, if model error is to be studied, it
is perhaps more pertinent to study the effect of modelling the geometry
through a small finite set of parameters, when real interfaces and faults
will have more nuanced structures, or the effect of modelling spatially
varying fields as constants. We leave the detailed study of these and
other grid-based model errors for separate study.

\subsection{Example. A Three-layer Model With Fault}

For this experiment we consider a permeability of the form
\begin{equation}\label{eq:num1}
\kappa(x)=\kappa_1 \chi_{D_1}(x)+\kappa_2 \chi_{D_2}(x)+\kappa_3 \chi_{D_3}(x)
\end{equation}
where $\{D_i\}_{i=1}^3$ are the open subsets defined by the geometric parameters $a=(a_{1},a_{2})$, $b=(b_{1},b_{2})$ and $c$ as in Figure \ref{Figure1A}. Therefore, the unknown parameter is $u=(a_{1},a_{2},b_{1},b_2,c,\kappa_{1},\kappa_{2},\kappa_{3})\in \mathbb{R}^{8}$. We consider the true $\kappa^{\dagger}(x)$ shown in Figure \ref{Figure2} (top-left) which corresponds to (\ref{eq:num1}) for the true values $u^{\dagger}$ of \Tref{Table1}.
\begin{table}[!ht]
\centering
\begin{tabular}{|cc|rl|cc|}
  \hline
  % after \\: \hline or \cline{col1-col2} \cline{col3-col4} ...
    parameter  &  true value  & data set 1& data set 1& data set 2& data set 2\\
& &  mean &  variance &  mean& variance\\
\hline
  $a_1$ & $0.39$ &  0.386& $3.4\times 10^{-3}$&  0.394& $2\times 10^{-4}$\\
  $b_1$ & $0.18$ &0.174& $4.6\times 10^{-3}$&  0.177& $7\times 10^{-4}$\\
  $a_2$ & $0.35$ &0.486& $6.9\times 10^{-3}$&  0.400& $6\times 10^{-4}$\\
  $b_2$ & $0.6$ &0.618& $8.9\times 10^{-3}$&  0.637& $3.5\times 10^{-3}$\\
  $c$ & $0.15$   &0.239& $1.04\times 10^{-2}$&  0.192& $1.2\times 10^{-3}$\\
 $\kappa_1$ & $12$ &11.718 & 1.796&  11.337& $1.0128$\\
  $\kappa_2$& $1$  &1.228& $4.57\times 10^{-2}$& 1.177& $9.5\times 10^{-3}$\\
    $\kappa_3$  &  $5$ &5.148&$ 6.65\times 10^{-1}$& 4.55& $1.378\times 10^{-1}$\\
  \hline
\end{tabular}

 \caption{Data relevant to the experiment with the 3-layers model}
 \label{Table1}
\end{table}
We consider 16 measurement functionals defined by (\ref{meas_func}) with measurement locations distributed as shown in Figure \ref{Figure2} (top-right). Synthetic data (data set 1) is generated as described above with $\gamma=2\times 10^{-3}$.

According to subsection \ref{ssec:prior}, the prior distribution for this parameter is defined by
\begin{equation*}
  \pi_0(u)=\pi_0^{A,g}(a)\pi_0^{A,g}(b)\pi_0^{C,g}(c)\pi_0^1(\kappa_1)\pi_0^2(\kappa_2)\pi_0^3(\kappa_3)
  \label{eq:prior_layer}
\end{equation*}
where each of the $\pi_0^i(\kappa_i)$ and $\pi_0^{C,g}(c)$ is a uniform distribution on a specified interval. The uniform prior associated to the geometrical parameters $a$ and $b$ is
defined with support everywhere on the edges of the squared domain.
However, specific restricted intervals are specified for the construction
of the uniform priors corresponding to the values of the permeabilities and,
in particular, we choose intervals which do not include the origin. The reason
for selecting these restrictive priors for the values of the permeabilities
is motivated by the subsurface flow application where prior knowledge of
a range of nominal values of permeabilities for each rock-type are typically
available from geologic data. The densities $\pi_0^{A,G}(a)$ and $\pi_0^{A,G}(b)$ are defined according to expression (\ref{pi0}) where $A=\{\textbf{x}\in \Real^{}|x_1+x_2\leq 1, x_1 \geq 0, x_2 \geq 0\}$. In the top row of Figure \ref{Figure3} we display some permeabilities defined by (\ref{eq:num1}) for parameters $u$ drawn from the prior distribution defined by the previous expression.

The forward model, the prior distribution and the synthetic data described above define the posterior measure given by (\ref{eq:post}), Theorem \ref{thm:existence_constant_case}, that we sample with a variant of the MCMC method of Algorithm \ref{MwG}. In concrete, we implement the outer Gibbs loop by considering groupings of the unknown $u=(a_{1},a_{2},b_{1},b_2,c,\kappa_{1},\kappa_{2},\kappa_{3})$ as follows: (i) the high permeabilities $\kappa_{1}$ and $\kappa_{3}$, (ii) the low permeability $\kappa_{2}$, (iii) the slip $c$, (iv) the right-hand lengths $a=(a_{1},a_{2})$ and (v) the left-hand lengths $b=(b_{1},b_{2})$. This separation of the unknown results in short decorrelation times compared to the ones when more variables of the unknown are updated simultaneously within the Metropolis part of Algorithm \ref{MwG}. We consider 20 different chains started with random draws from the prior. Trace plots from the first $10^5$ steps of one of the chains is presented in Figure \ref{Figure4}, together with the running mean and standard deviations. We monitor the convergence of the multiple chains with the multivariate potential scale reduction factor (MPSRF) \cite{Gelman}. Once approximate convergence has been reached, all chains are merged and the samples from the combined chains are used for the subsequent results. In the middle row of Figure \ref{Figure3} we show permeabilities (defined by (\ref{eq:num1})) for some samples $u$ of the posterior distribution. In \Tref{Table1} we display the values of the mean $\hat{u}$ and the variance of the posterior measure characterized with a total of $1.5\times 10^7$ samples from our MCMC chains. The permeability that results from (\ref{eq:num1})  for the mean parameter $\hat{u}$ of the posterior is presented in Figure \ref{Figure2} (bottom-left).

We now repeat this experiment with a different set of synthetic data (data set 2) generated as before but with a smaller error with standard deviation $\gamma=5\times 10^{-4}$. This set of synthetic data defines a new posterior measure that we characterize with the MCMC Algorithm \ref{MwG} used for the previous experiments (using exactly the same parameters in the proposal generators). The corresponding values of $u$ are also displayed in \Tref{Table1} and the resulting permeability is displayed in Figure \ref{Figure2} (bottom-right). Some permeabilities associated to the samples of the posterior are presented in Figure \ref{Figure3} (bottom row ). For both experiments, the marginals of the posterior distribution are displayed in \Fref{Figure5} along with the marginals of the prior. The marginals of the posterior pushed forward by the forward operator are displayed in Figure \ref{Figure6}. In Figure \ref{Figure7} we display the integrated autocorrelation for each of the components of the unknown for both posterior that arise from the data set 1 (top) and the data set 2 (bottom).

From Table \ref{Table1} we observe that, for both experiments, the mean of the parameters are in very good agreement with the truth. In fact, the corresponding permeabilities from Figures \ref{Figure2} produce very similar results. However, it comes as no surprise that more accurate synthetic data (data set 2) result in a posterior density that is more peaked around the truth ( see Figure \ref{Figure5}). In other words, the posterior associated with smaller error variance quantifies less uncertainty in the unknown parameters. While both estimates provide a good approximation of the truth, the associated uncertainties are substantially different from one another. Indeed, from Figure \ref{Figure3} we observe a larger variability in the samples of the posterior that we obtain from the data set 1.

From Figure \ref{Figure7} we note that the correlation of the samples is larger for observational error with smaller covariance. Indeed, during the MCMC algorithm for sampling of the posterior, the local move in the proposal is more likely to be rejected for smaller values of $\gamma$ (recall we are using the same MCMC algorithm for both experiments).

\begin{figure}[htbp]
\begin{center}
\includegraphics[scale=0.30]{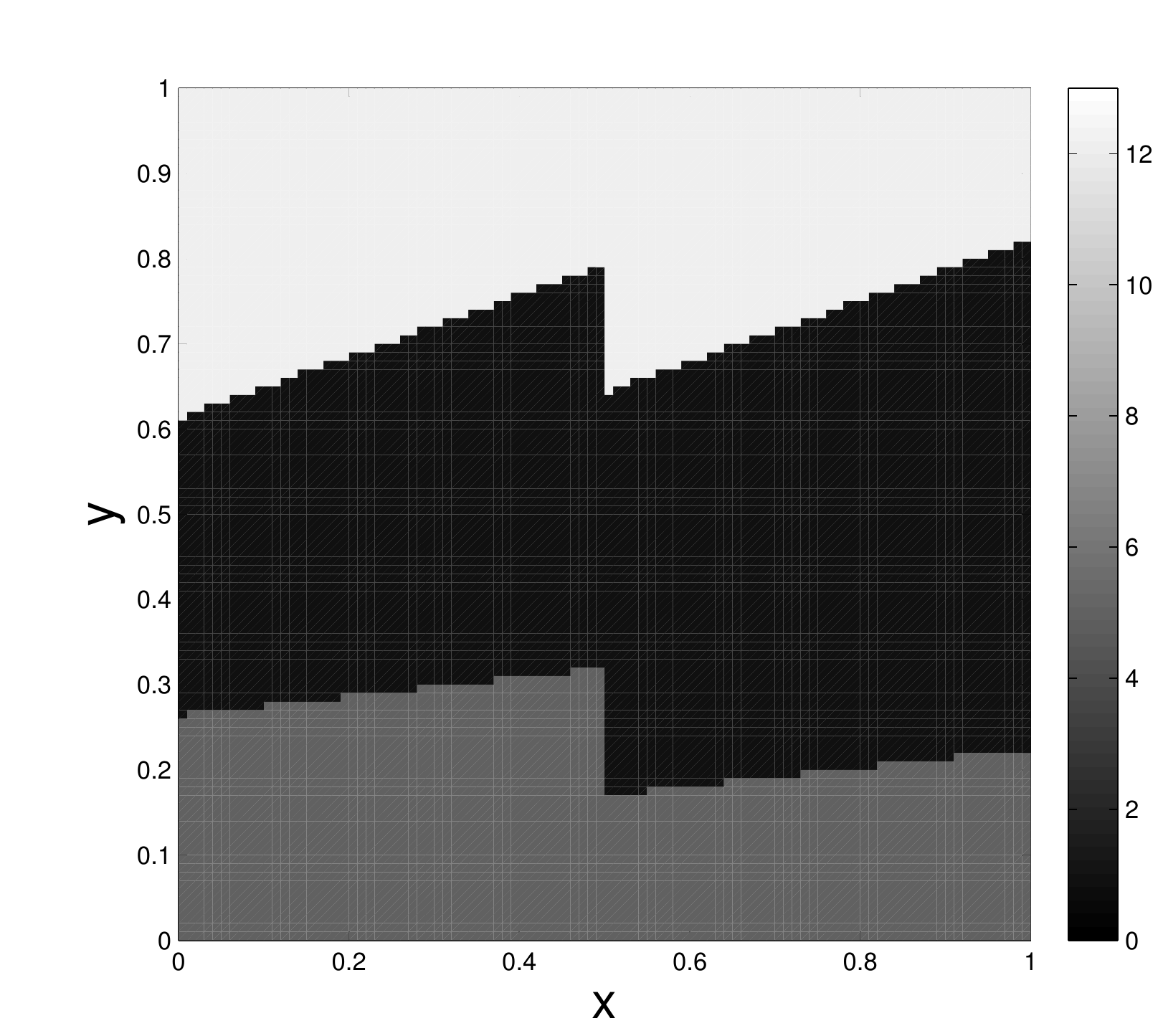}
\includegraphics[scale=0.30]{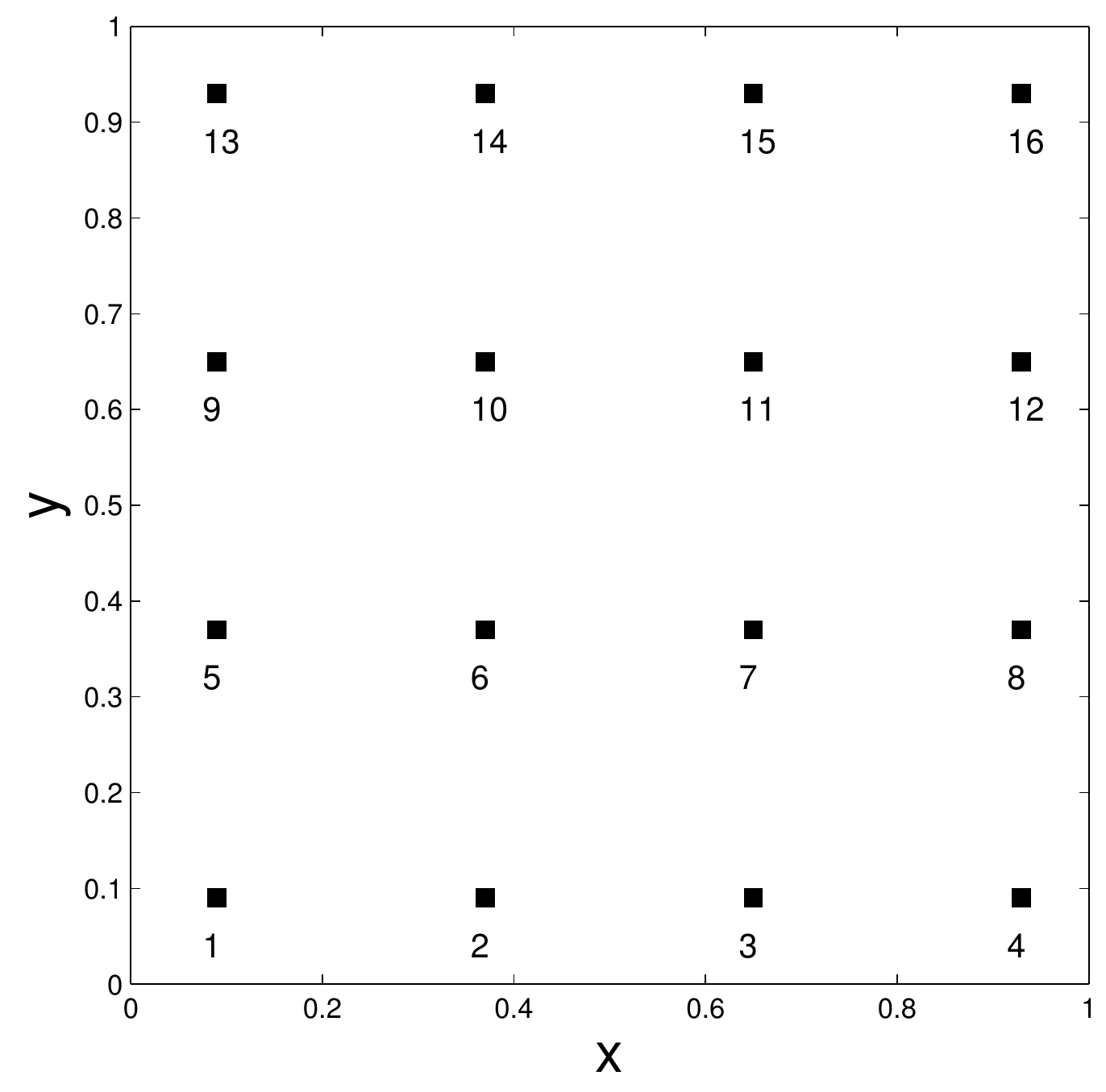}\\
\includegraphics[scale=0.30]{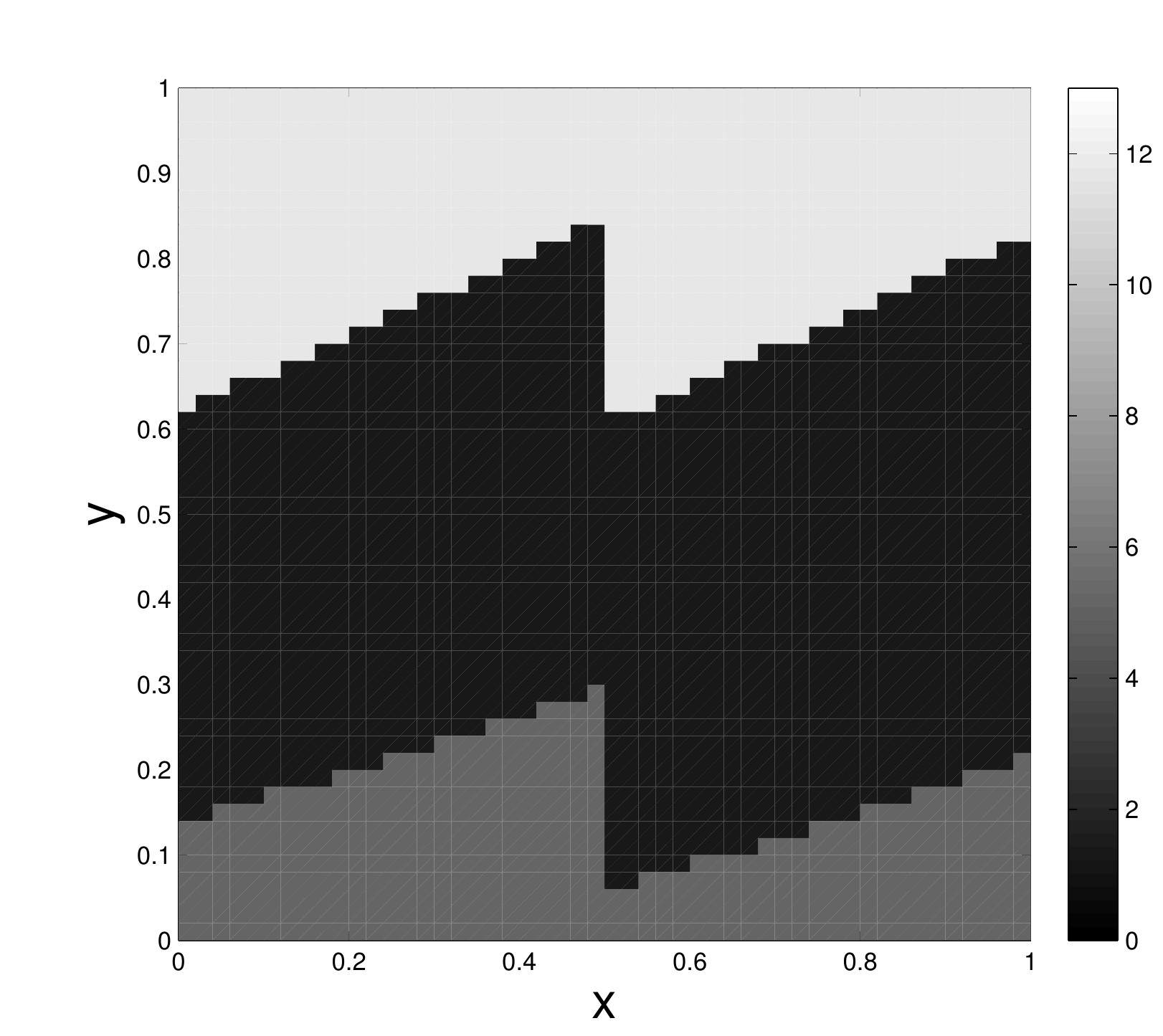}
\includegraphics[scale=0.30]{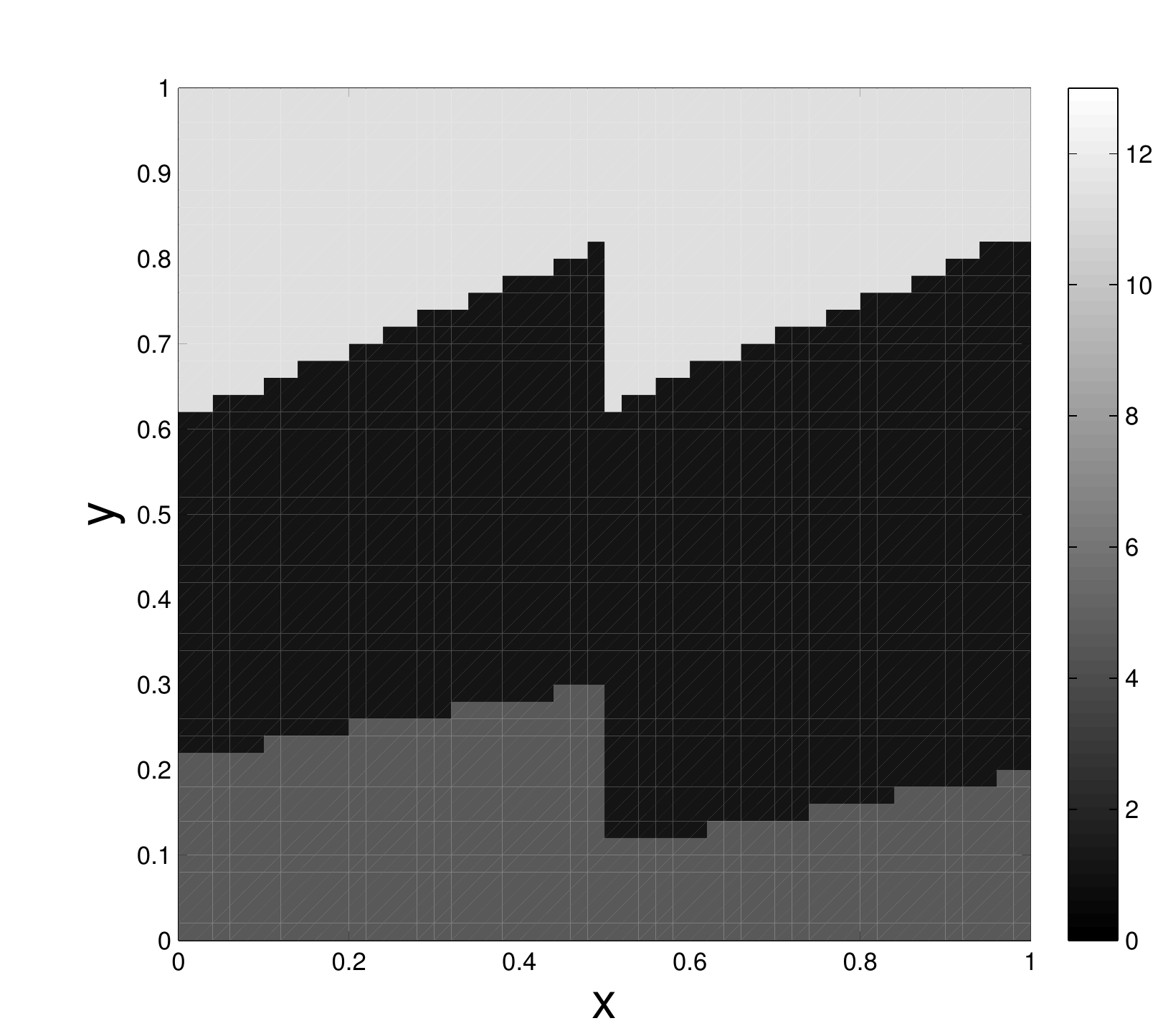}
 \caption{Permeabilities for the 3-layers fault model. Top-left: truth. Top-right: measurement locations. Bottom-left: mean (from less accurate data). Bottom-right: mean (from more accurate data).}

    \label{Figure2}
\end{center}
\end{figure}

\begin{figure}[htbp]
\begin{center}
\includegraphics[scale=0.2]{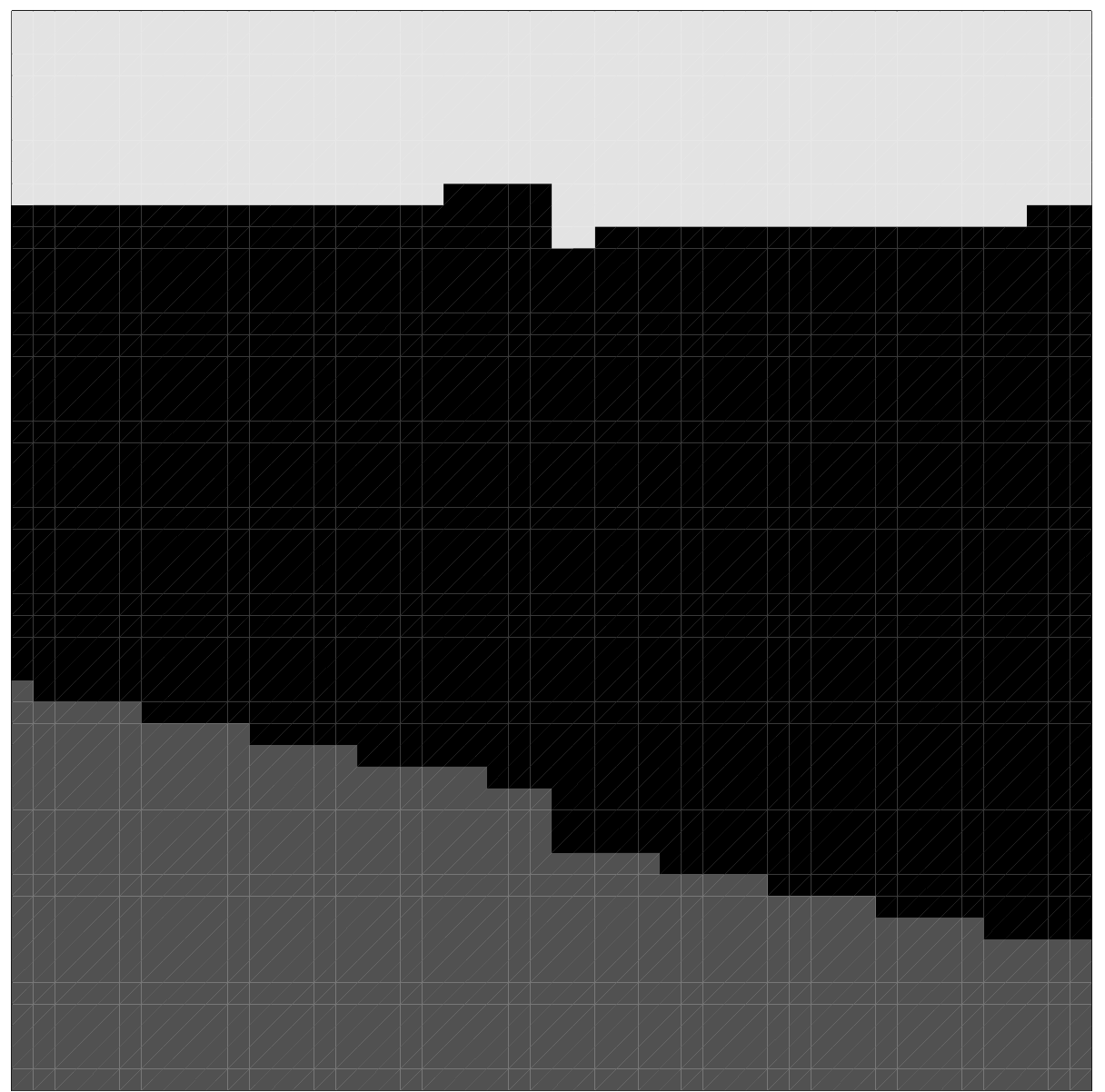}
\includegraphics[scale=0.2]{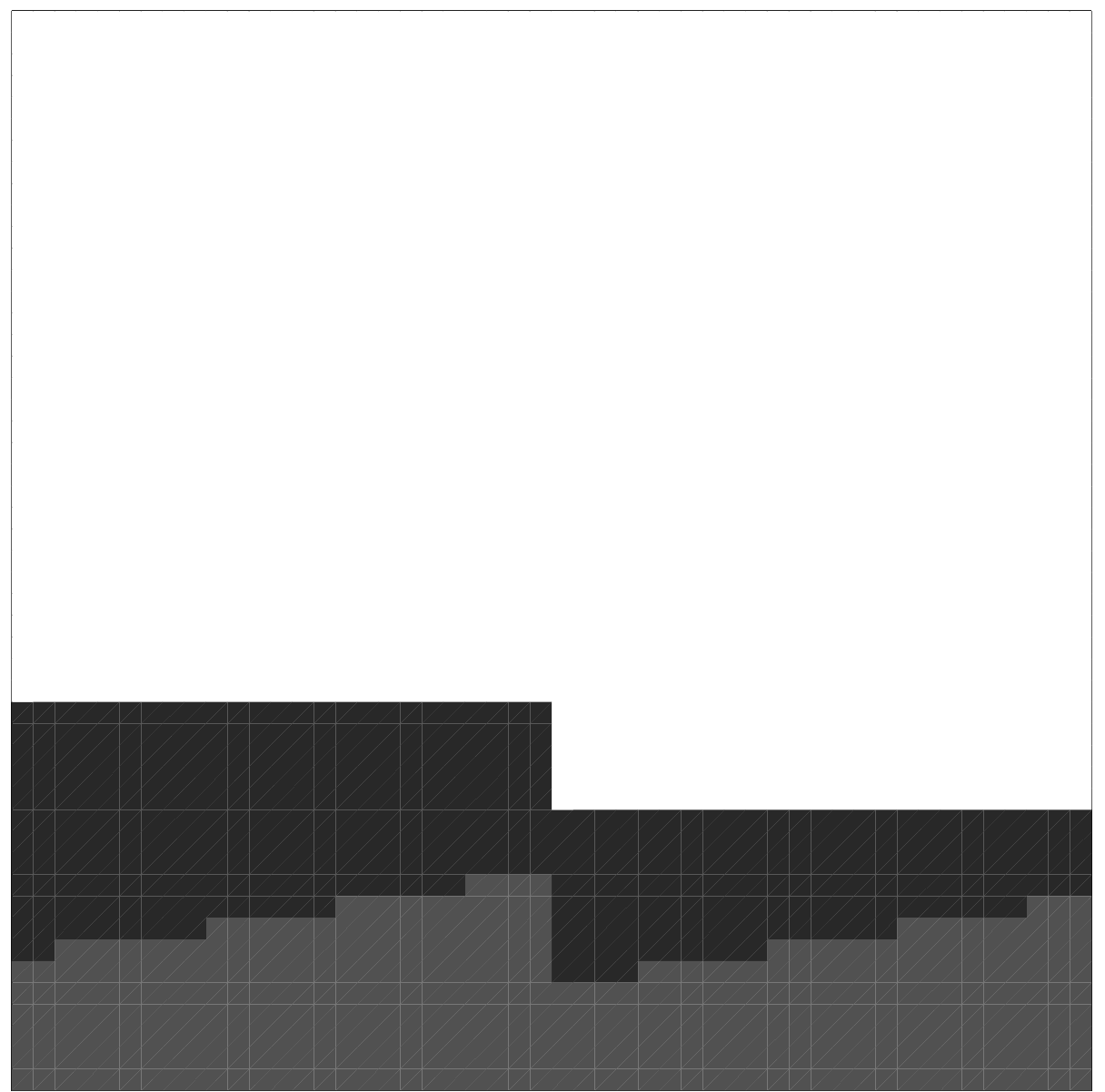}
\includegraphics[scale=0.2]{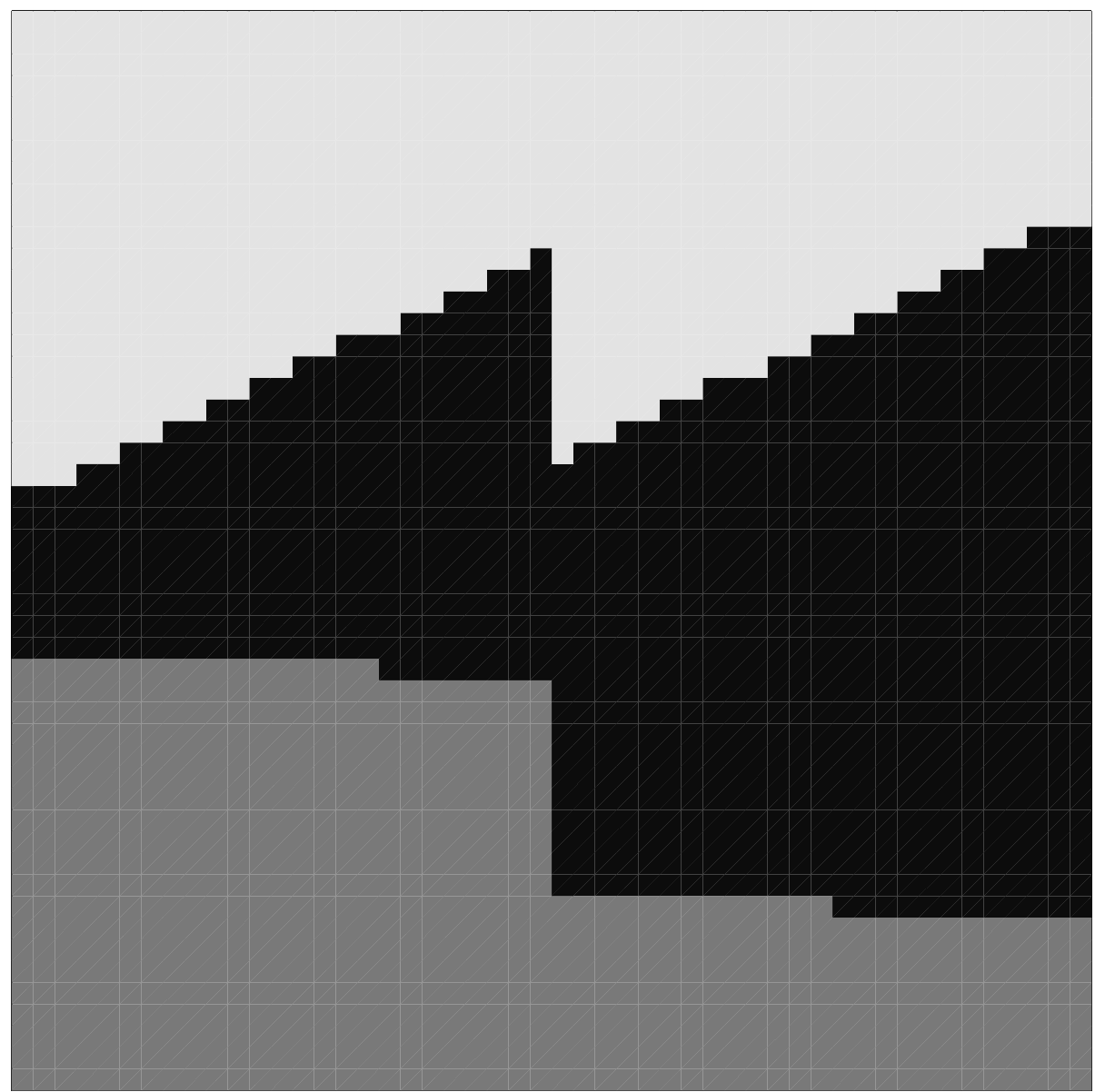}
\includegraphics[scale=0.2]{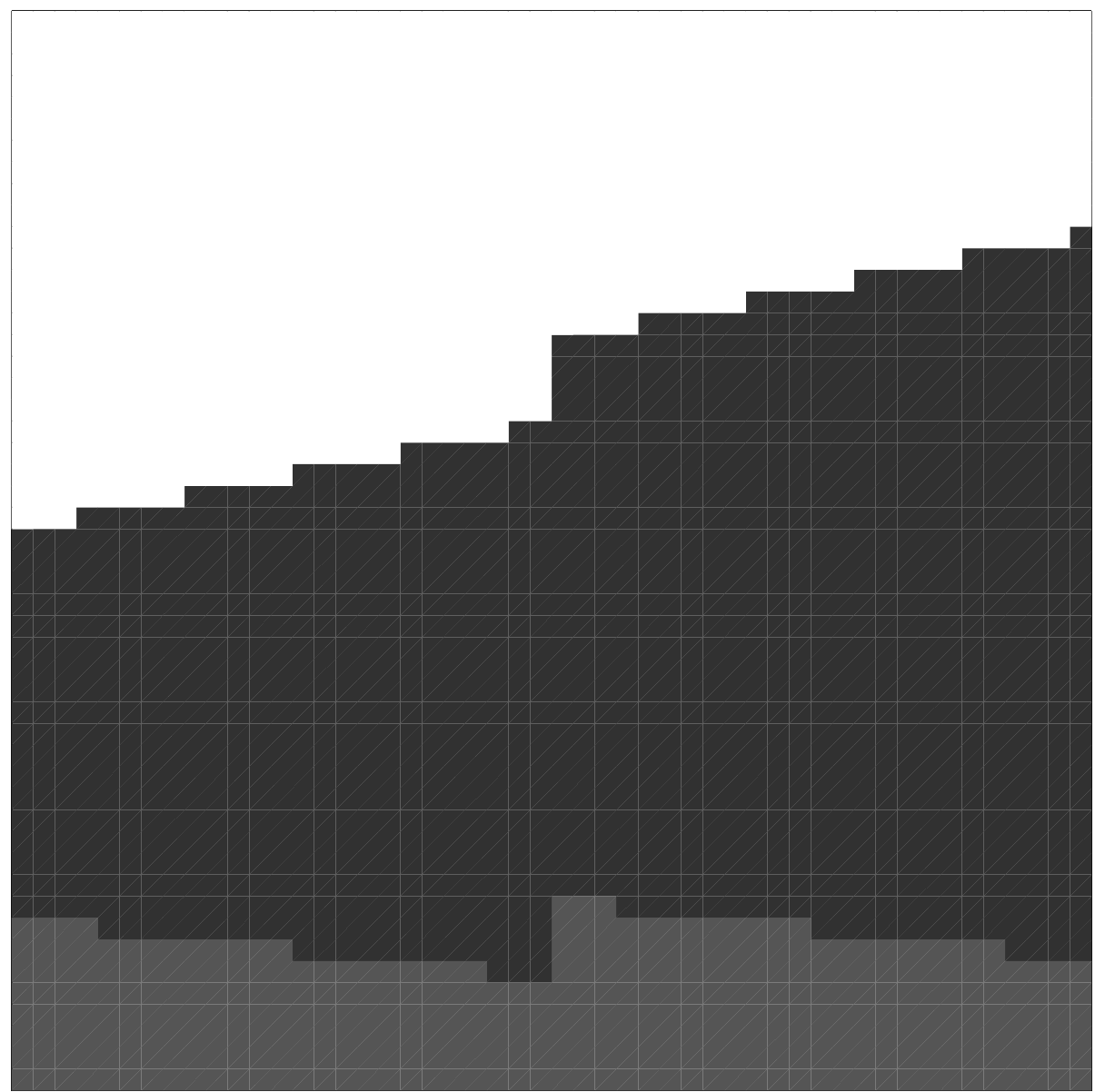}
\includegraphics[scale=0.2]{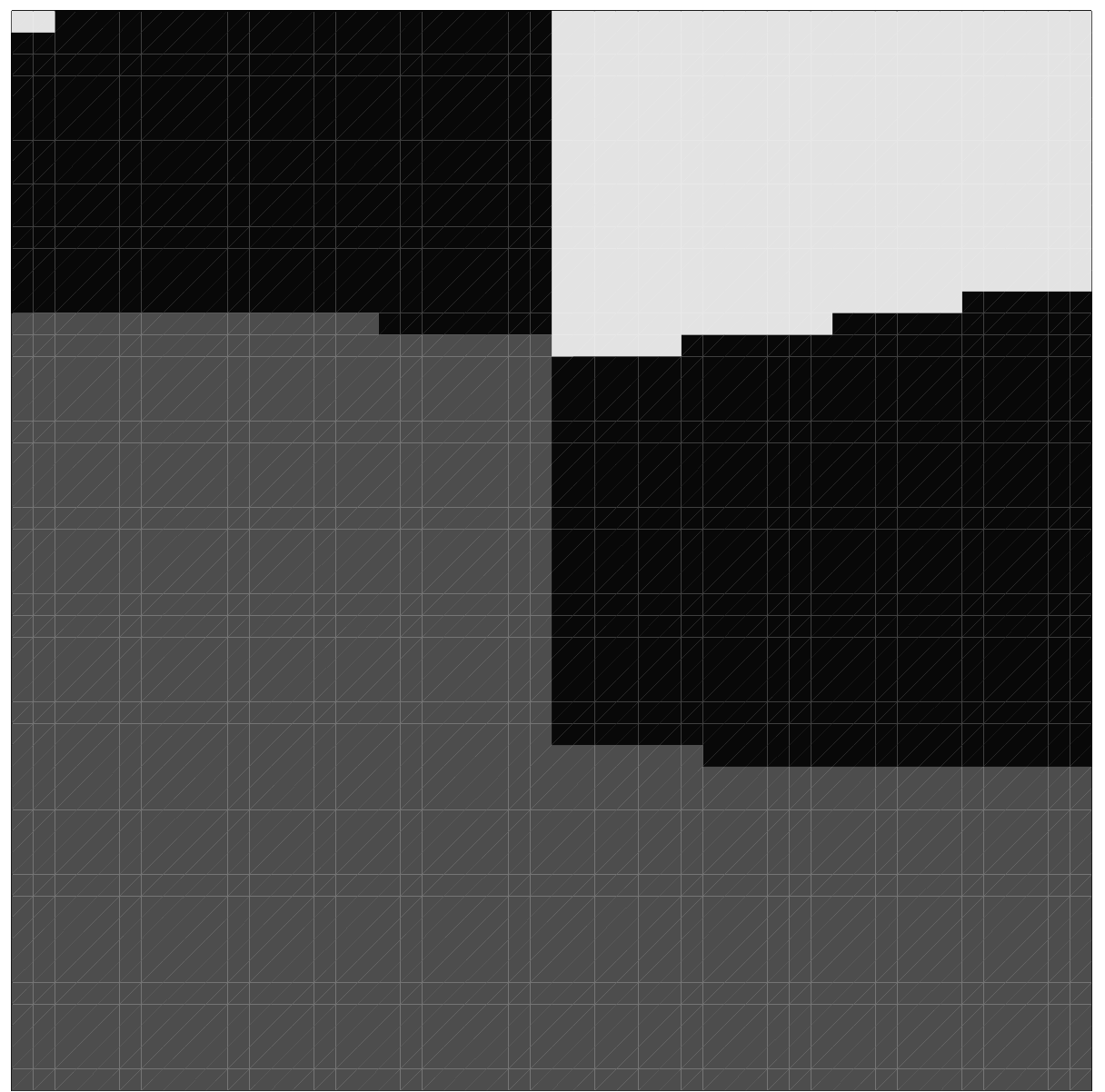}
\includegraphics[scale=0.2]{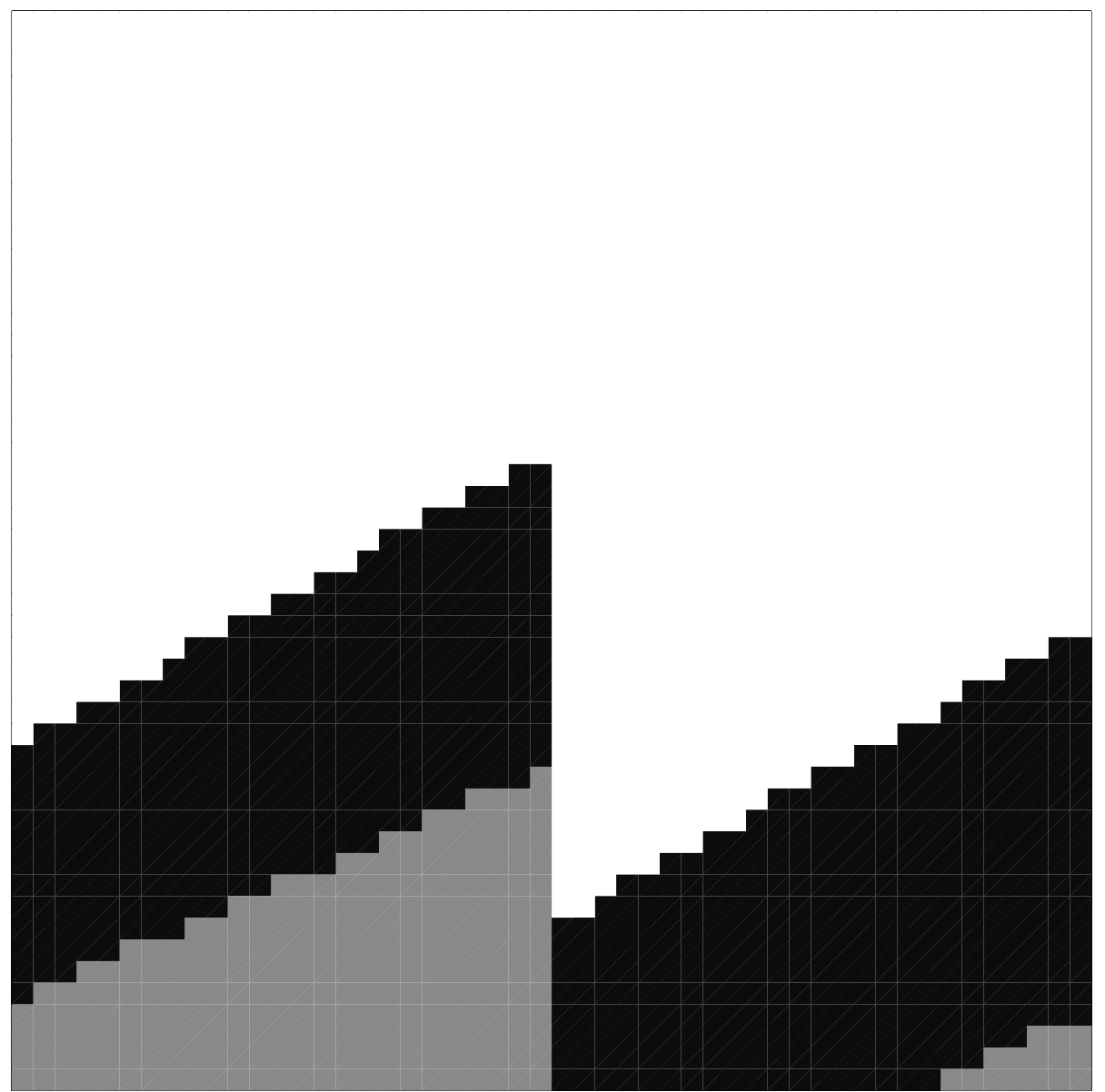}\\
\includegraphics[scale=0.2]{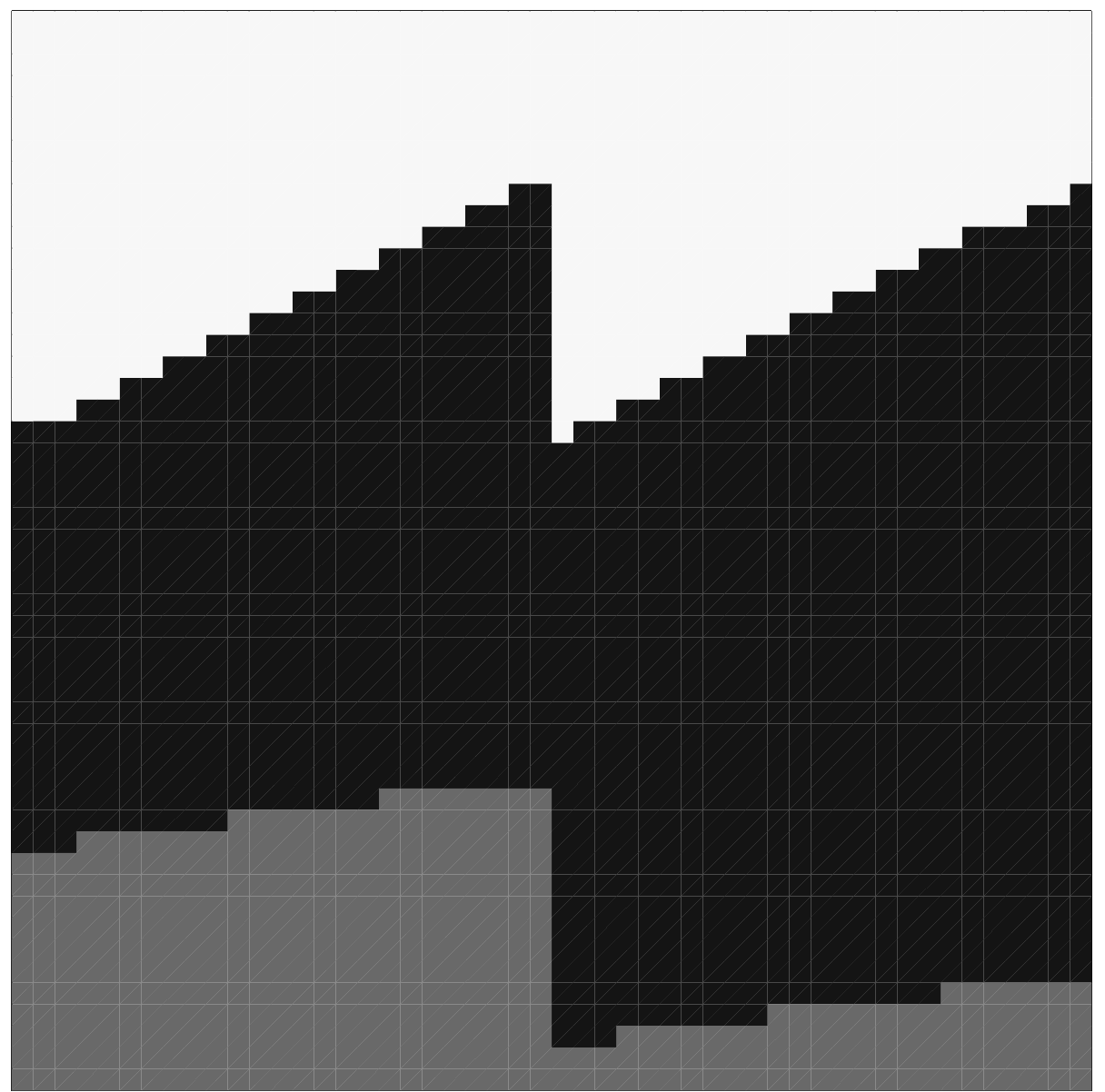}
\includegraphics[scale=0.2]{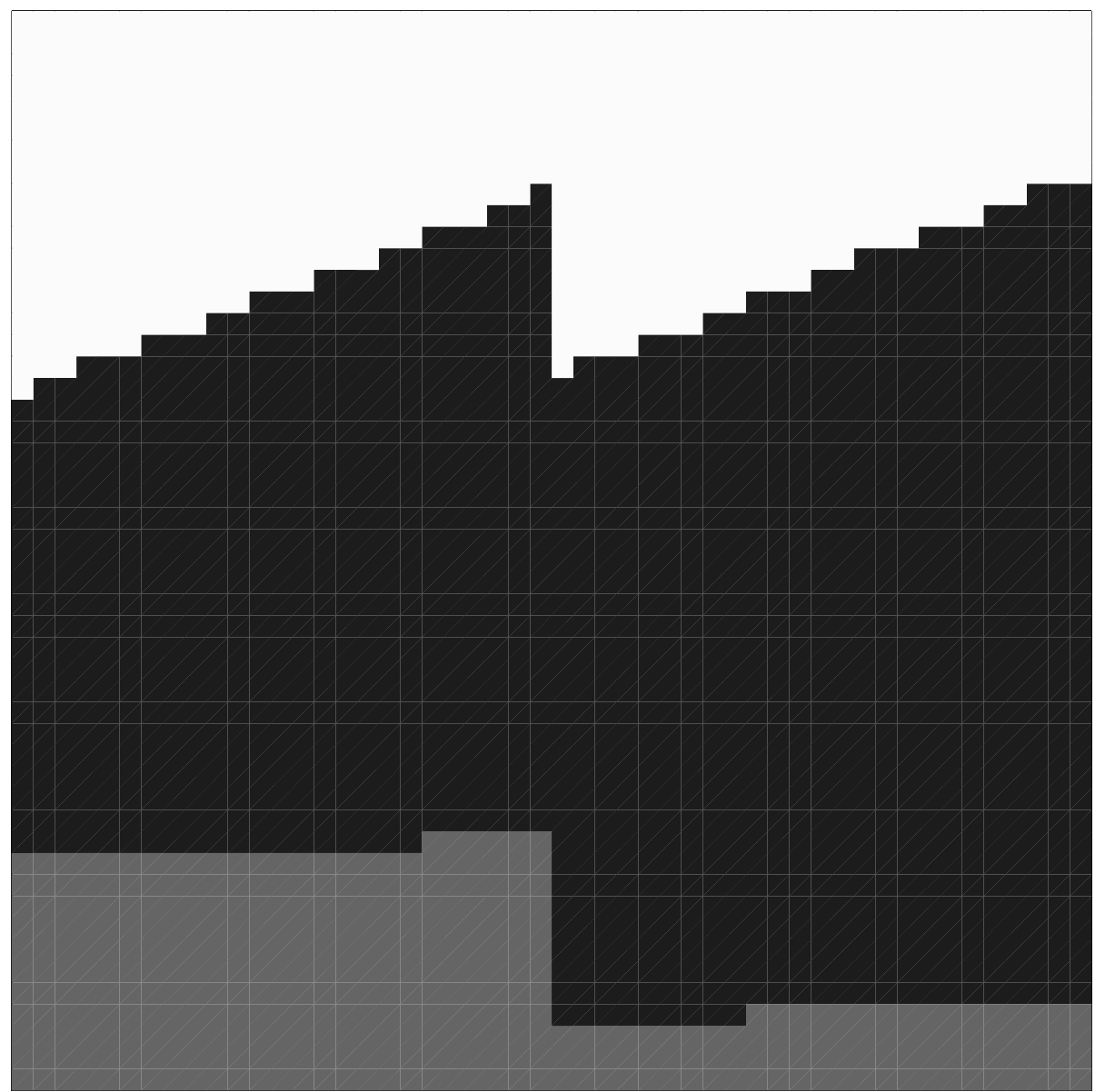}
\includegraphics[scale=0.2]{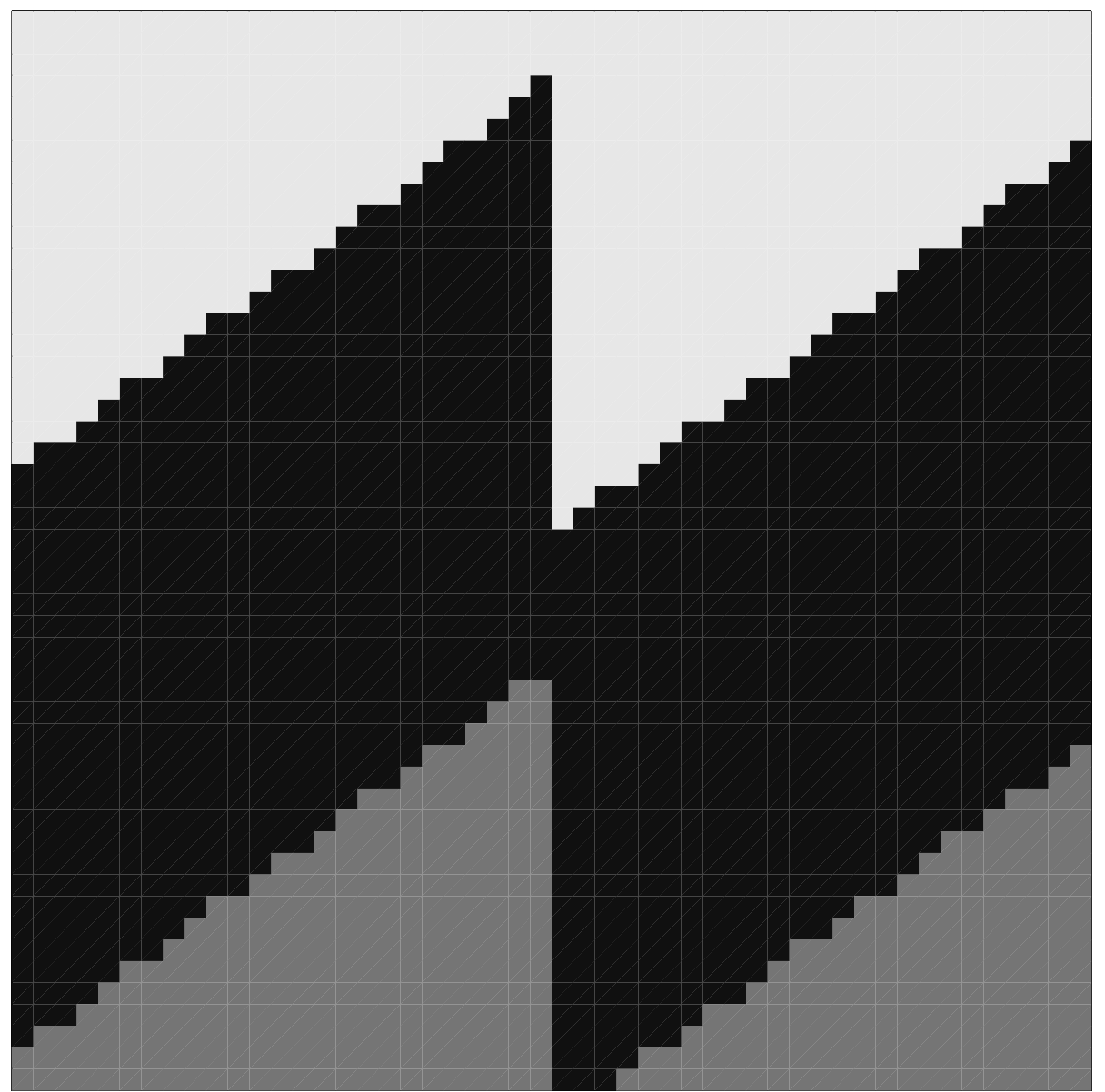}
\includegraphics[scale=0.2]{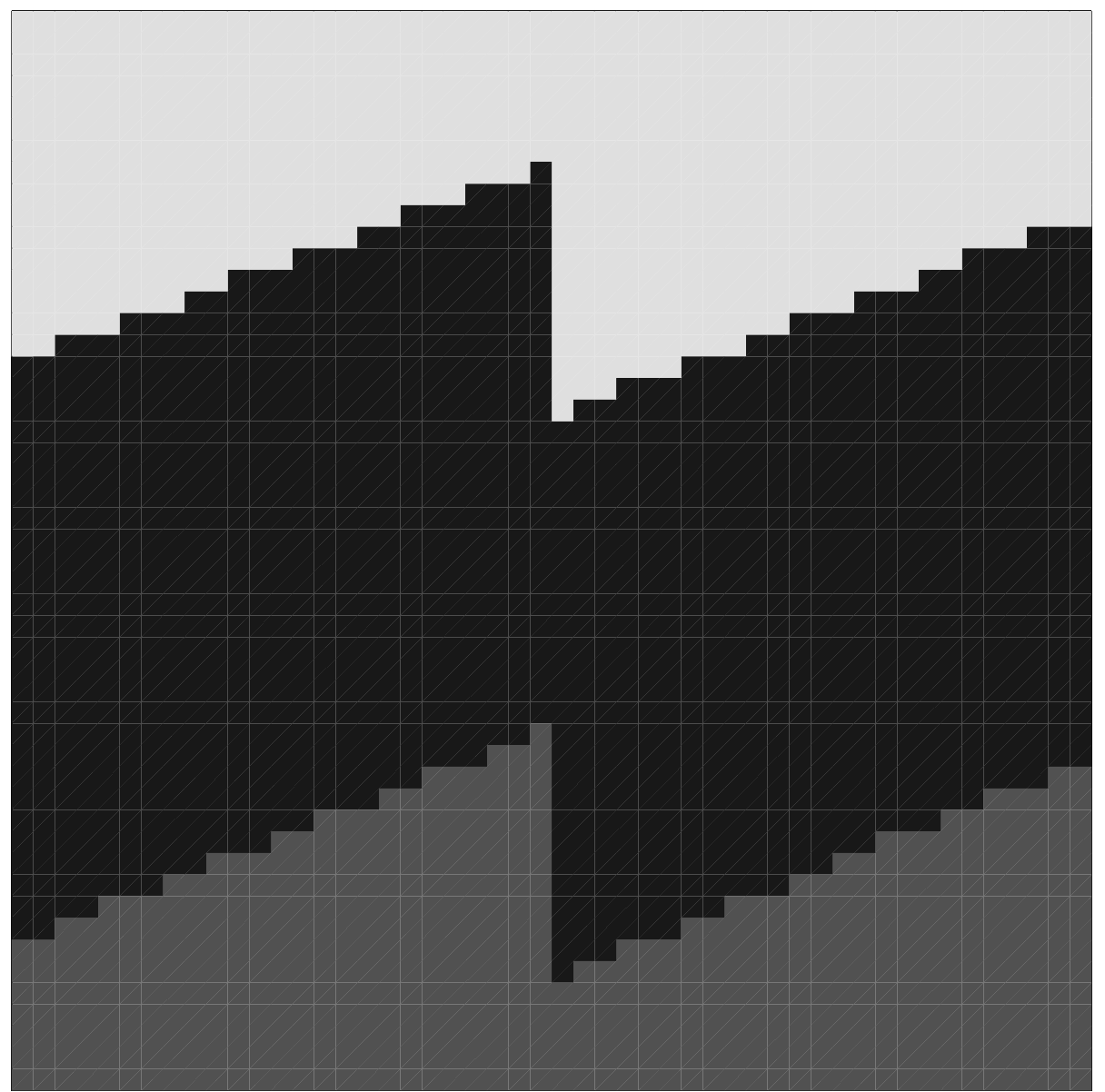}
\includegraphics[scale=0.2]{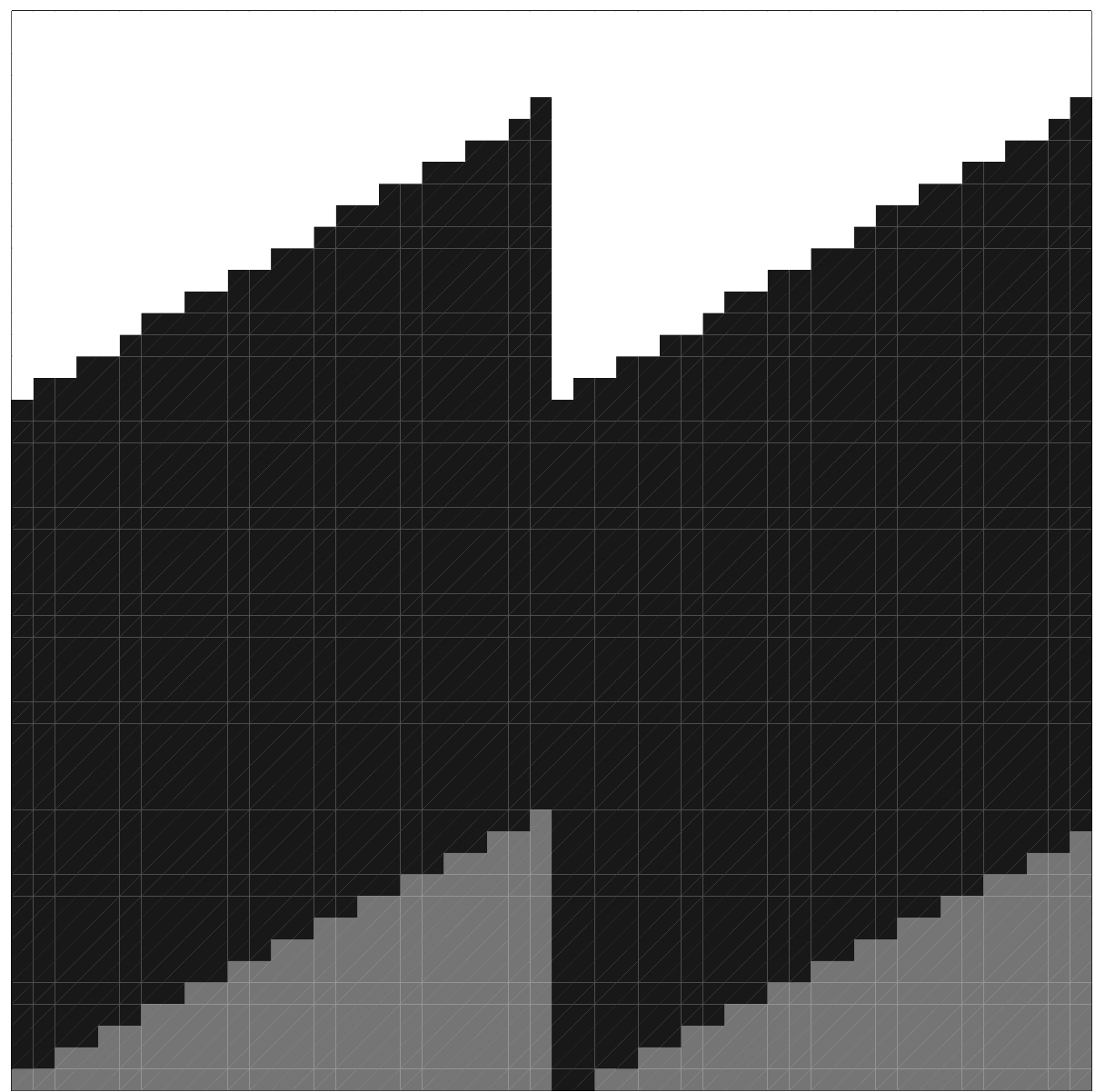}
\includegraphics[scale=0.2]{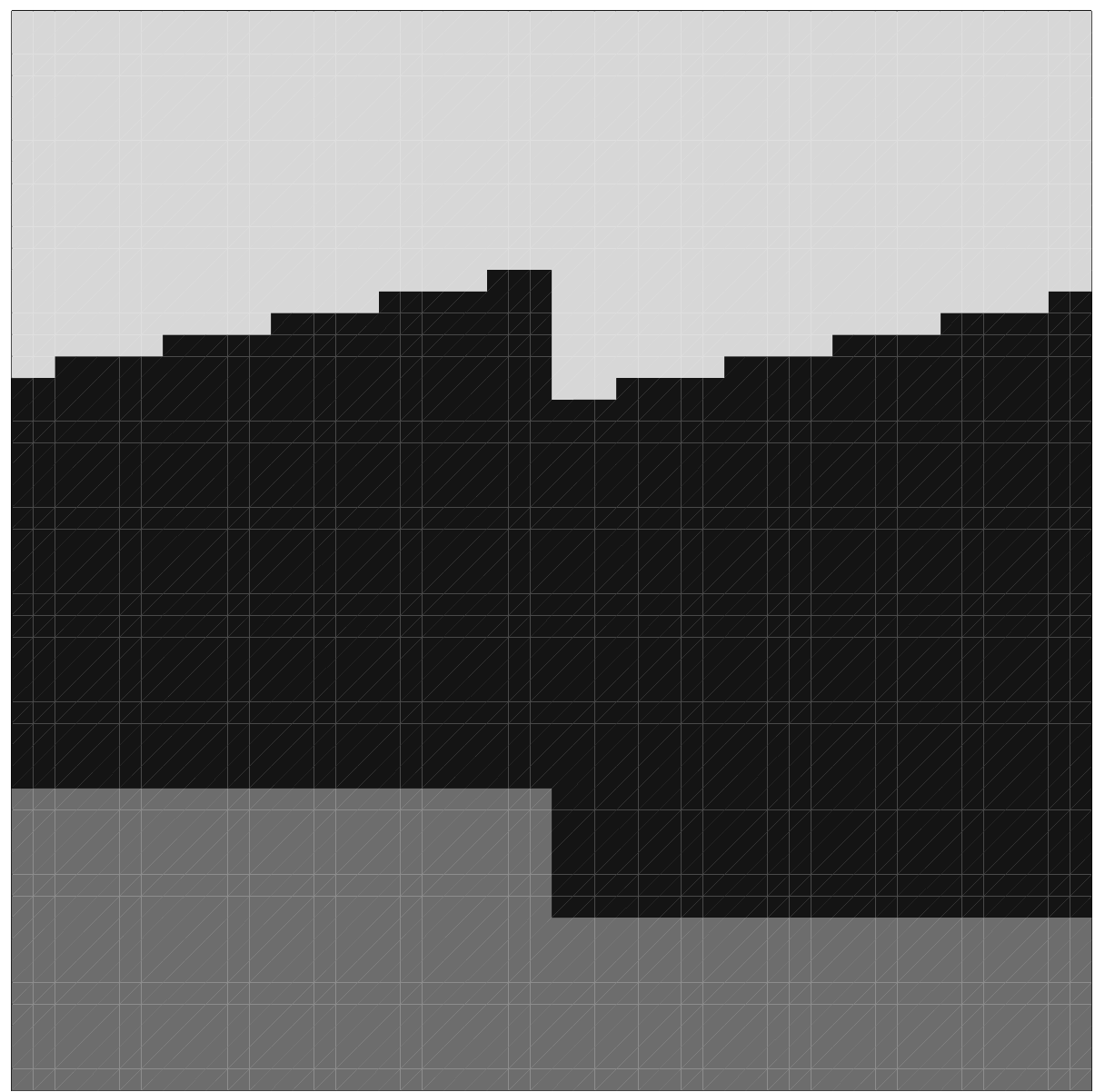}\\
\includegraphics[scale=0.2]{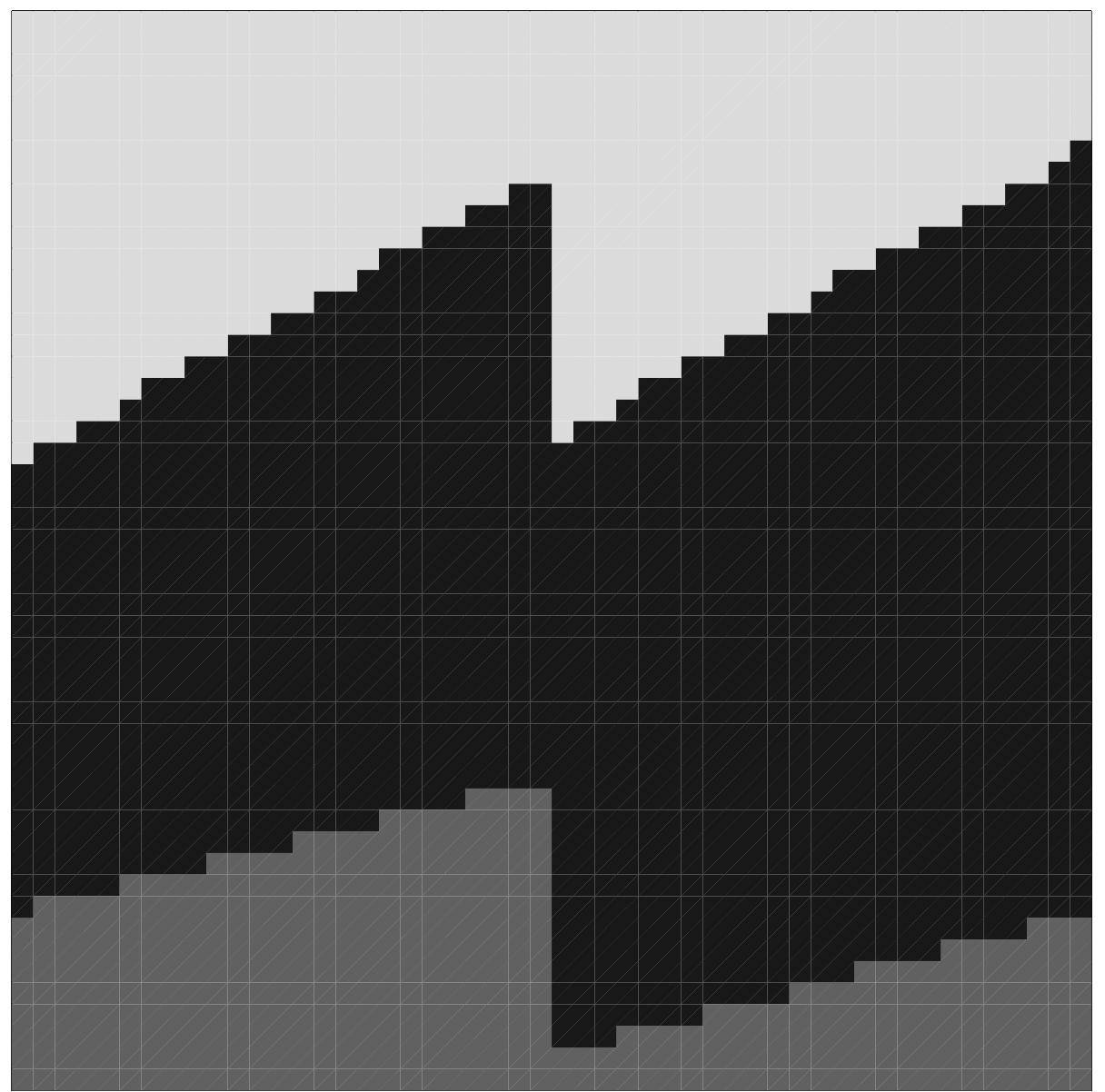}
\includegraphics[scale=0.2]{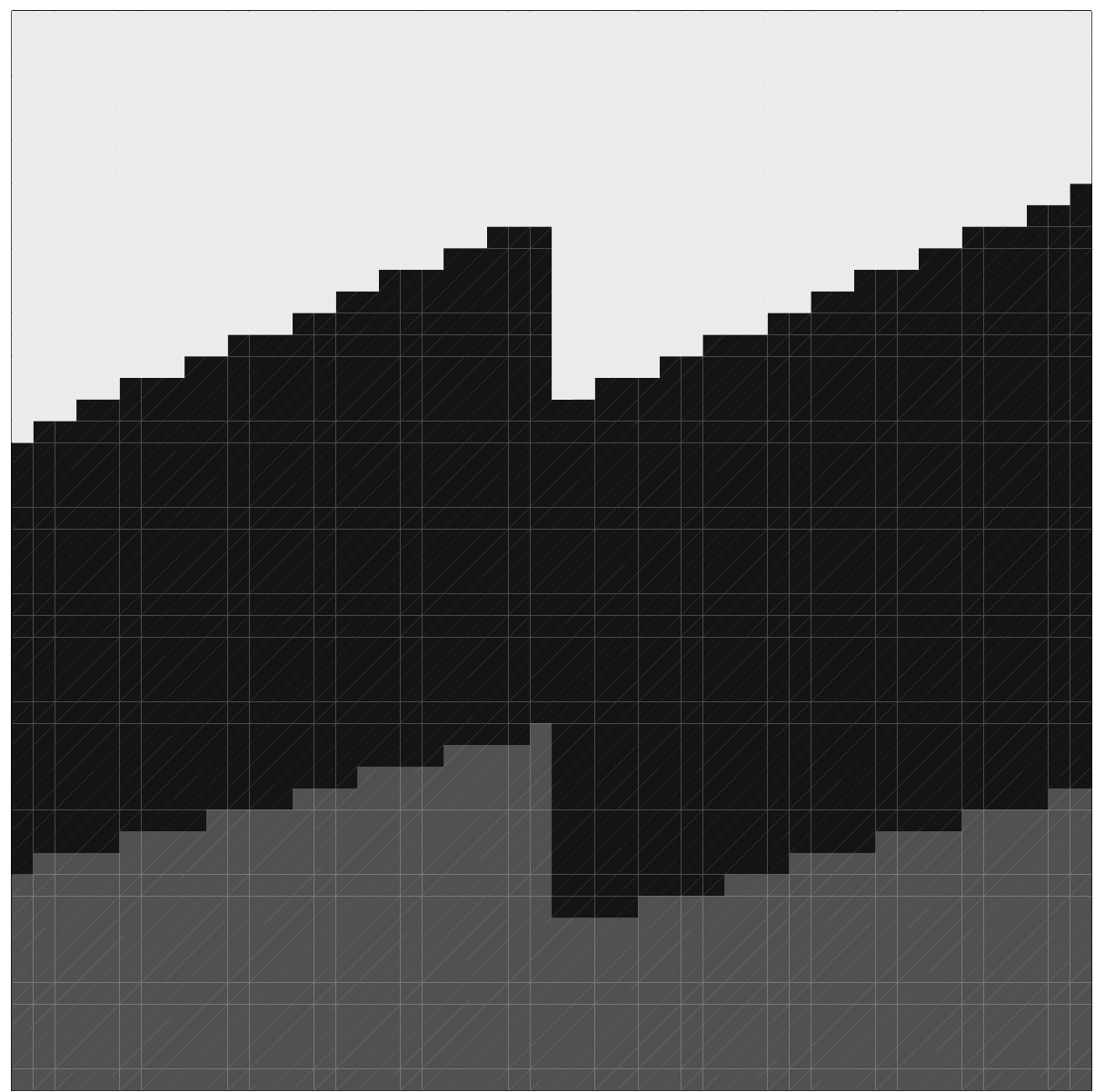}
\includegraphics[scale=0.2]{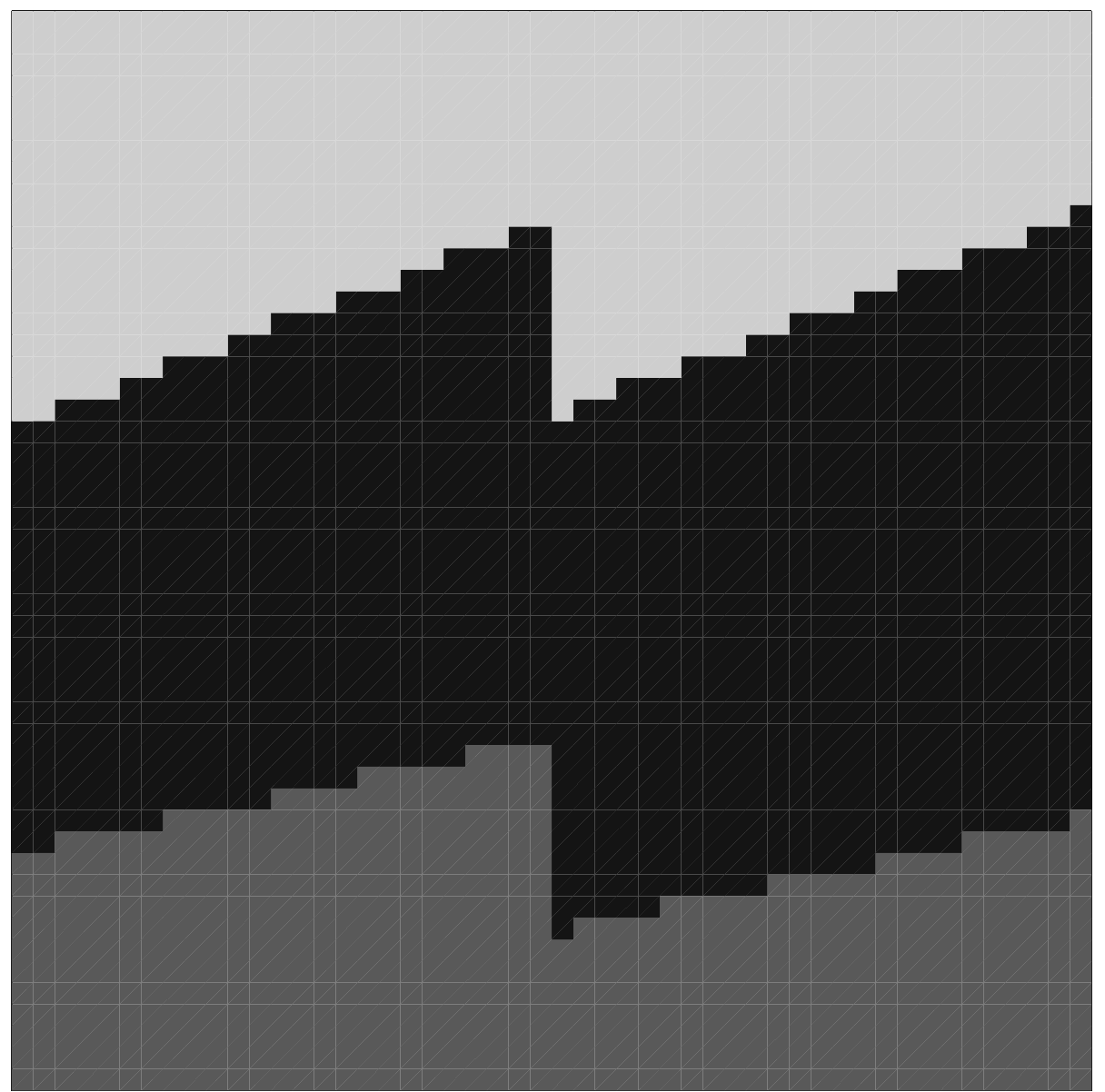}
\includegraphics[scale=0.2]{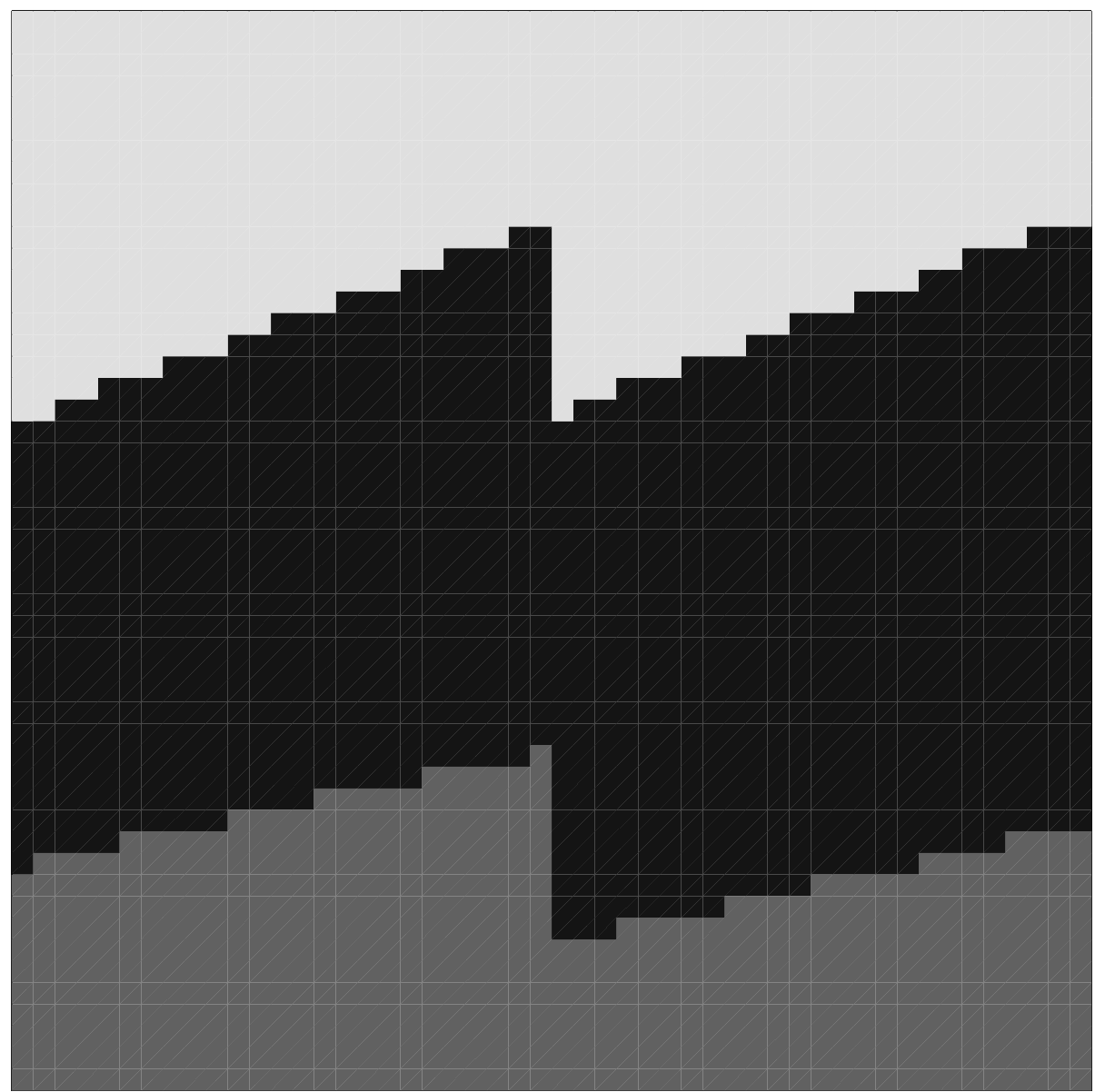}
\includegraphics[scale=0.2]{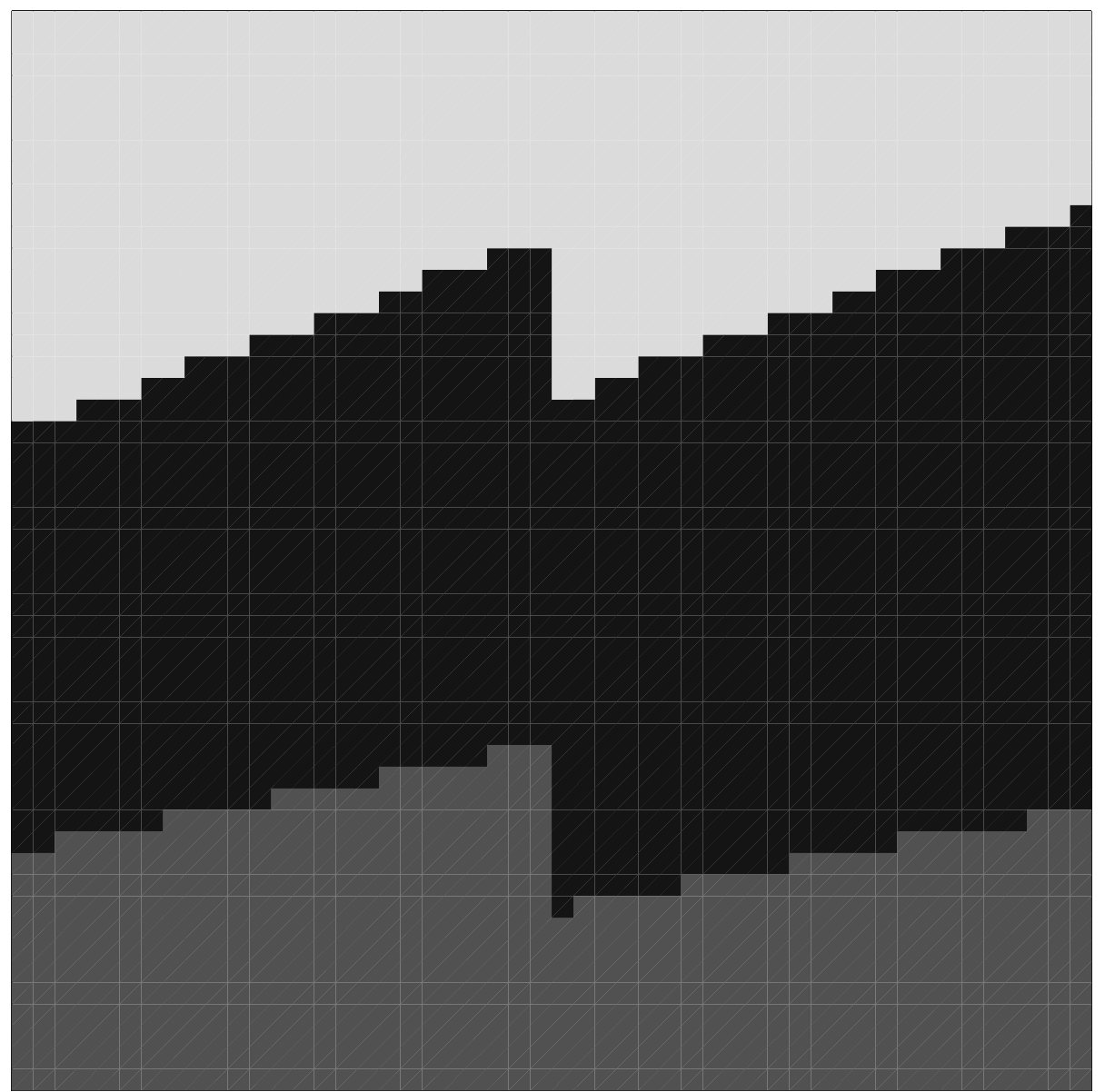}
\includegraphics[scale=0.2]{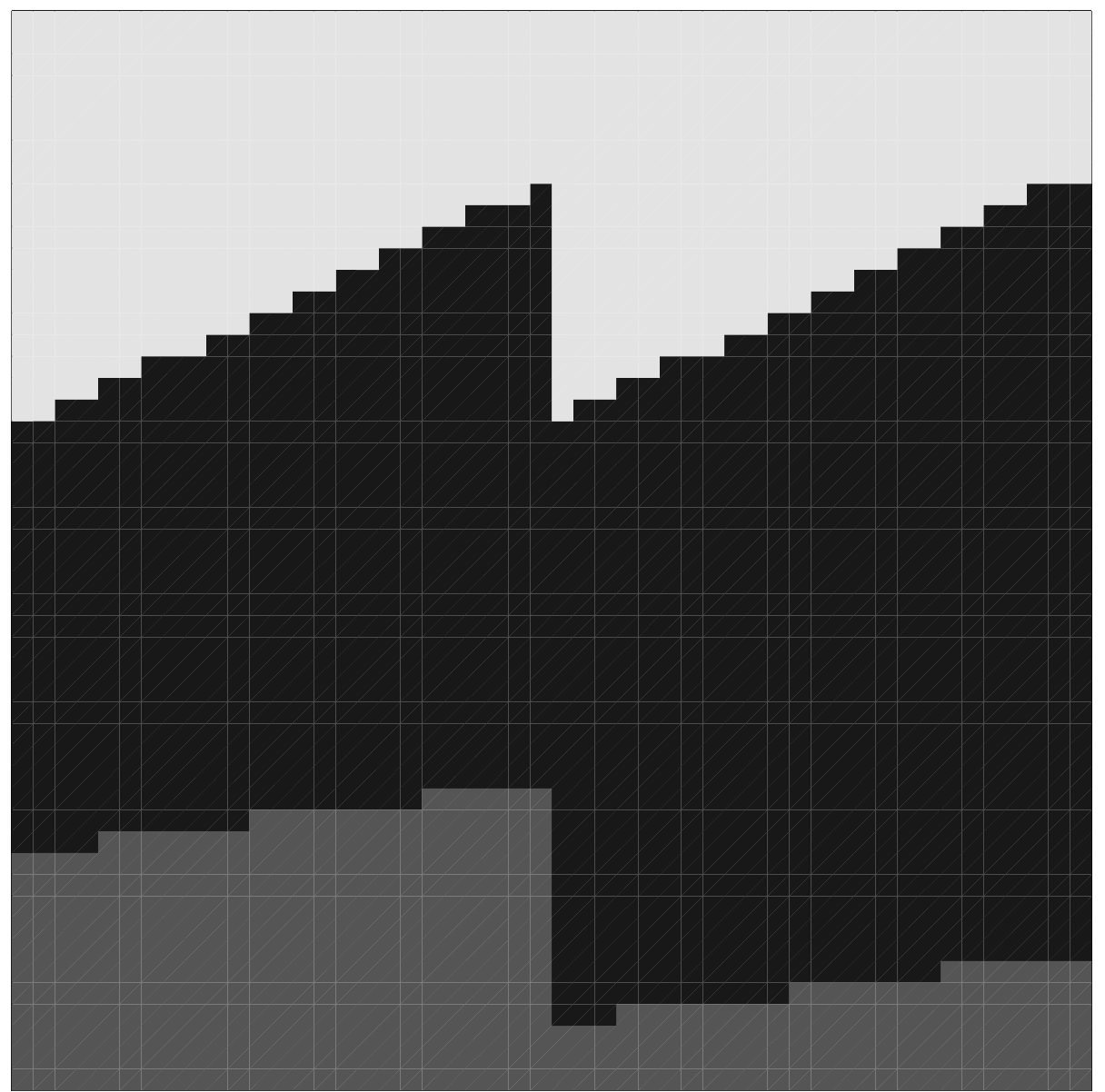}\\
 \caption{Permeabilities defined by (\ref{eq:num1}) from samples of the prior (top row), posterior with less accurate measurements (middle row) and posterior with more accurate measurements (bottom-row)}    \label{Figure3}
\end{center}
\end{figure}

\begin{figure}[htbp]
\begin{center}
\includegraphics[scale=0.28]{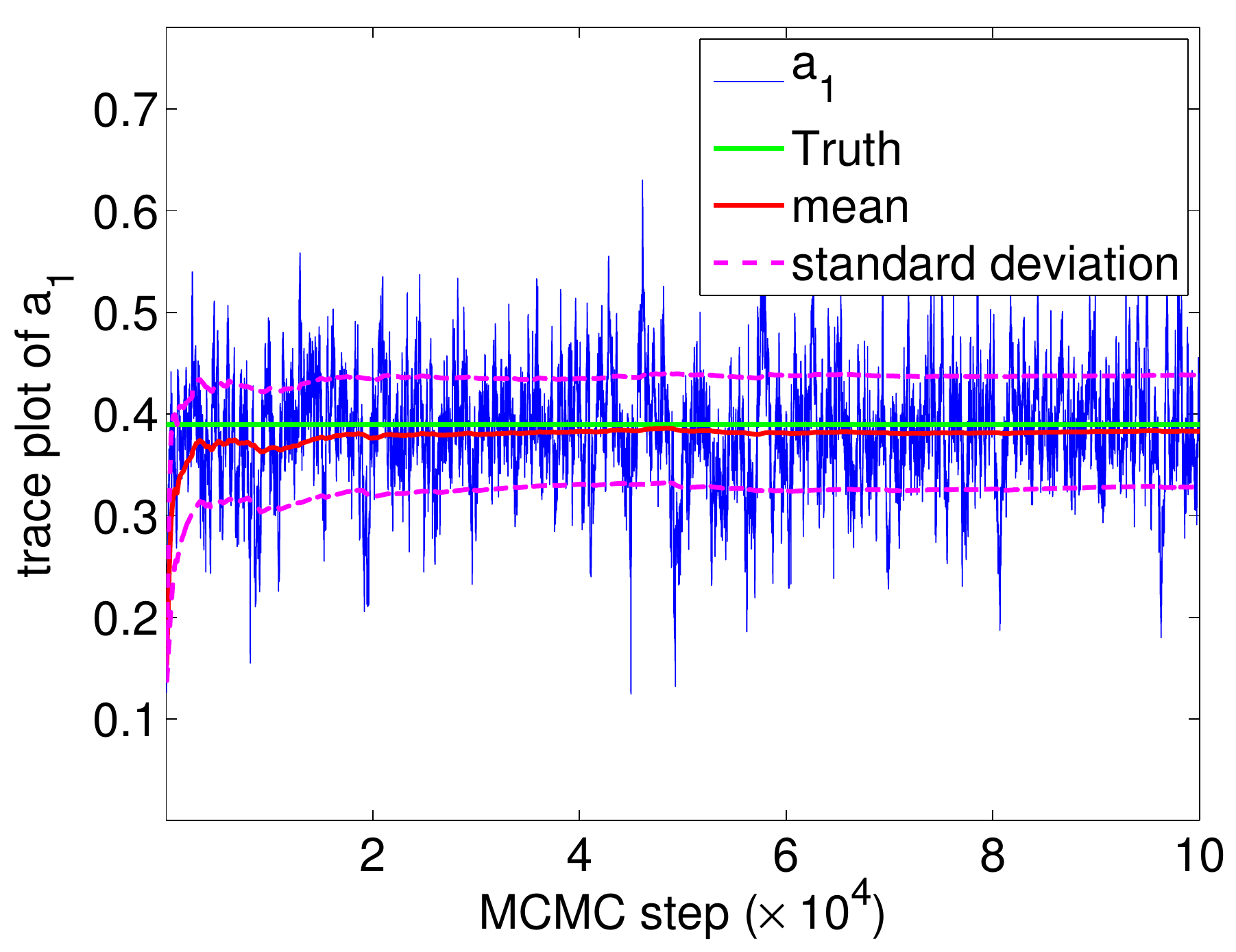}
\includegraphics[scale=0.28]{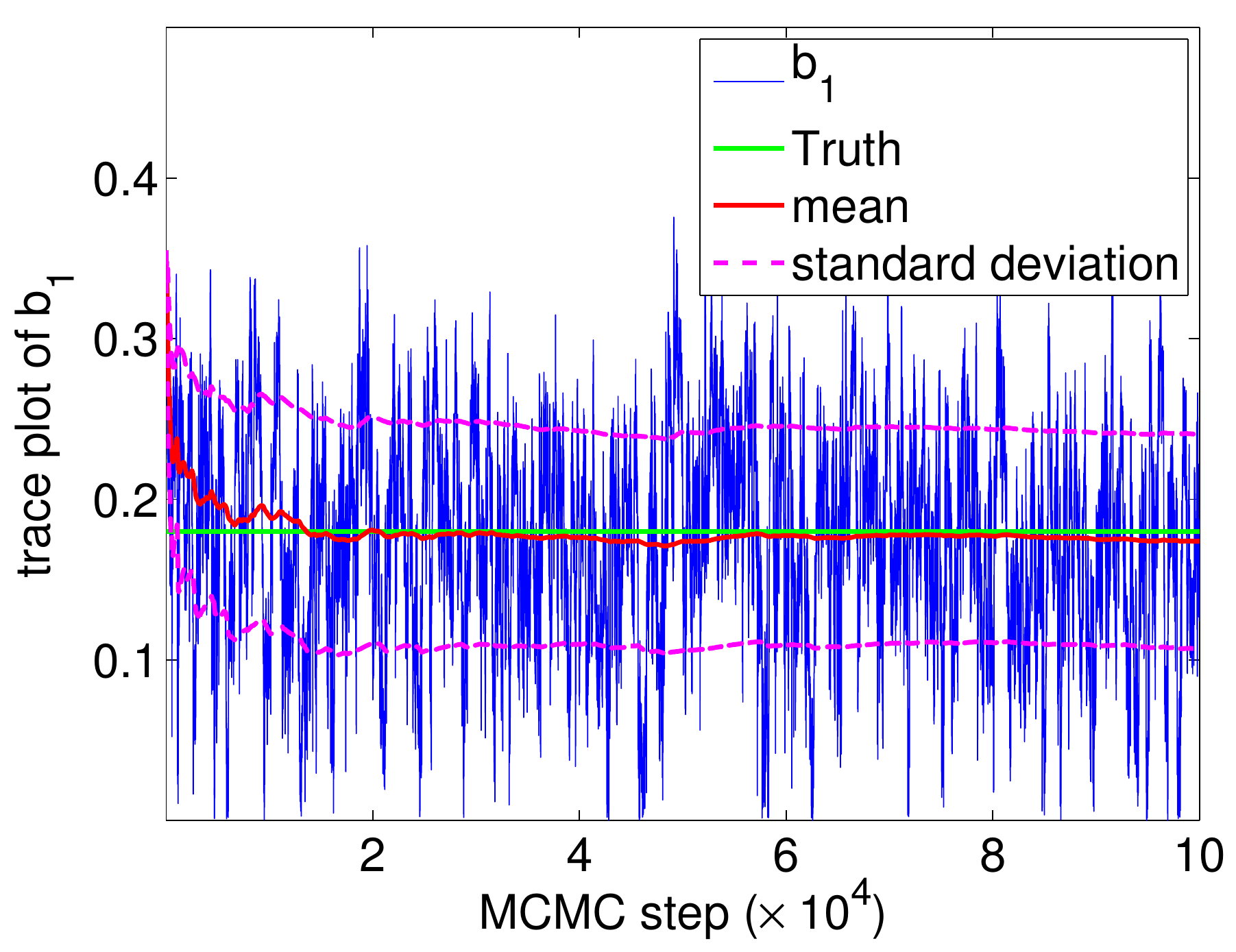}
\includegraphics[scale=0.28]{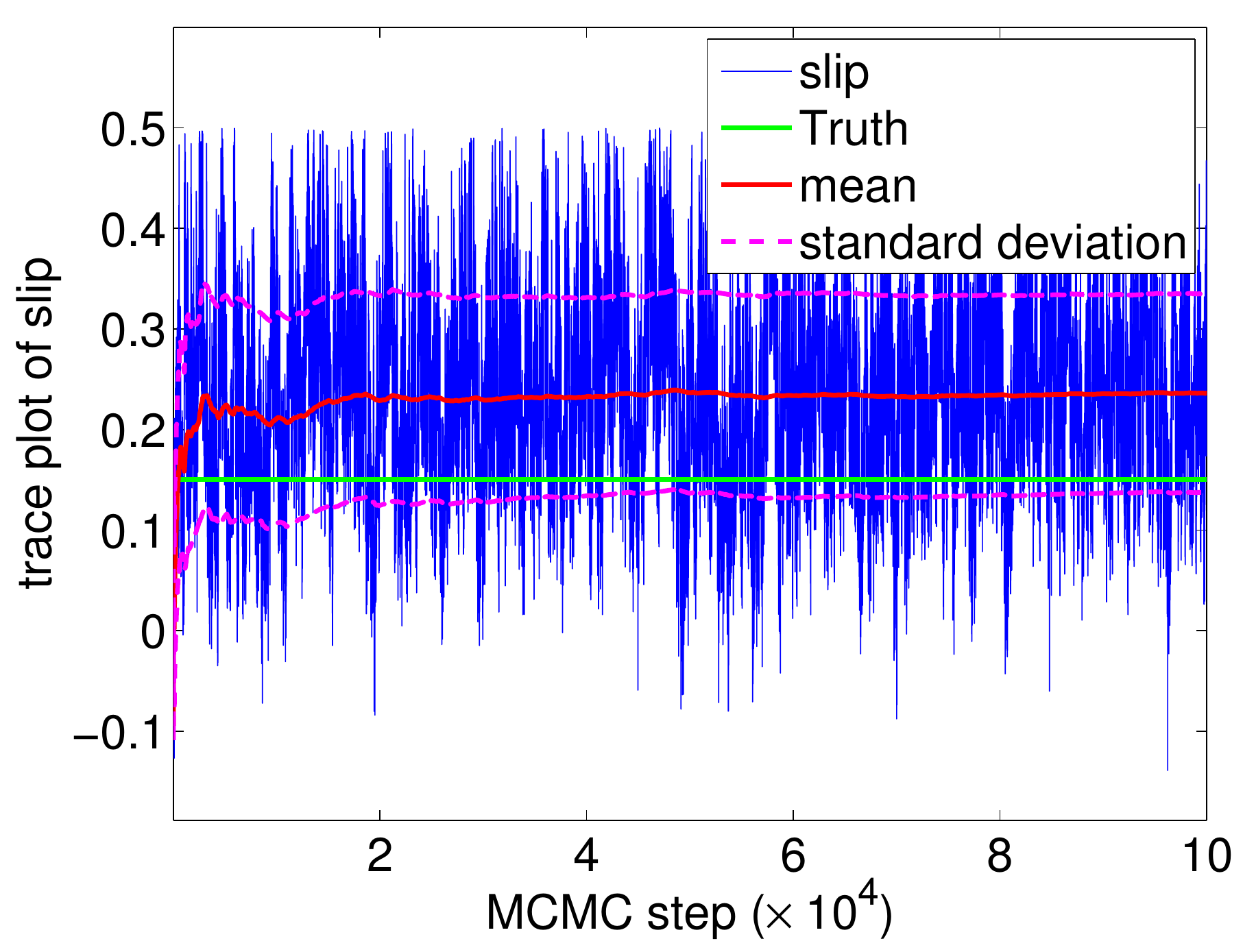}\\
\includegraphics[scale=0.28]{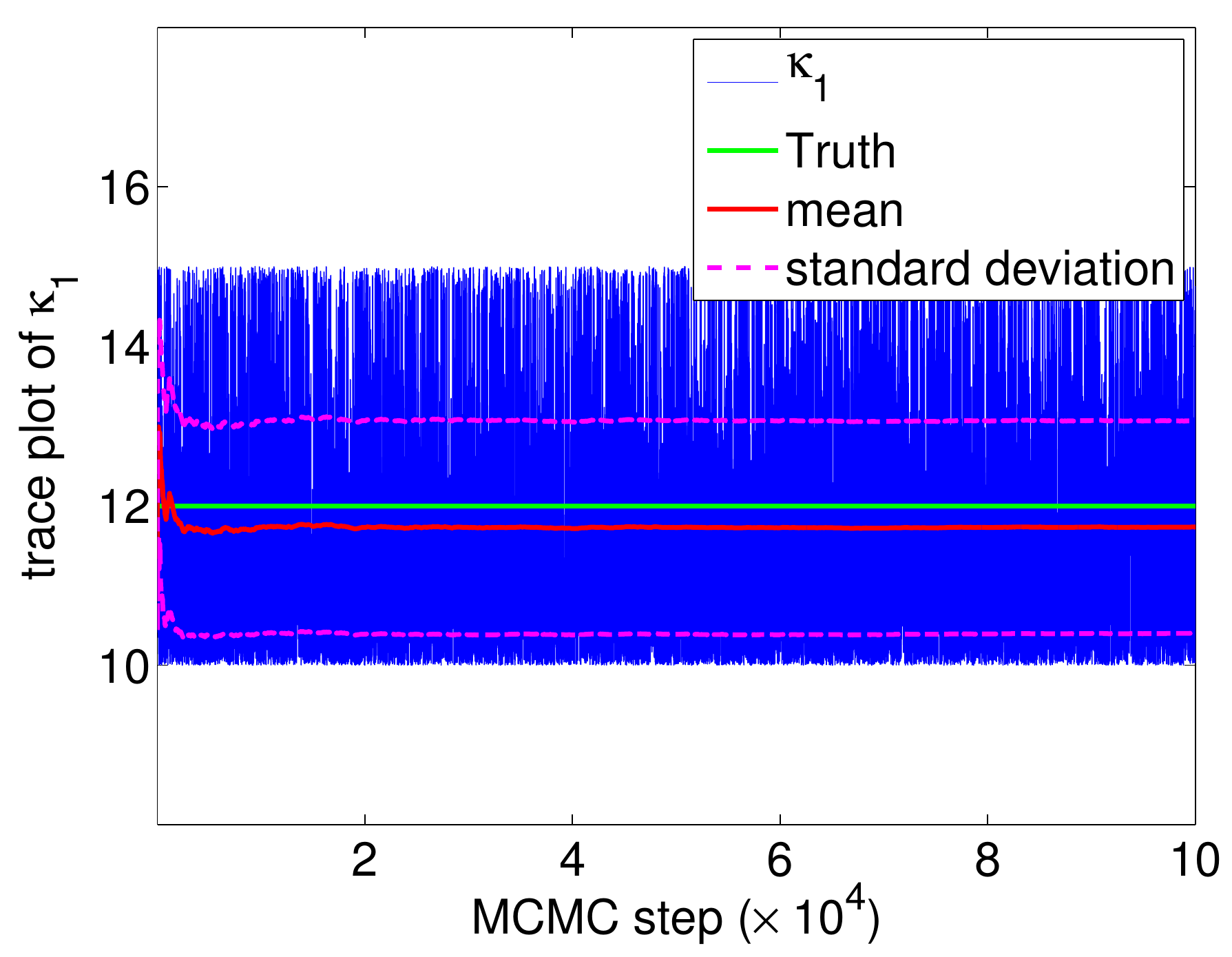}
\includegraphics[scale=0.28]{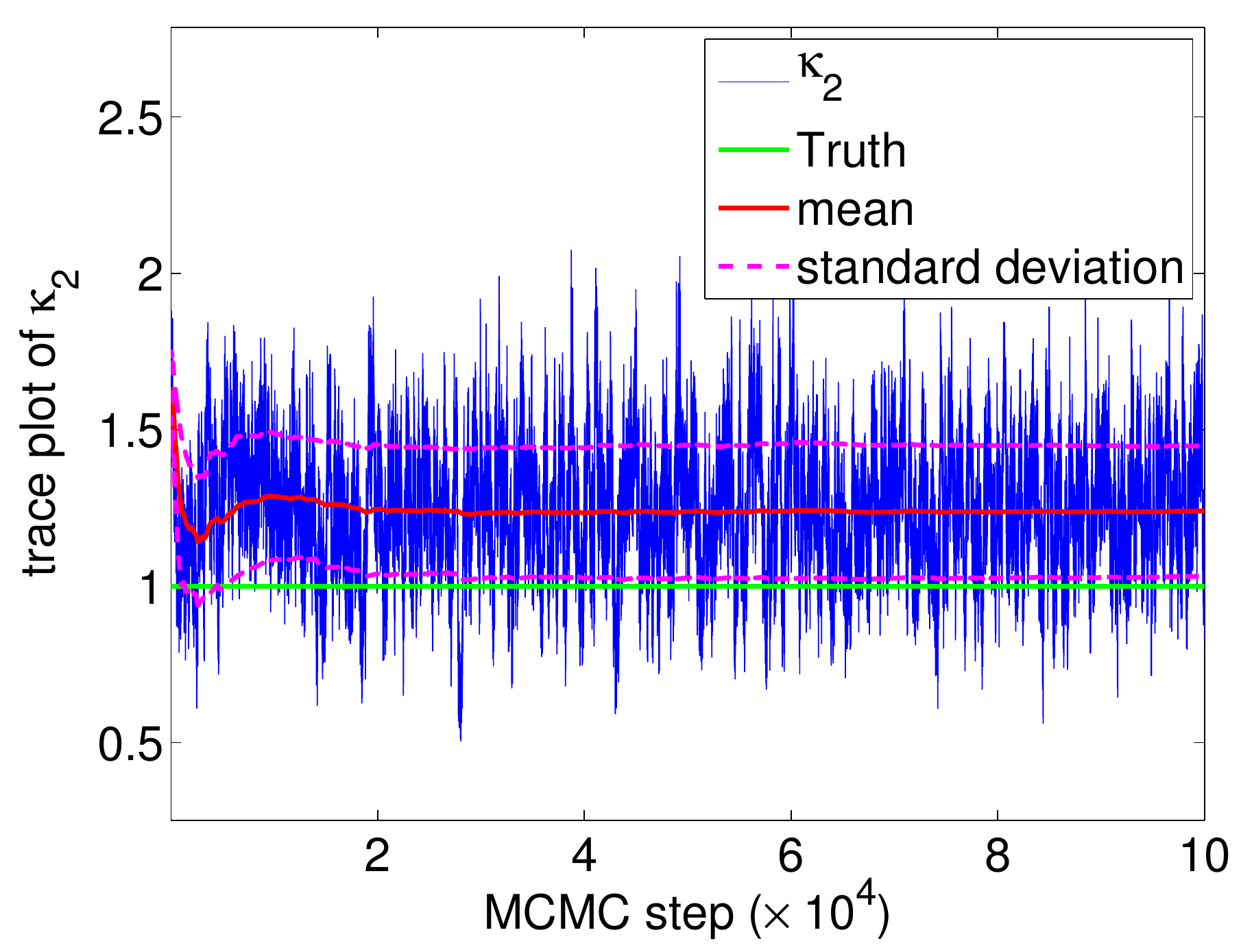}
\includegraphics[scale=0.28]{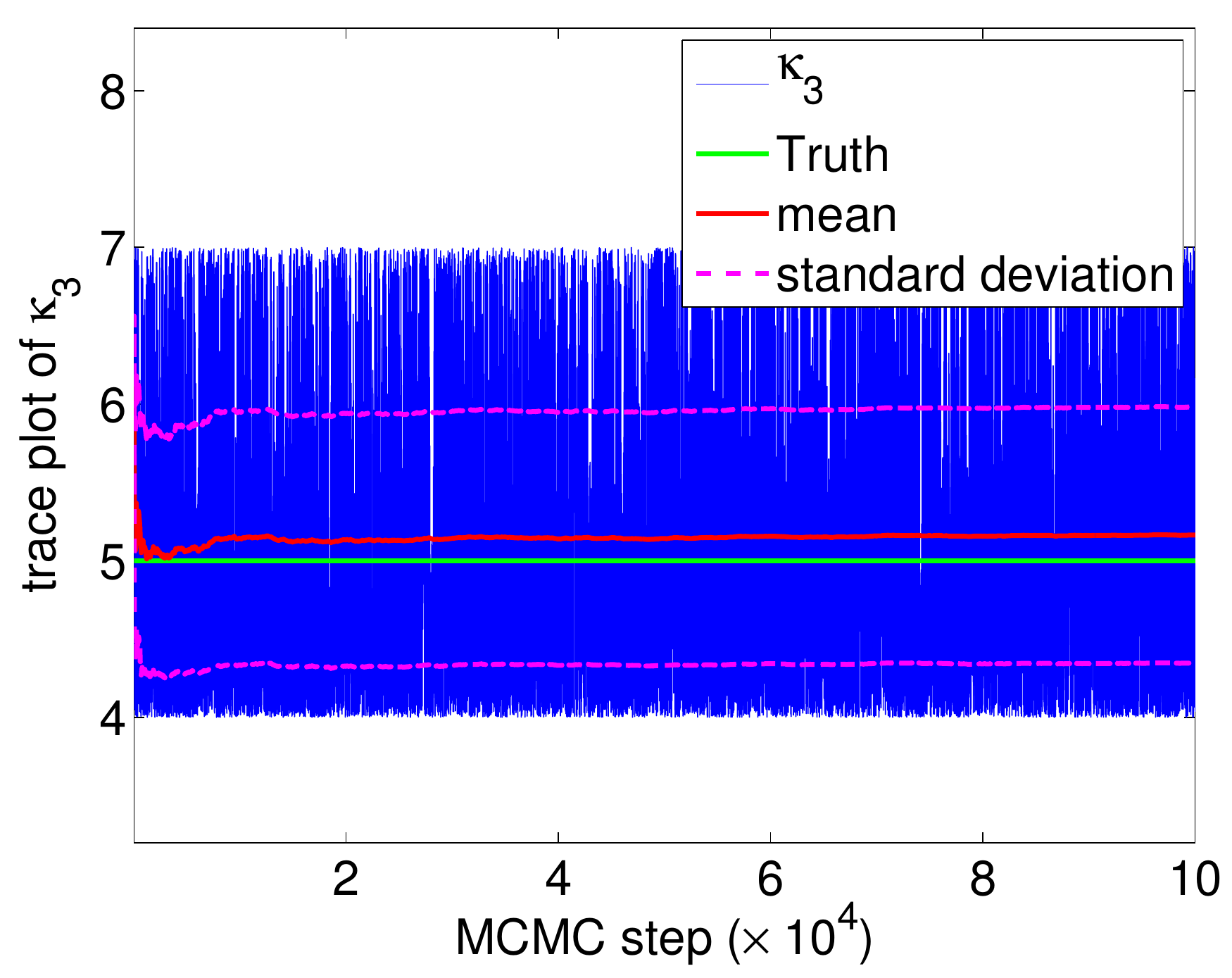}
 \caption{Trace plot from one MCMC chain (data set 1). Top: three geometric parameters. Bottom: permeability values}
    \label{Figure4}
\end{center}
\end{figure}

\begin{figure}[htbp]
\begin{center}
\includegraphics[scale=0.25]{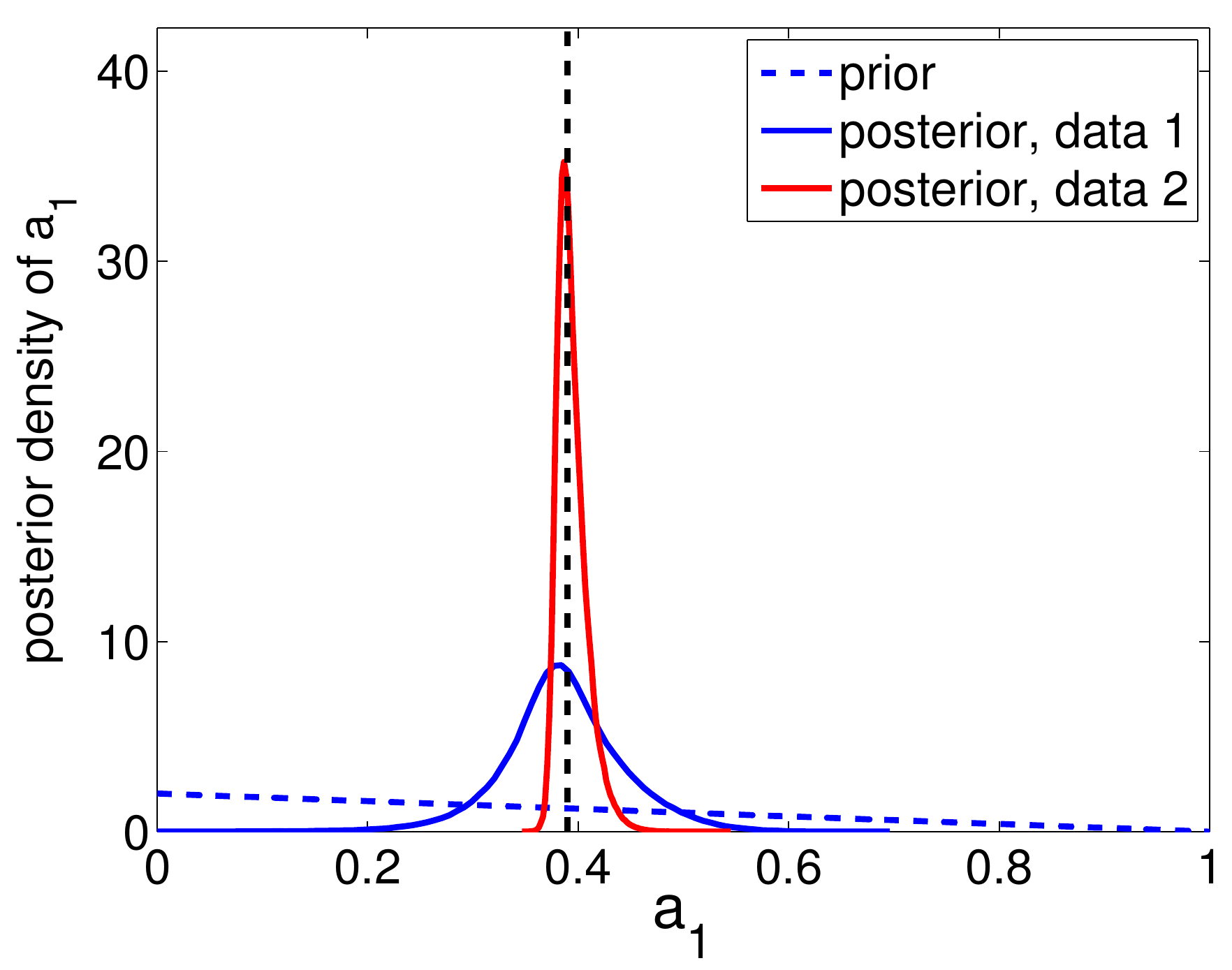}
\includegraphics[scale=0.25]{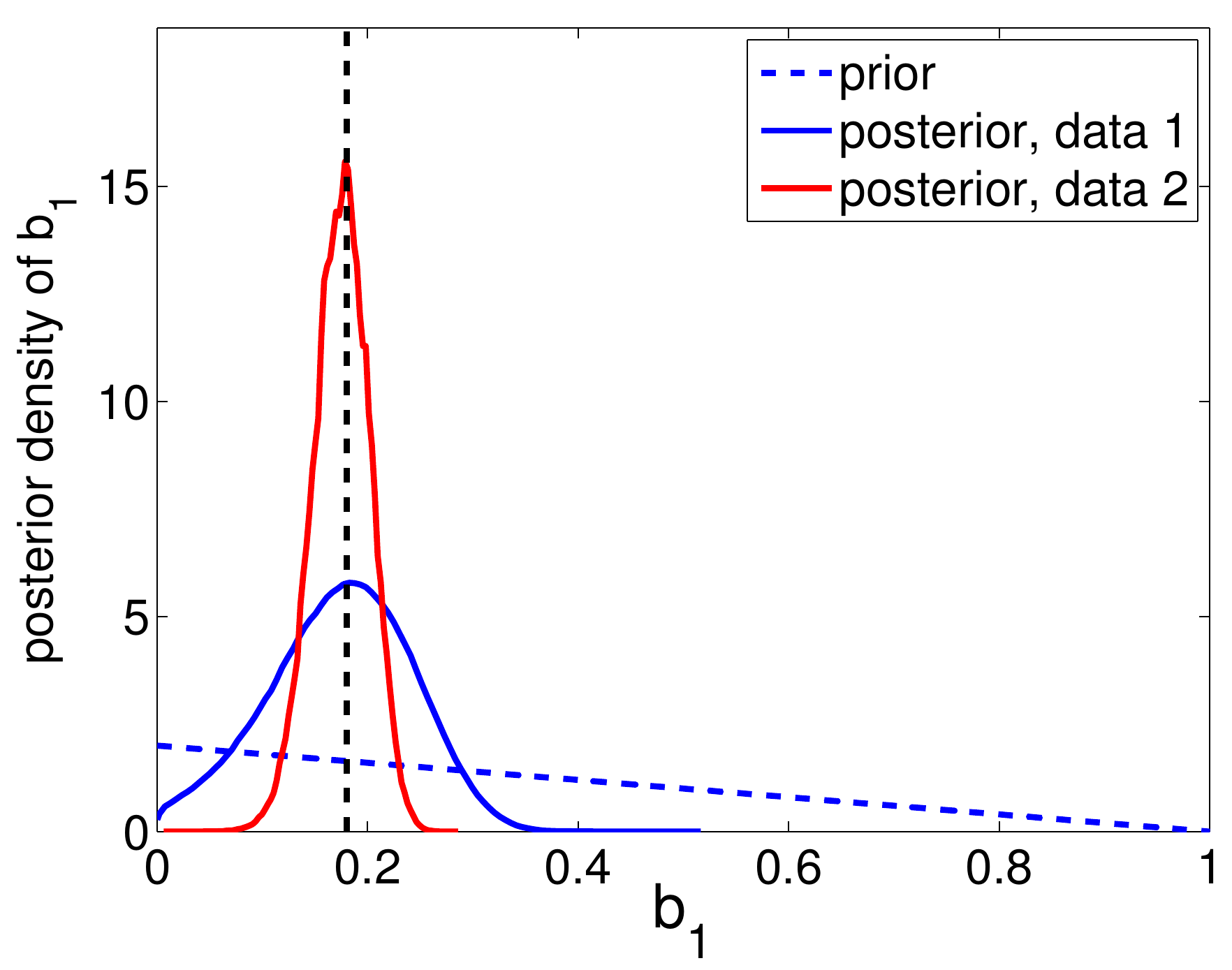}
\includegraphics[scale=0.25]{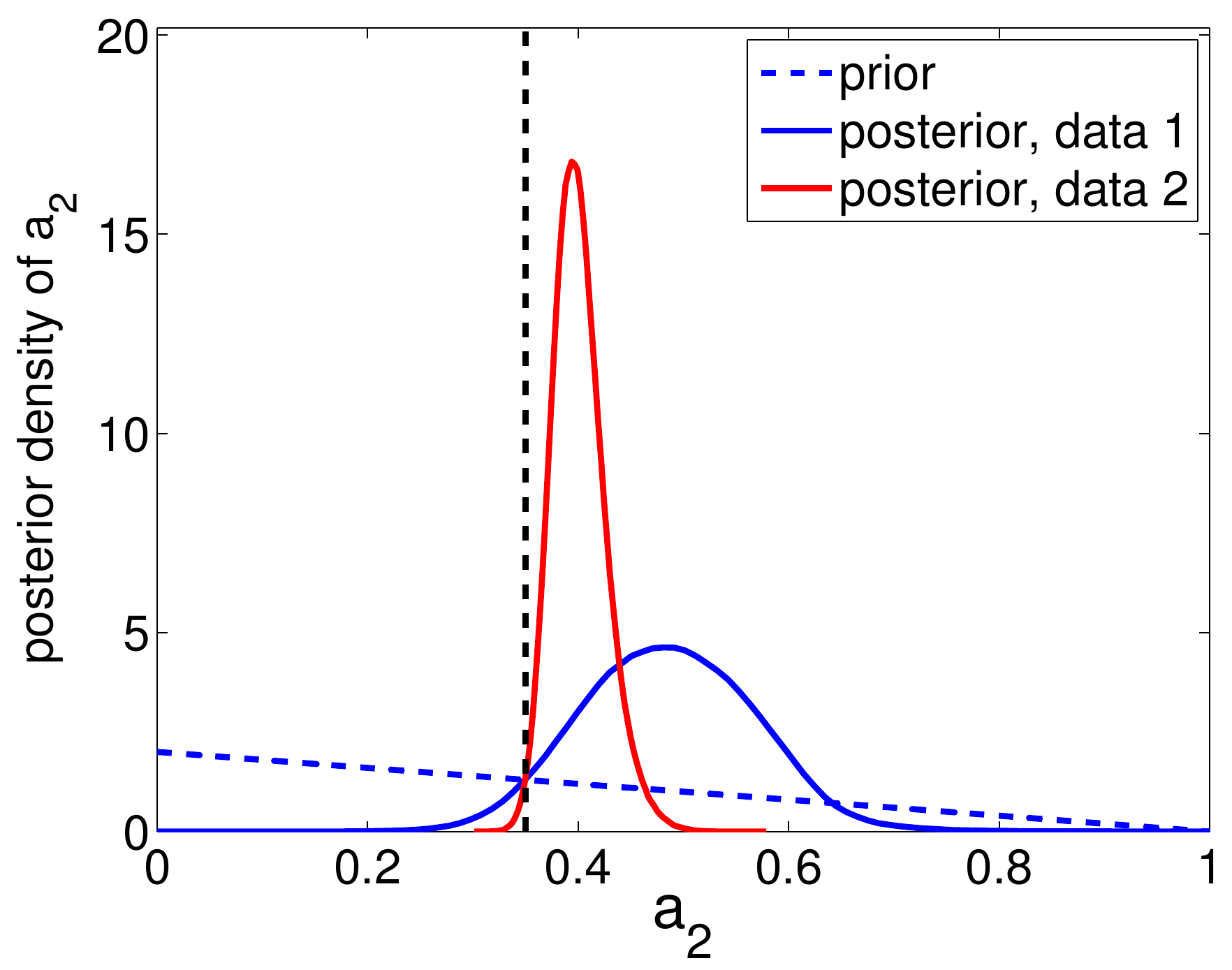}
\includegraphics[scale=0.25]{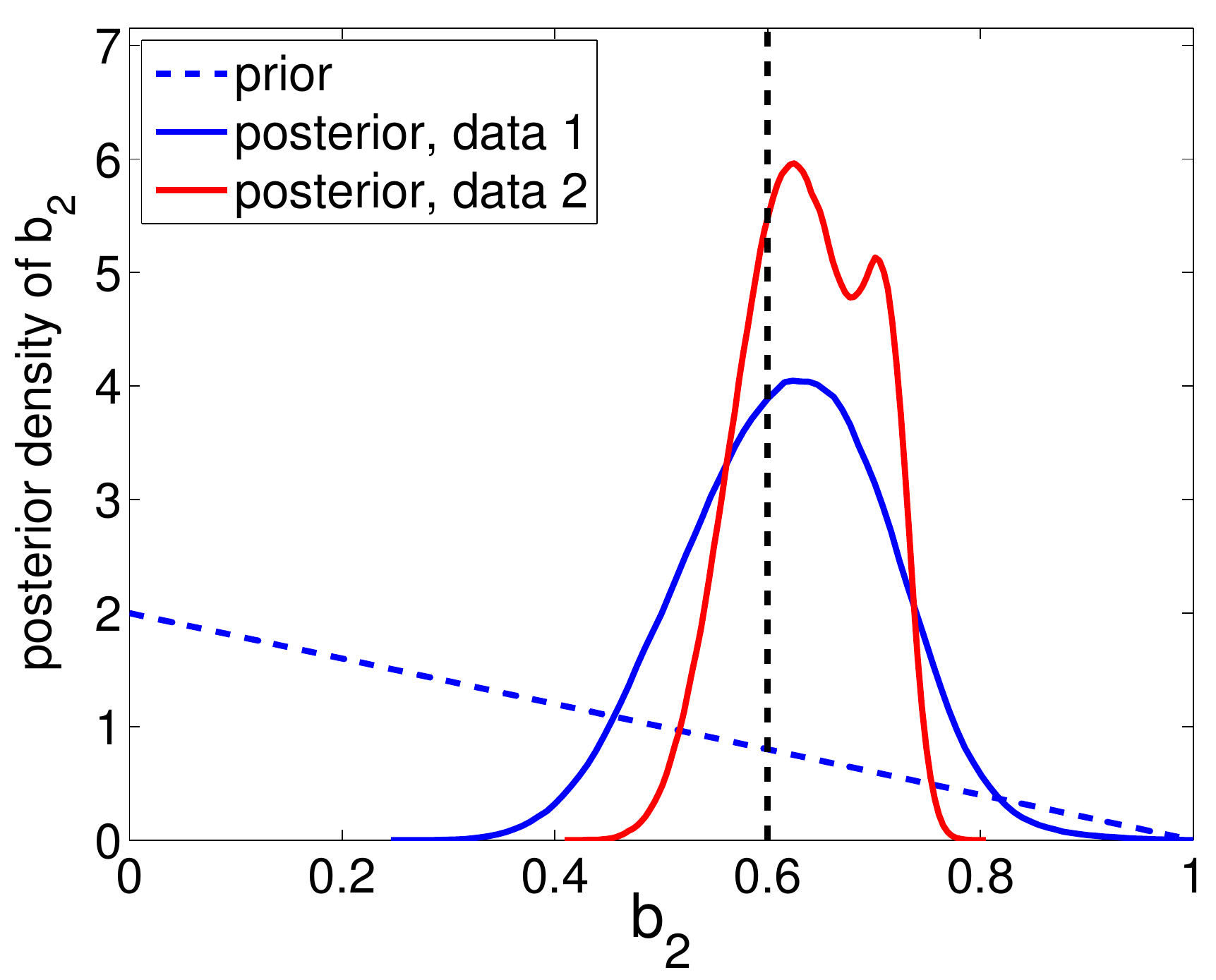}
\includegraphics[scale=0.25]{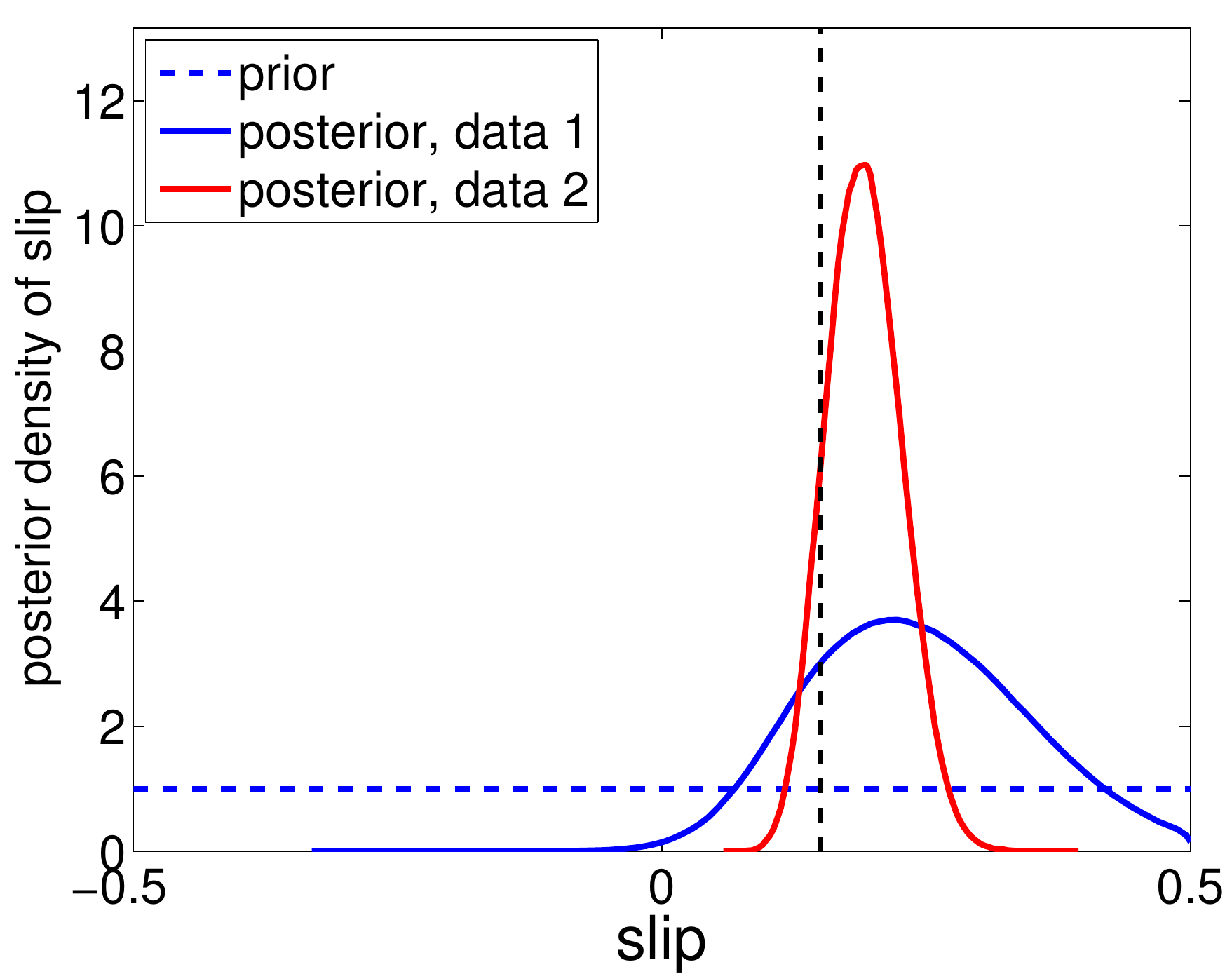}
\includegraphics[scale=0.25]{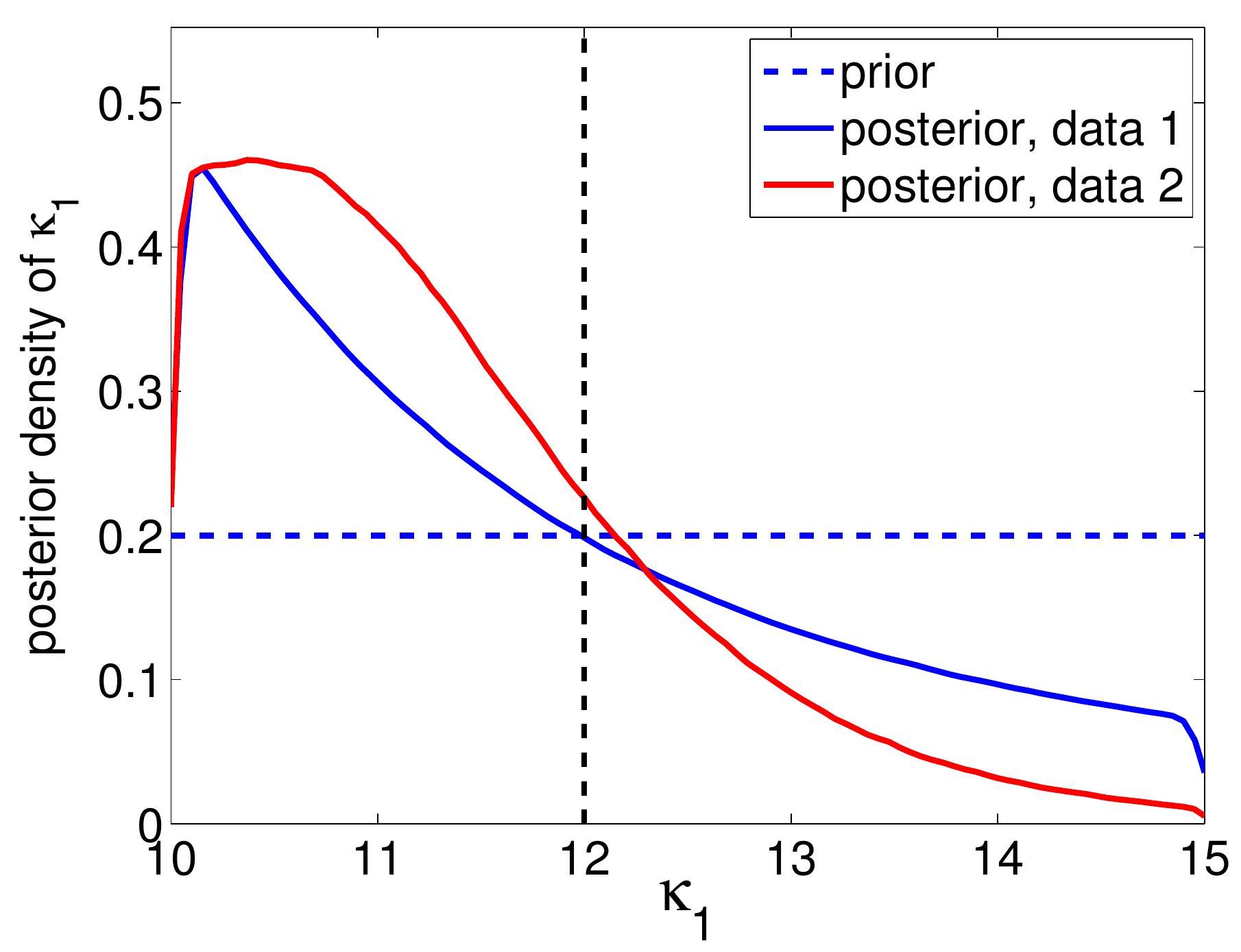}
\includegraphics[scale=0.25]{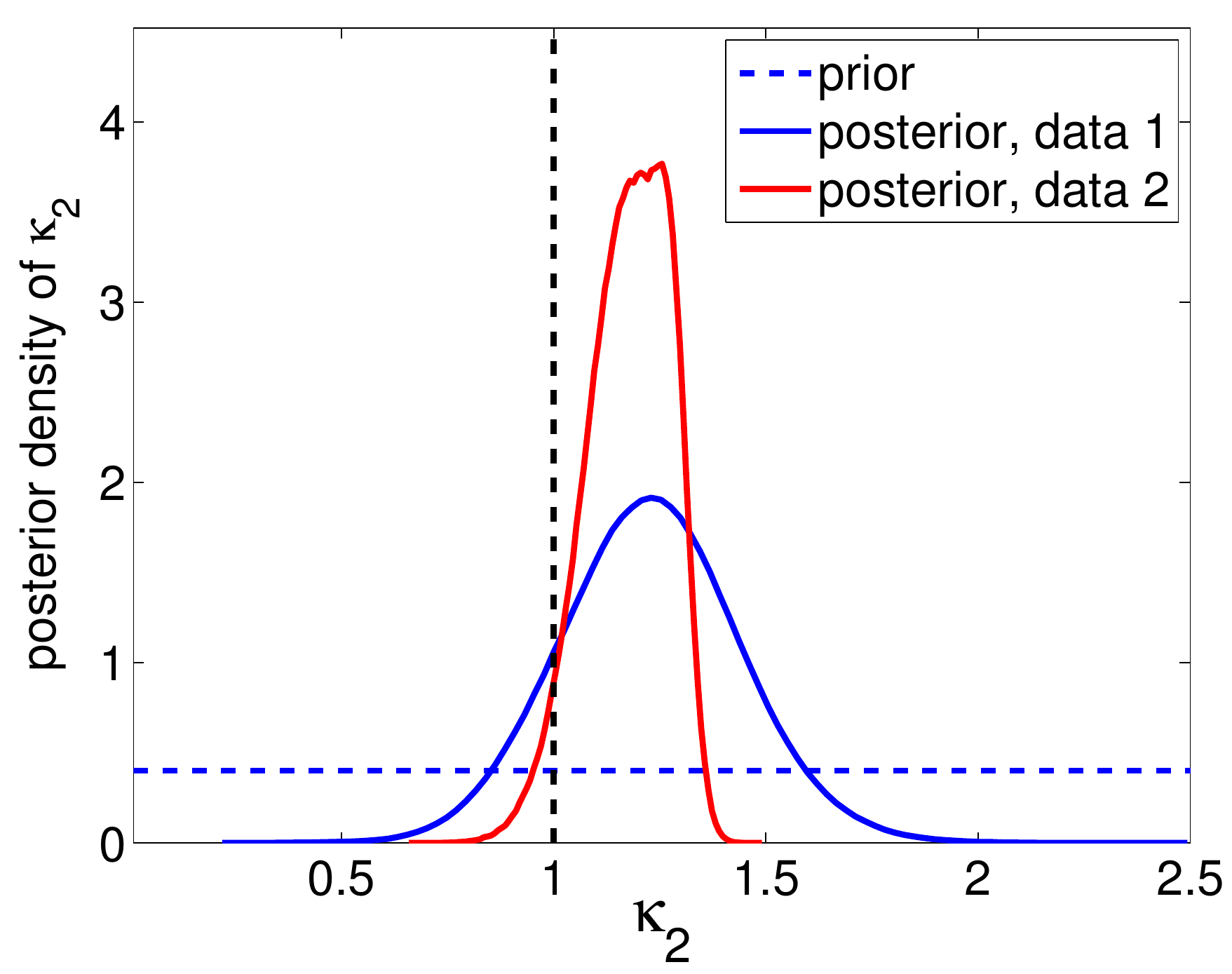}
\includegraphics[scale=0.25]{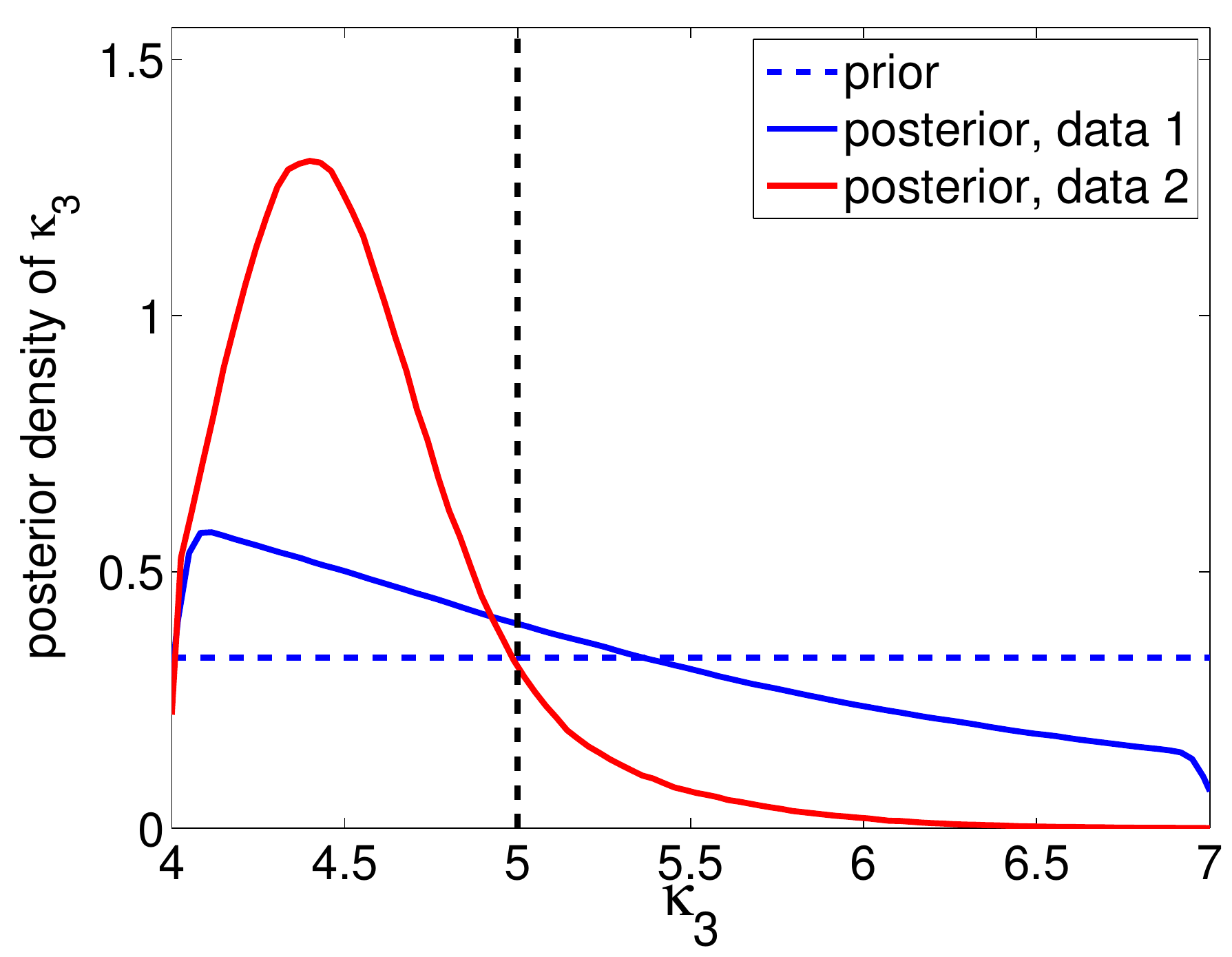}

 \caption{Prior and posterior densities of the unknown $u$. Vertical line indicates the true value of $u$.}
    \label{Figure5}
\end{center}
\end{figure}

\begin{figure}[htbp]
\begin{center}
\includegraphics[scale=0.25]{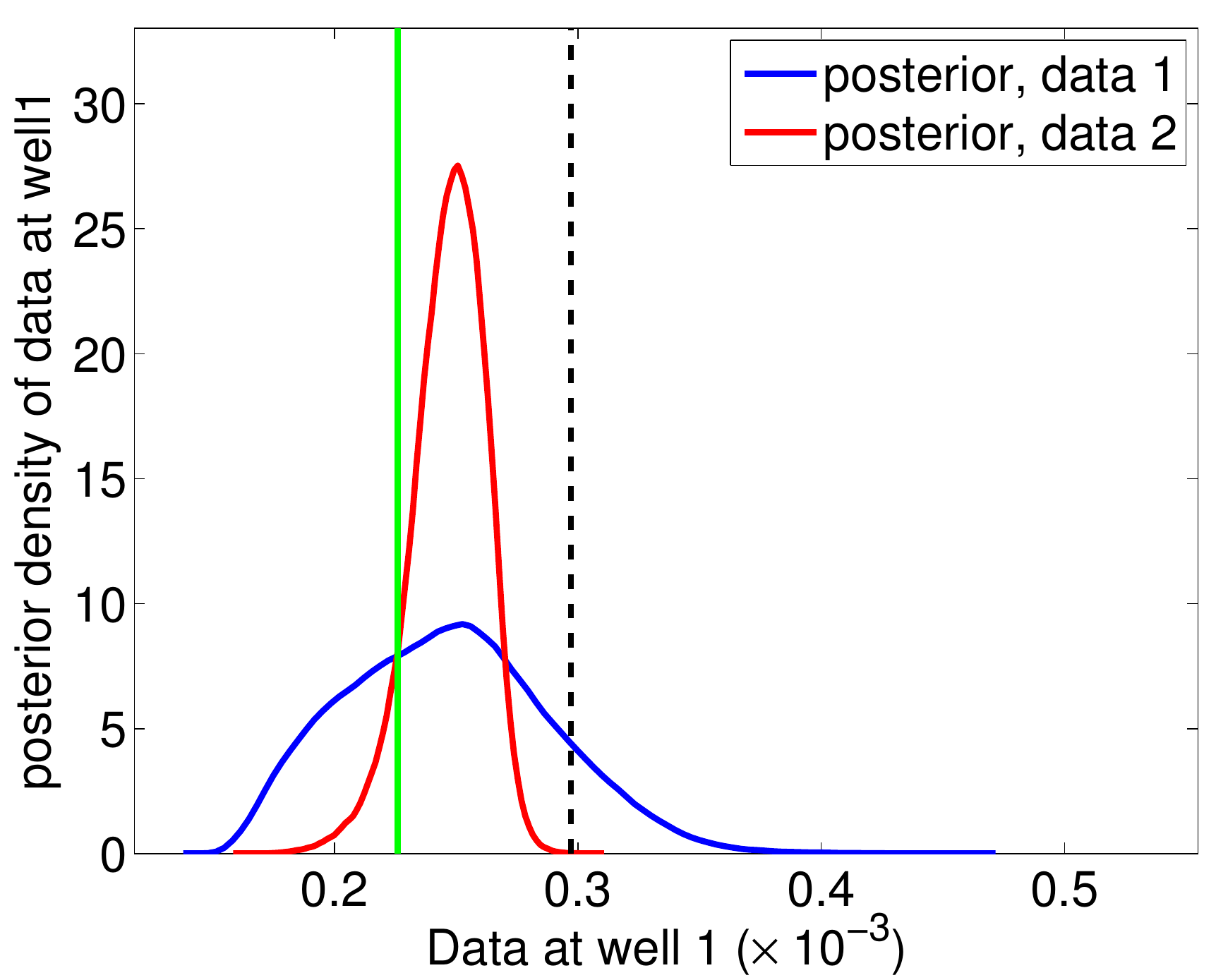}
\includegraphics[scale=0.25]{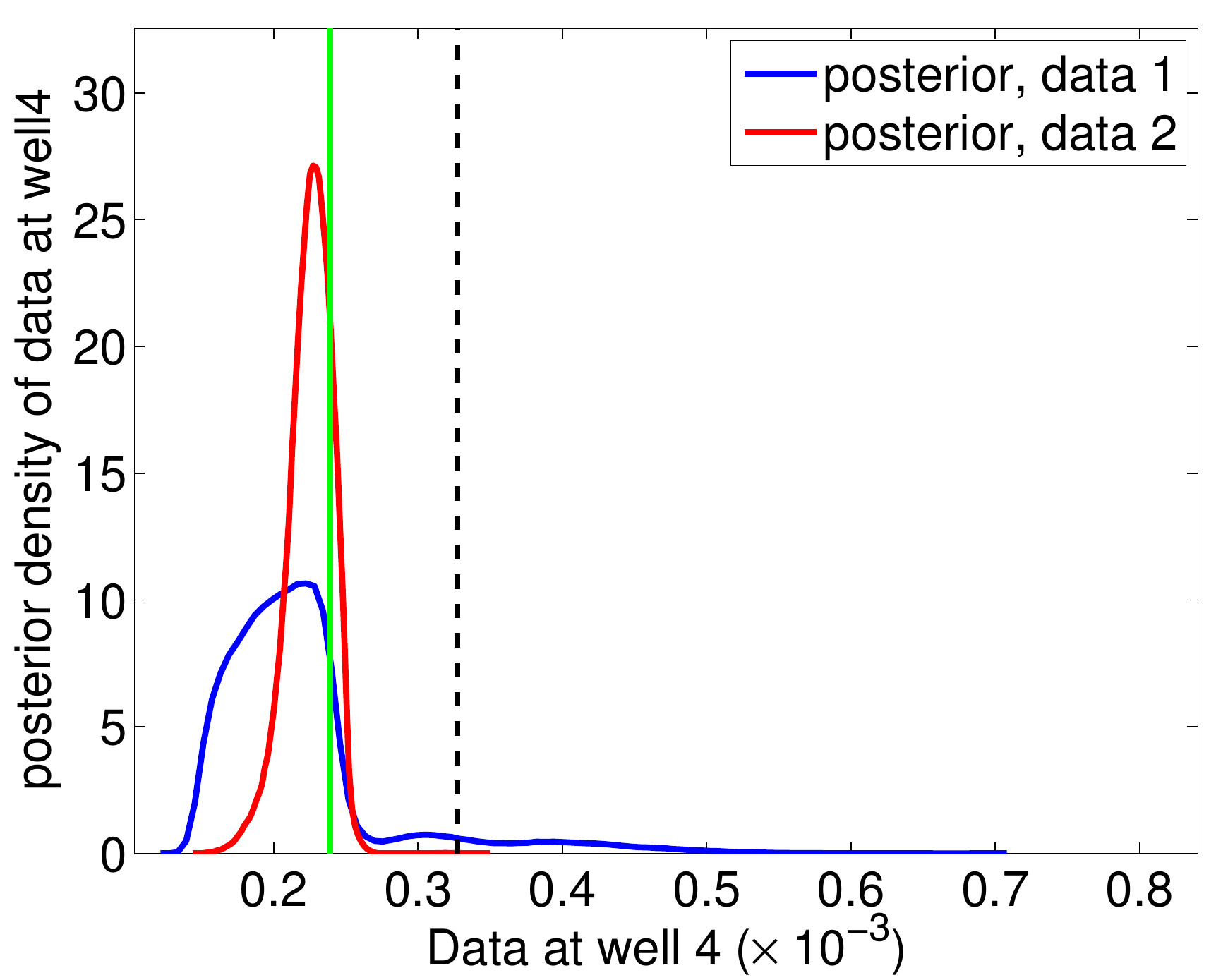}
\includegraphics[scale=0.25]{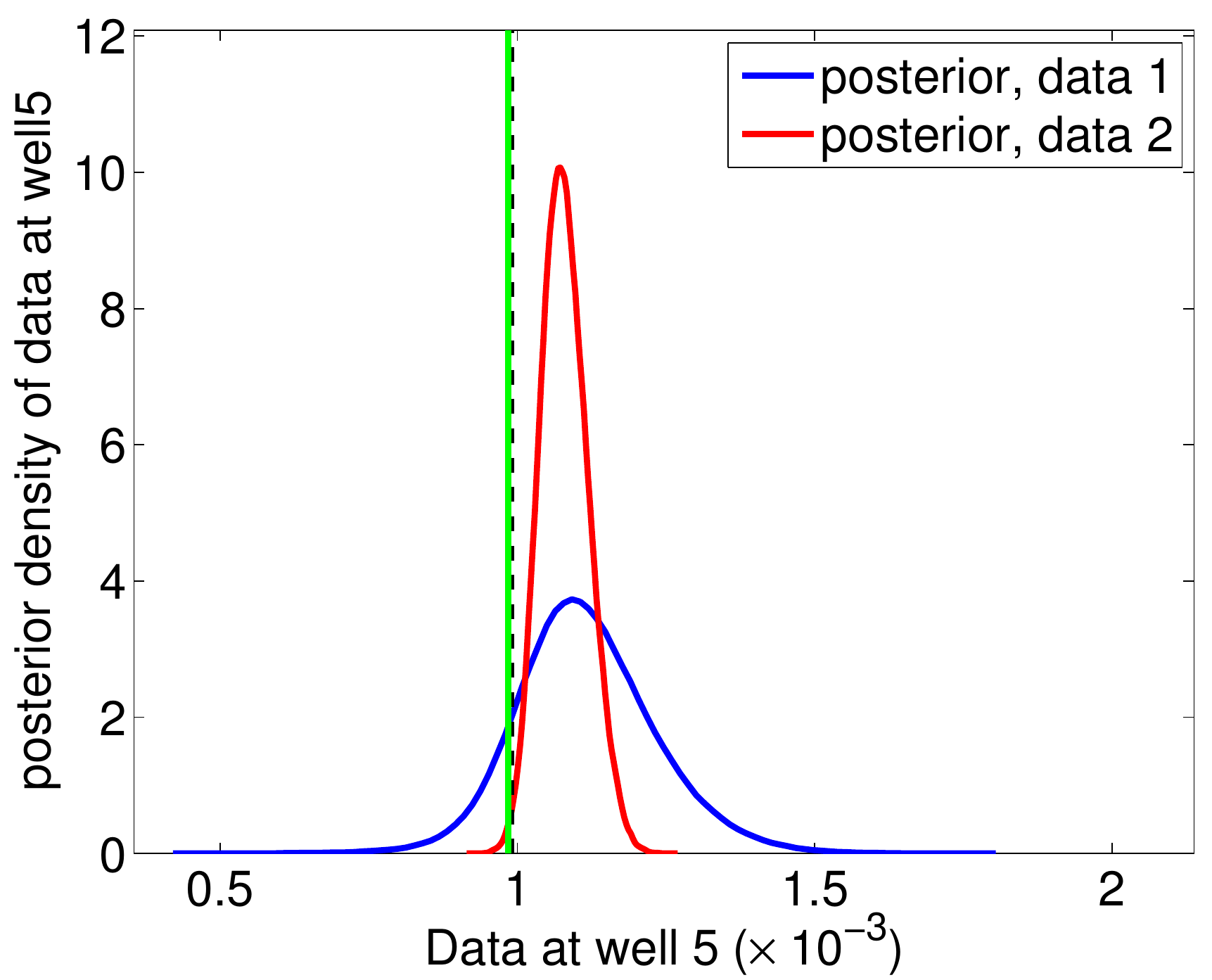}
\includegraphics[scale=0.25]{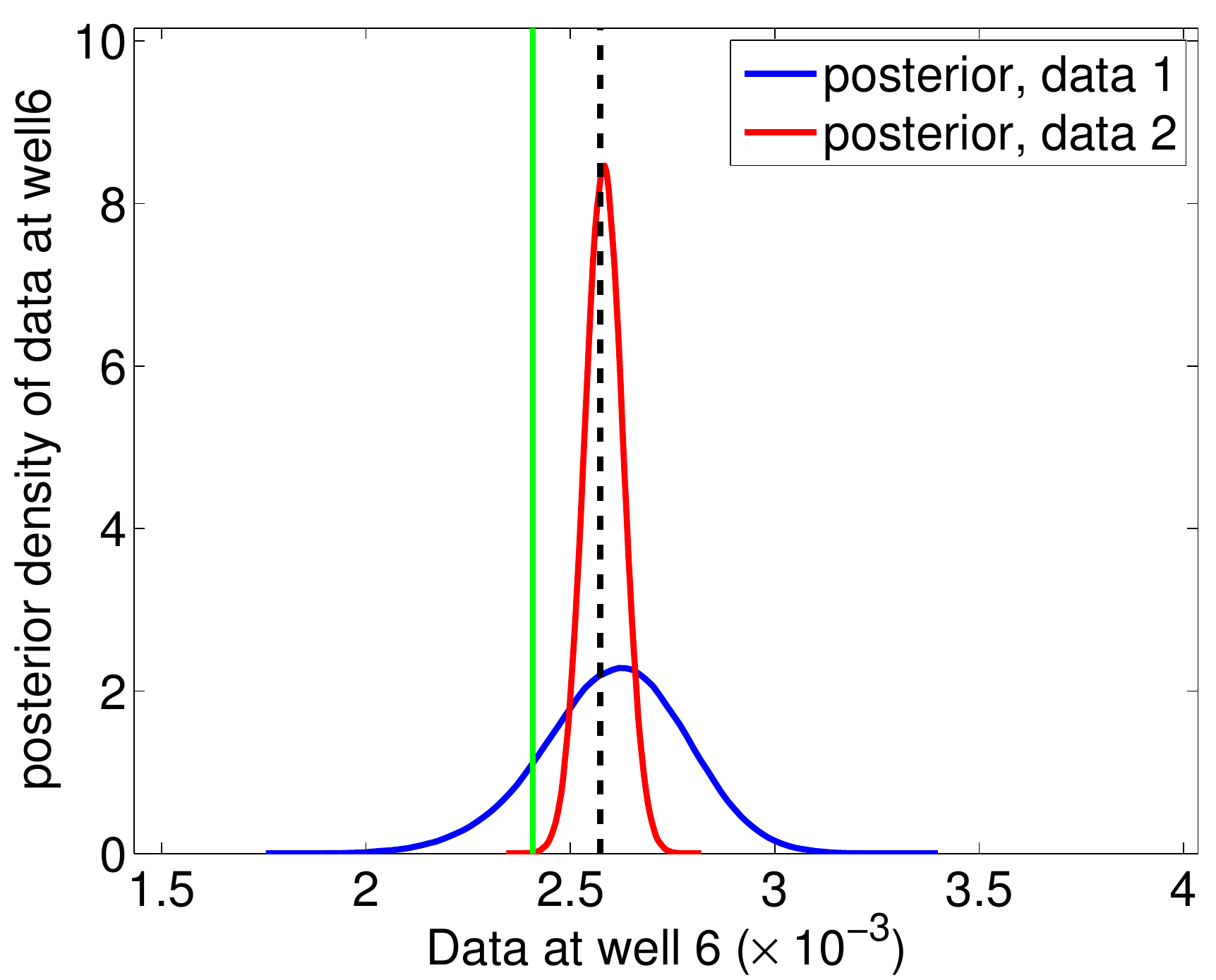}
\includegraphics[scale=0.25]{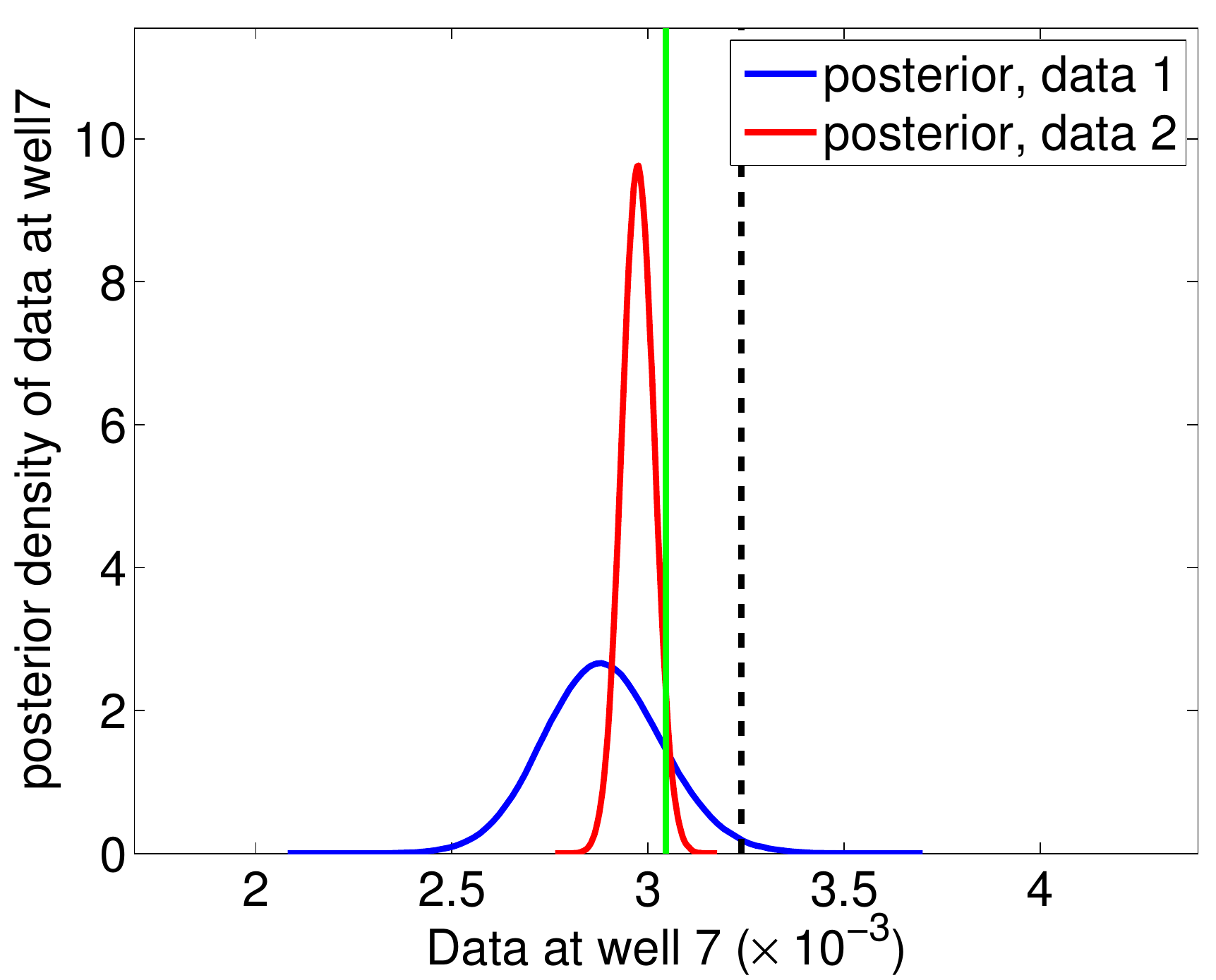}
\includegraphics[scale=0.25]{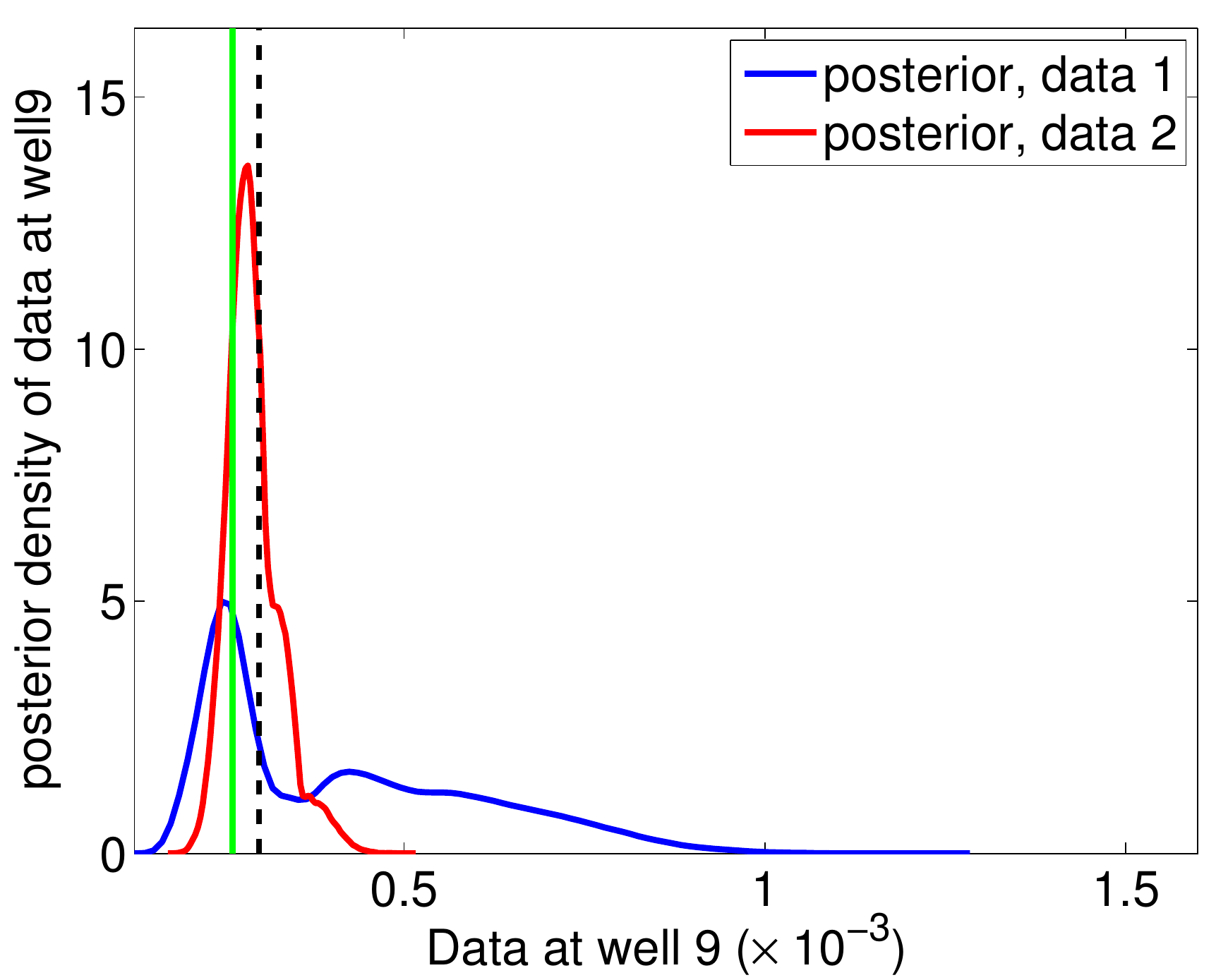}
\includegraphics[scale=0.25]{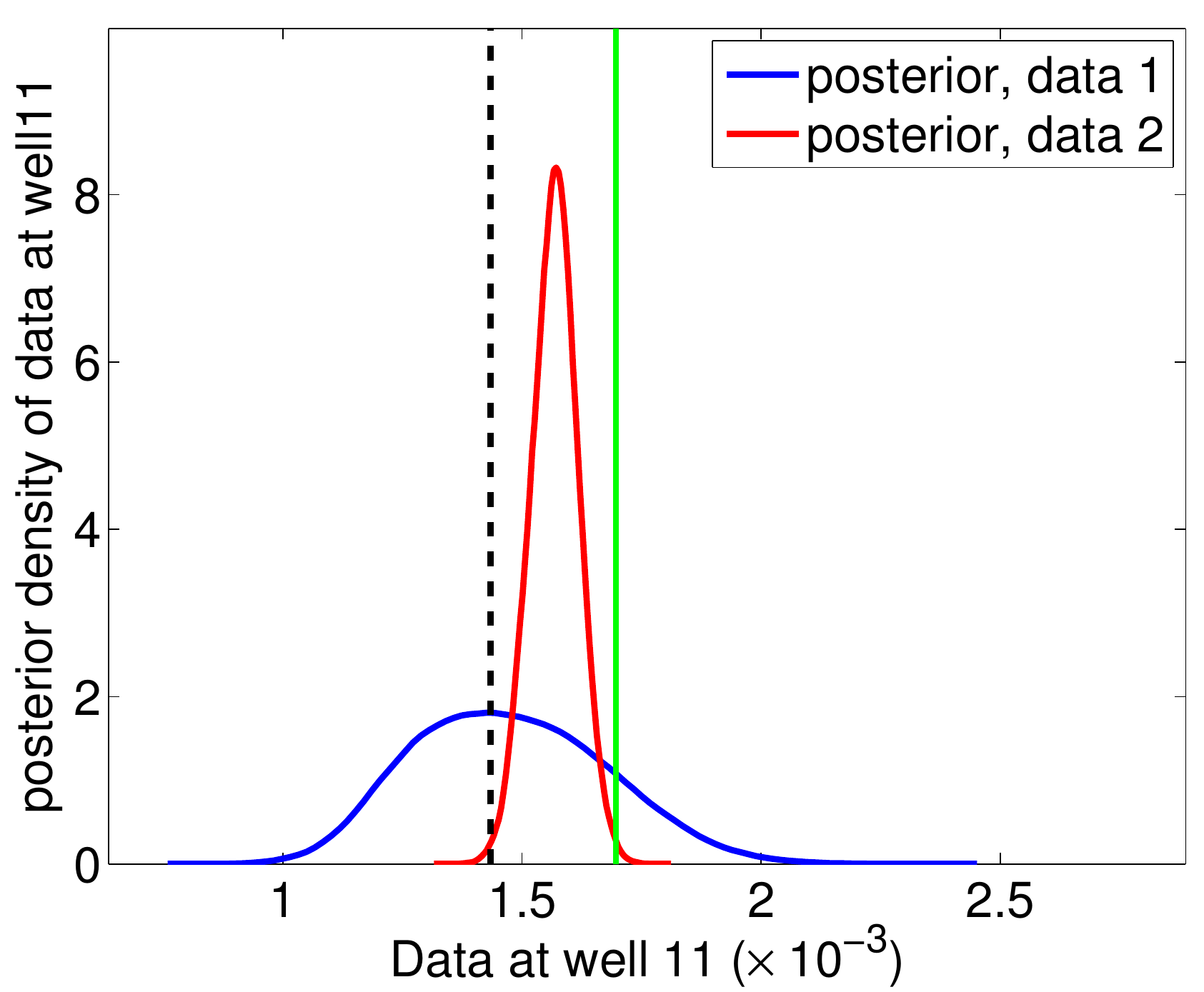}
\includegraphics[scale=0.25]{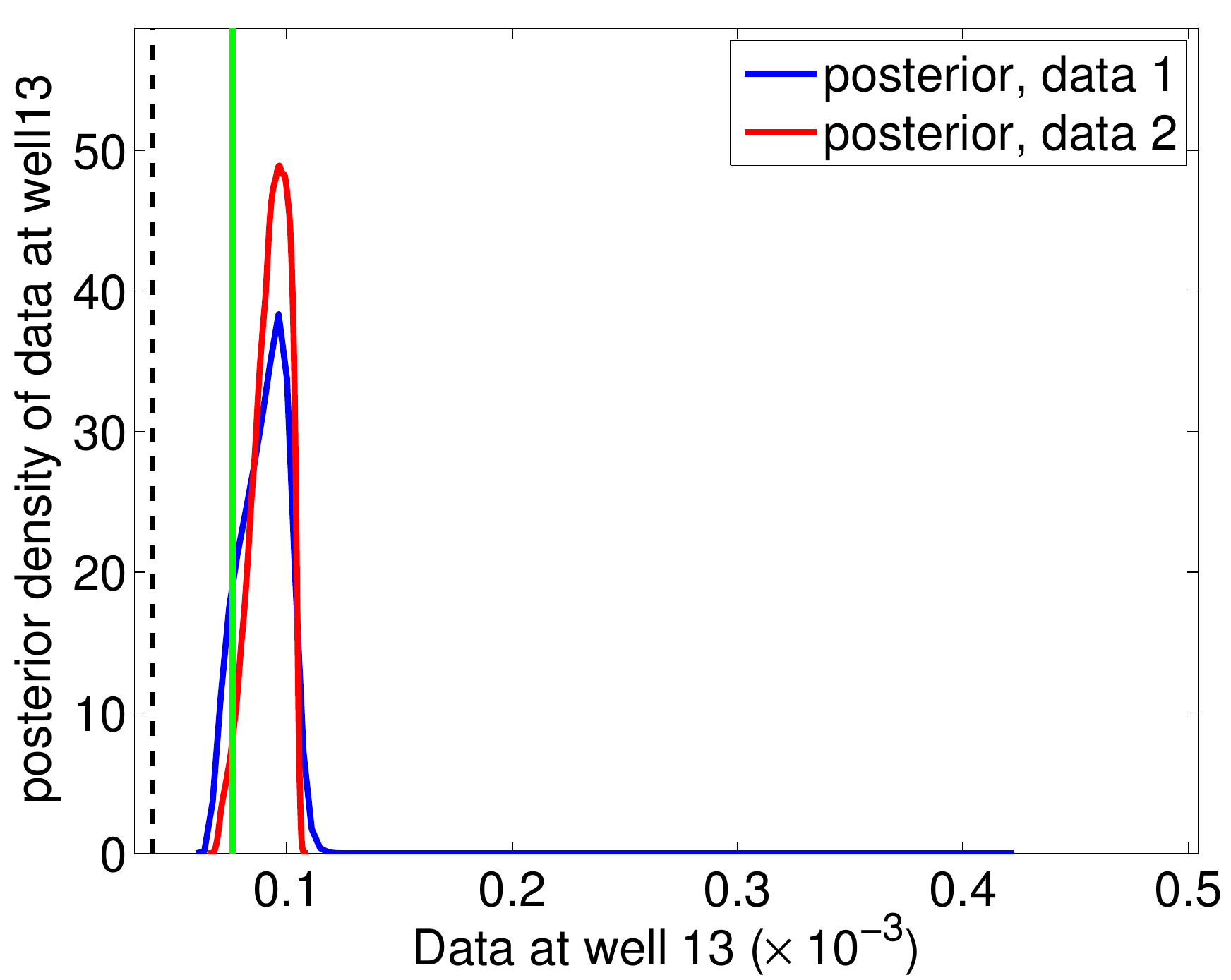}
\includegraphics[scale=0.25]{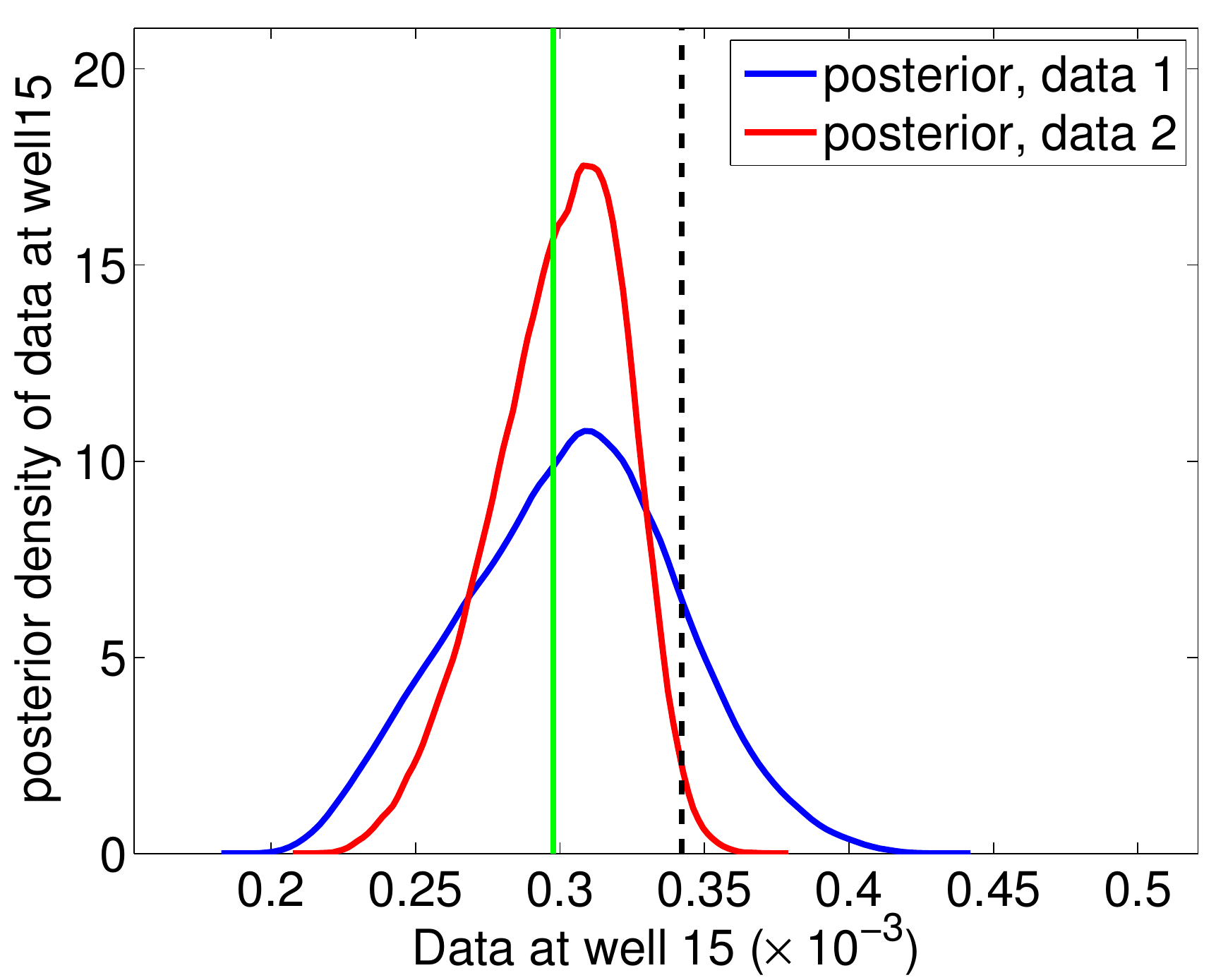}
 \caption{Posterior densities of the data predictions $G(u)$. Vertical lines indicate the nominal value of synthetic observations: dotted black (less accurate), green (more accurate) }
    \label{Figure6}
\end{center}
\end{figure}

\begin{figure}[htbp]
\begin{center}
\includegraphics[scale=0.35]{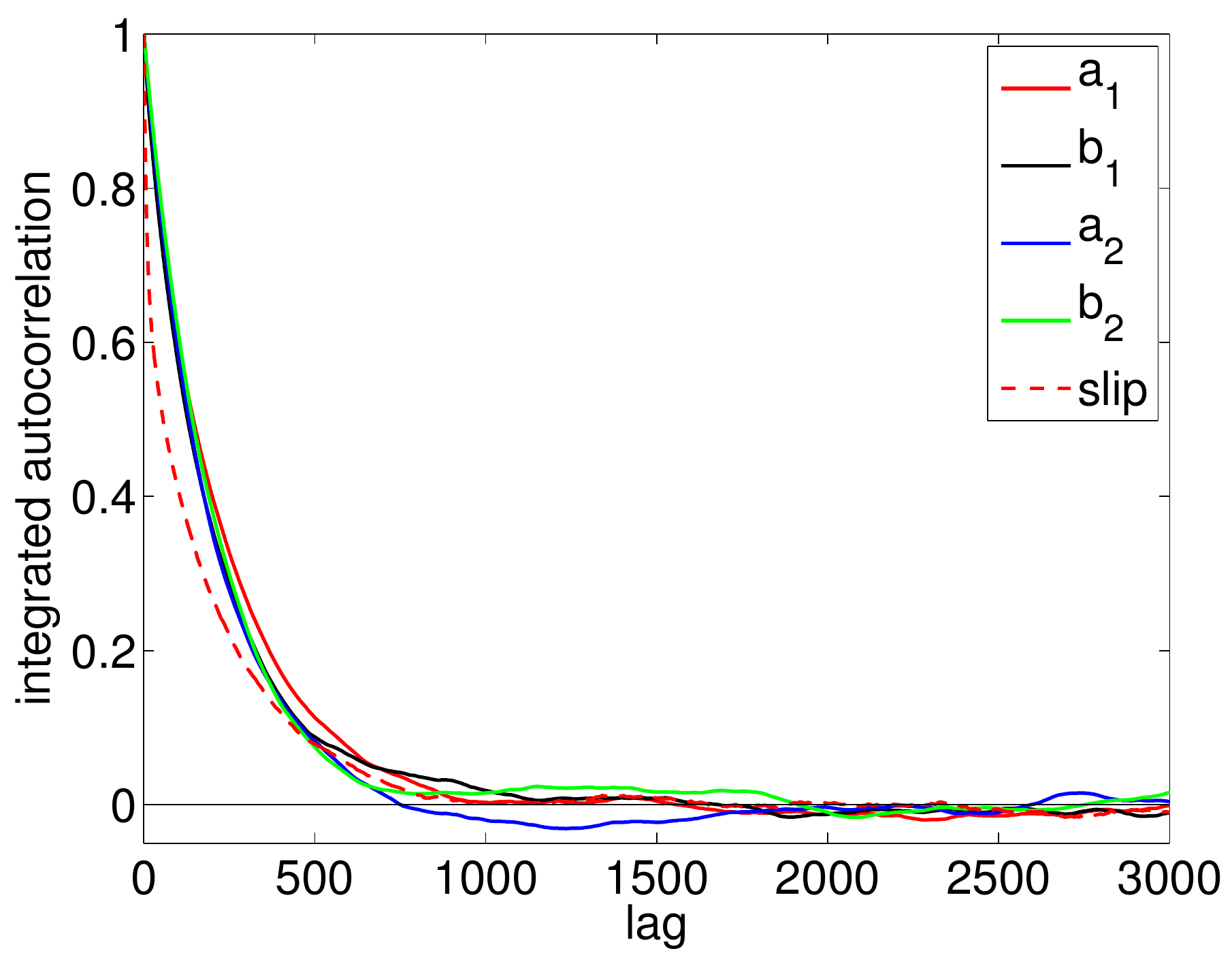}
\includegraphics[scale=0.35]{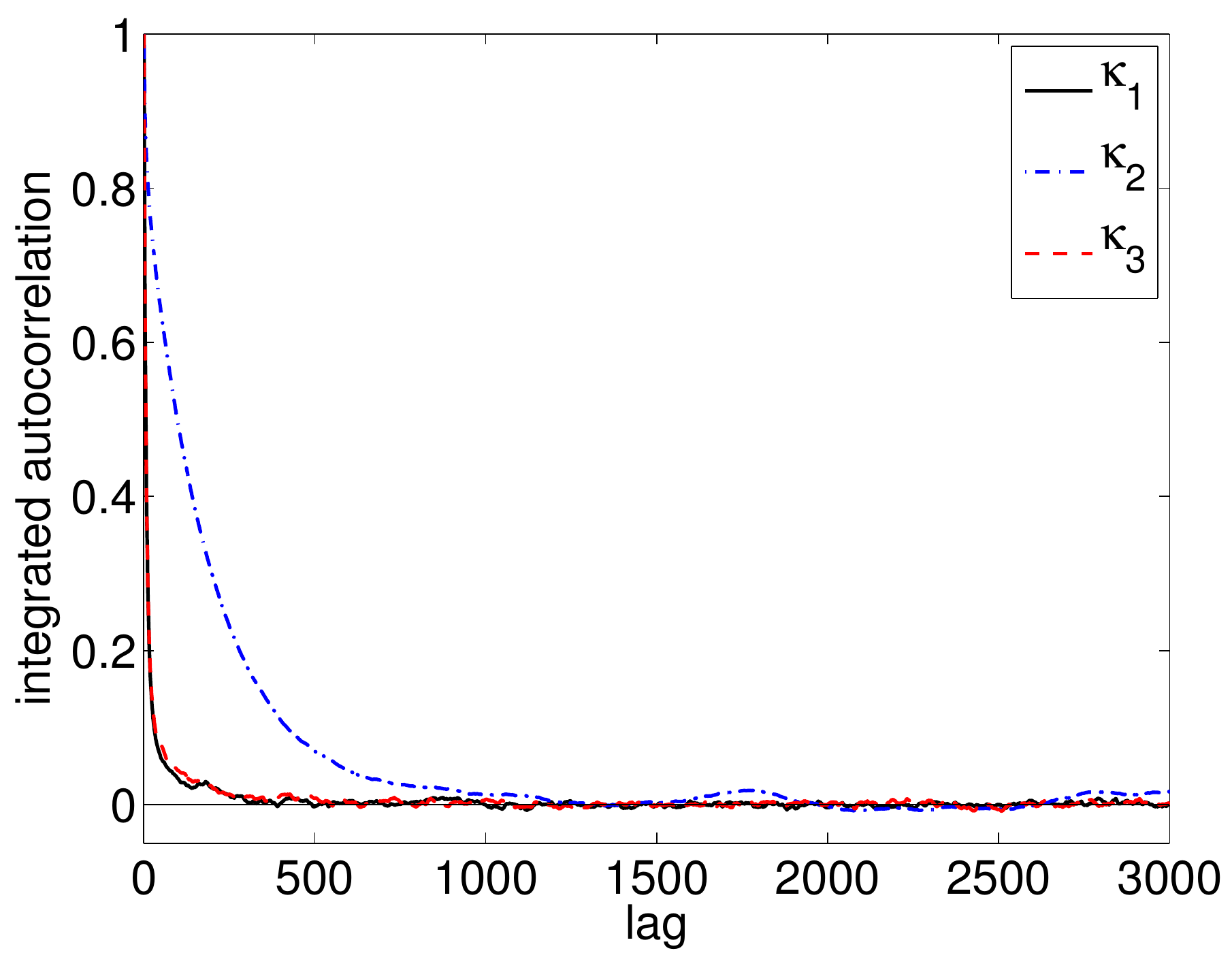}\\
\includegraphics[scale=0.35]{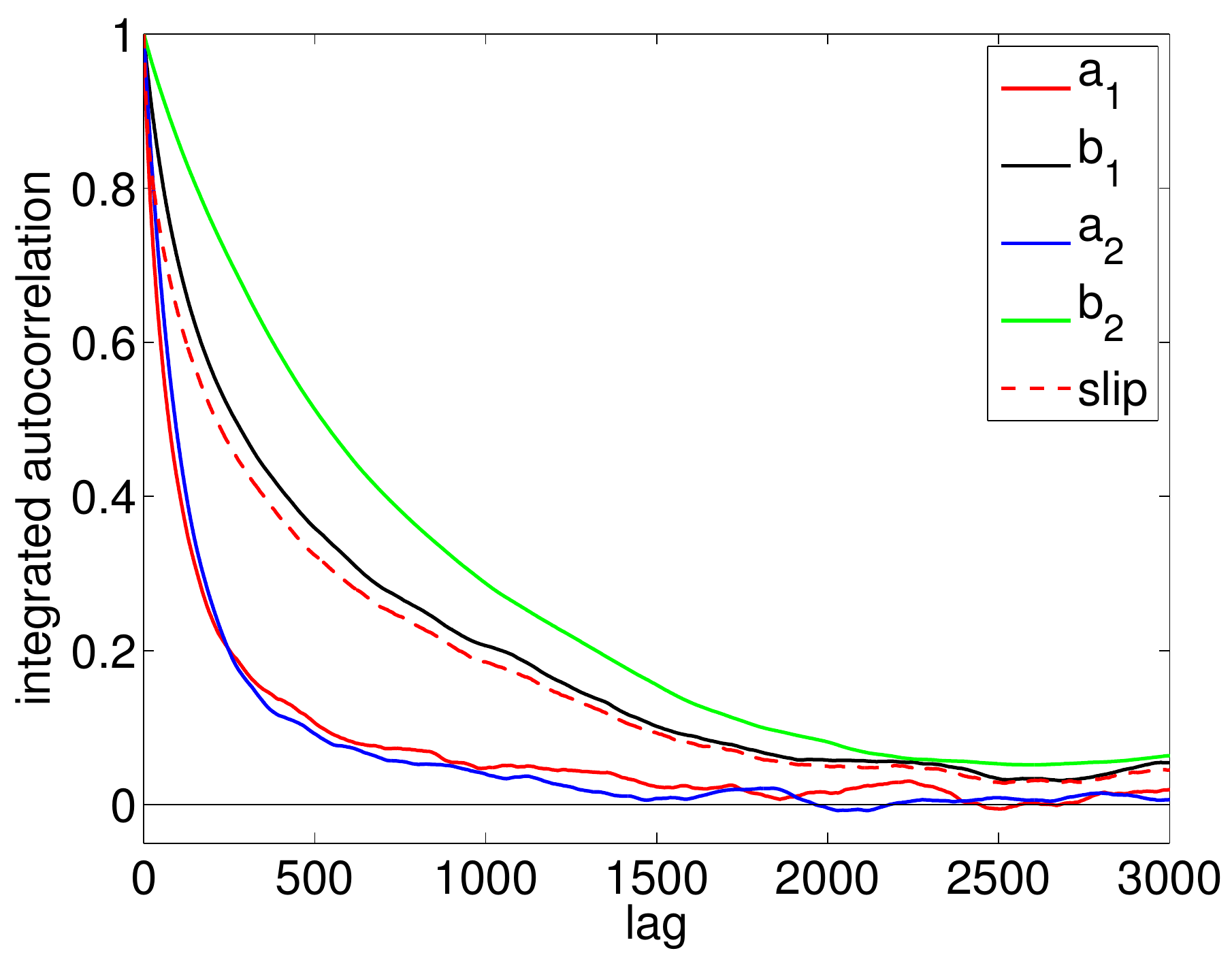}
\includegraphics[scale=0.35]{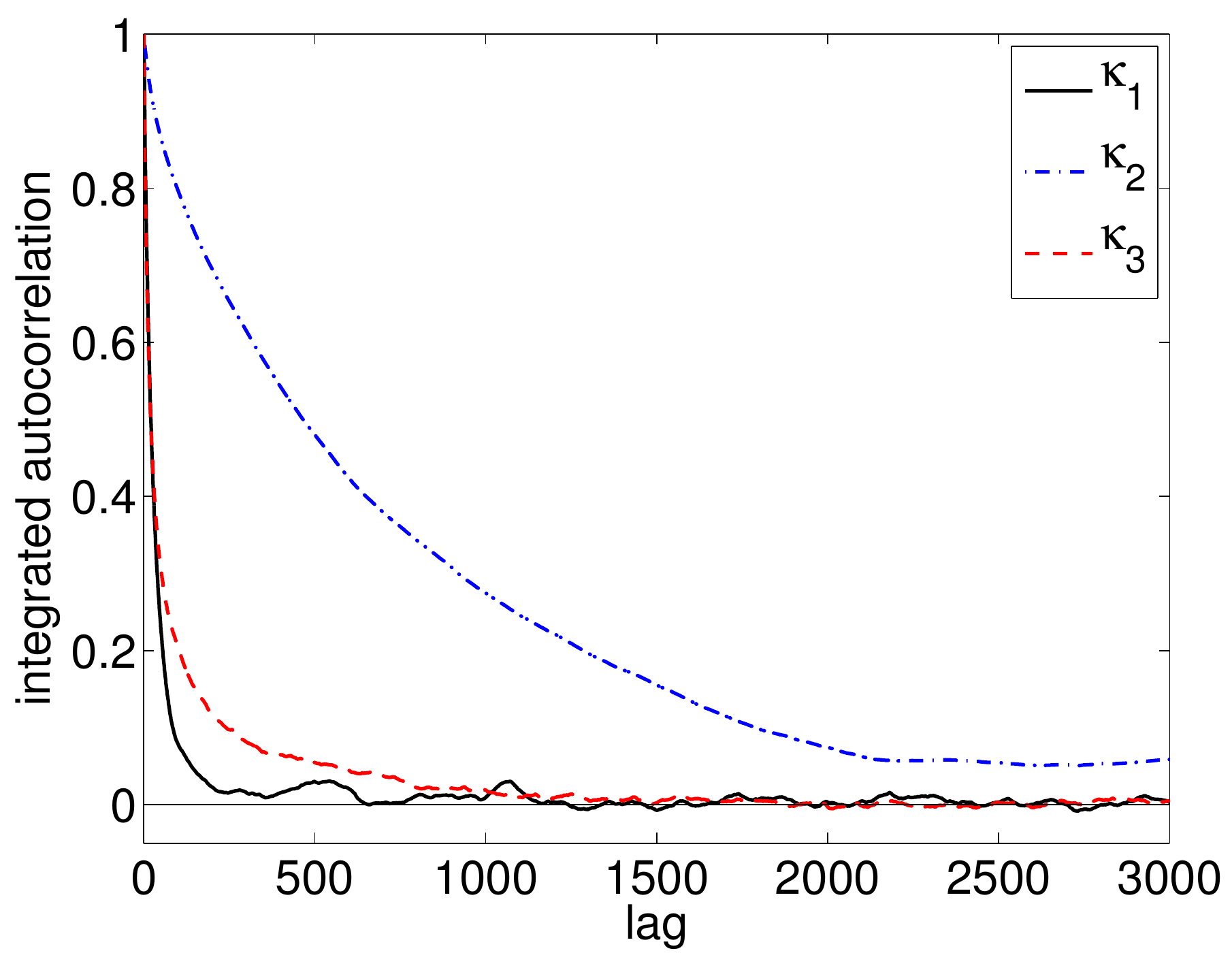}
 \caption{Autocorrelation from one MCMC chain. Top: data 1. Bottom: data 2. Left: geometric parameters. Right: permeability values.}
    \label{Figure7}
\end{center}
\end{figure}
\clearpage

%\subfigure[Layer Model]
%\label{fig:test1_layer}
%\subfigure[Layer Model with a fault]
%\label{fig:test2_fault}
%\subfigure[Channelized Model]
%\label{fig:test3_channel}

\subsection{Example. Two-layer Model With Spatially-varying Permeabilities}
This is an infinite-dimensional example where the permeability is described by
\begin{equation}\label{eq:num3}
\kappa(x)=\kappa_{1}(x) \chi_{D_1}(x)+\kappa_2(x) \chi_{D_2}(x)
\end{equation}
where $D_{1}$  and $D_{2}$ are the subsets defined by $a$ and $b$ as shown in Figure \ref{fig:test1_layer}. For this case we consider the case that $c=0$ is a known parameter. The unknown parameter in this case is $u=(a,b,\log \kappa_{1}, \log \kappa_{2})\in [0,1]^2\times C(\bD;\bbR^2)$.

We consider a prior measure of the form
\begin{equation}\label{eq:num4}
\mu_{0}(du)= \pi_0^{A,g}(a)da~\pi_0^{A,g}(b)db~\delta(c)\otimes N(m_1,C_{1}) N(m_2,C_{2})
\end{equation}
where $\delta$ is the Dirac distribution ($c=0$ is known) and where $A=[0,1]$ and $C_{1}$ and $C_{2}$ are covariance operators defined from correlation functions typical in geostatistics \cite{Oliver}. In concrete the correlation functions of $C_{1}$ (resp. $C_{2}$) is spherical (resp. exponential) and its maximum correlation is along $\pi/4$ (resp. $3\pi/4$) \cite{Oliver}. It is convenient, both conceptually and computationally, to parameterize the
log permeabilities via their Karhunen-Loeve (KL) expansions of the form
\begin{equation}\label{eq:num5}
\log(\kappa_{i}(x)) =m_i+\sum_{j=1}^{\infty}\sqrt{\lambda_{j,i}} v_{j,i}(x)\xi_{j,i}
\end{equation}
where $\lambda_{j,i}$ and $v_{j,i}(x)$ are the eigenvalues and eigenfunctions of $C_{i}$ and $\xi_{j,i}\in \mathbb{R}$. We assume that in the previous expression $\{\lambda_{j,i}\}_{j=1}^{\infty}$ are ordered as $\lambda_{1,i}\ge\lambda_{2,i} \ge\cdots$. Under this representation, $\xi_{j,i}\sim N(0,1)$ produces $\log \kappa_{i}\sim N(m_i,C_{i})$. Thus, the KL representation enables us to sample from the prior distribution and therefore generate proposals in Algorithm \ref{MwG}. For the numerical implementation of the Bayesian inversion, we consider expression (\ref{eq:num5}) truncated to $N = 50^2$ terms corresponding to all the eigenvectors of the discretized (on a $50\times 50$ grid) covariance $C_{i}$. However we emphasize that the approach we adopt corresponds to a well-defined
limiting Bayesian inverse problem in which the series is not truncated and
the PDE is solved exactly. The theory of paper \cite{CDS09} may be used to quantify
the error arising from the truncation of the KL expansion and the approximation
of the solution of the PDE; this gives a distinct advantage to the
``apply then discretize'' approach advocated here since all sufficiently resolved
computations are approximating the same limiting problem.  In Figure \ref{Figure8} (top row) we show log-permeabilities computed with (\ref{eq:num3}) with parameter $u$ sampled from the prior (\ref{eq:num4}).
There is substantial variability in these samples.

The true log-permeability $\log \kappa^{\dagger}$, shown in Figure \ref{Figure9} (top-left),  is obtained from (\ref{eq:num3}) with $\log \kappa_{1}^{\dagger}$ displayed in Figure \ref{Figure10} (top-left) $\log \kappa_{2}^{\dagger}$ shown in Figure \ref{Figure10} bottom-left and from $a^{\dagger}=0.11$ and $b^{\dagger}=0.86$. The functions $\log \kappa_{1}^{\dagger}$ and $\log \kappa_{2}^{\dagger}$ are draws from the Gaussian measures $N(m_1,C_{1})$ and $N(m_2,C_{2})$ that we use to define the prior (\ref{eq:num3}). However, in order to avoid the inverse crime, these priors that we use to generate $\log \kappa_{1}^{\dagger}$ and $\log \kappa_{2}^{\dagger}$ (and so $u^{\dagger}$) are defined on a discretized domain (of $100\times 100$ grids) that is finer than the one used for the inversion. Synthetic data (data set 1) is generated by using the true permeability in the elliptic PDE and applying the measurement functional with 9 measurement locations as displayed in Figure \ref{Figure9} (bottom-left). Gaussian noise of standard deviation $\gamma=1.0\times 10^{-3}$ is added to the observations. For this case, we sample the posterior with the MCMC Algorithm \ref{MwG} using the following splitting of the unknown $u=(a,b,\log \kappa_{1}, \log \kappa_{2})$ for the outer
Gibbs loop: (i) $\log \kappa_{1}$, (ii) $\log \kappa_{2}$ and (iii) $(a,b)$. Similar to the previous experiment, splitting the unknown yields the best performance in terms of decorrelation times. As before, 20 parallel chains were generated and the MPSRF was computed for assessing the convergence of the independent chains. Trace plots from one of these chains are presented in Figure \ref{Figure11}. The samples from all chains were combined to produce the mean and variance of the unknown. The mean and variance of $\log \kappa_{1}$ and $\log \kappa_{2}$ are shown in Figure \ref{Figure10} (second column) and Figure \ref{Figure12}, respectively  (first column). The mean for the geometric parameters are $\hat{a}=0.198$ and $\hat{b}=0.639$. The corresponding variances are $\sigma_{a}=1.03\times 10^{-2}$ and $\sigma_{b}=9.5\times 10^{-3}$, respectively. The permeability (\ref{eq:num3}) corresponding to the mean $\hat{u}$ is displayed in Figure \ref{Figure9} (top-middle).

We repeat this experiment with a synthetic data (data set 2) that we now generate from a configuration of 36 measurement locations as specified in Figure \ref{Figure9} (bottom-right). Mean and variance of $\log \kappa_{1}$ and $\log \kappa_{2}$ are displayed in Figure \ref{Figure10} (third column) and Figure \ref{Figure12}  (second column), respectively. The mean and variances for the geometric parameters are $\hat{a}=0.088$ and $\hat{b}=0.822$ and $\sigma_{a}=8\times 10^{-4}$ and $\sigma_{b}=7\times 10^{-4}$, respectively. The permeability (\ref{eq:num3}) corresponding to the mean $\hat{u}$ is displayed in Figure \ref{Figure9} (top-right).

For both experiments, Figure \ref{Figure13} shows the autocorrelation of the geometric parameters as well as the KL-coefficients $\xi_{1,i}$ and $\xi_{10,i}$ from (\ref{eq:num5}). In Figure \ref{Figure14} we present the posterior and the prior densities for geometric parameters and the aforementioned KL coefficients. From this figure we note that for some of the components of the unknown (e.g. $a$, $b$, $\xi_{1,2}$), the corresponding posterior density tends to be more peaked around the truth when more measurements are assimilated. Some less sensitive parameters seemed to be unaffected when more measurements are included. In particular, we see that the densities for the coefficients of smaller eigenvalues (e.g. $\xi_{9,2}$, $\xi_{10,2}$) for the permeability in $D_{2}$ is almost identical to the prior; indicating the uninformative effect of the data.

\begin{figure}[htbp]
\begin{center}
\includegraphics[scale=0.2]{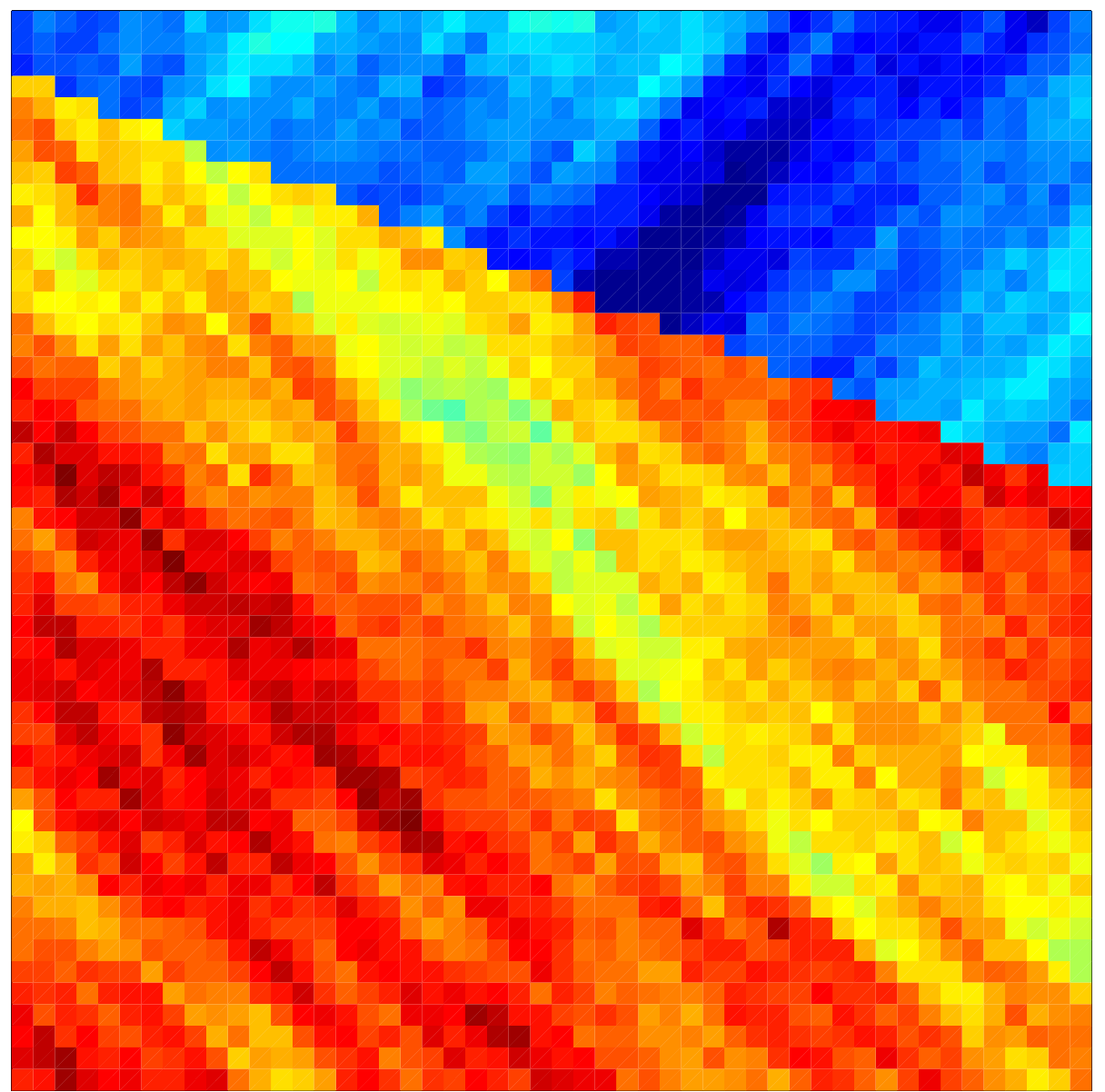}
\includegraphics[scale=0.2]{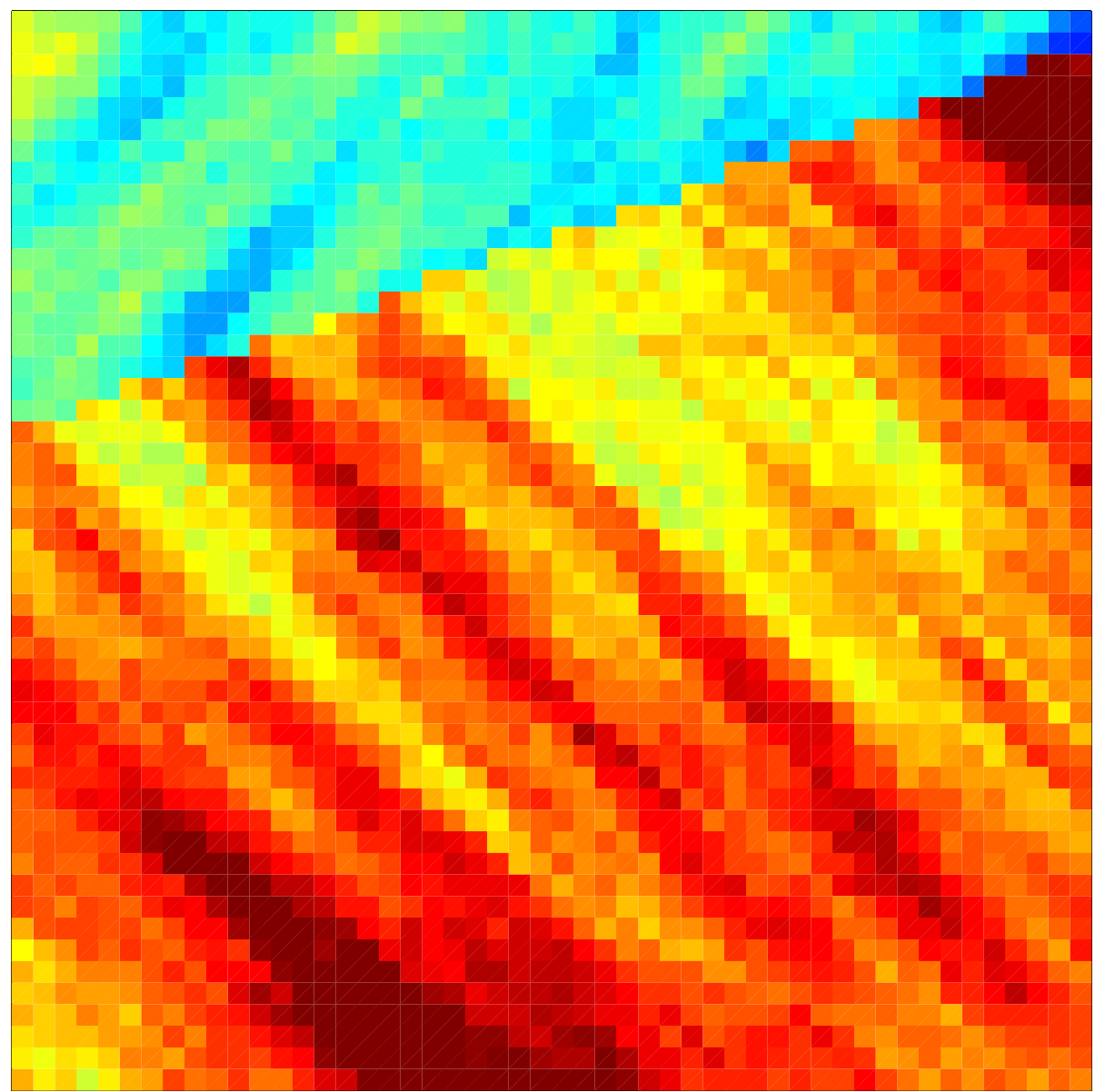}
\includegraphics[scale=0.2]{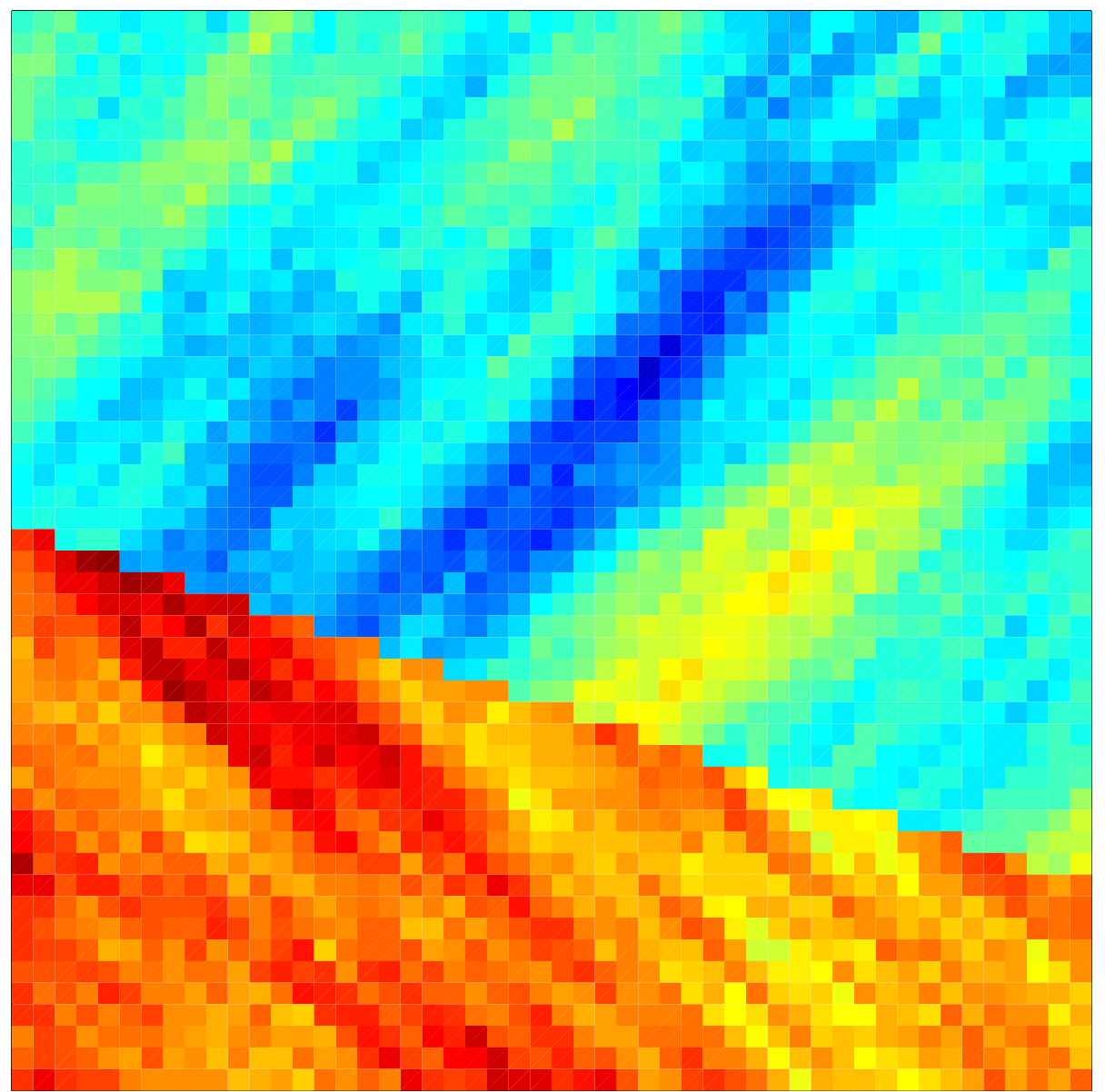}
\includegraphics[scale=0.2]{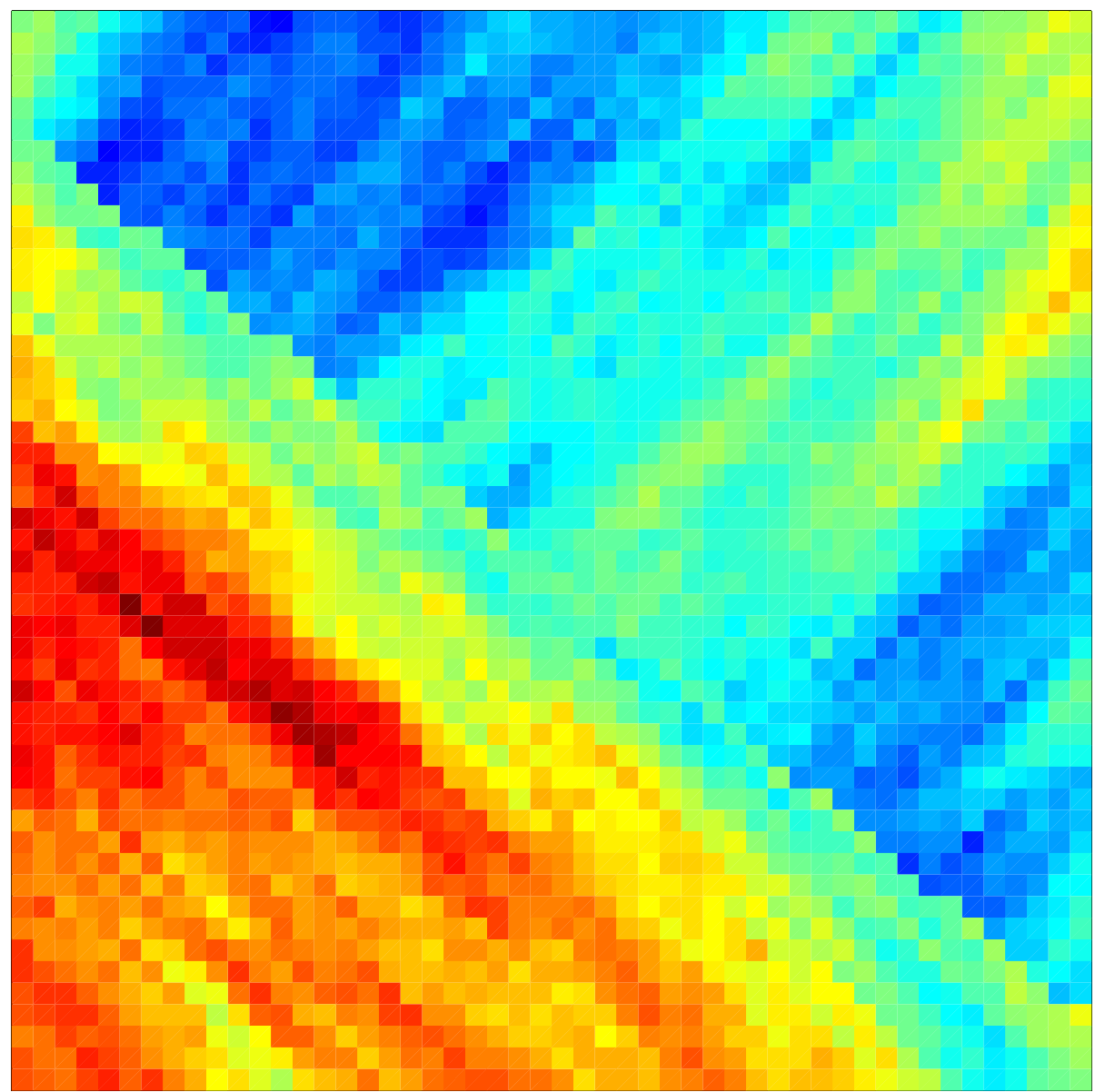}
\includegraphics[scale=0.2]{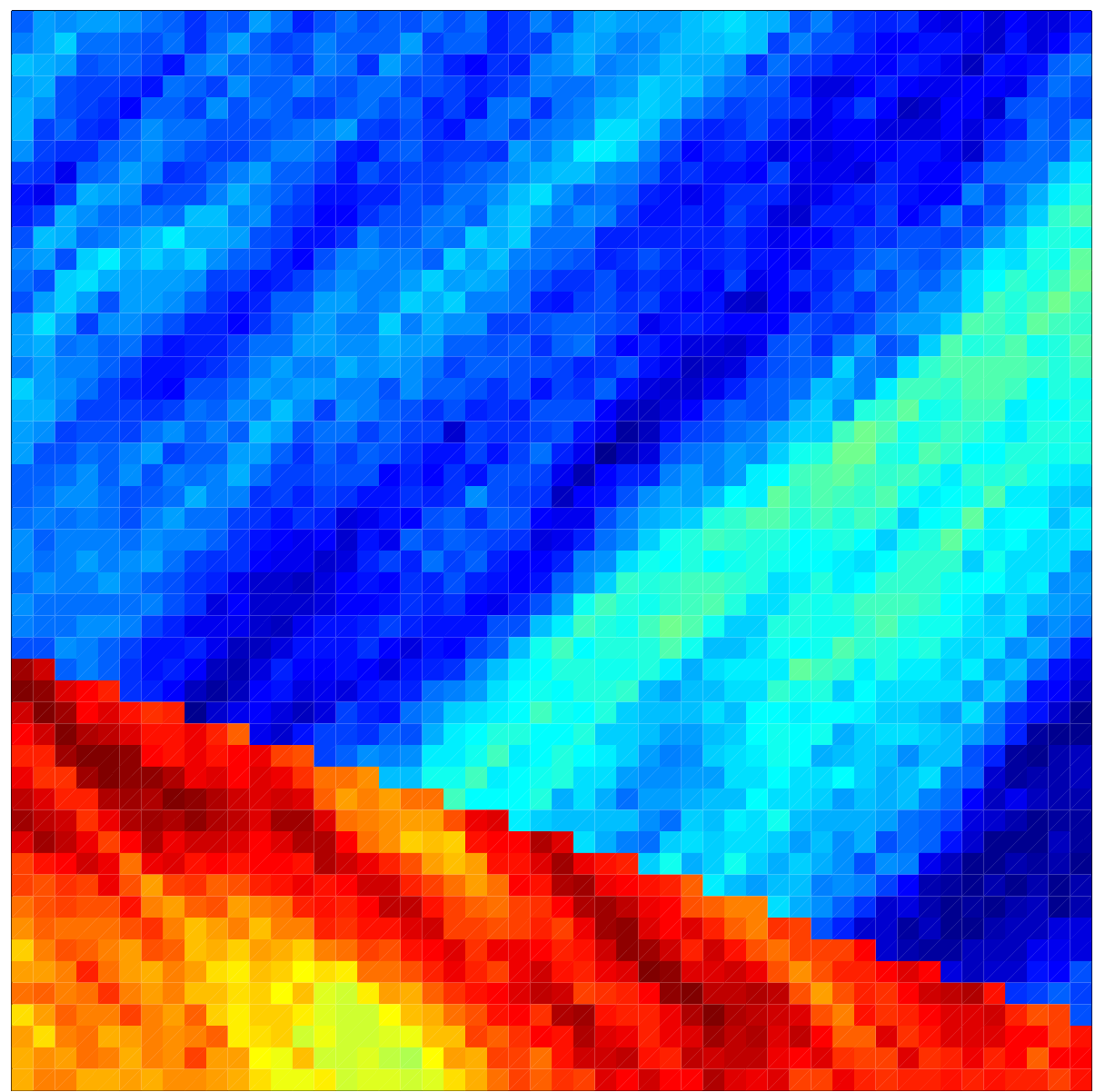}
\includegraphics[scale=0.2]{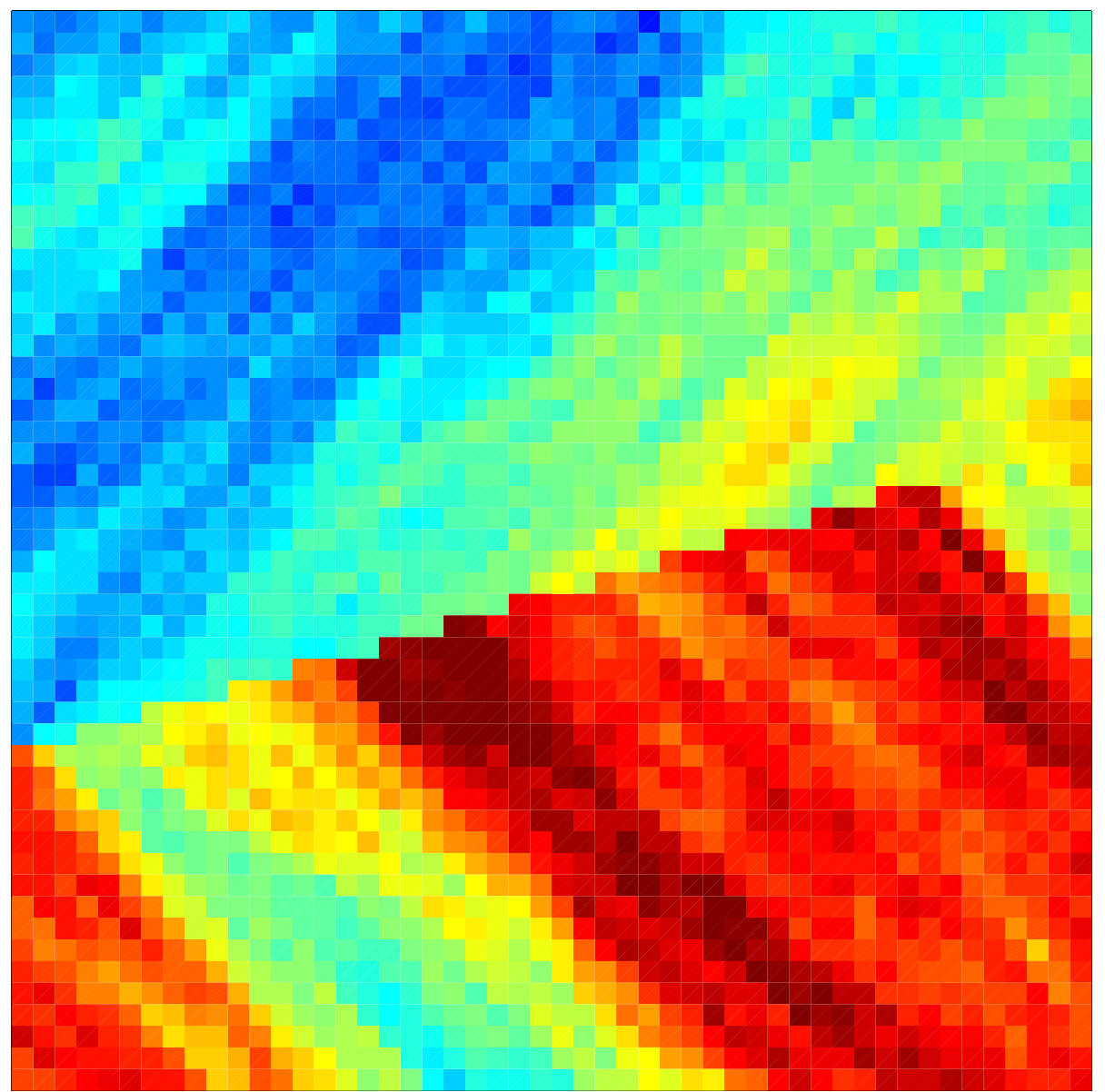}\\
\includegraphics[scale=0.2]{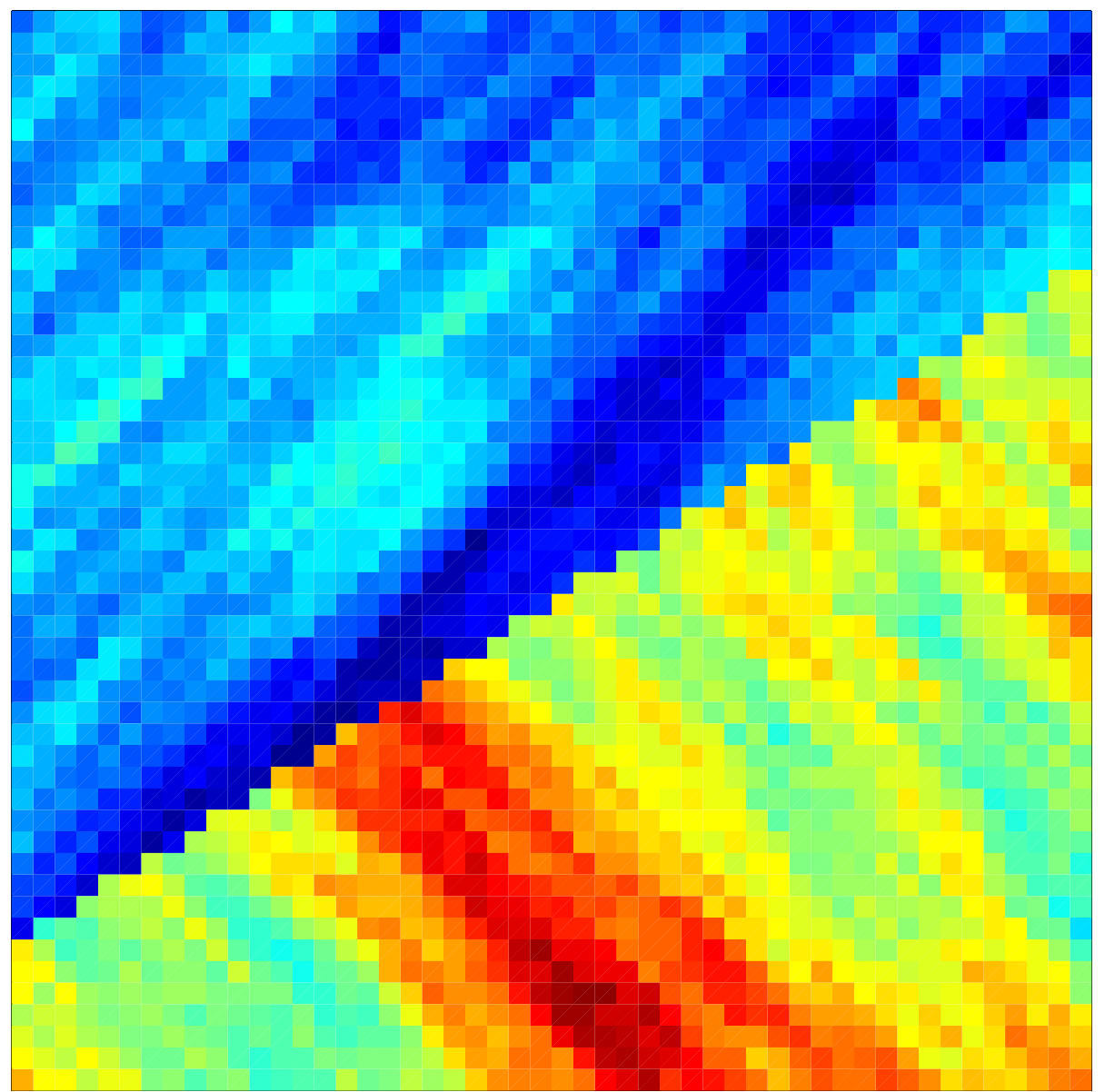}
\includegraphics[scale=0.2]{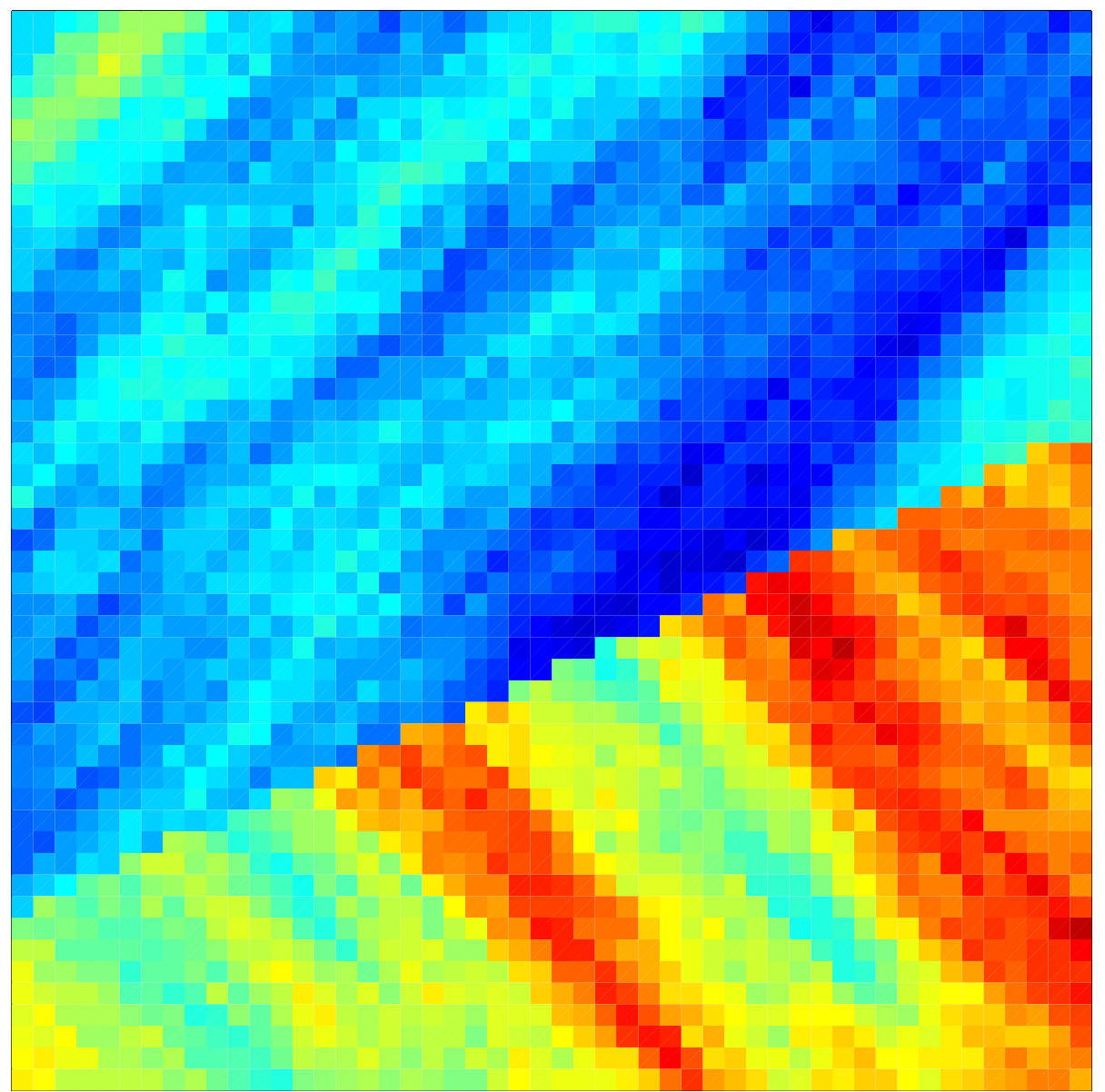}
\includegraphics[scale=0.2]{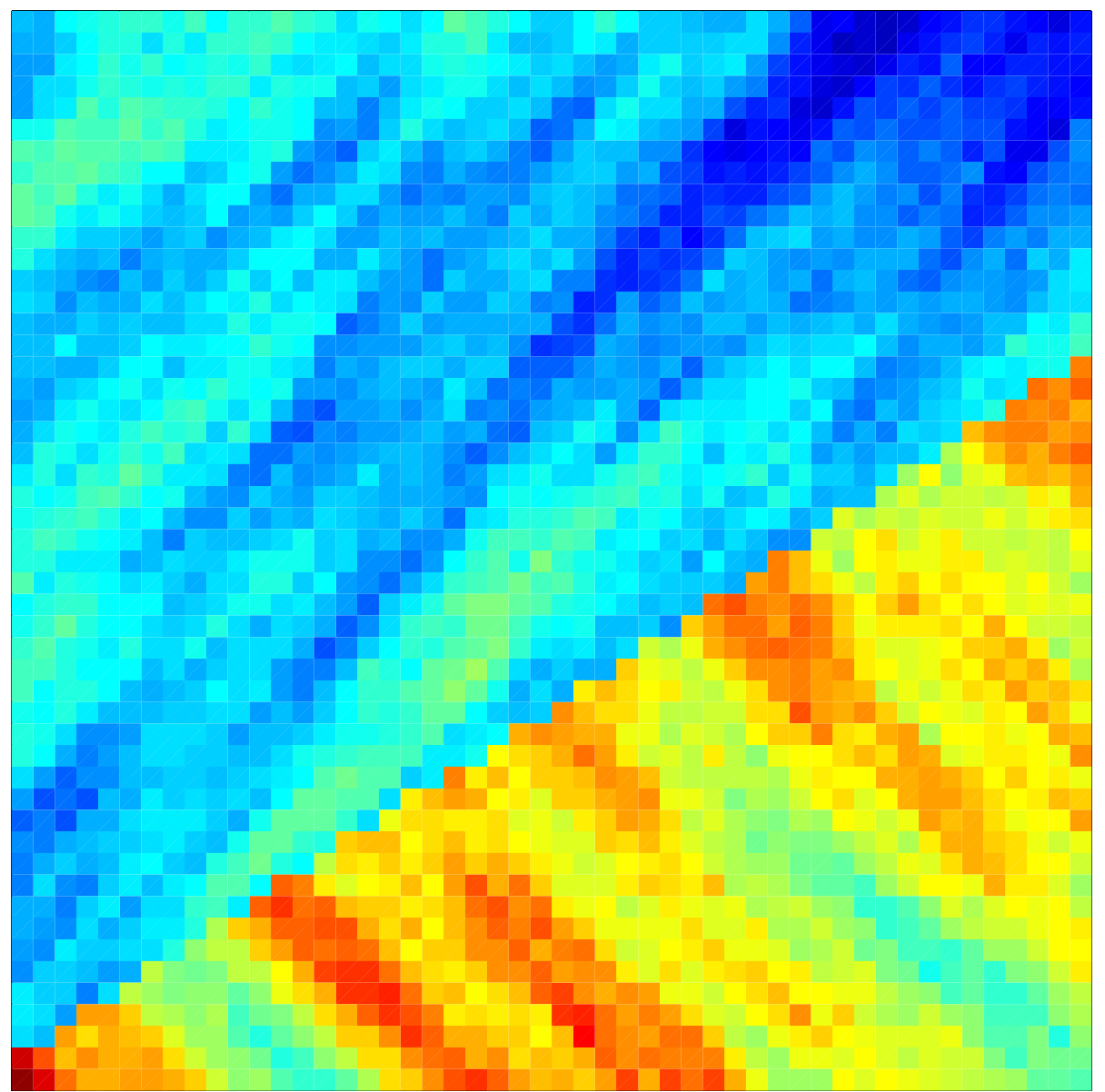}
\includegraphics[scale=0.2]{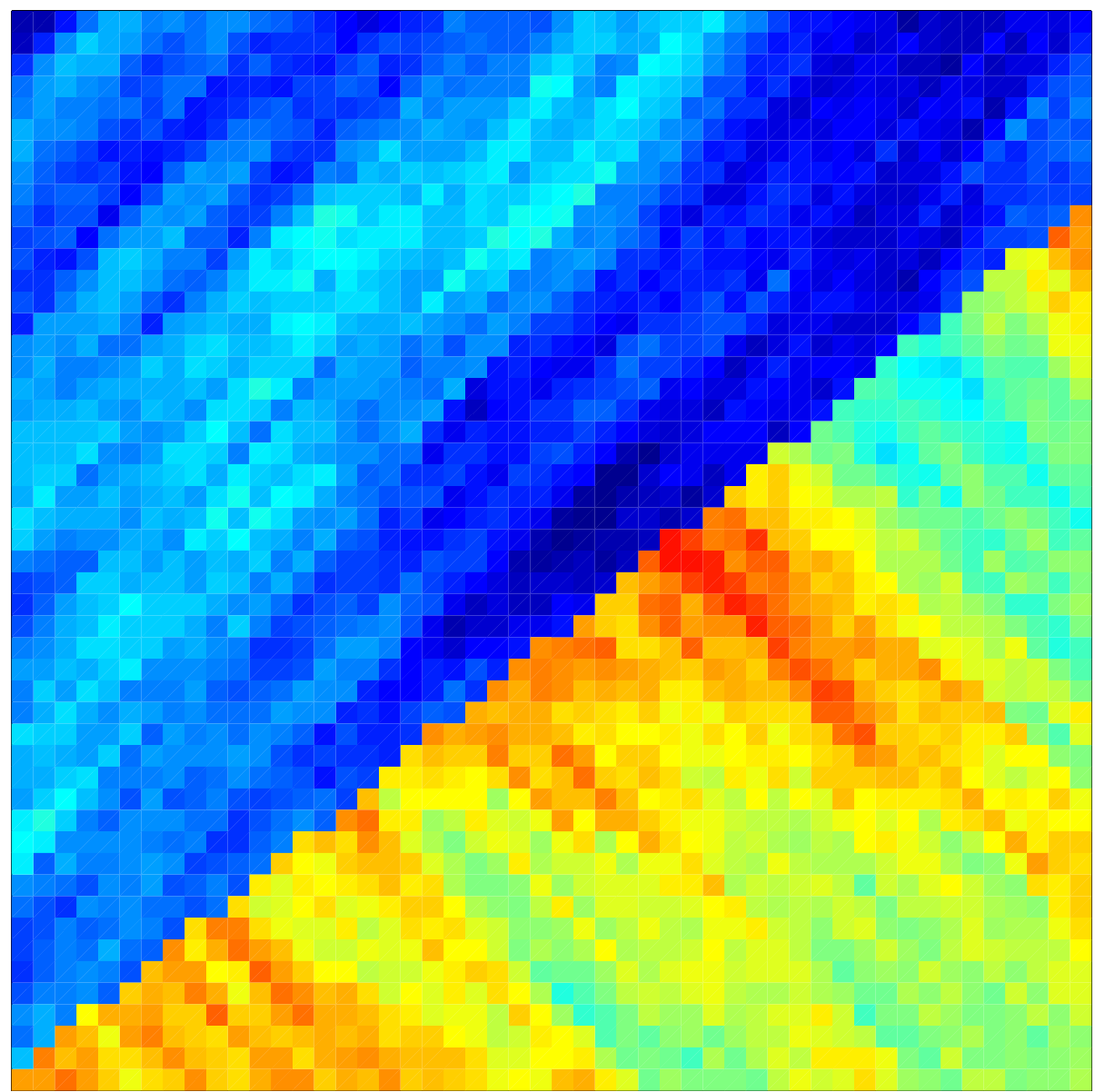}
\includegraphics[scale=0.2]{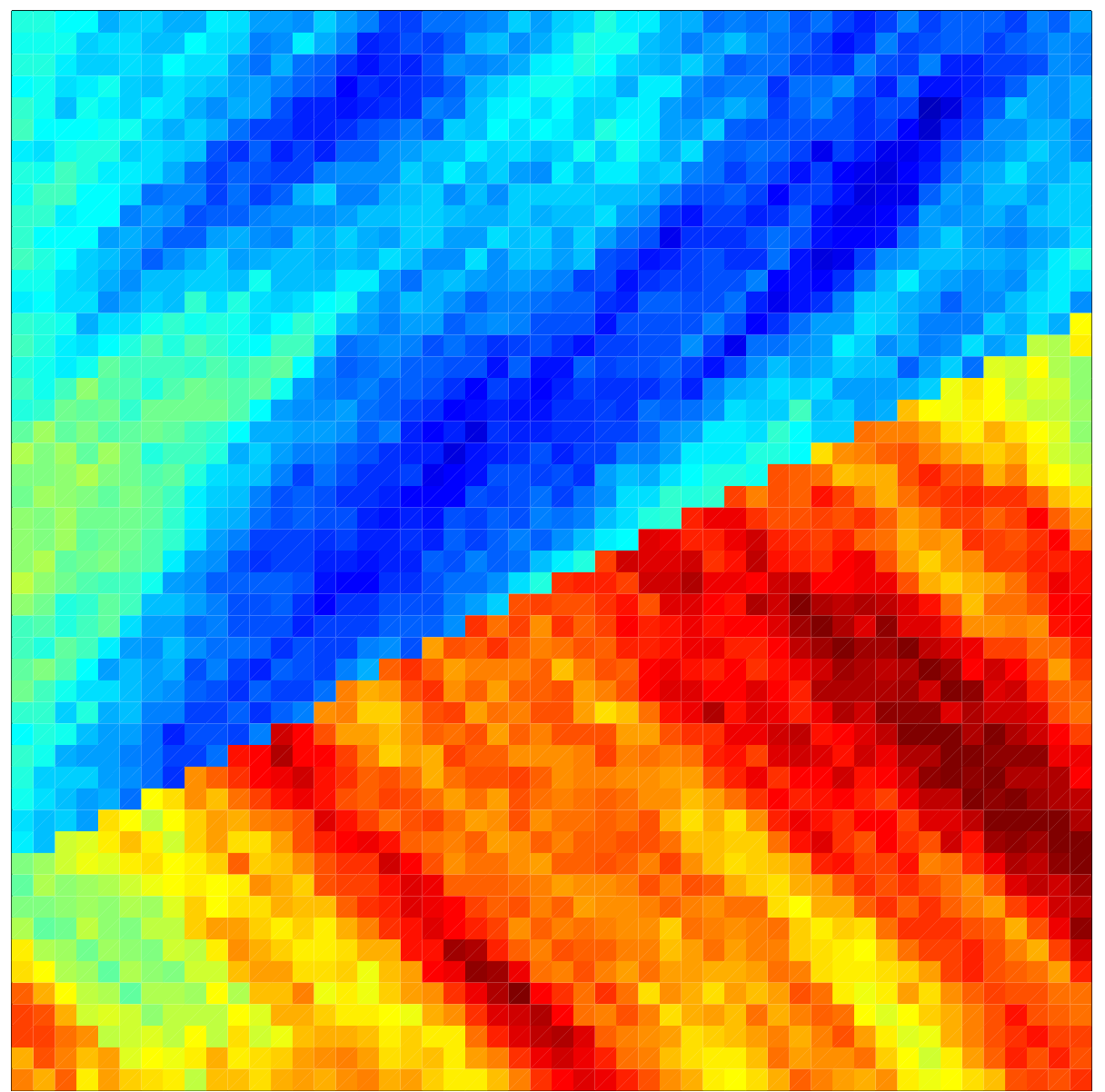}
\includegraphics[scale=0.2]{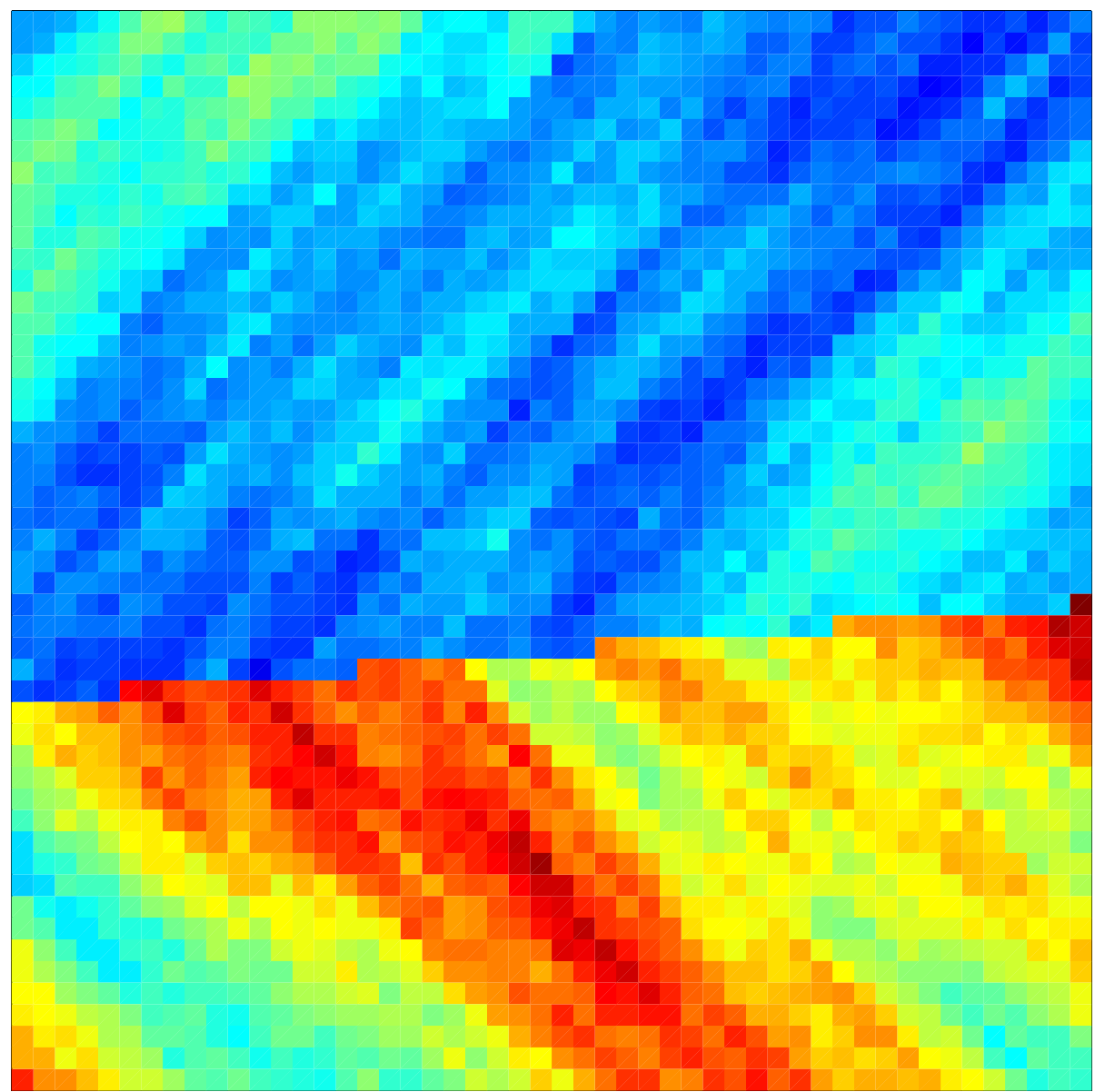}\\
\includegraphics[scale=0.2]{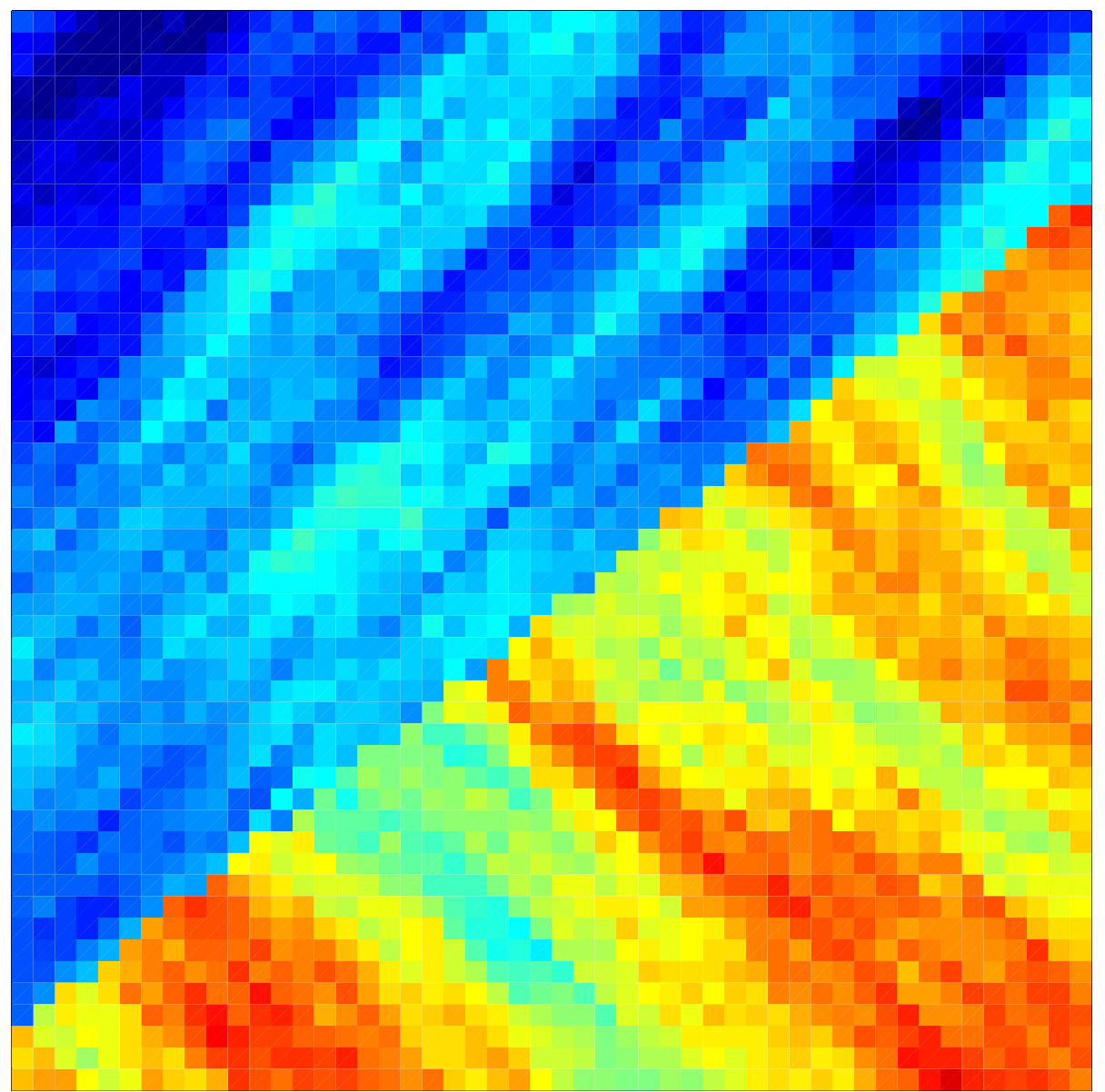}
\includegraphics[scale=0.2]{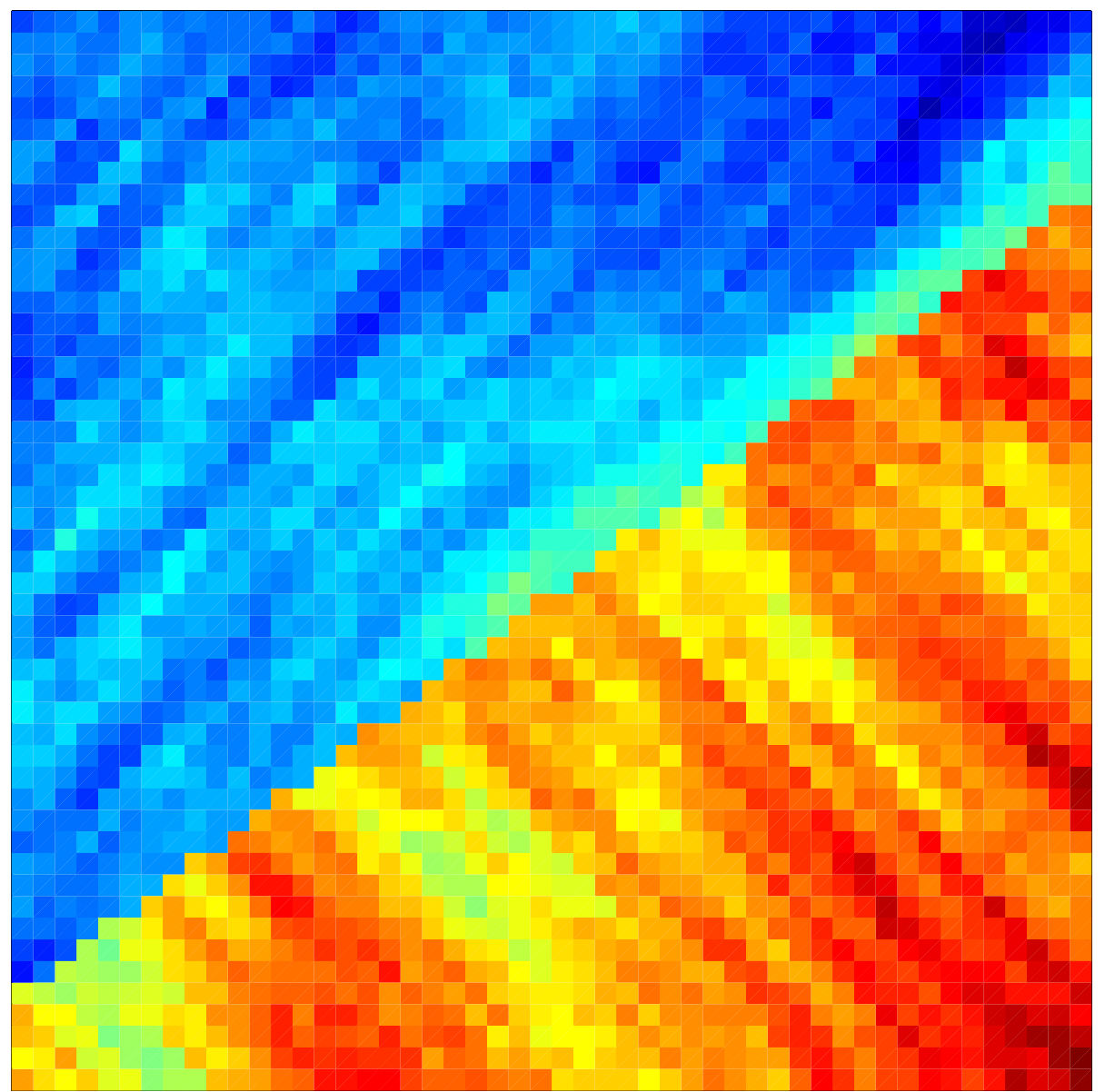}
\includegraphics[scale=0.2]{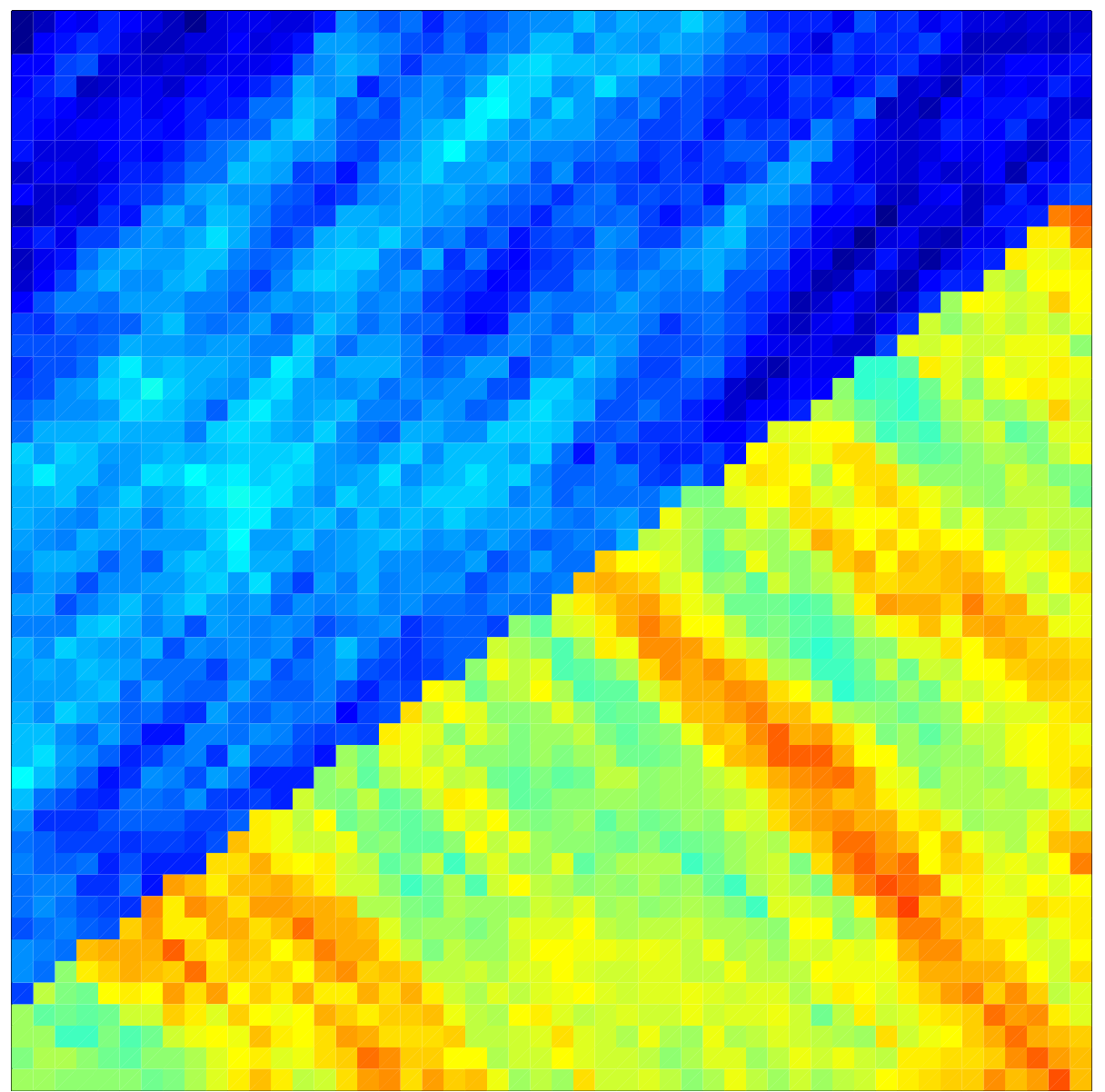}
\includegraphics[scale=0.2]{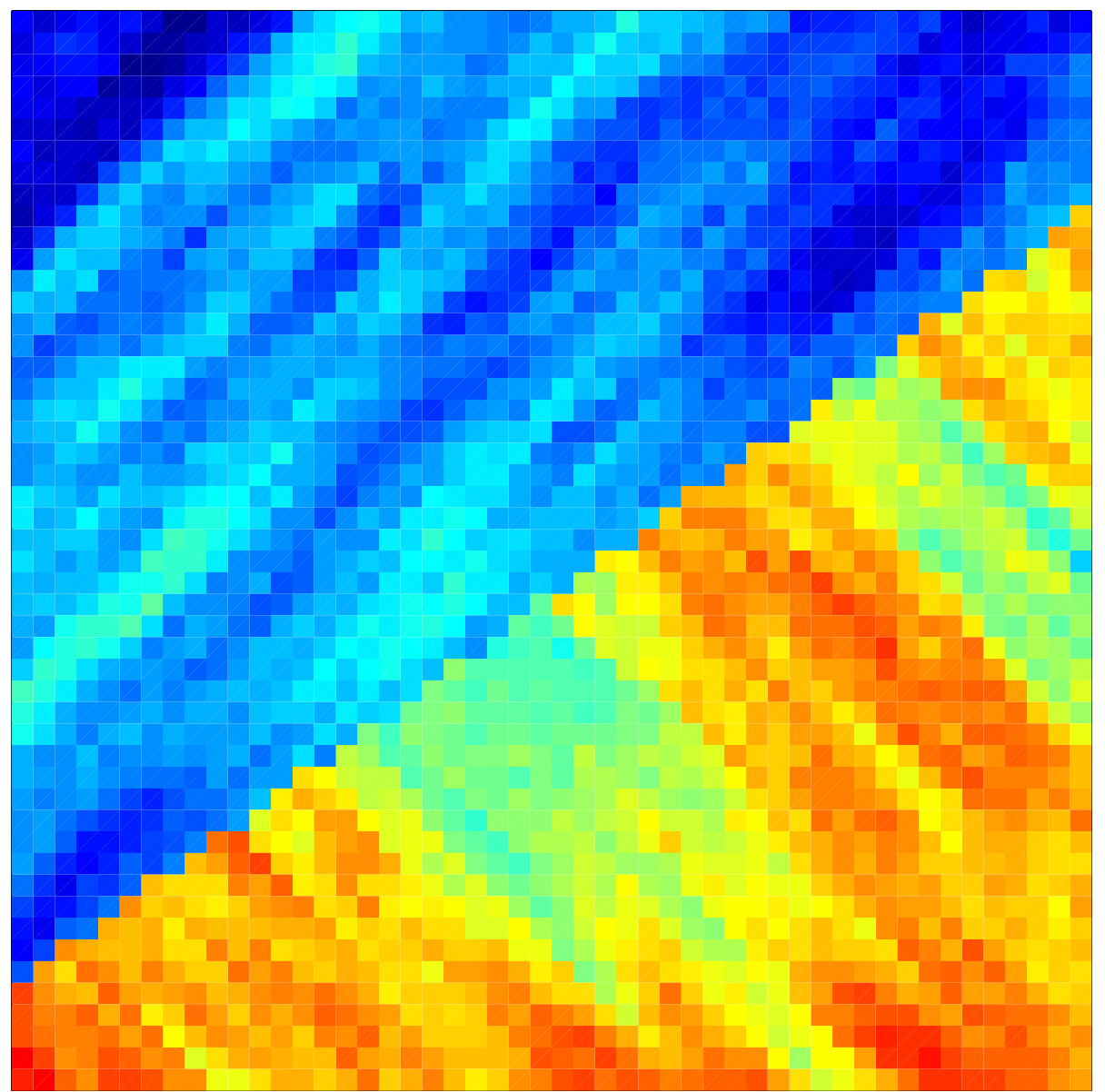}
\includegraphics[scale=0.2]{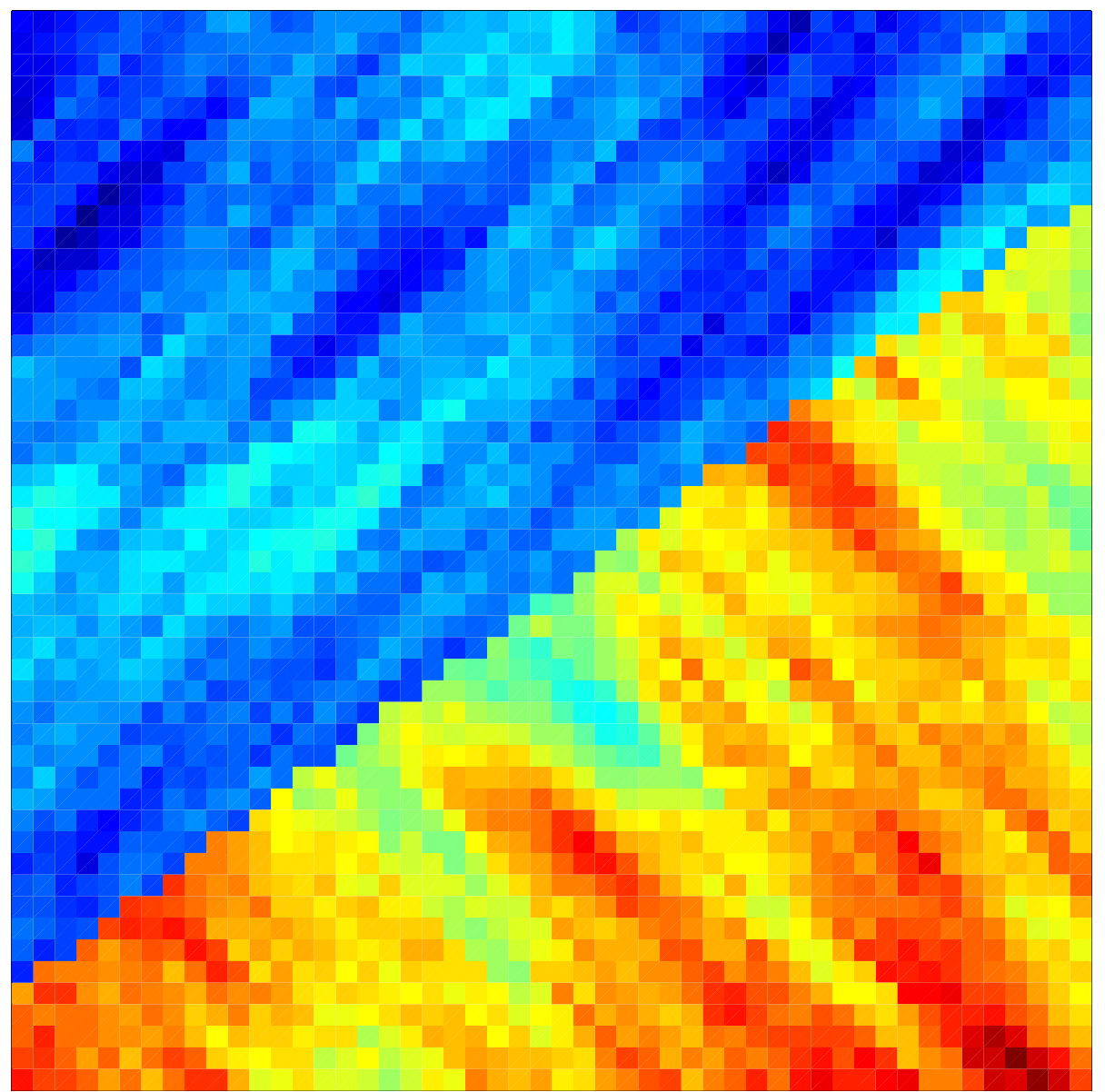}
\includegraphics[scale=0.2]{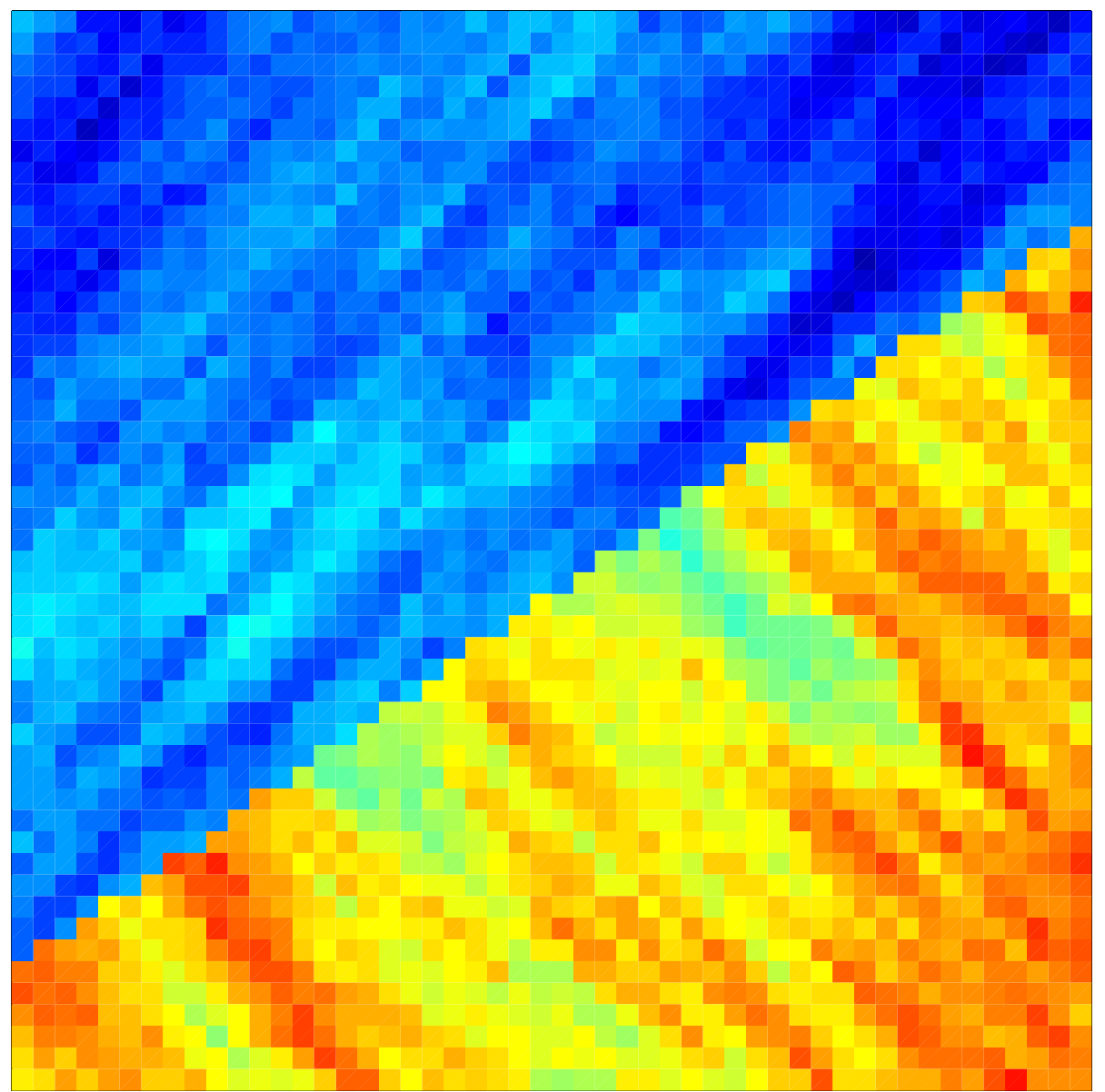}\\
 \caption{$\log \kappa$'s defined by (\ref{eq:num3}) from samples of the prior (top row), posterior with fewer measurements (middle row) and posterior with more measurements (bottom-row)}
    \label{Figure8}
\end{center}
\end{figure}

\begin{figure}[htbp]
\begin{center}
\includegraphics[scale=0.30]{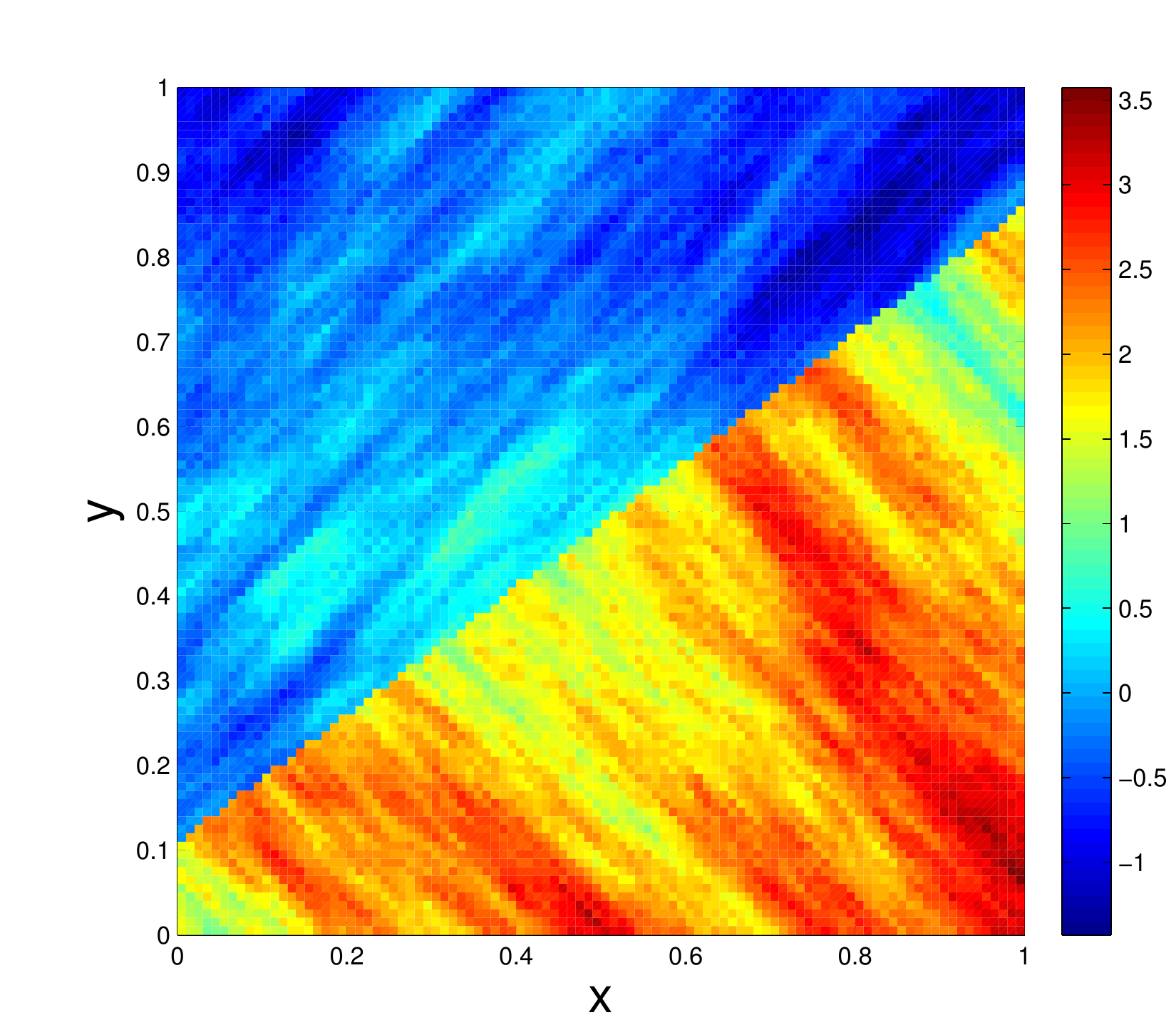}
\includegraphics[scale=0.30]{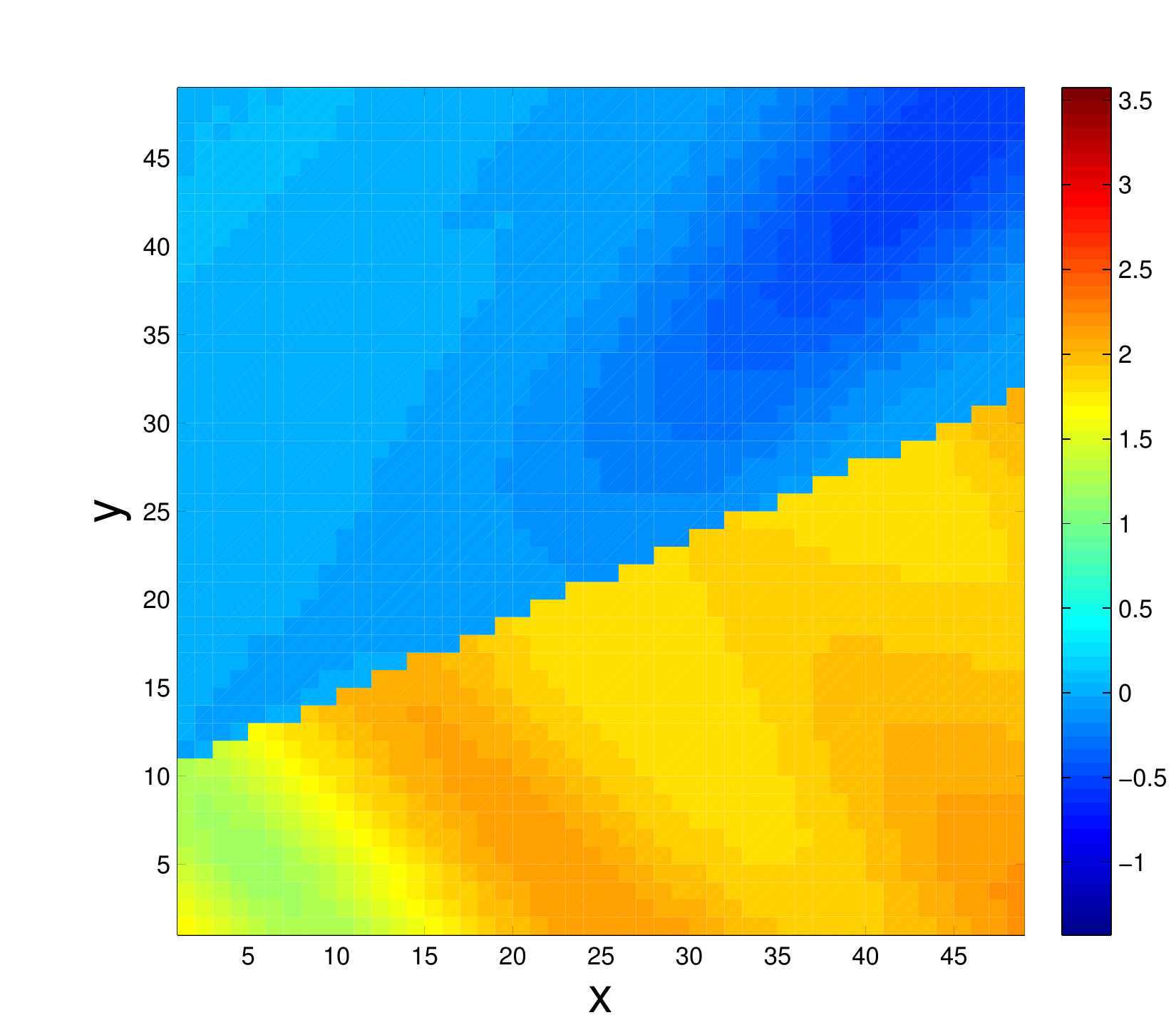}
\includegraphics[scale=0.30]{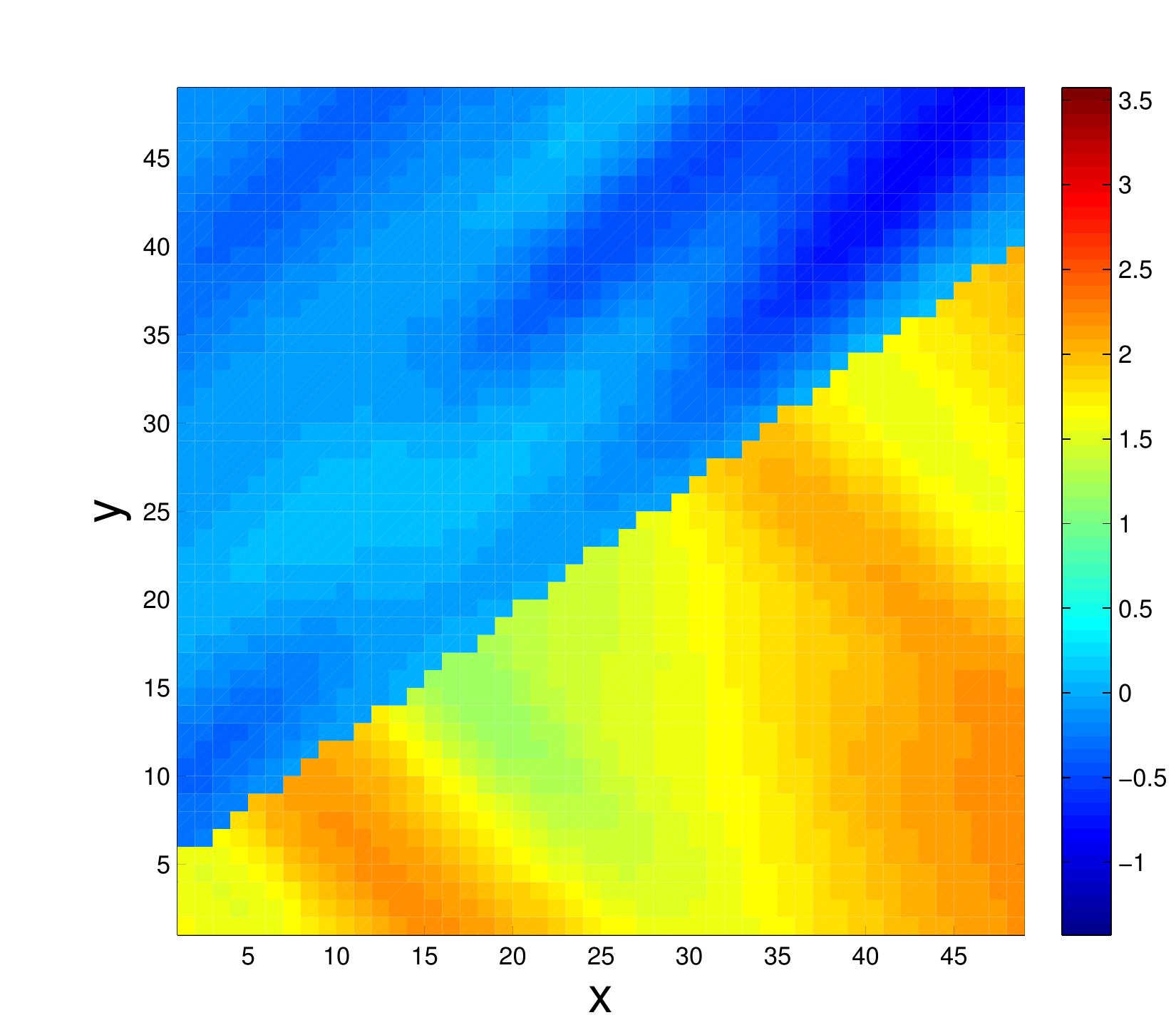}\\
\includegraphics[scale=0.30]{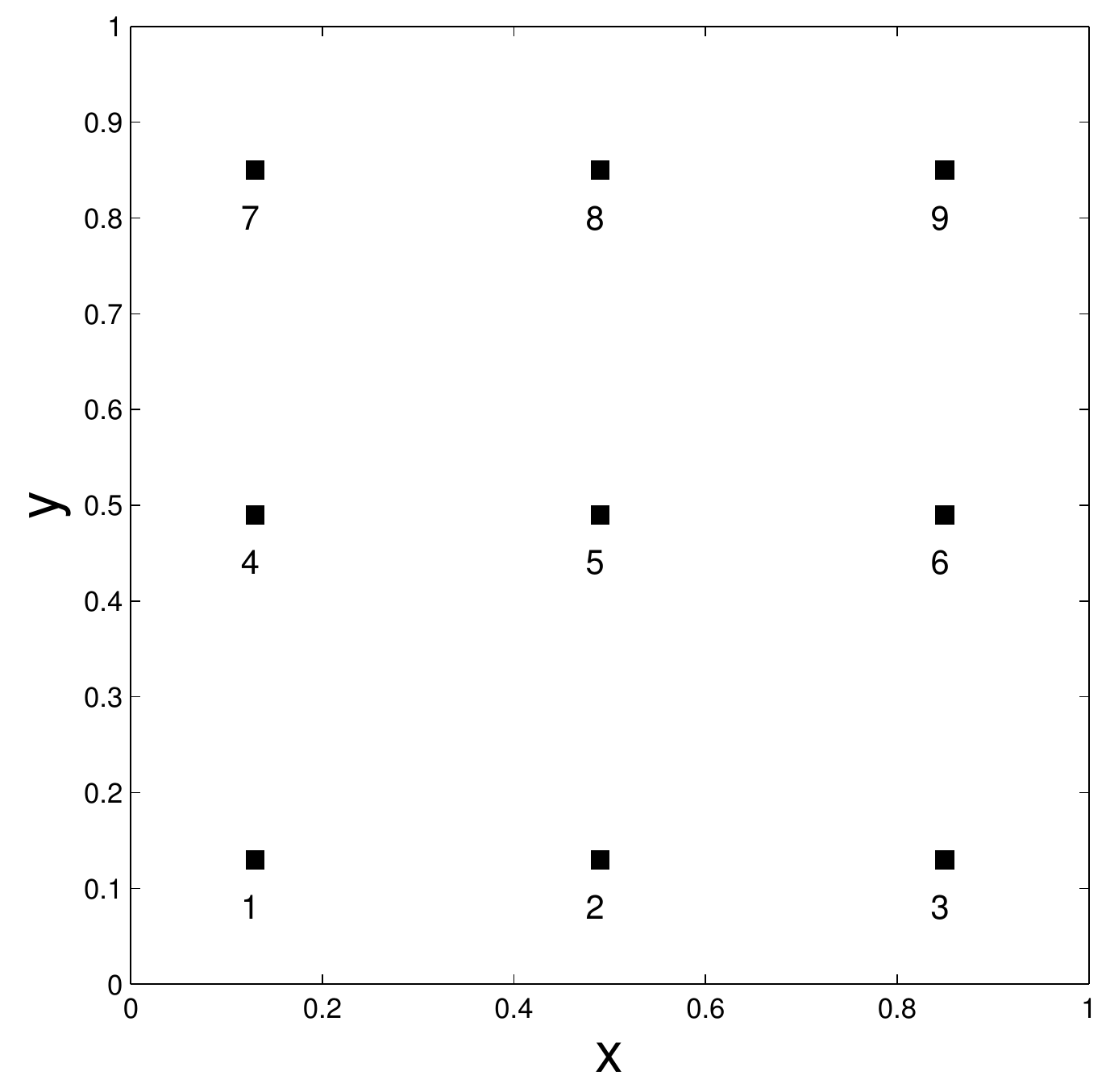}
\includegraphics[scale=0.30]{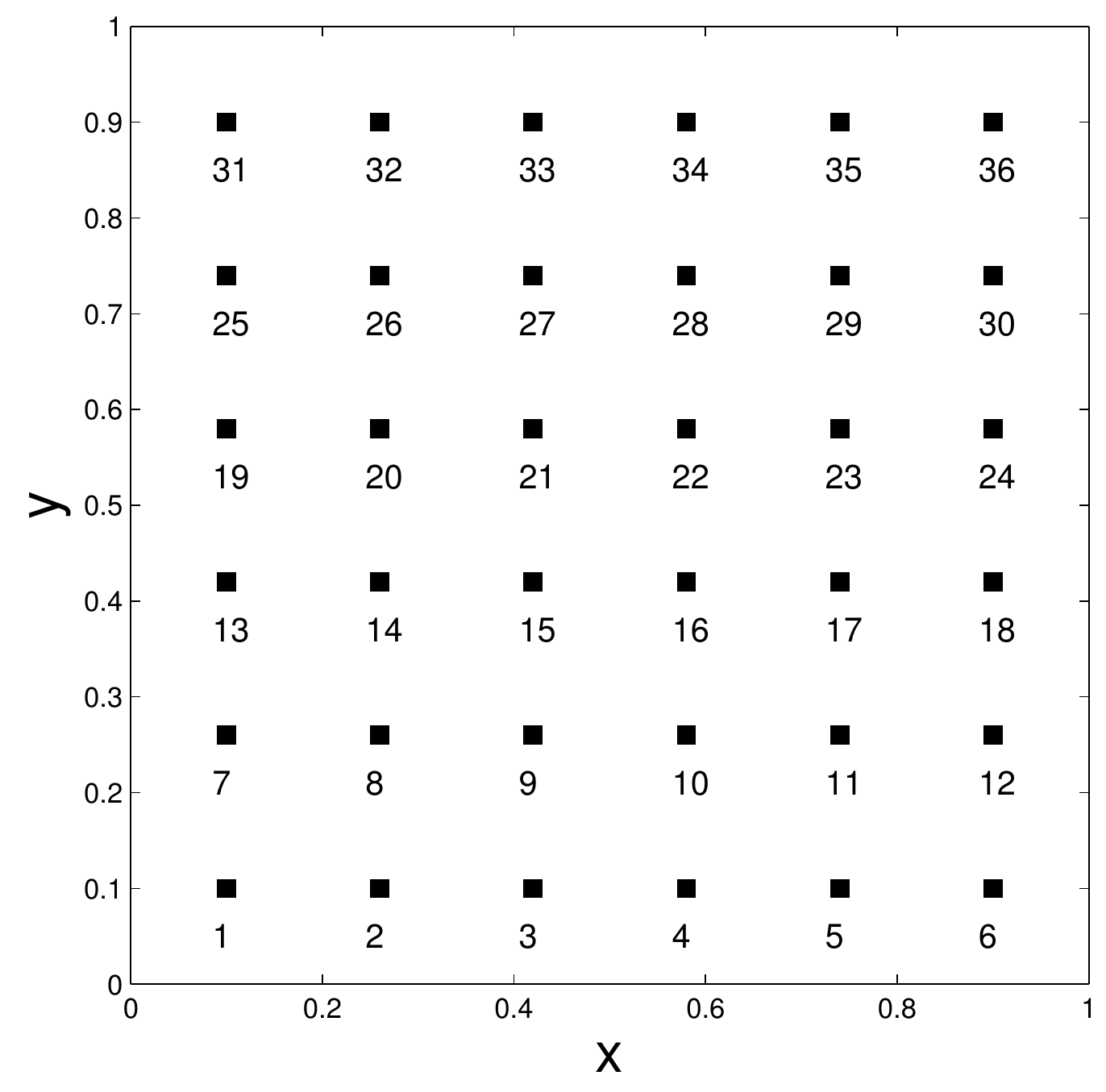}
 \caption{Top, from left to right: truth, mean (fewer measurements), mean (more measurements) of $\log \kappa$. Bottom: measurement locations for data with less (left) and more (right) measurements.}

    \label{Figure9}
\end{center}
\end{figure}

\begin{figure}[htbp]
%\begin{center}
\includegraphics[scale=0.30]{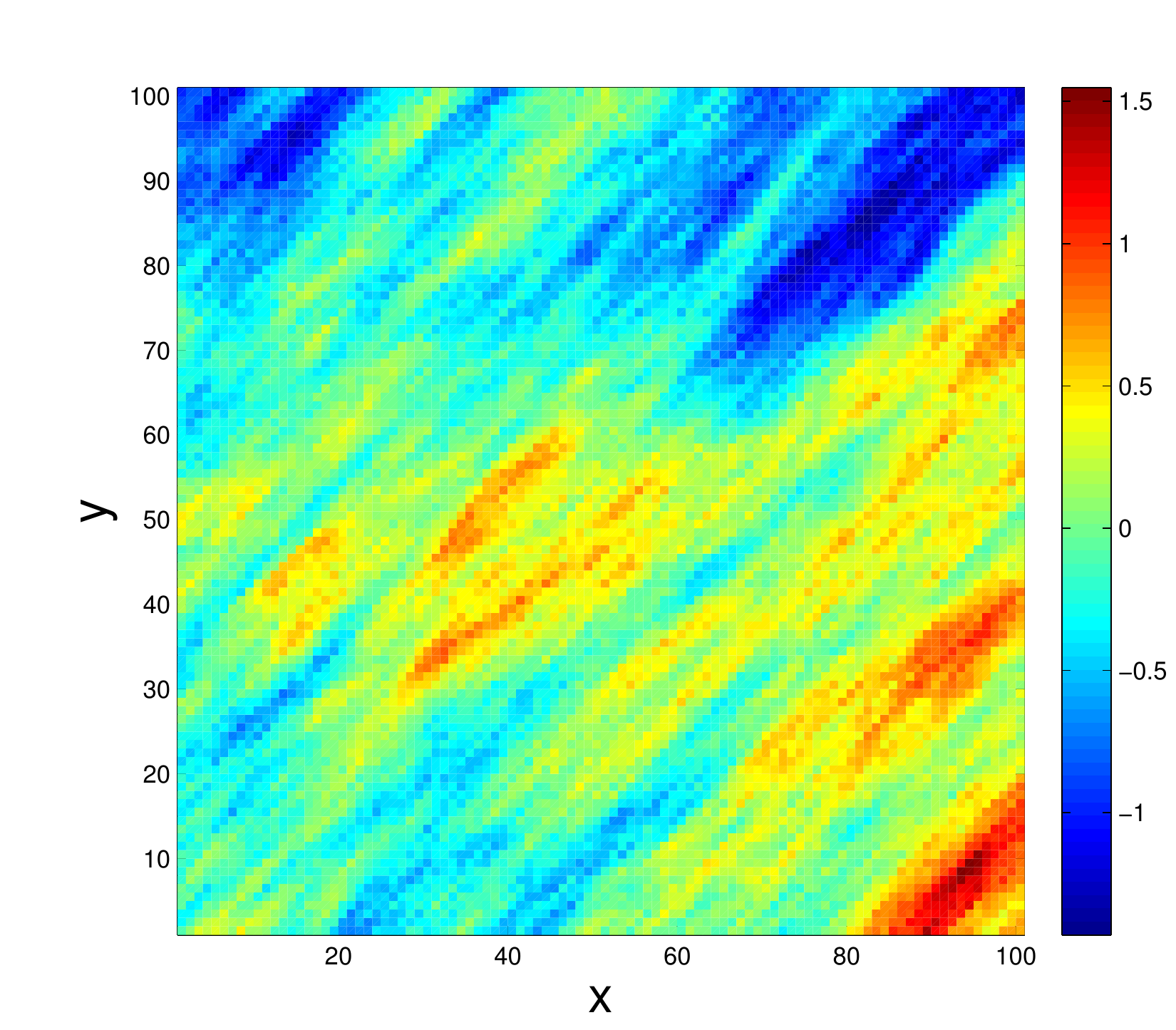}
\includegraphics[scale=0.30]{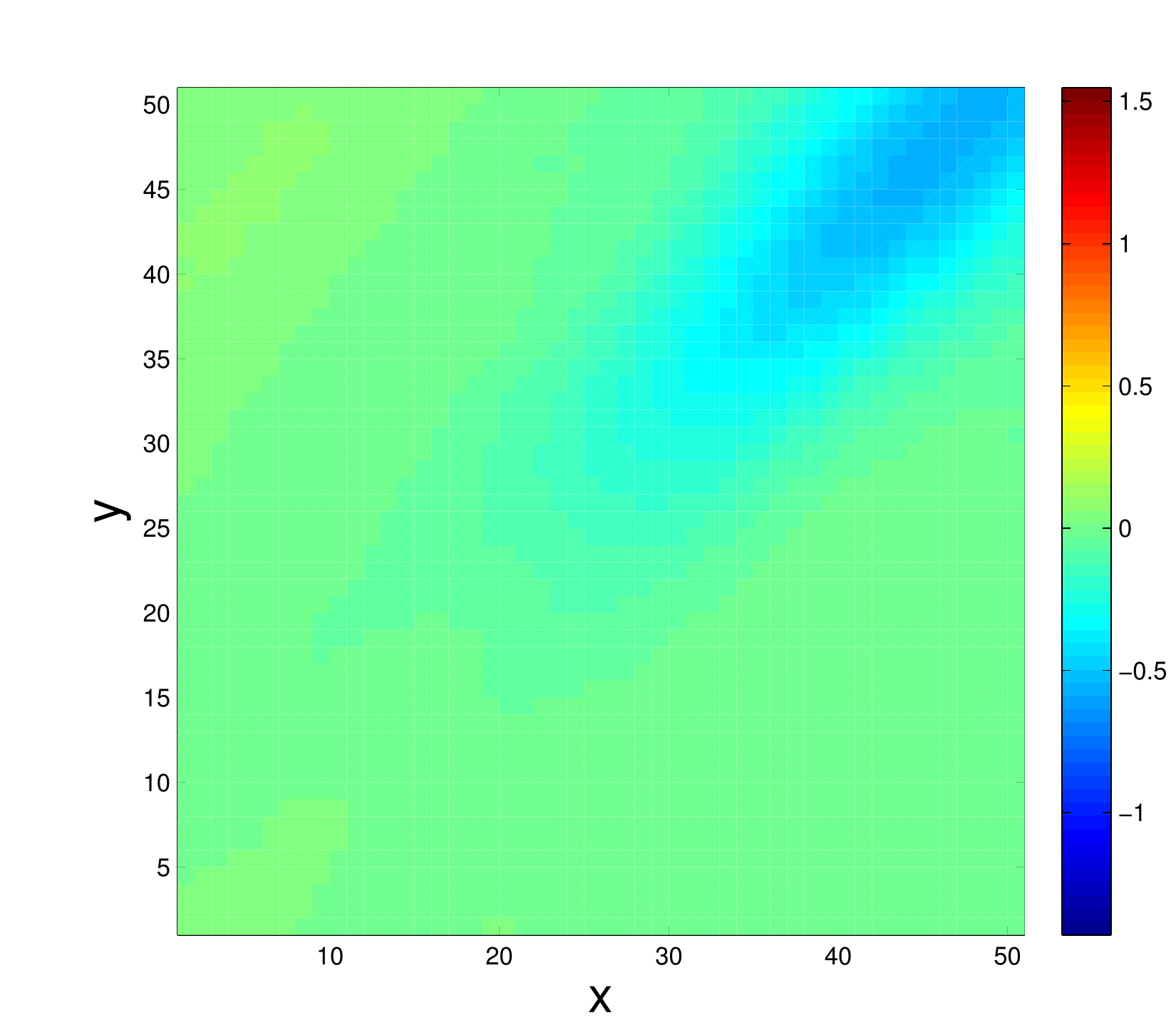}
\includegraphics[scale=0.30]{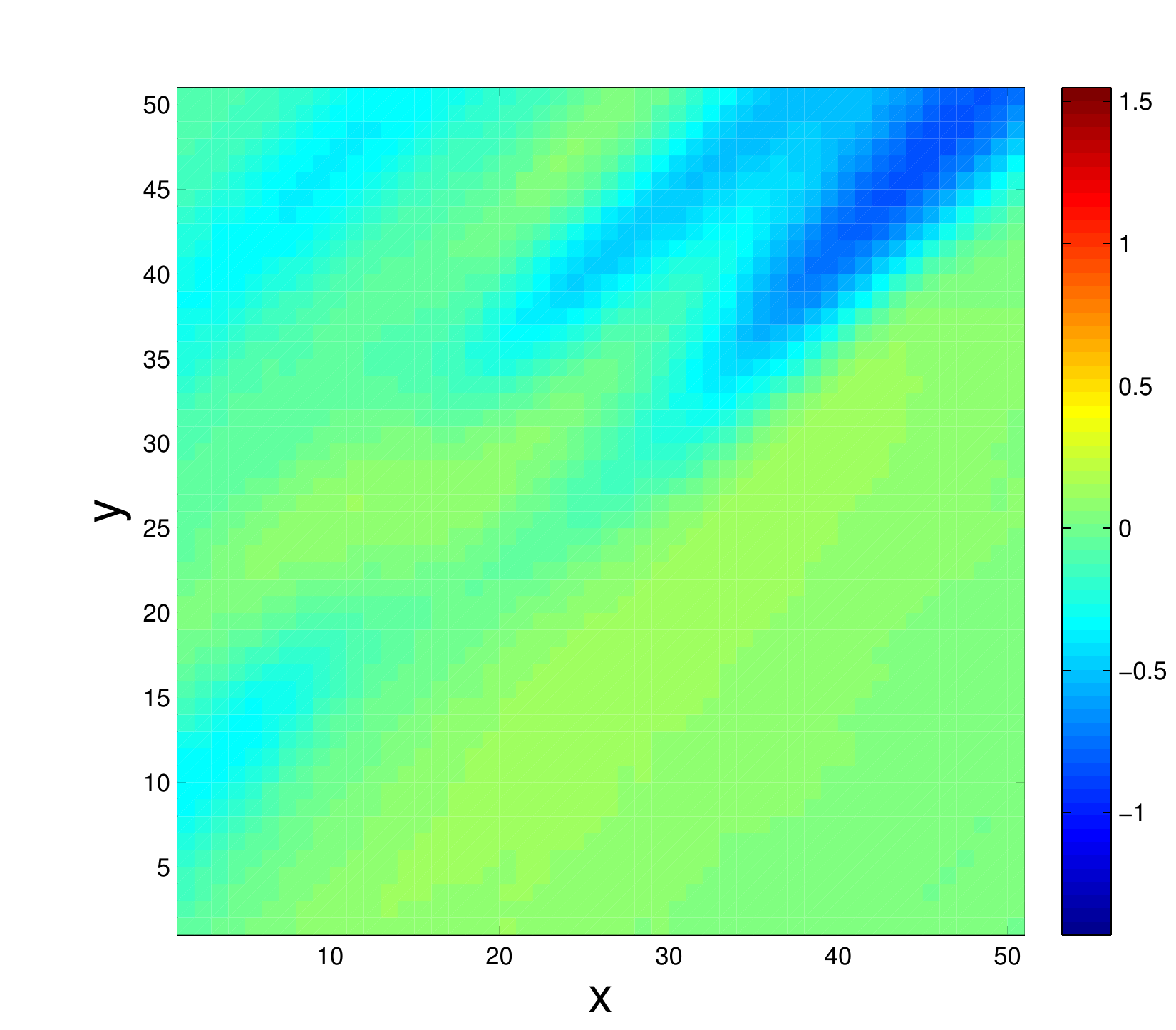}\\
\includegraphics[scale=0.30]{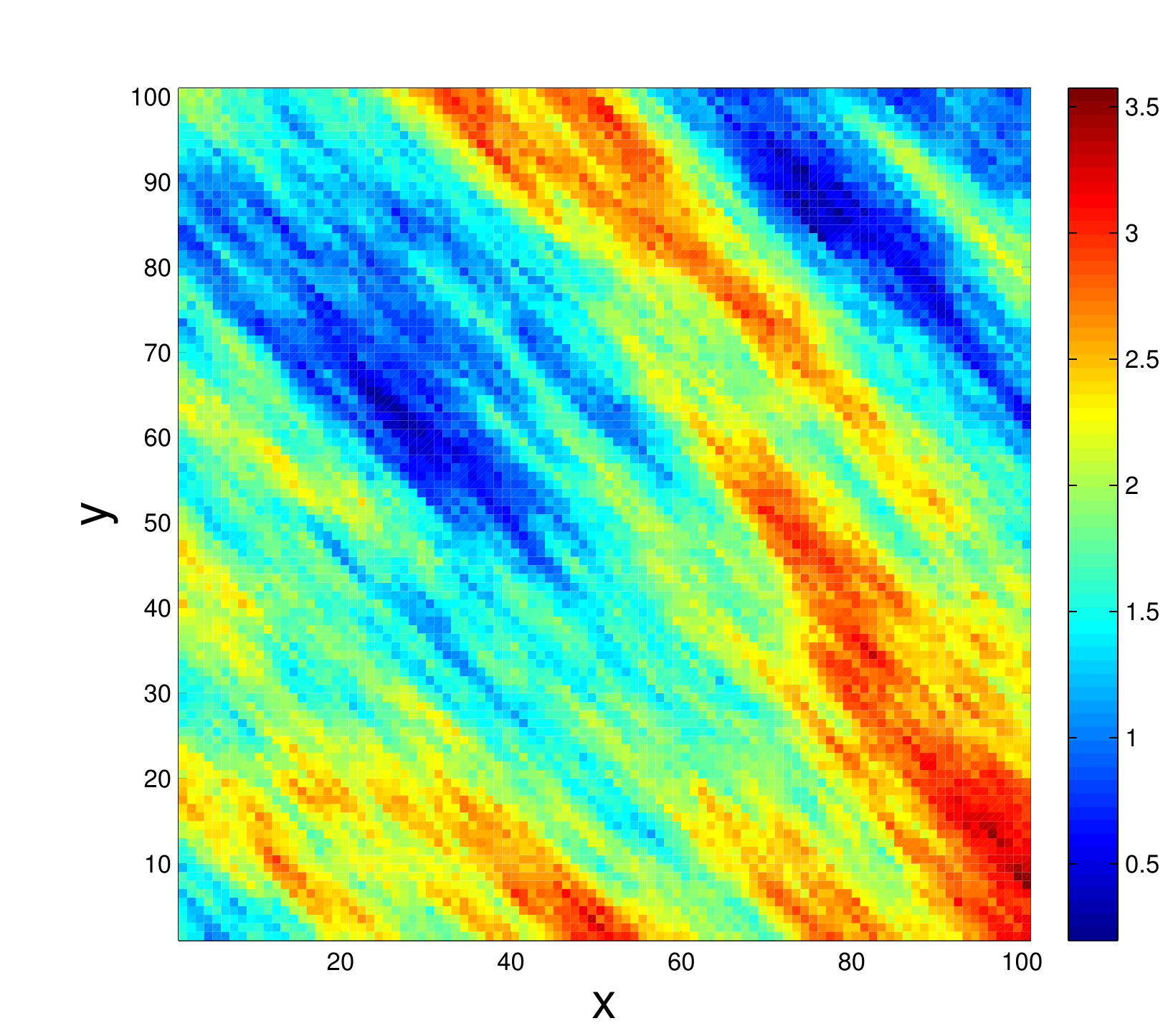}
\includegraphics[scale=0.30]{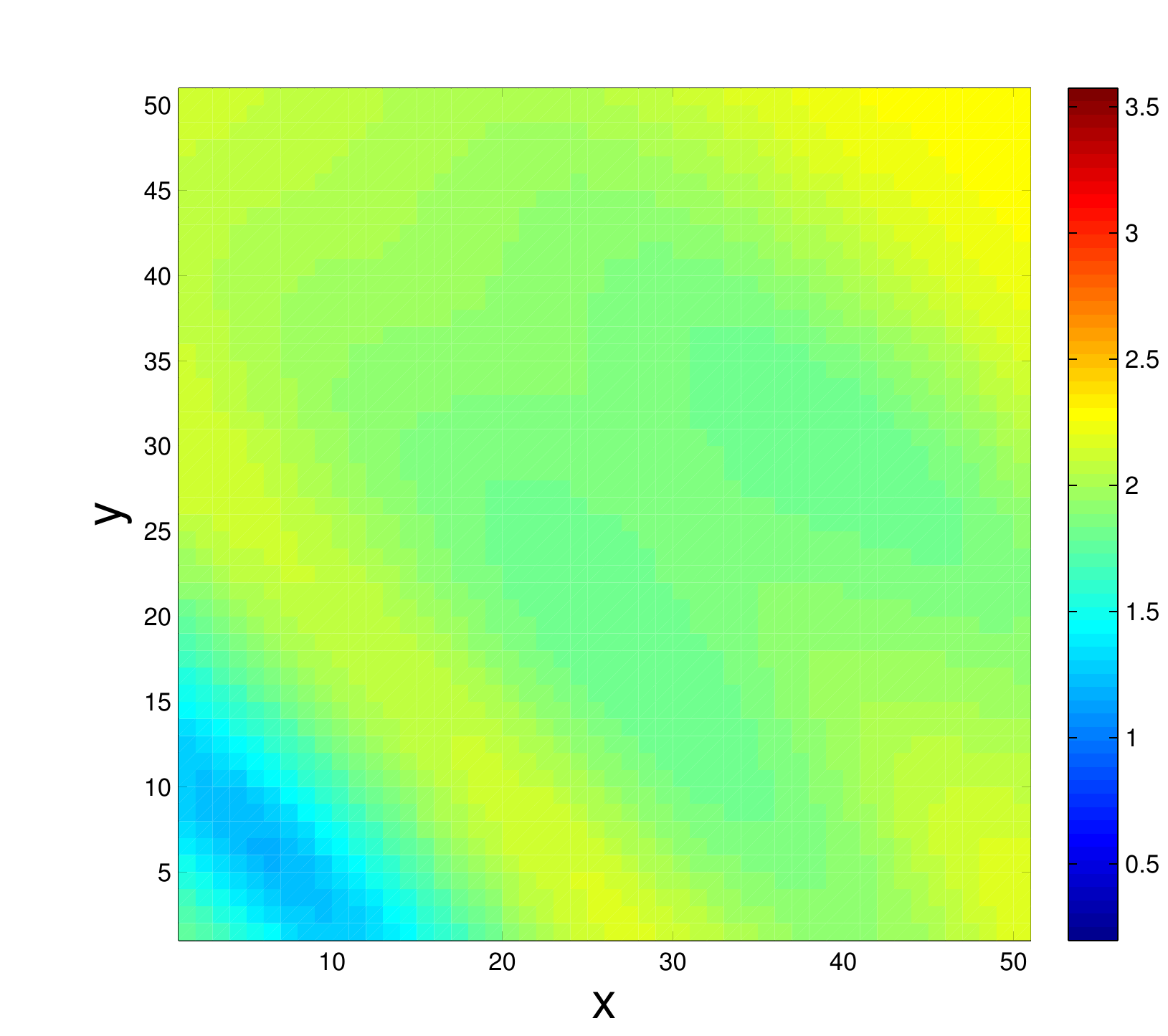}
\includegraphics[scale=0.30]{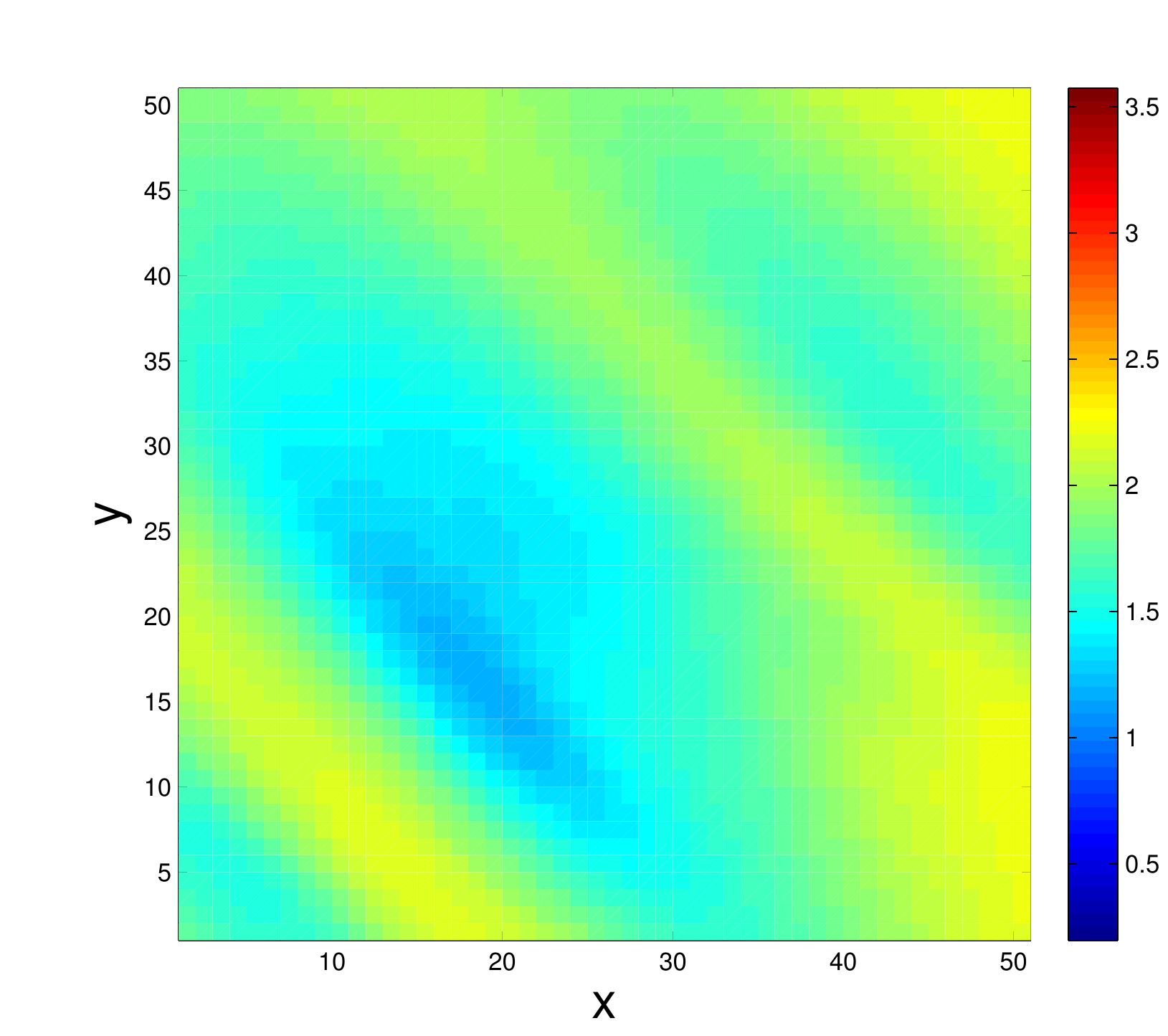}
 \caption{Top:  $\log \kappa_{1}$. Bottom: $\log \kappa_{2}$. From left to right: truth, mean (fewer measurements), mean (more measurements).  }

    \label{Figure10}
%\end{center}
\end{figure}

\begin{figure}[htbp]
\begin{center}
\includegraphics[scale=0.285]{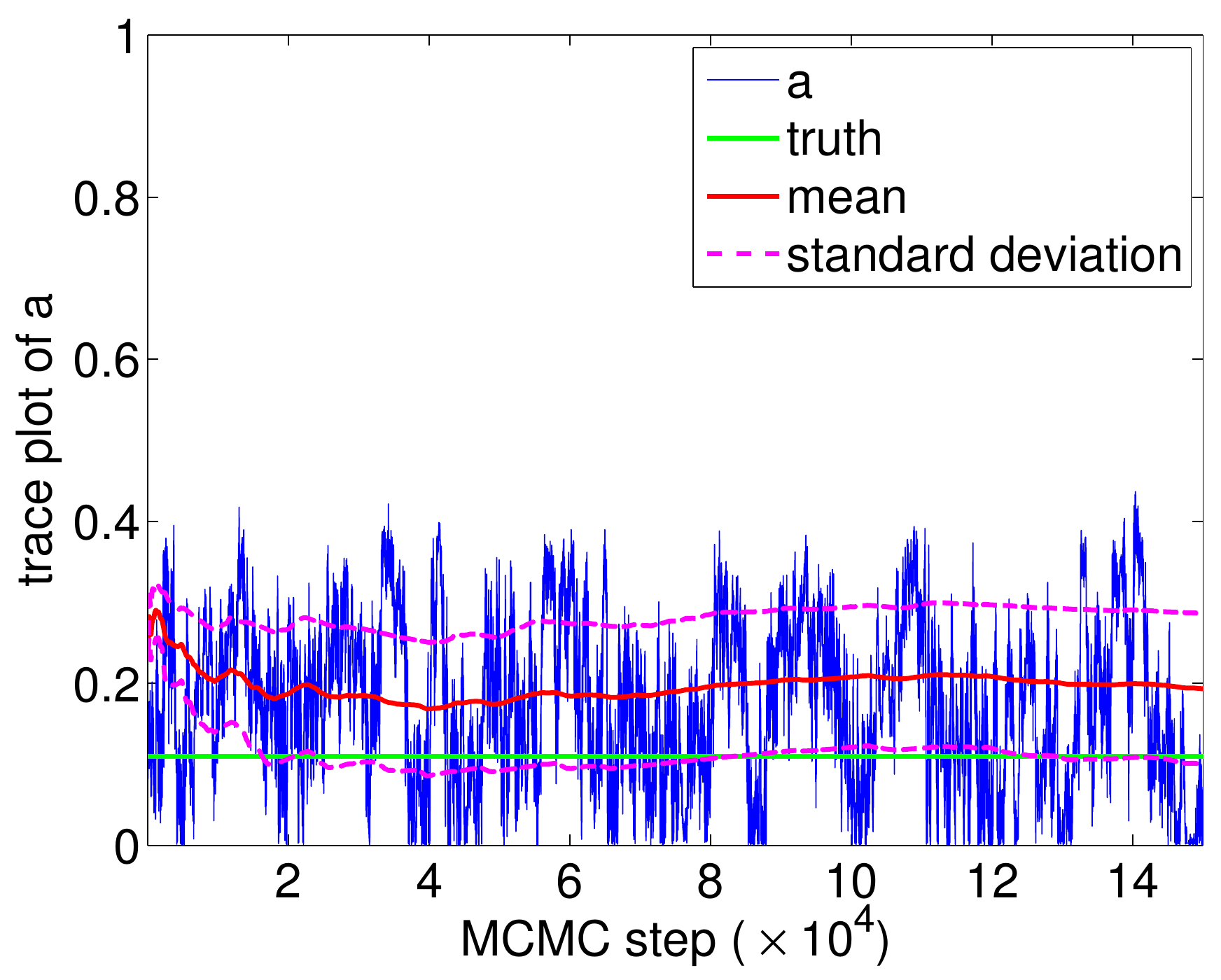}
\includegraphics[scale=0.285]{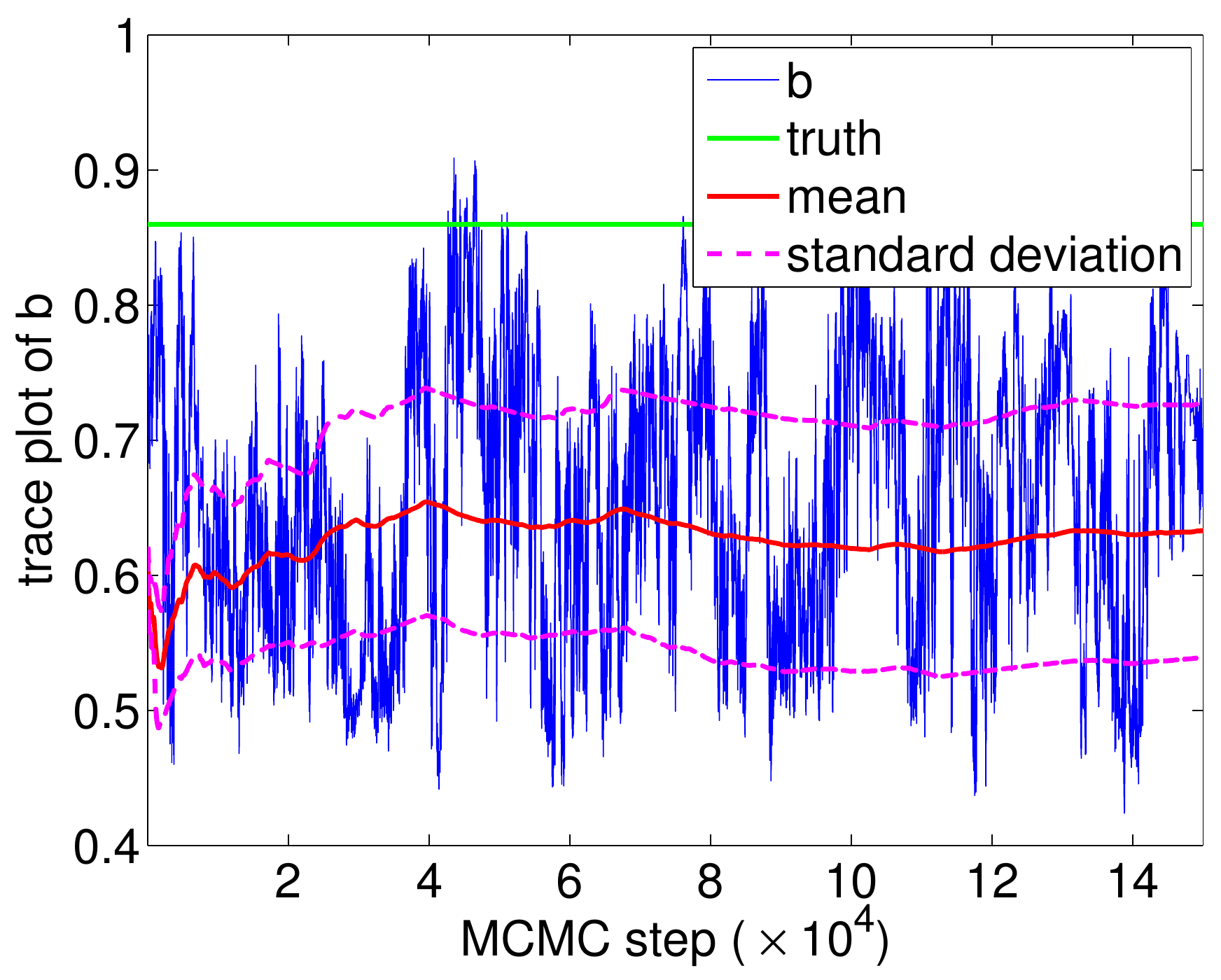}
\includegraphics[scale=0.285]{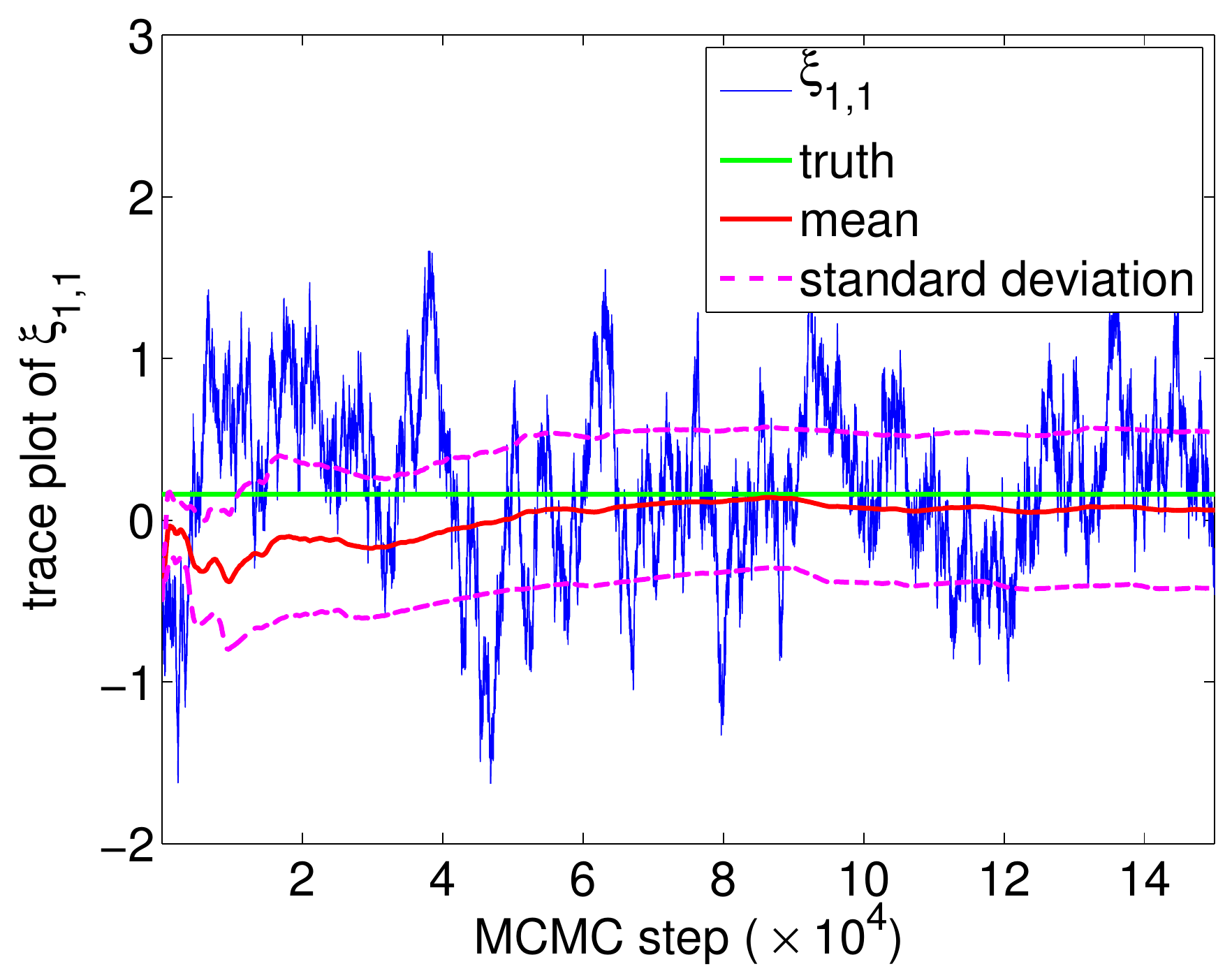}\\
\includegraphics[scale=0.285]{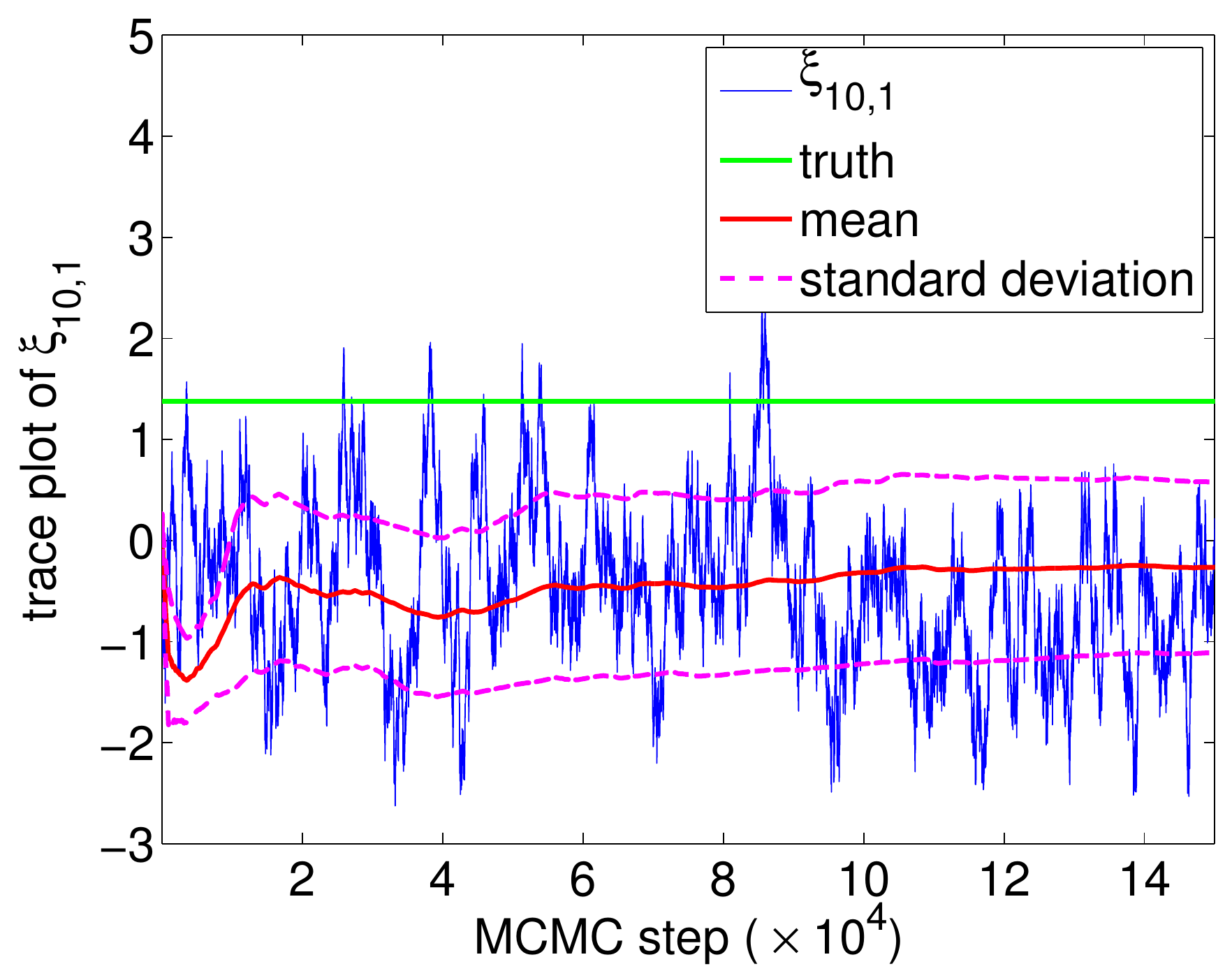}
\includegraphics[scale=0.285]{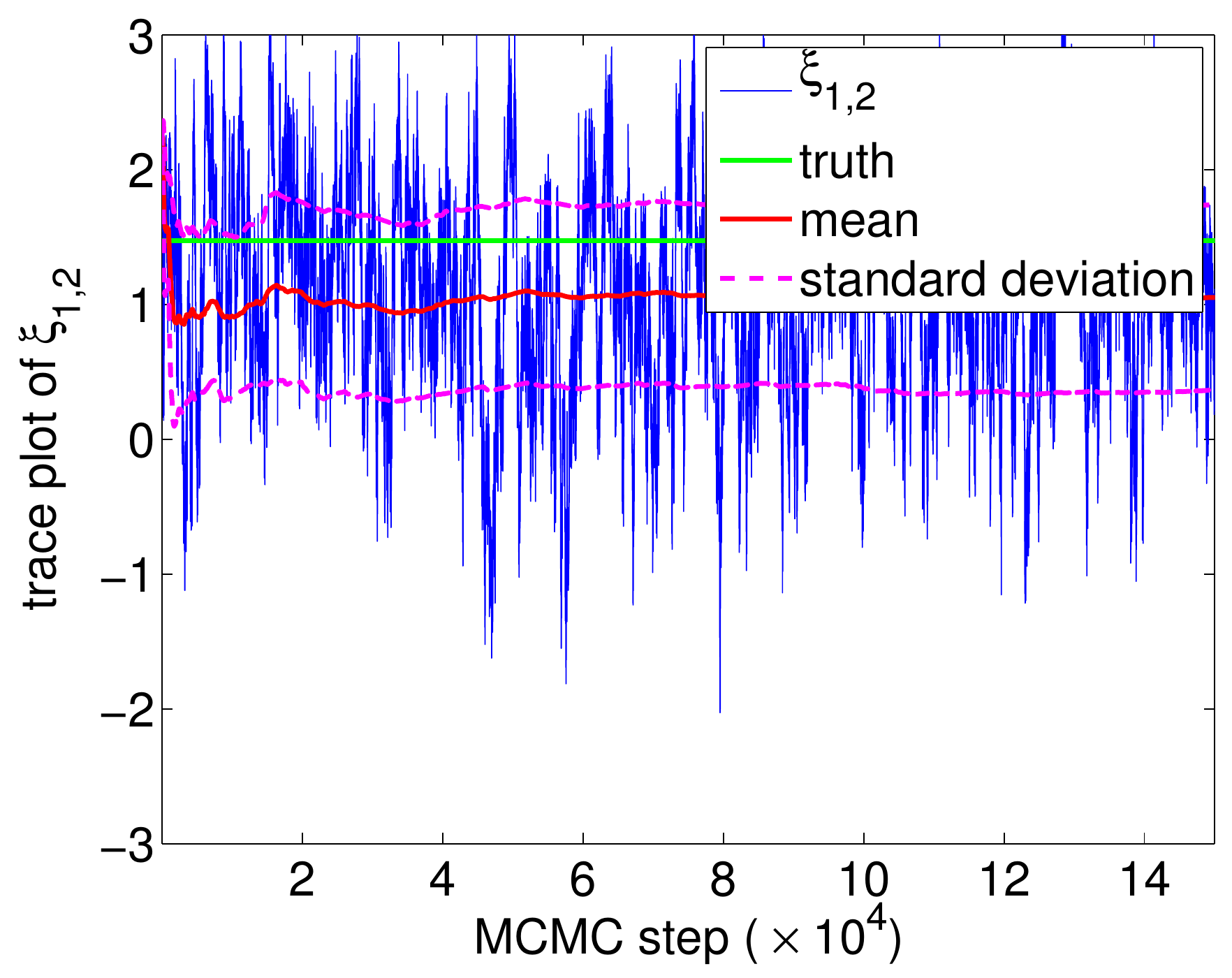}
\includegraphics[scale=0.285]{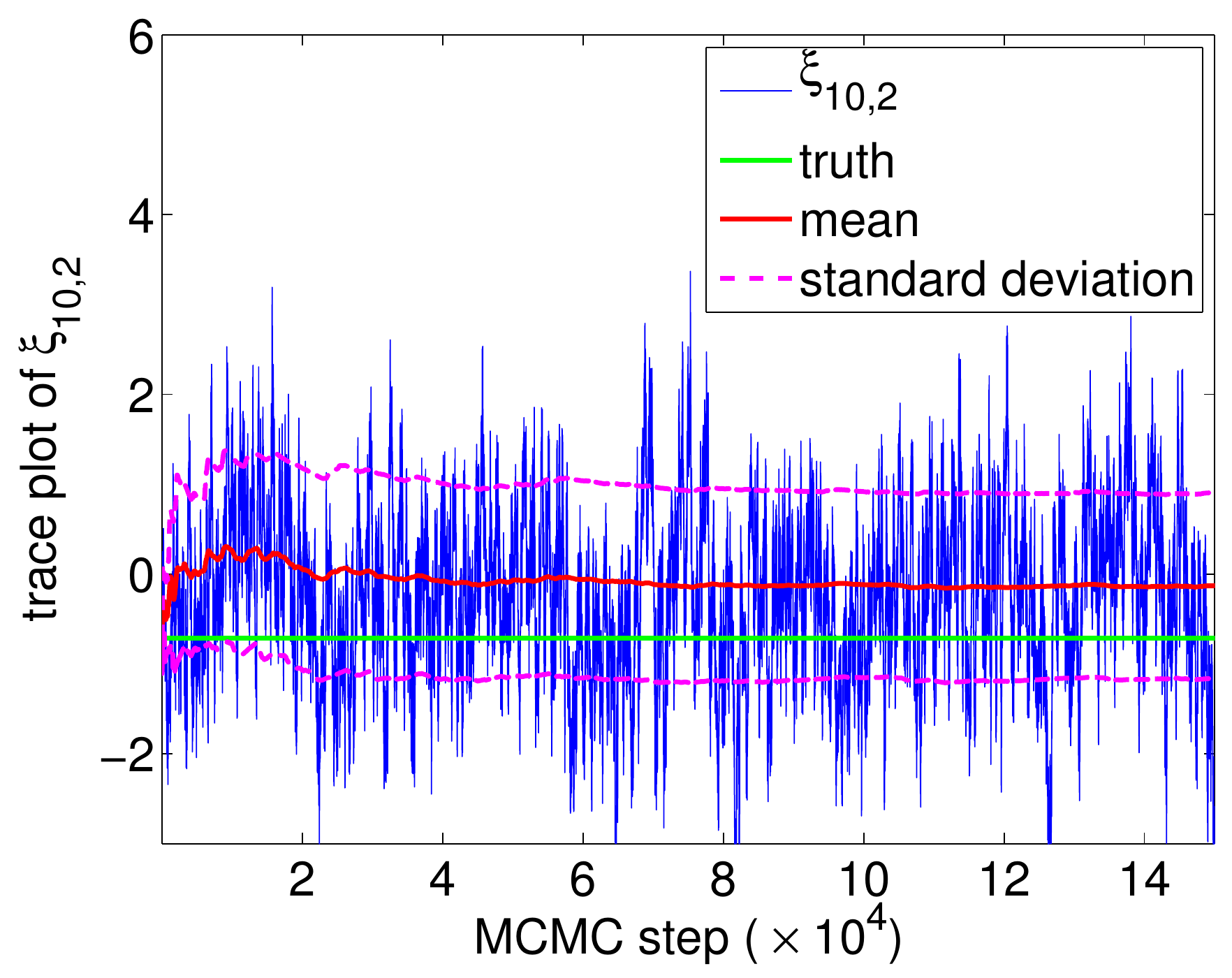}

 \caption{Trace plot from one MCMC chain (data 1). Top: geometric parameters. Bottom: Some KL modes of $\log \kappa_{1}$ and $\log \kappa_{2}$}
    \label{Figure11}
\end{center}
\end{figure}

\begin{figure}[htbp]
\begin{center}
\includegraphics[scale=0.30]{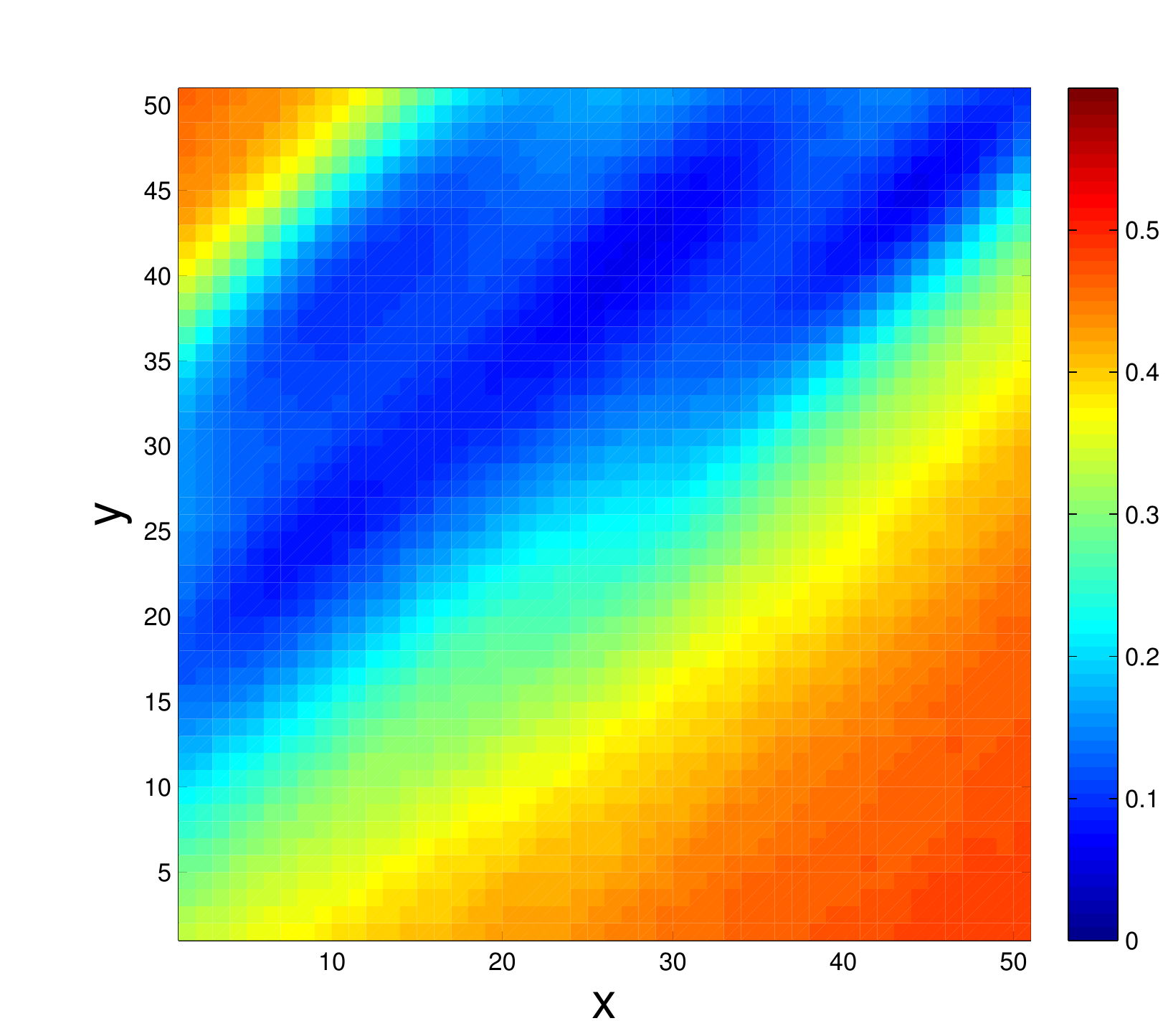}
\includegraphics[scale=0.30]{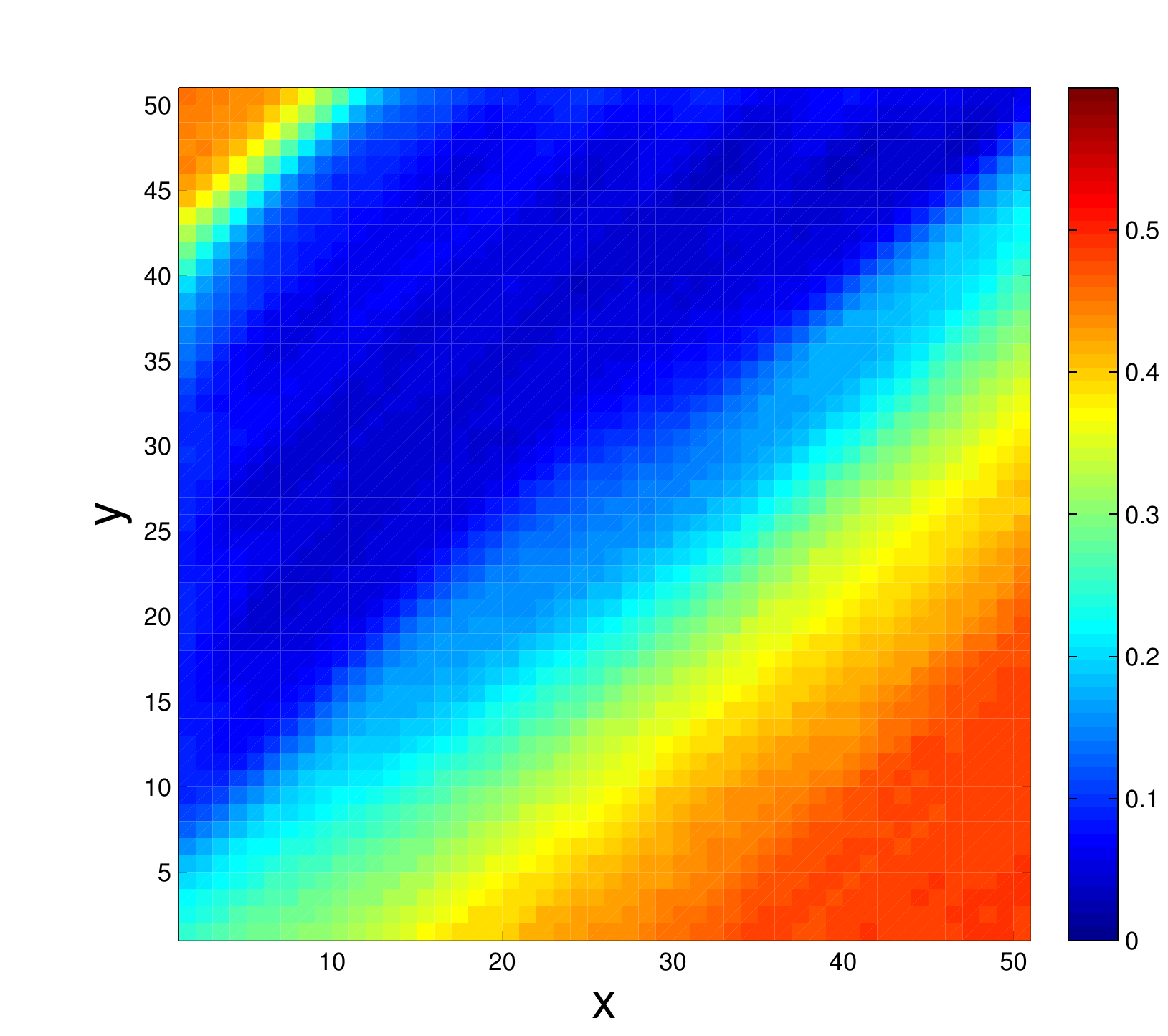}\\
\includegraphics[scale=0.30]{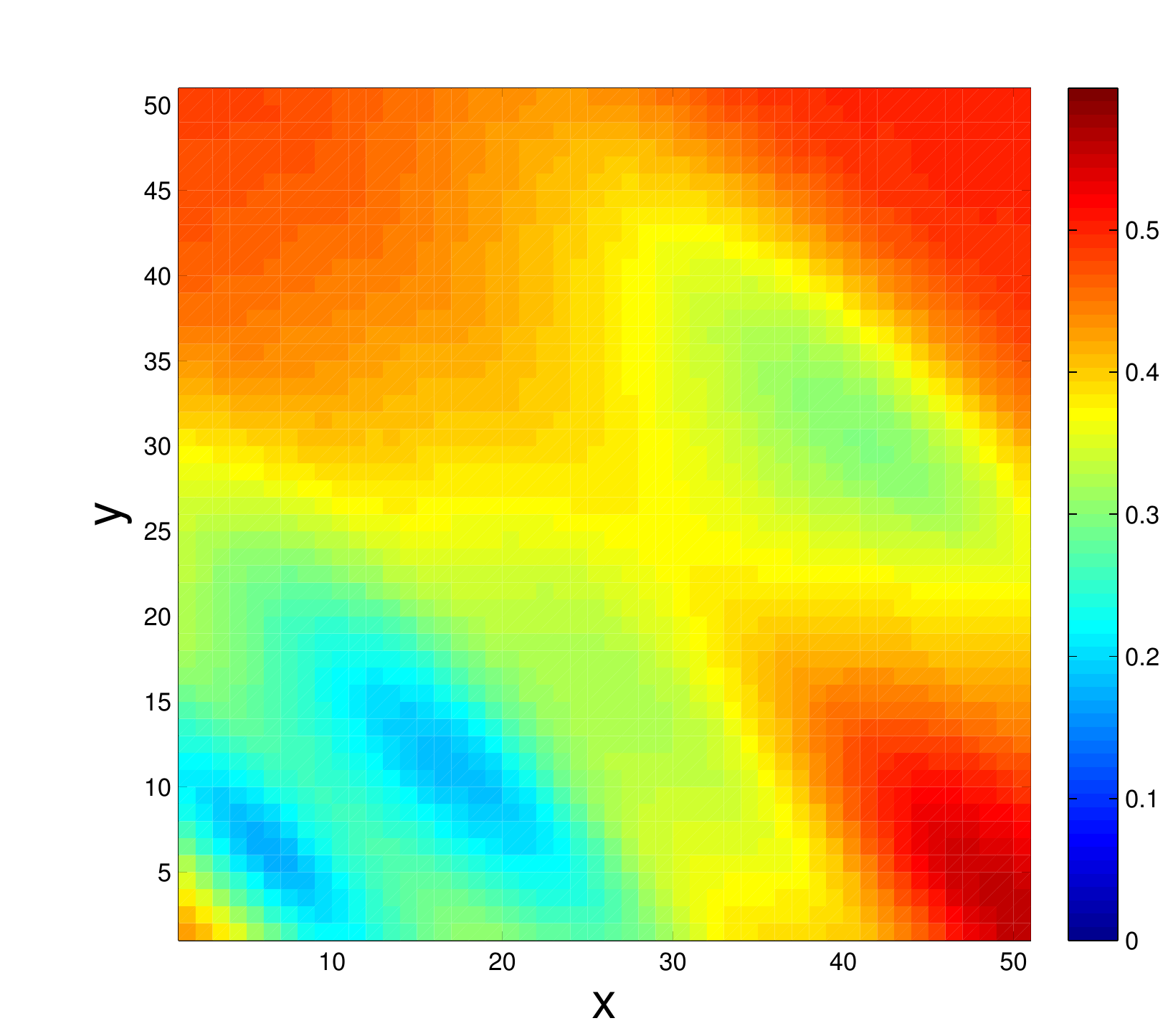}
\includegraphics[scale=0.30]{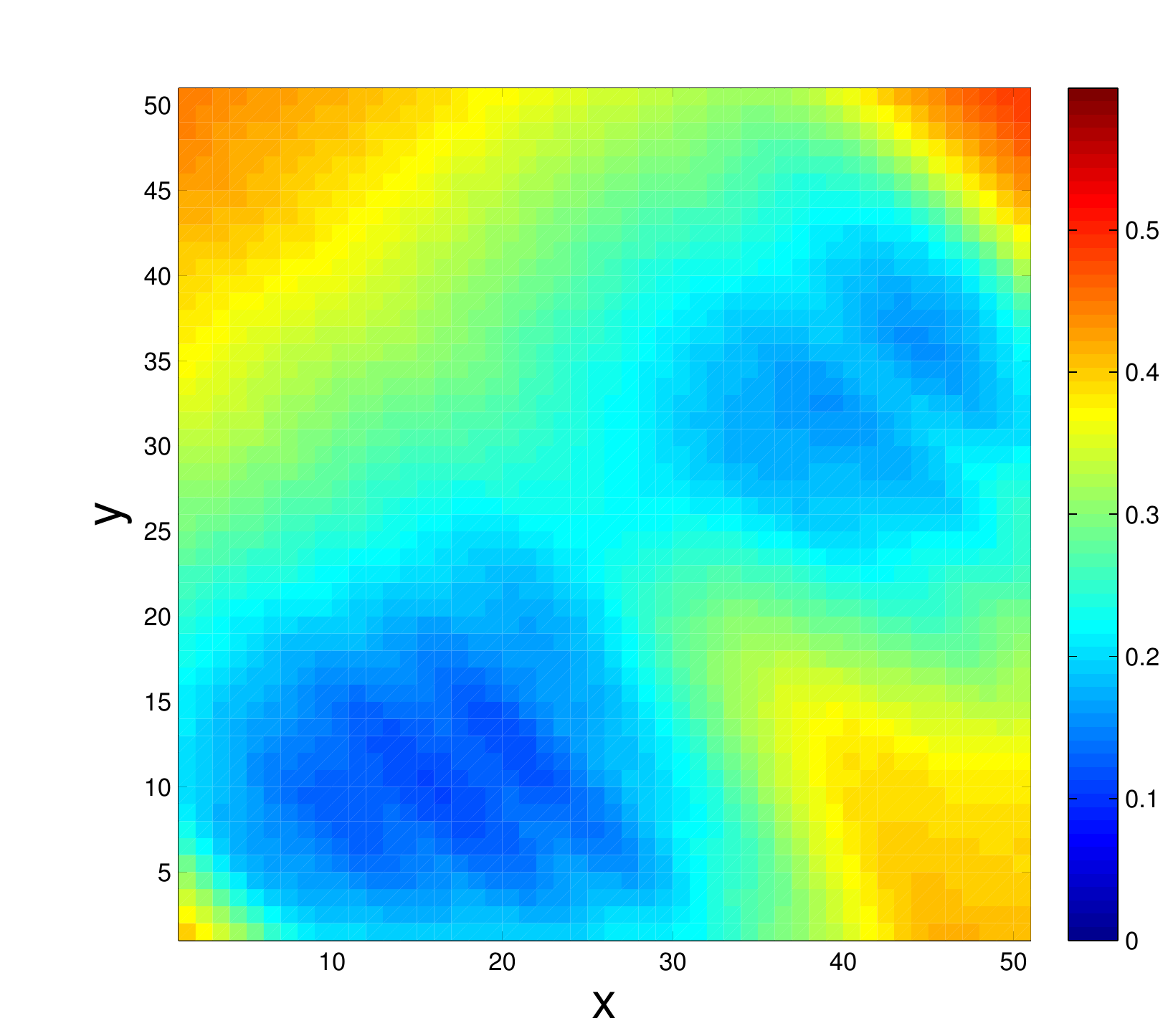}
 \caption{Top to bottom: fewer measurements to more measurements. Left to right: variance of $\log \kappa_{1}$, Right: variance $\log \kappa_{2}$.}

    \label{Figure12}
\end{center}
\end{figure}

\begin{figure}[htbp]
\begin{center}
\includegraphics[scale=0.35]{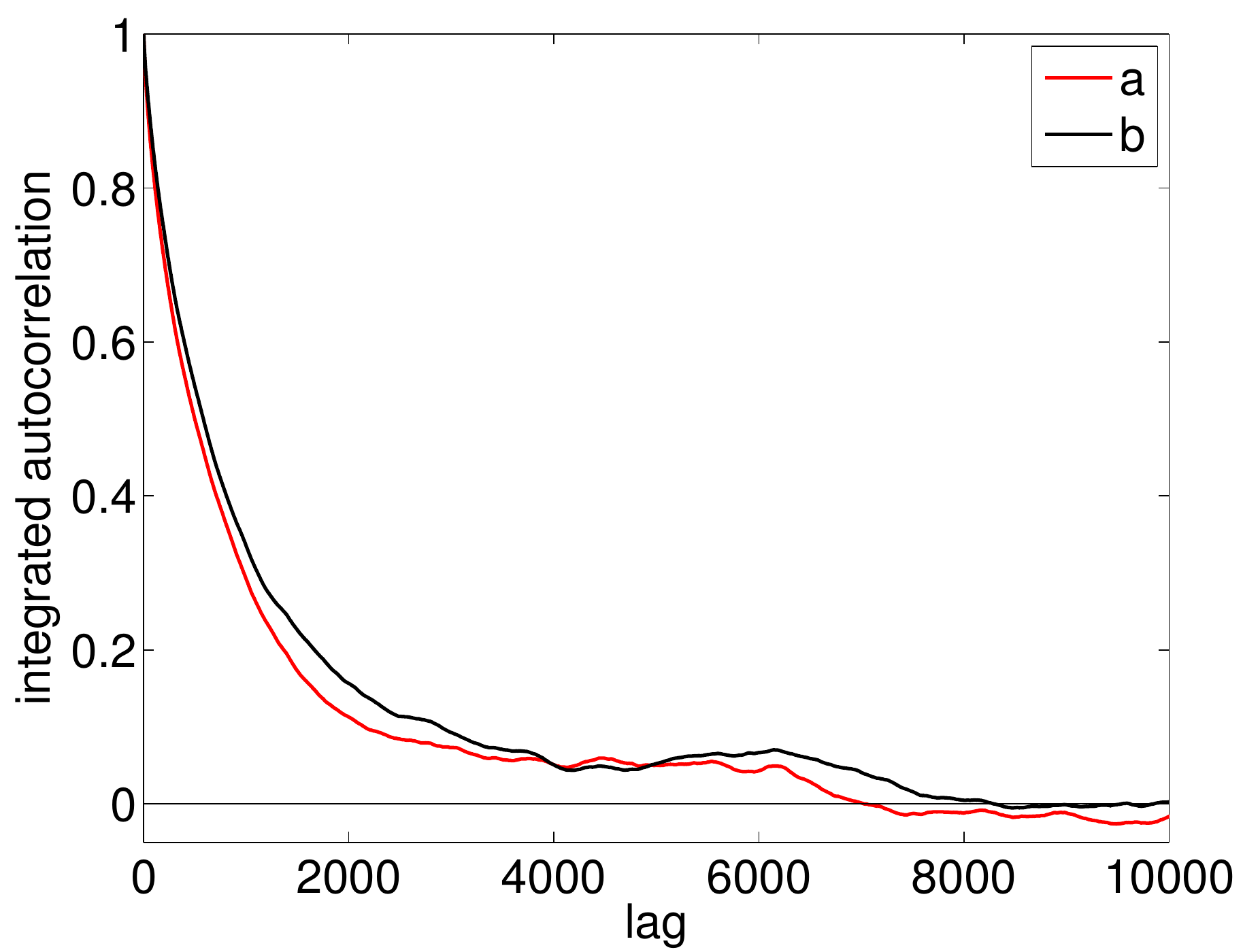}
\includegraphics[scale=0.35]{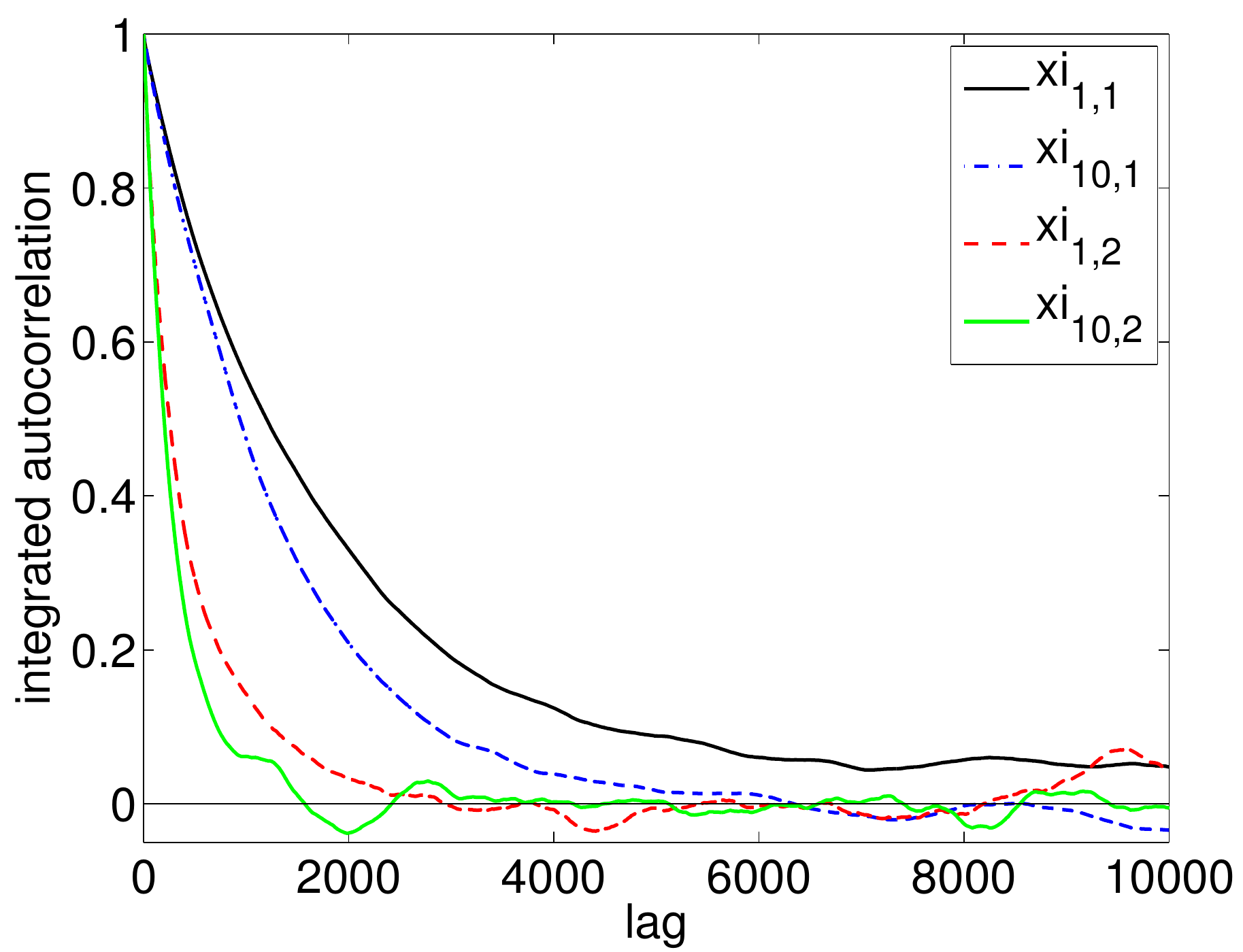}\\
\includegraphics[scale=0.35]{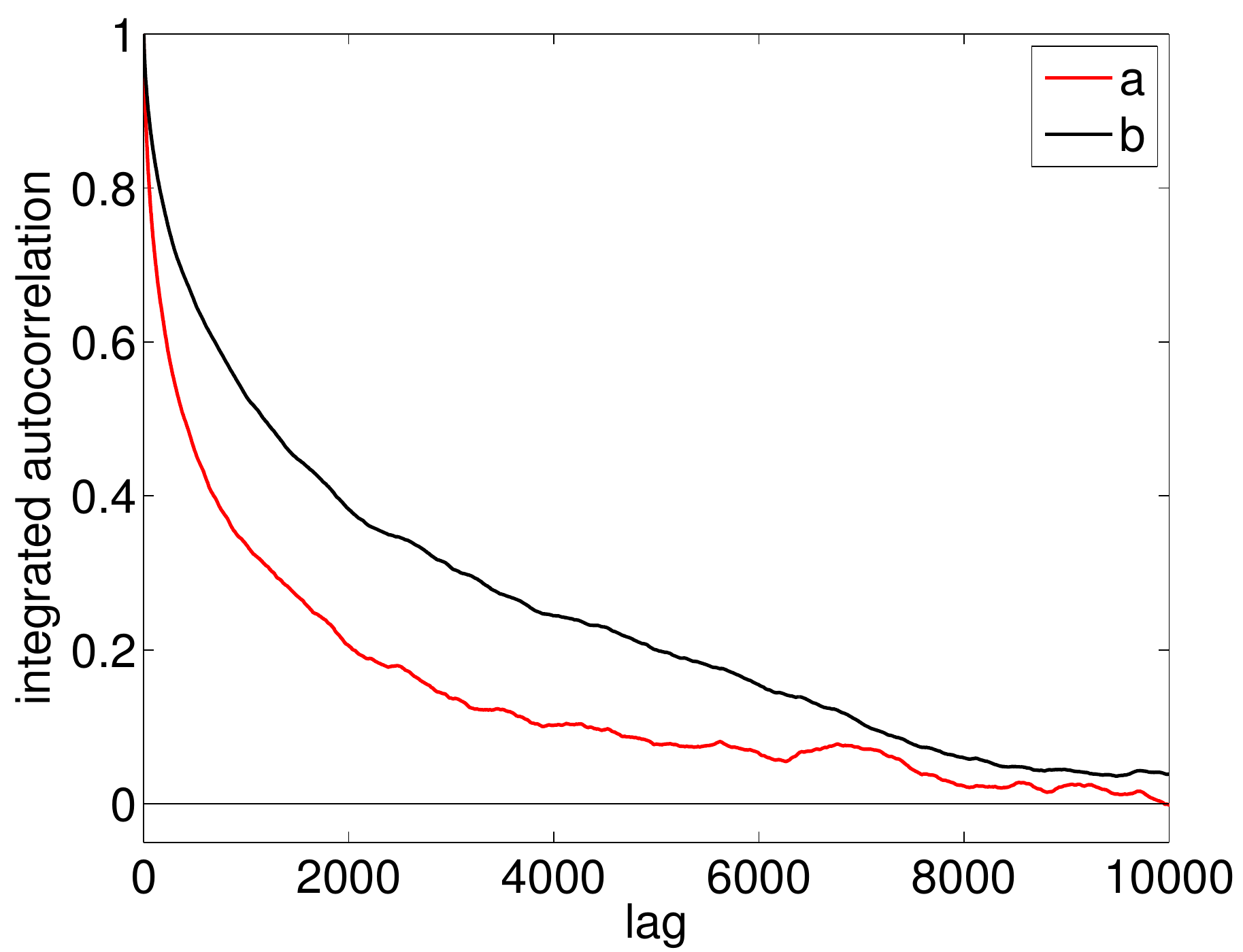}
\includegraphics[scale=0.35]{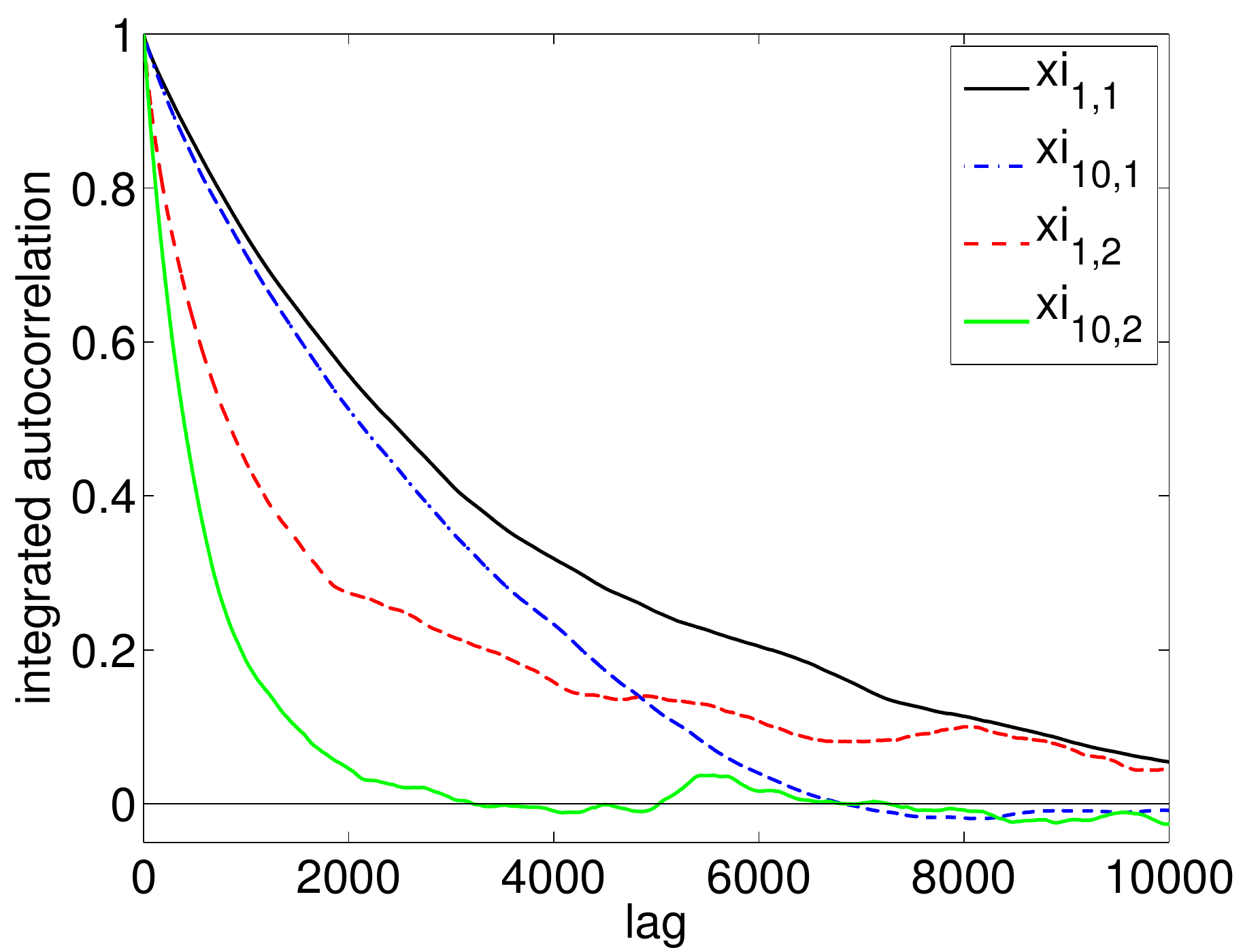}
 \caption{Autocorrelation from one MCMC chain. Top: data 1. Bottom: data 2. Left: geometric parameters. Right: Some KL model of  $\log \kappa_{1}$ and $\log \kappa_{2}$}
    \label{Figure13}
\end{center}
\end{figure}

\begin{figure}[htbp]
\begin{center}
\includegraphics[scale=0.25]{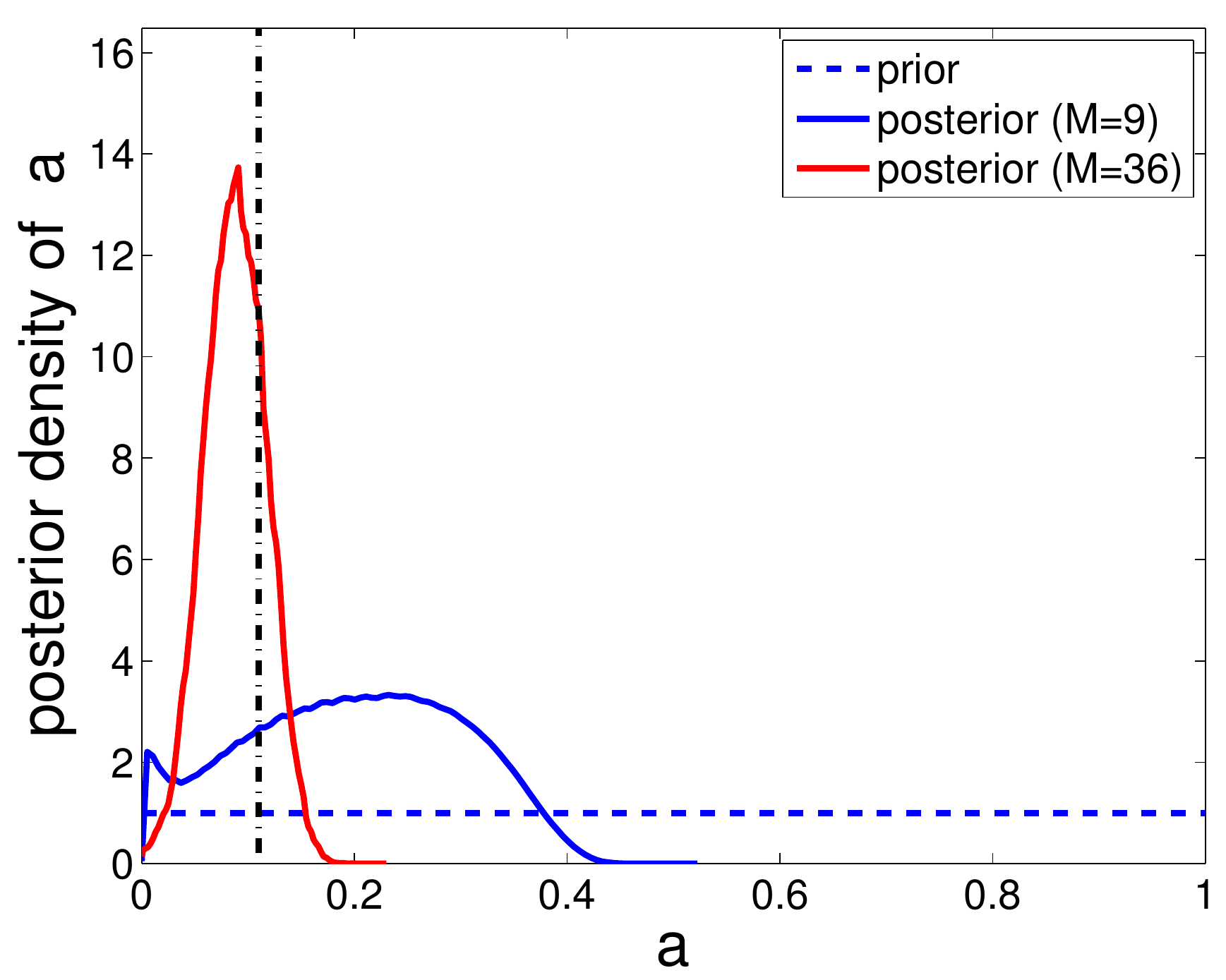}
\includegraphics[scale=0.25]{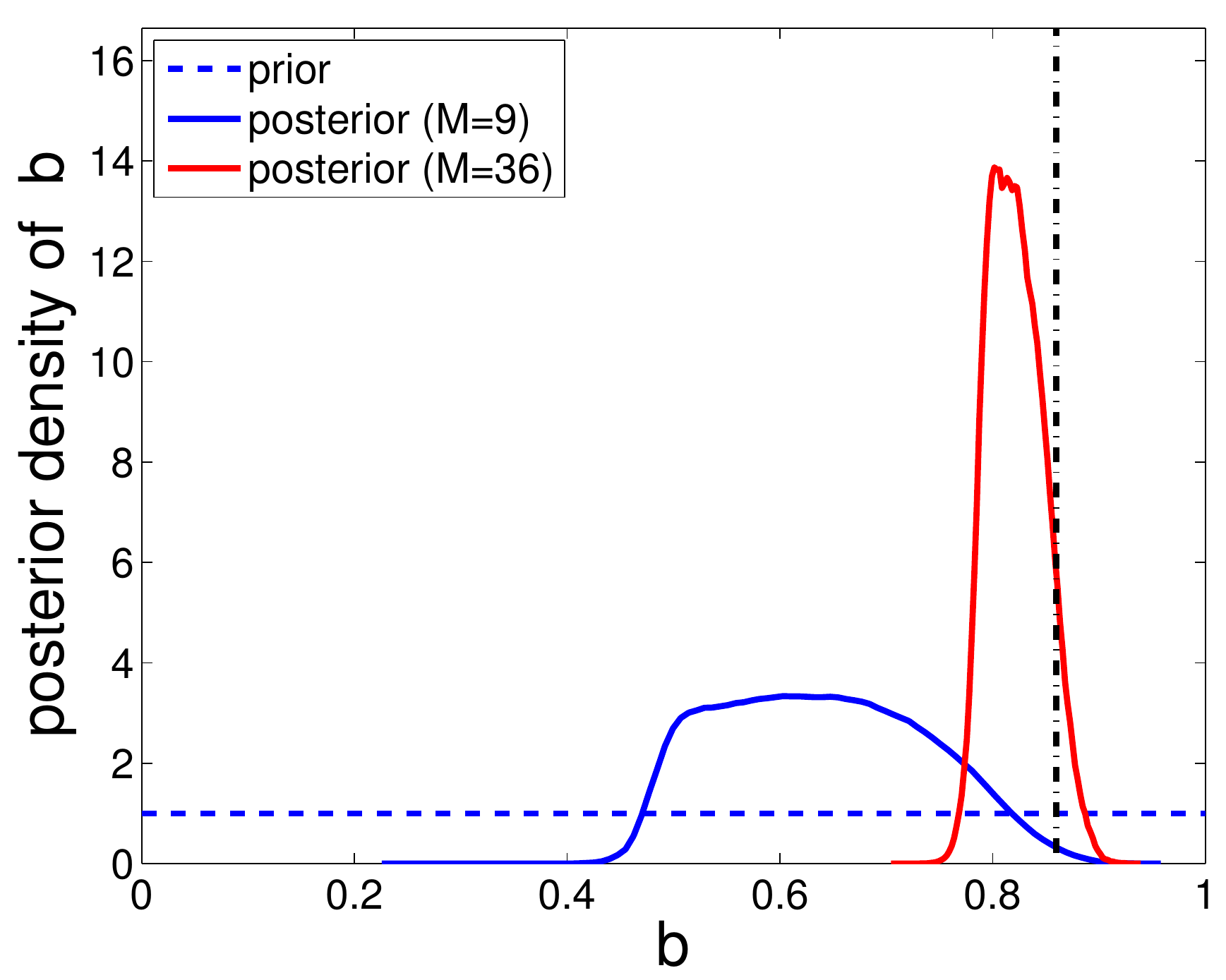}
\includegraphics[scale=0.25]{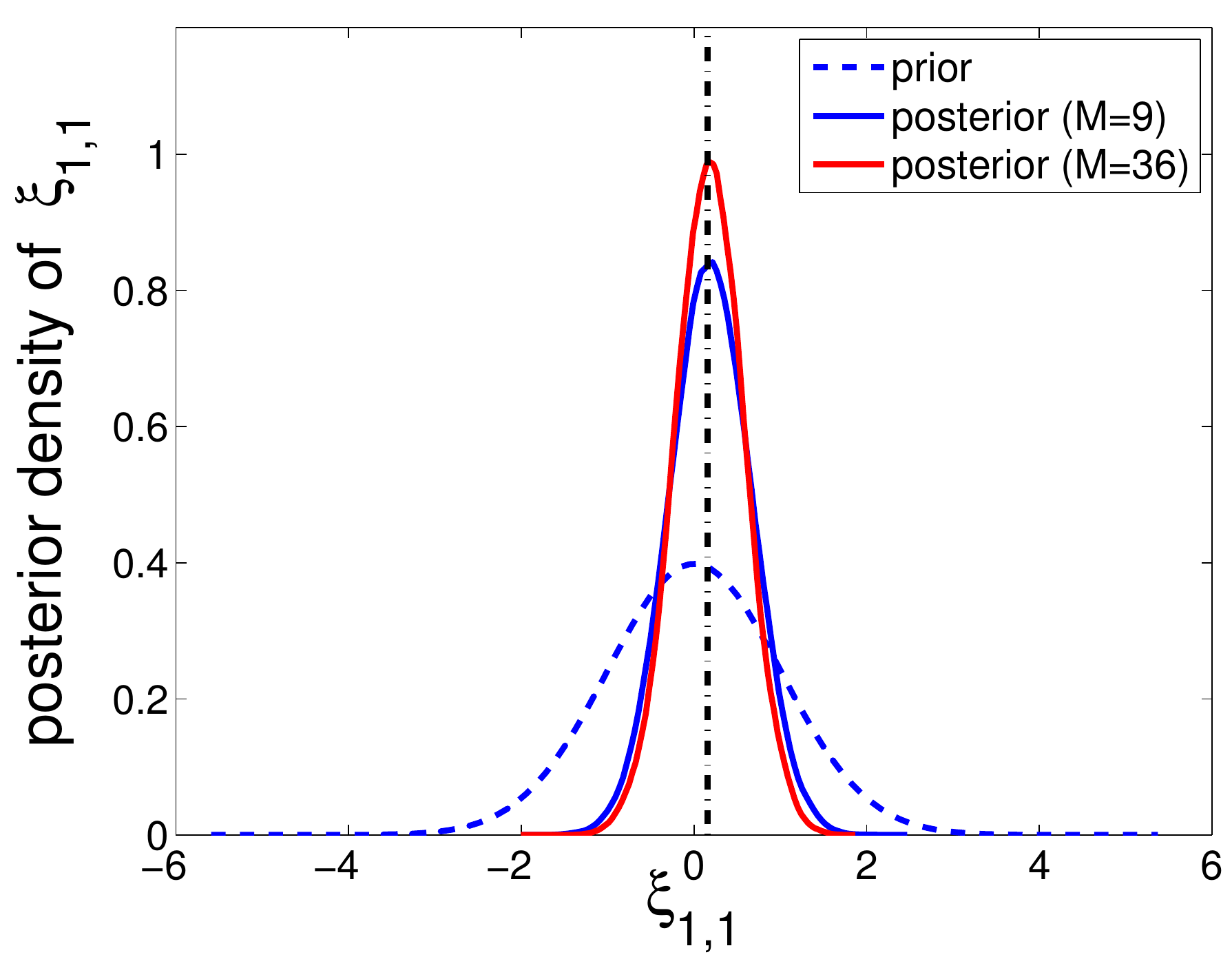}\\
\includegraphics[scale=0.25]{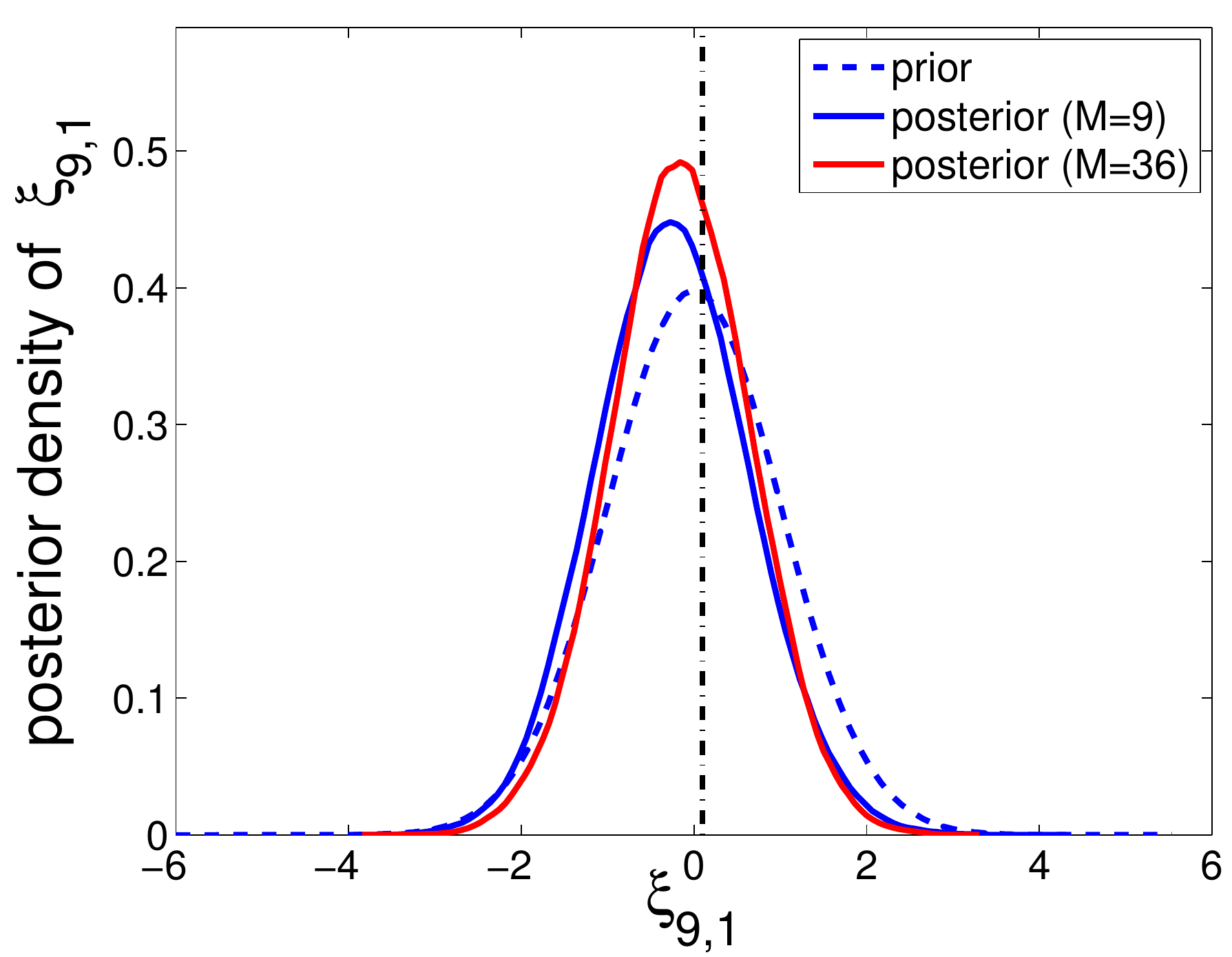}
\includegraphics[scale=0.25]{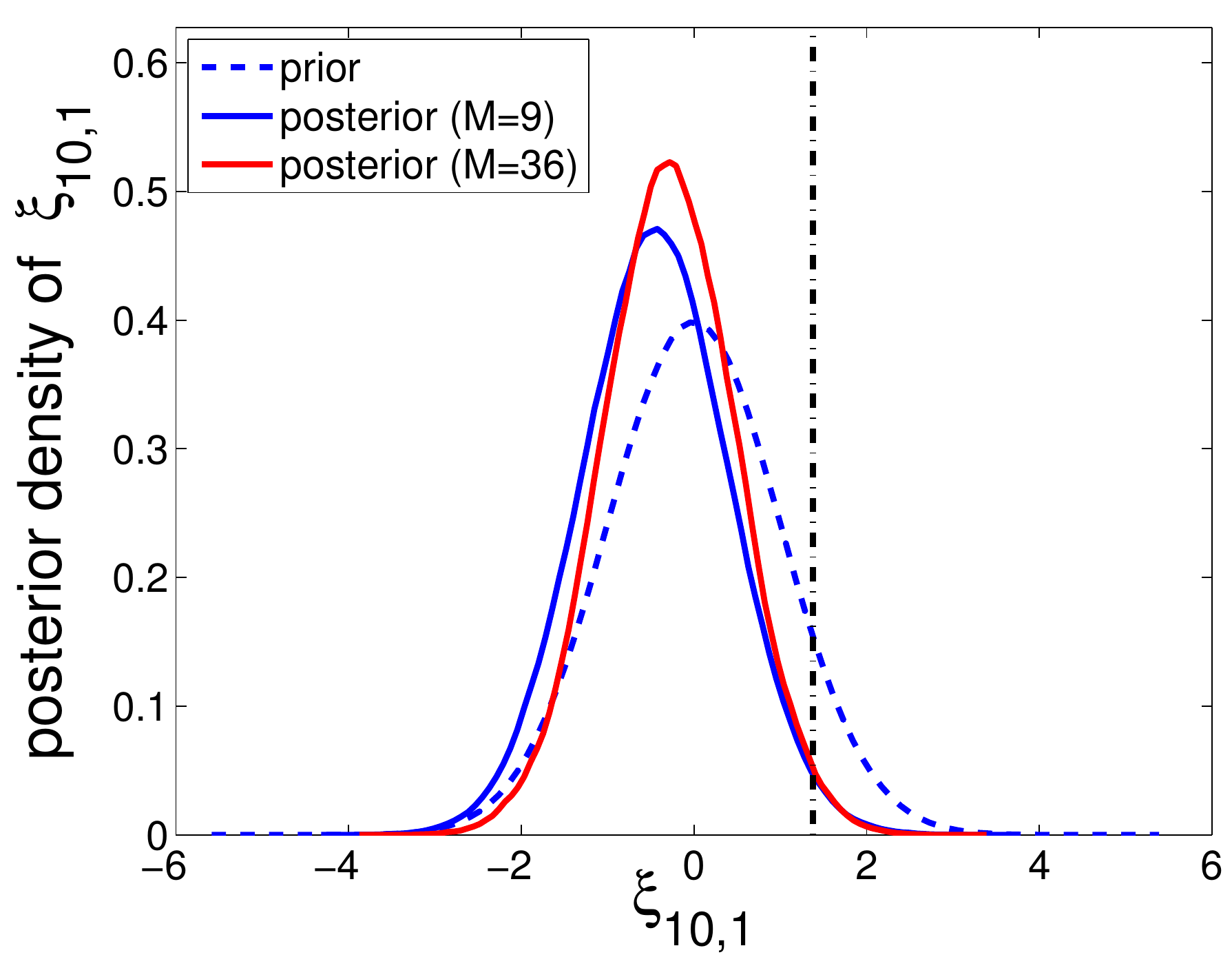}
\includegraphics[scale=0.25]{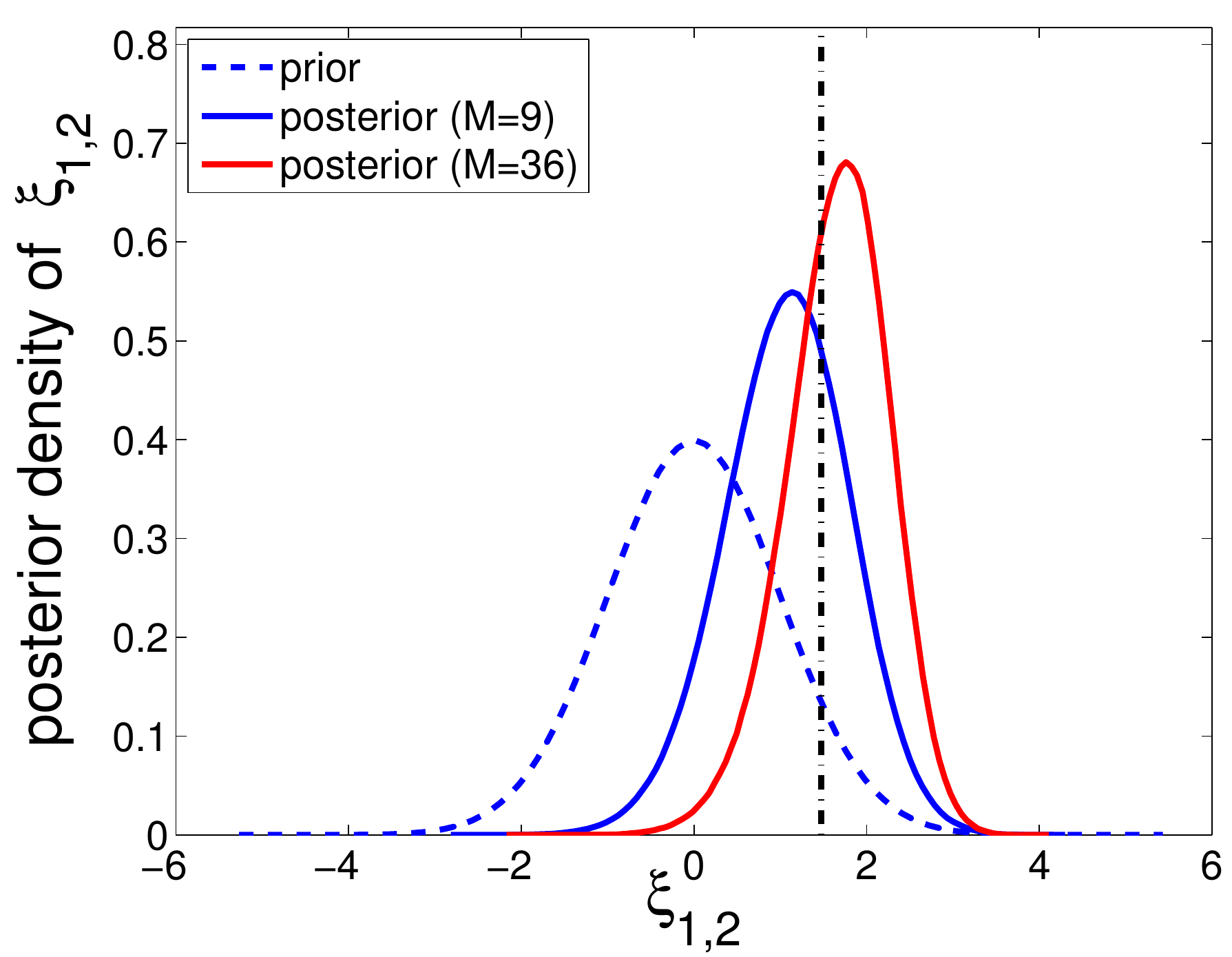}
\includegraphics[scale=0.25]{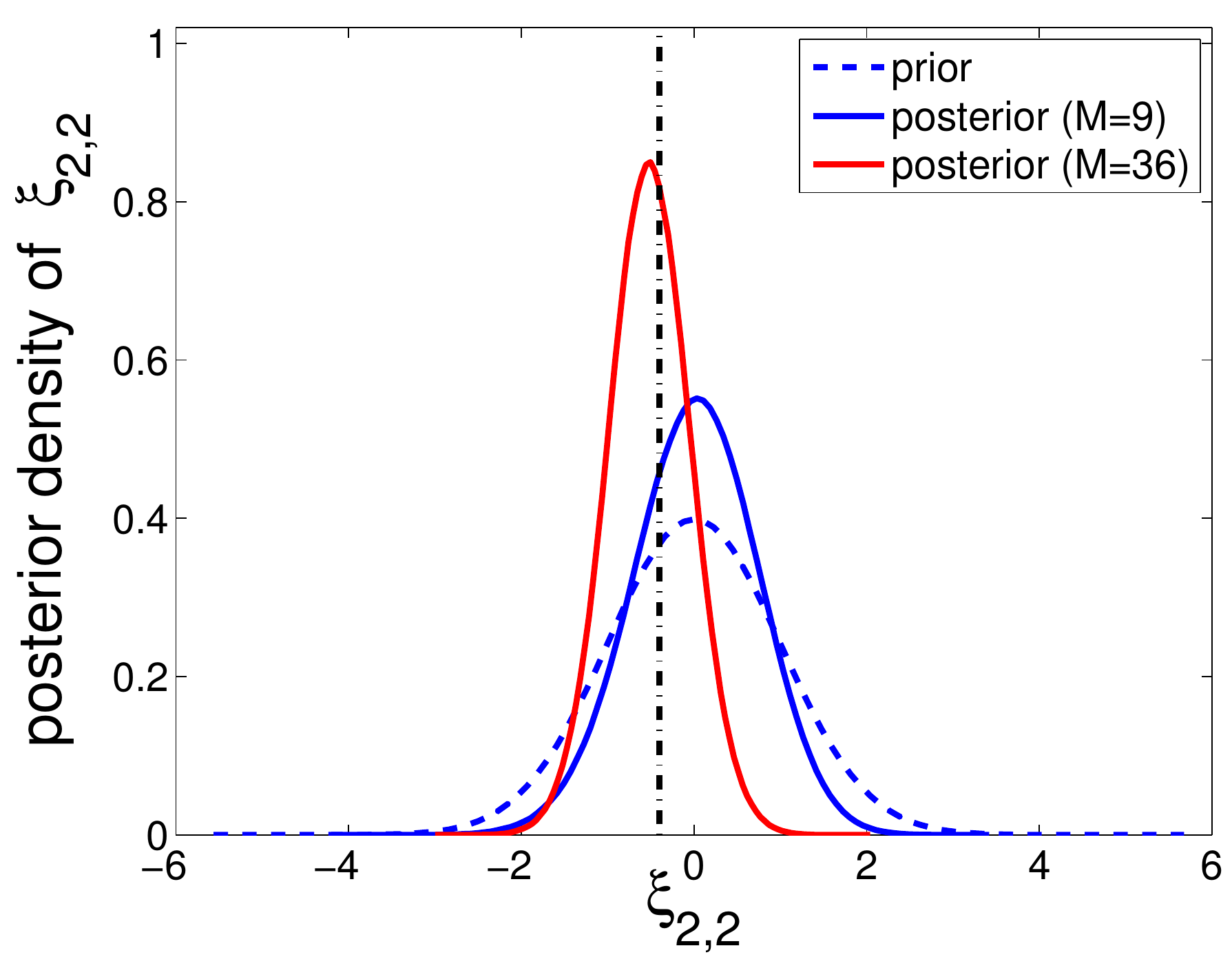}
\includegraphics[scale=0.25]{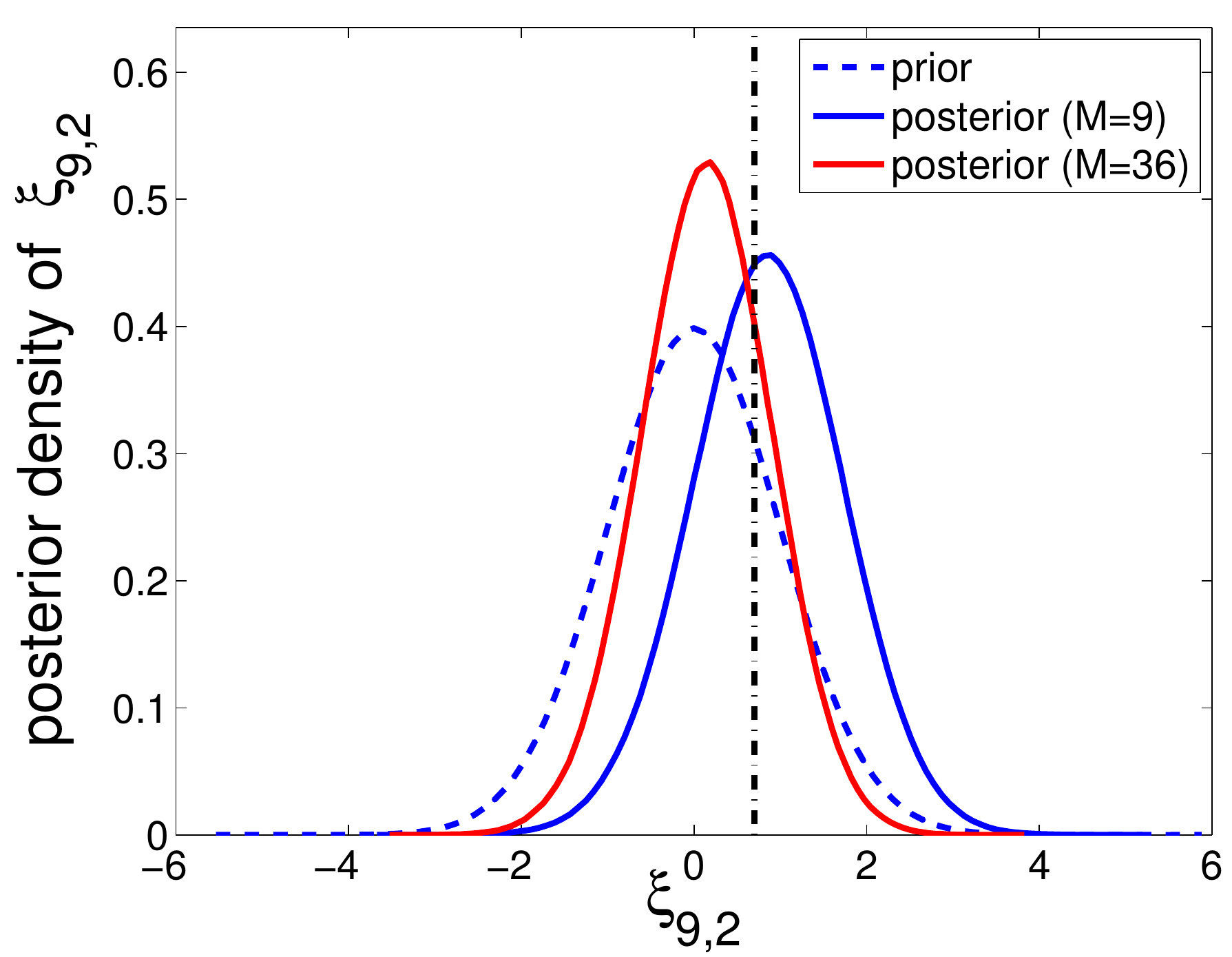}
\includegraphics[scale=0.25]{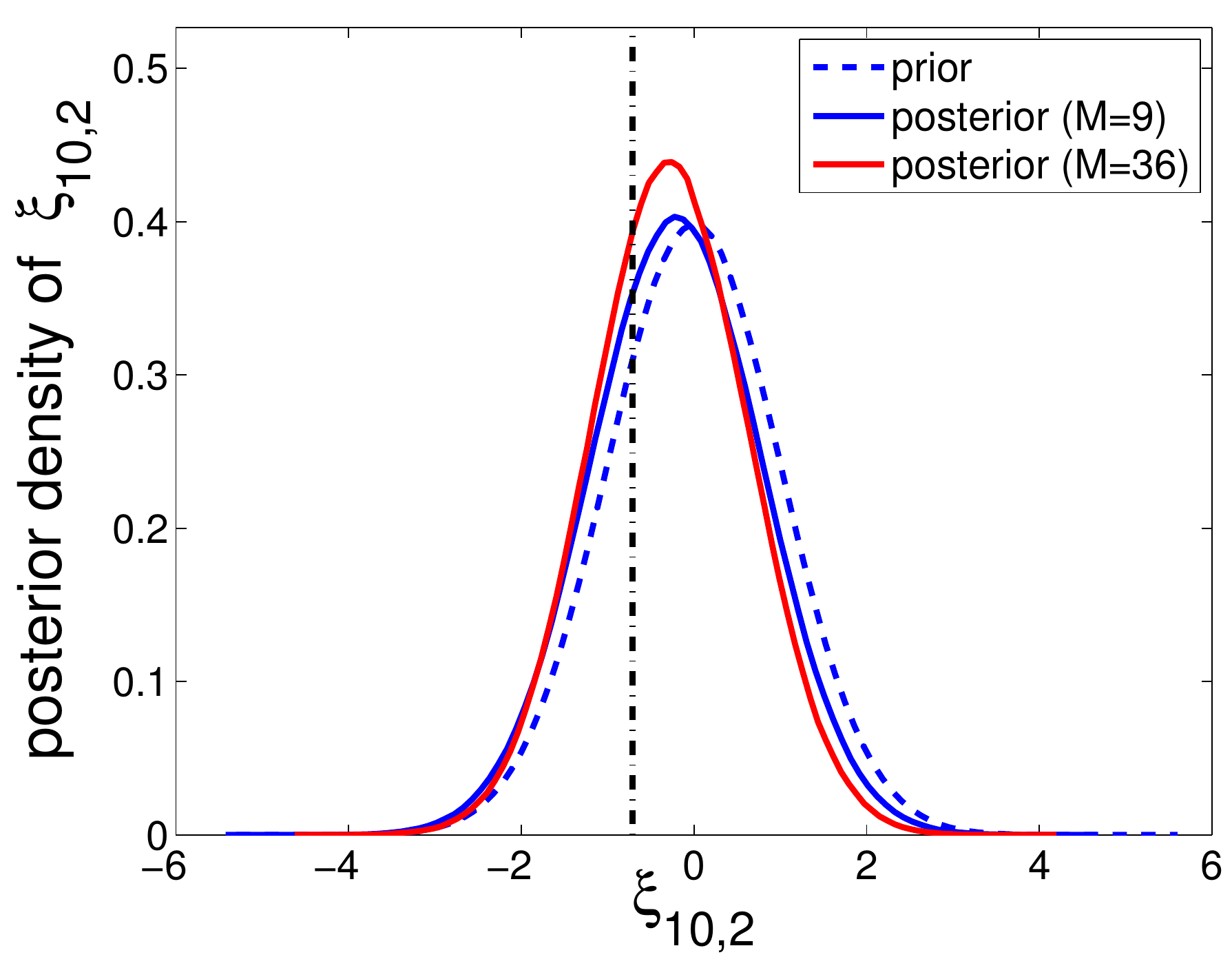}

 \caption{Prior and posterior densities of the unknown $u$}
    \label{Figure14}
\end{center}
\end{figure}

%\begin{figure}[htbp]
%\begin{center}
%\includegraphics[scale=0.25]{density_meas_1B}
%\includegraphics[scale=0.25]{density_meas_4B}
%\includegraphics[scale=0.25]{density_meas_5B}
%\includegraphics[scale=0.25]{density_meas_6B}
%\includegraphics[scale=0.25]{density_meas_7B}
%\includegraphics[scale=0.25]{density_meas_9B}
 %\caption{Posterior densities of the data predictions $G(u)$ (for the case of 9 measurements). Vertical lines indicate the nominal value of synthetic observations: dotted black}
 %   \label{Figure15}
%\end{center}
%\end{figure}

\clearpage

\subsection{Channelized Permeability}

In this experiment we consider channelized permeabilities of the form (\ref{eq:num3}) where $D_{1}$ and $D_{2}$ are the  domains corresponding to the interior and exterior of the channel, respectively. These domains are  parametrized with five parameters as shown in Figure \ref{Figure1B}. Two fields for $\log \kappa_{i}$ are considered as in the previous experiments.  The unknown parameter in this case is $u=(d_{1},\dots,d_{5},\log \kappa_{1}, \log\kappa_{2})\in \mathbb{R}^5\times C(\bD;\bbR^2)$. We consider a prior distribution of the form
\begin{equation}\label{eq:num4B}
\fl \mu_{0}(du)=  \Pi_{i=1}^{5}\pi_0^{B_{i},g}(d_{i})\otimes N(m_1,C_{1}) N(m_2,C_{2})
\end{equation}
where the set $B_{i}$ in the definition of $\pi_0^{B_{i},g}(\cdot)$ (for each parameter)  is a specified interval in $\mathbb{R}$, and $C_{1}$ and $C_{2}$ are covariance operators as in the previous experiment.

The true log-permeability $\log \kappa^{\dagger}$, shown in Figure \ref{Figure16} (left),  is obtained from (\ref{eq:num3}) with $\log \kappa_{1}^{\dagger}$ displayed in Figure \ref{Figure17} (top-left) $\log \kappa_{2}^{\dagger}$ shown in Figure \ref{Figure17} (bottom-left) and from the geometric parameters specified in Table \ref{Table3}. The true fields $\log \kappa_{i}^{\dagger}$ are generated as described in the preceding subsection. We consider the measurement configuration displayed in Figure \ref{Figure16} (right) and generate synthetic data as before. For the present example, the measurements are corrupted with noise with standard deviation of $\gamma=2.5\times 10^{-4}$. Algorithm \ref{MwG} is applied to characterize the posterior distribution with 20 parallel chains that passed Gelman-Rubin MPSRF diagnostic \cite{Gelman}. For this experiment, we use an outer Gibbs loop in which we update each of the geometric parameters independently; then, both log-permeabilities field are updated simultaneously. Combined samples from all chains are used to compute the mean and variance of the unknown. Trace plots can be found in Figure \ref{Figure19}. In Figure \ref{Figure18} we show log-permeabilities with parameter $u$ sampled from the prior (\ref{eq:num4B}) (top row)  and the posterior  (bottom row) respectively. In Figure \ref{Figure17} we show the posterior mean and variance for the fields $\log \kappa_{1}$ (second column) and $\log \kappa_{2}$ (third column), respectively. The mean and variance for the geometric parameters are reported in Table \ref{Table3}. The permeability (\ref{eq:num3}) corresponding to the mean $\hat{u}$ is displayed in Figure \ref{Figure16} (middle). The autocorrelation of the geometric parameters and KL-coefficients $\xi_{1,i}$ and $\xi_{10,i}$ are shown in Figure \ref{Figure20}. For these variables, in Figure \ref{Figure21} we display the posterior and the prior densities. Similar to our previous experiments, we observe that the log-permeability obtained with the mean parameters resembles the truth. However, the uncertainty in the problem is reflected in the variability in the possible inverse estimates of the log-permeability.

\begin{table}[!ht]
\centering
\begin{tabular}{|c|c|c|c|}
  \hline
  % after \\: \hline or \cline{col1-col2} \cline{col3-col4} ...
    parameter  &  true value  &   mean &  variance \\
\hline
(amplitude) & $0.2$ & 0.225& $1.3\times 10^{-3}$\\
(frecuency) & $11$ &11.161& $1.12\times 10^{-1}$\\
(angle) & $0.39$ &0.363 &$5.3\times 10^{-3}$\\
(initial point) & $0.4$ &0.388& $1.2\times 10^{-3}$\\
(width) & $0.3$   &0.262& $1.6\times 10^{-3}$\\
  \hline
\end{tabular}
 \caption{Data relevant to the experiment with the channelized permeability.}
 \label{Table3}
\end{table}

\begin{figure}[htbp]
%\begin{center}
\includegraphics[scale=0.30]{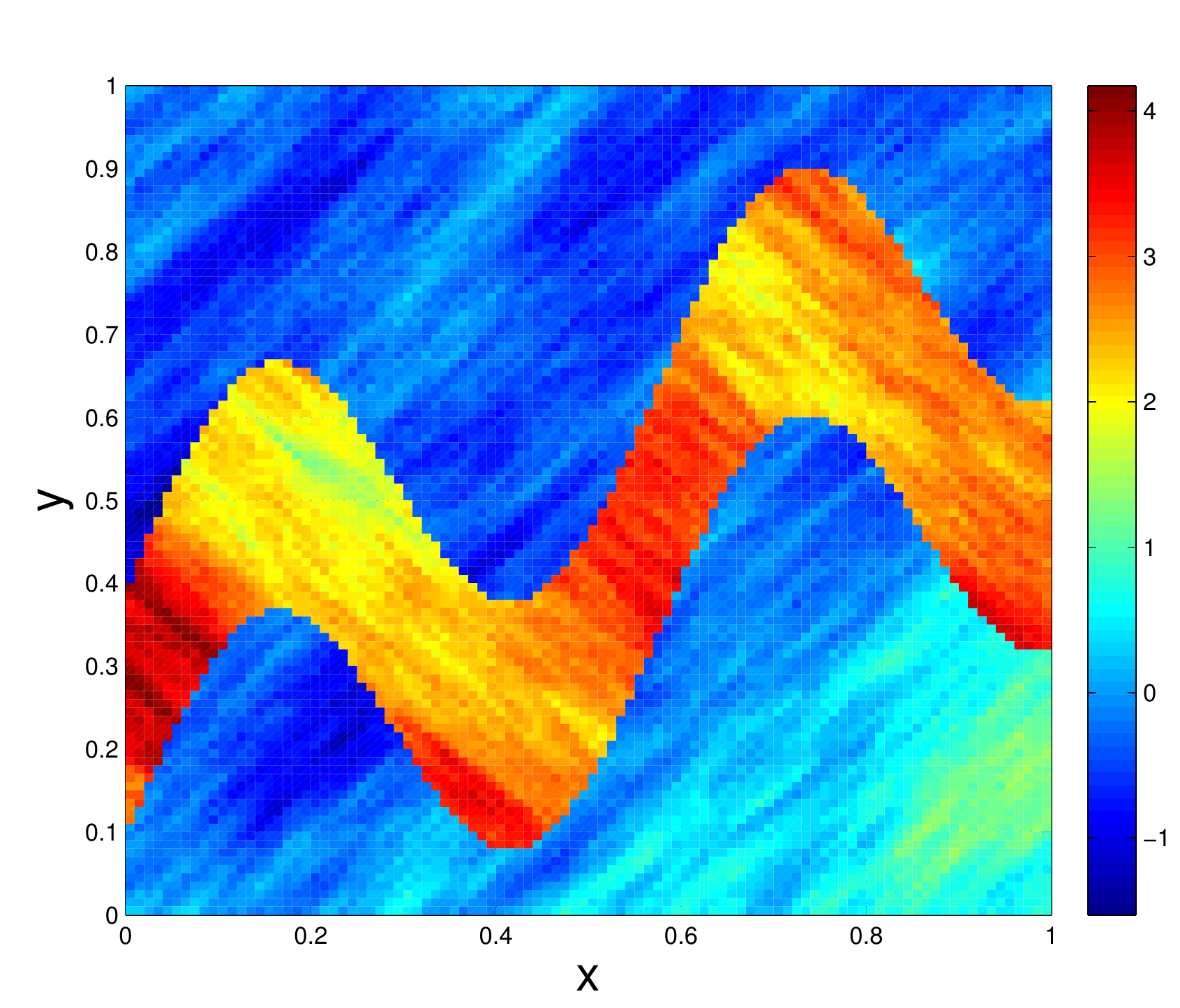}
\includegraphics[scale=0.30]{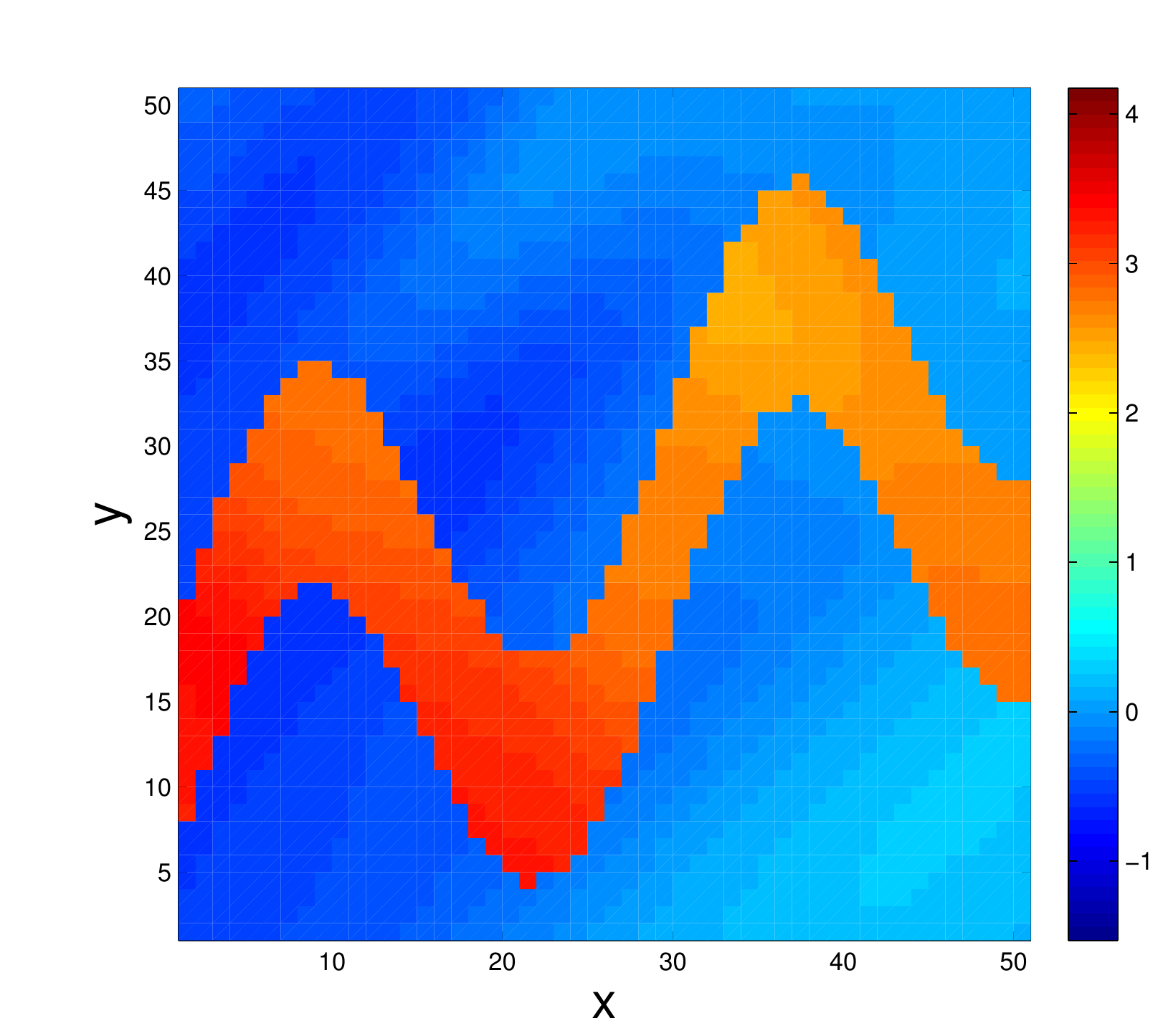}
\includegraphics[scale=0.30]{Fault_wells}

 \caption{Left: true $\log \kappa$. Middle: mean $\log \kappa$.  Right: Measurement configuration.}
    \label{Figure16}
%\end{center}
\end{figure}

\begin{figure}[htbp]
%\begin{center}
\includegraphics[scale=0.30]{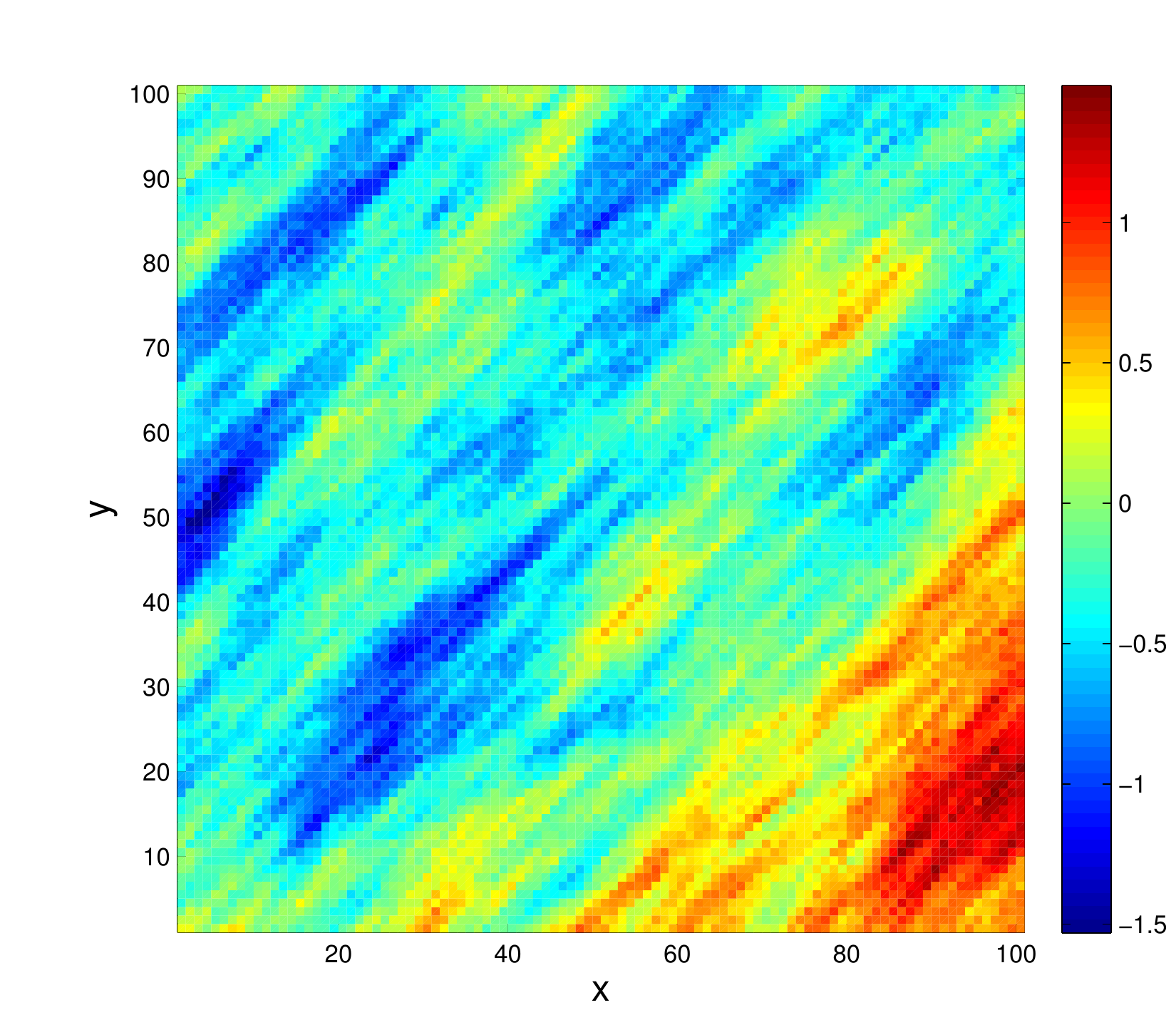}
\includegraphics[scale=0.30]{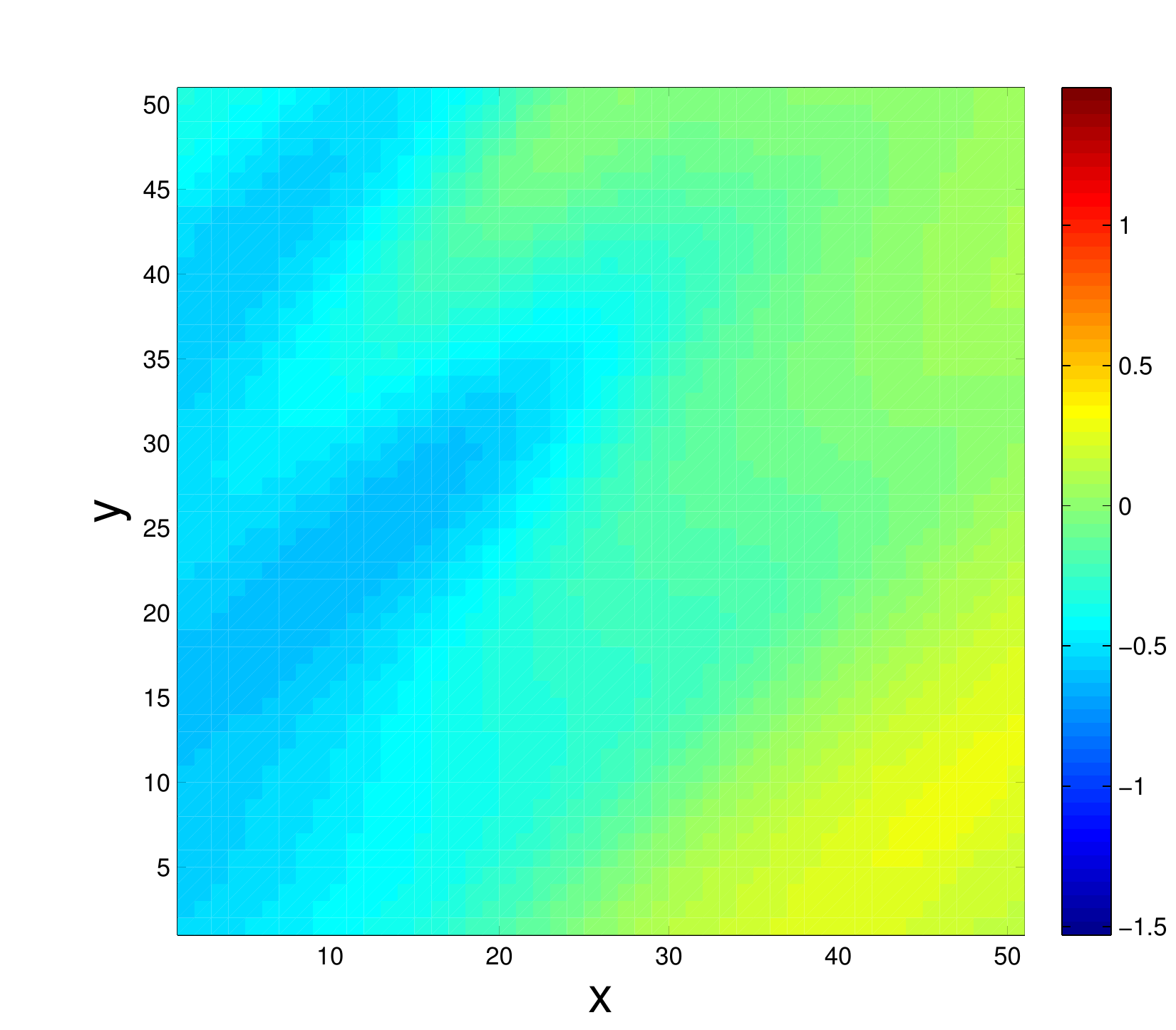}
\includegraphics[scale=0.30]{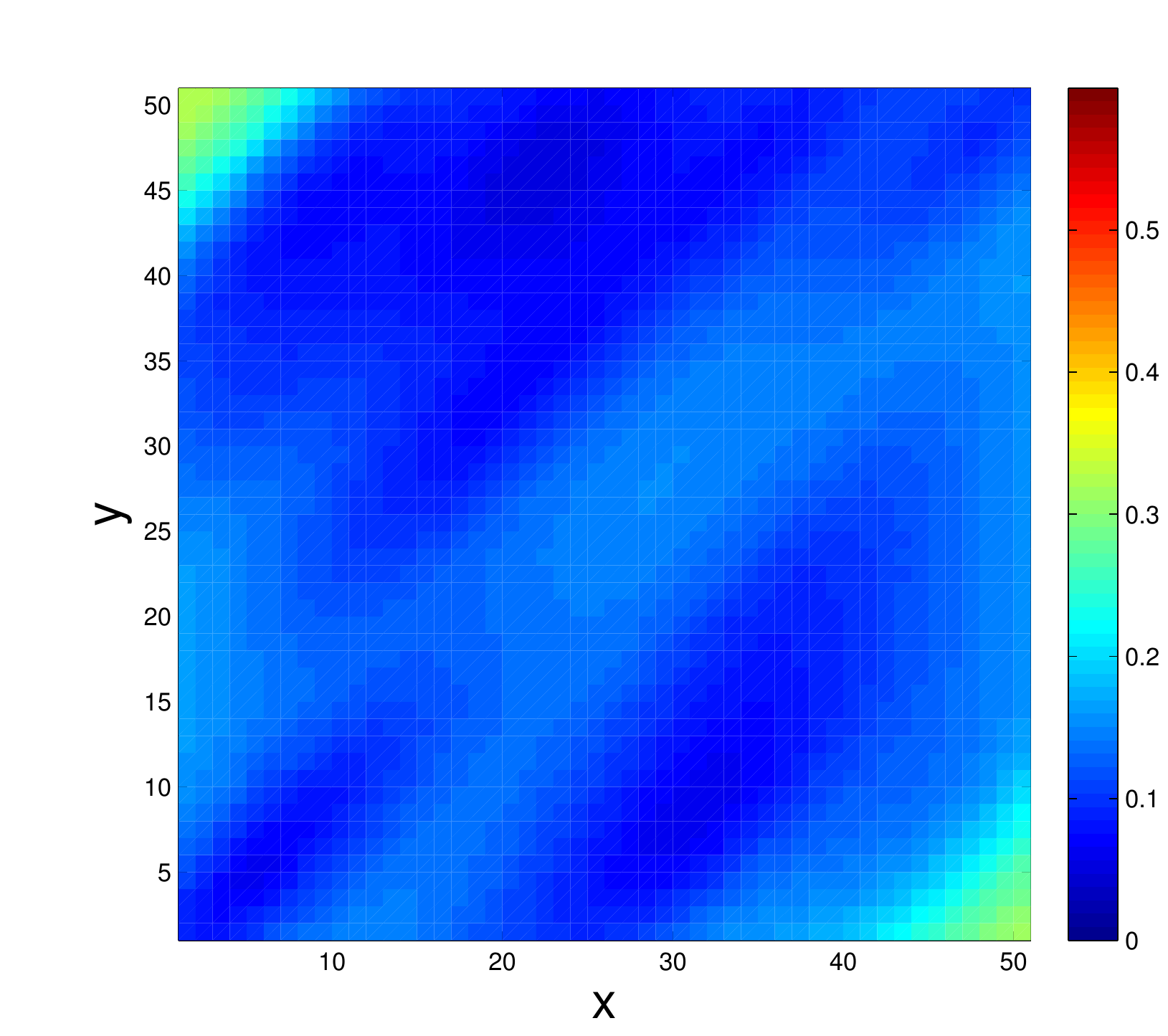}\\
\includegraphics[scale=0.30]{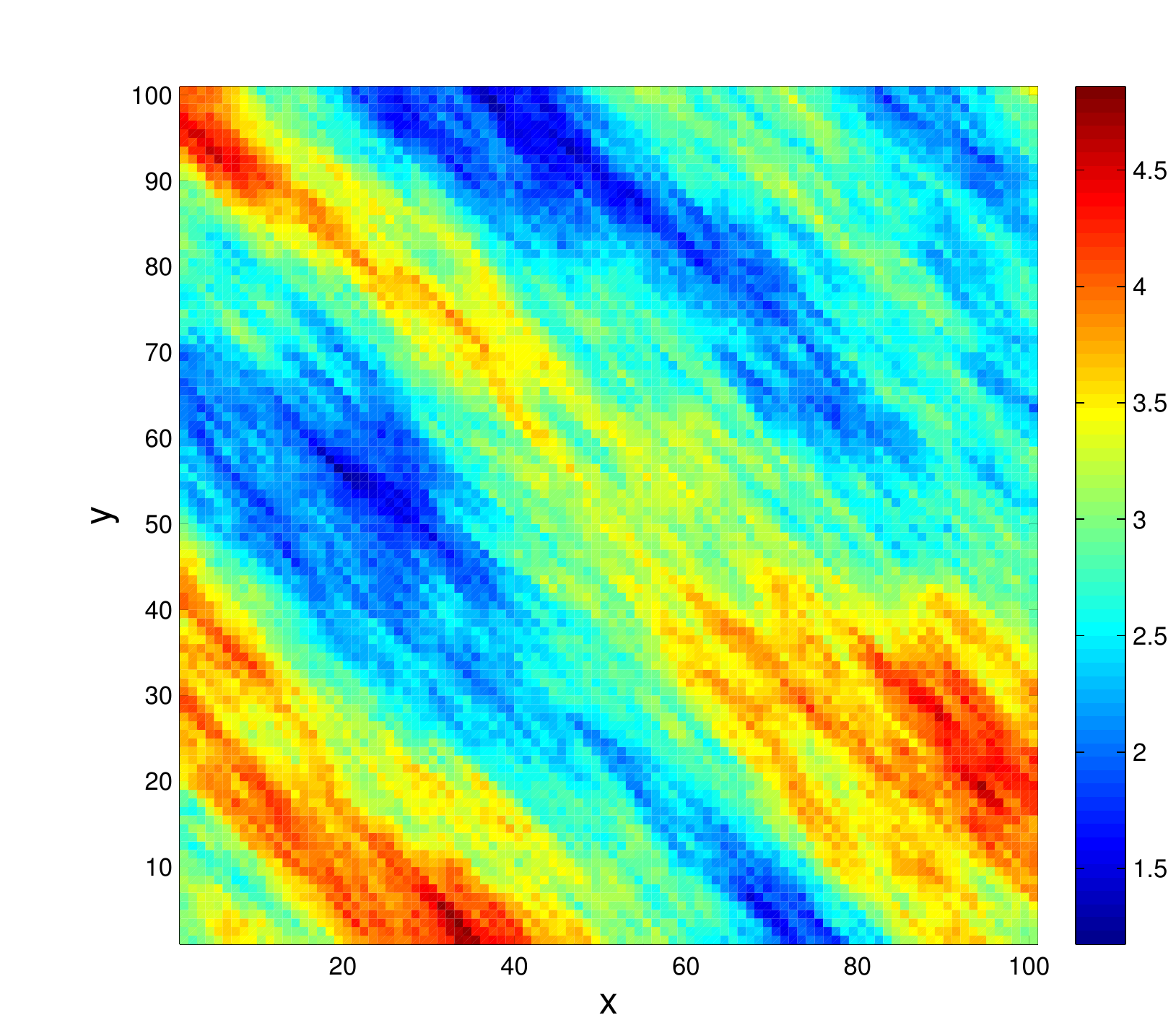}
\includegraphics[scale=0.30]{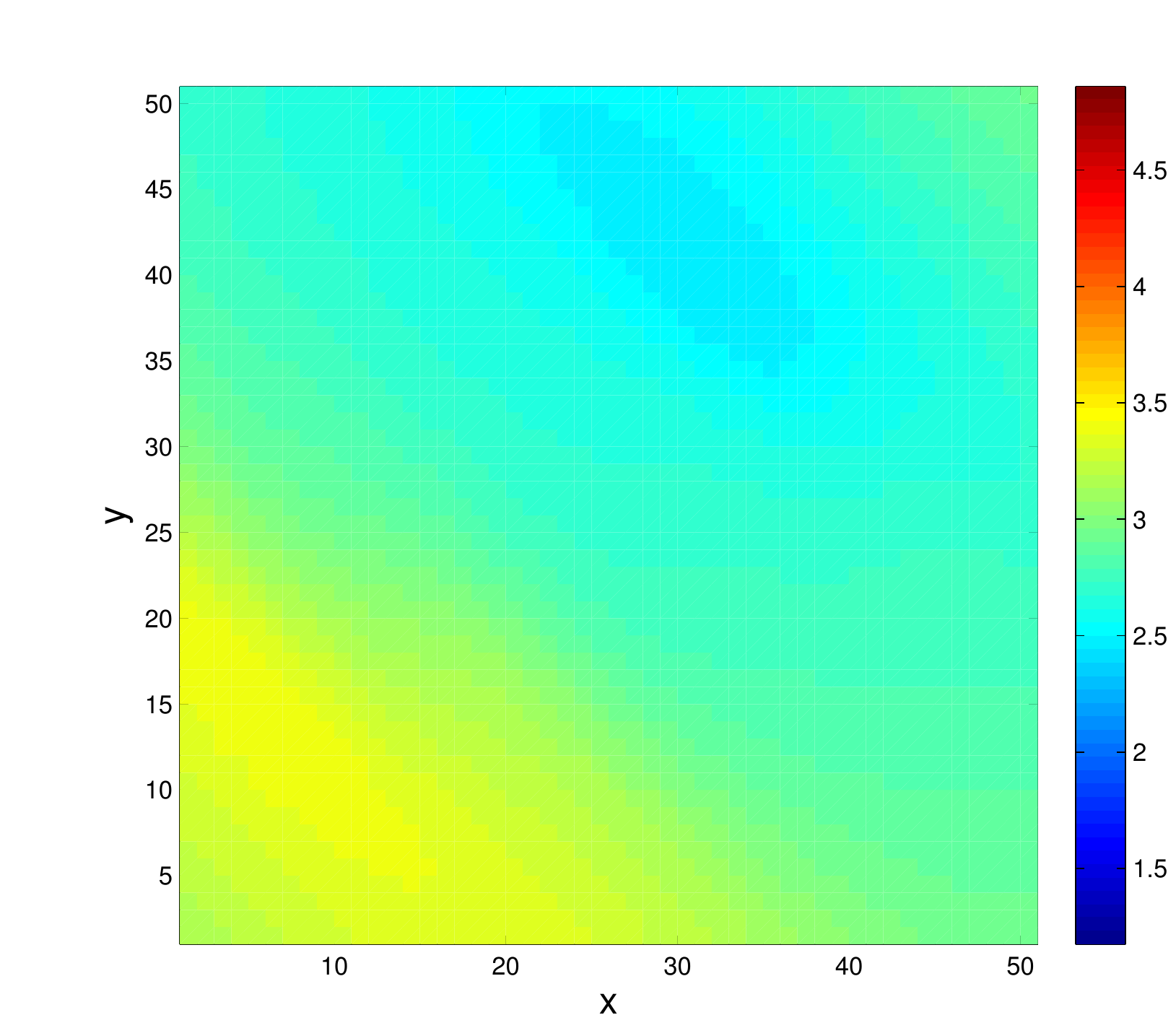}
\includegraphics[scale=0.30]{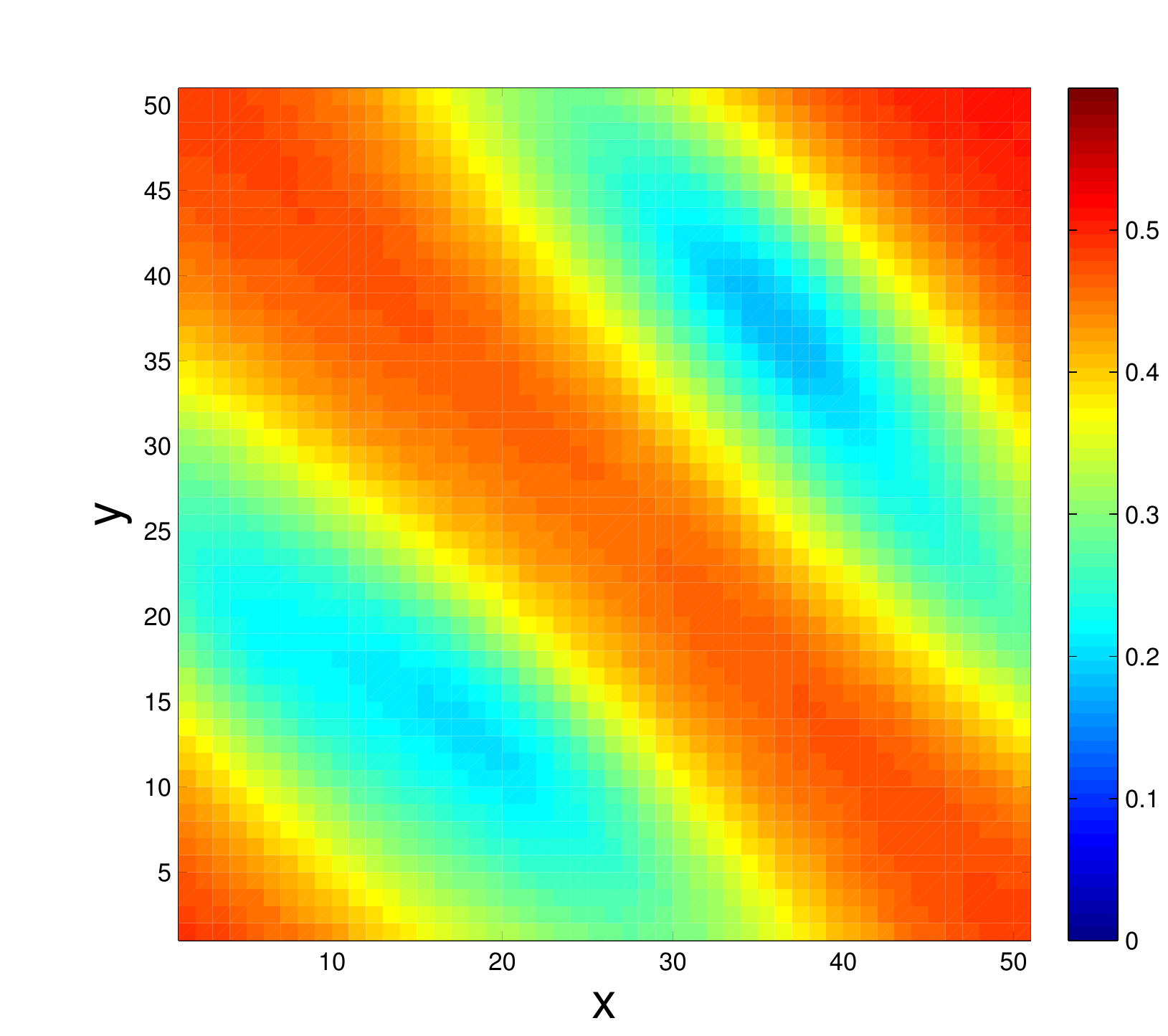}
 \caption{From left to right: truth, mean and variance of $\log \kappa_{1}$ (top) and $\log \kappa_{2}$ (bottom)}
    \label{Figure17}
%\end{center}
\end{figure}

\begin{figure}[htbp]
\begin{center}
\includegraphics[scale=0.2825]{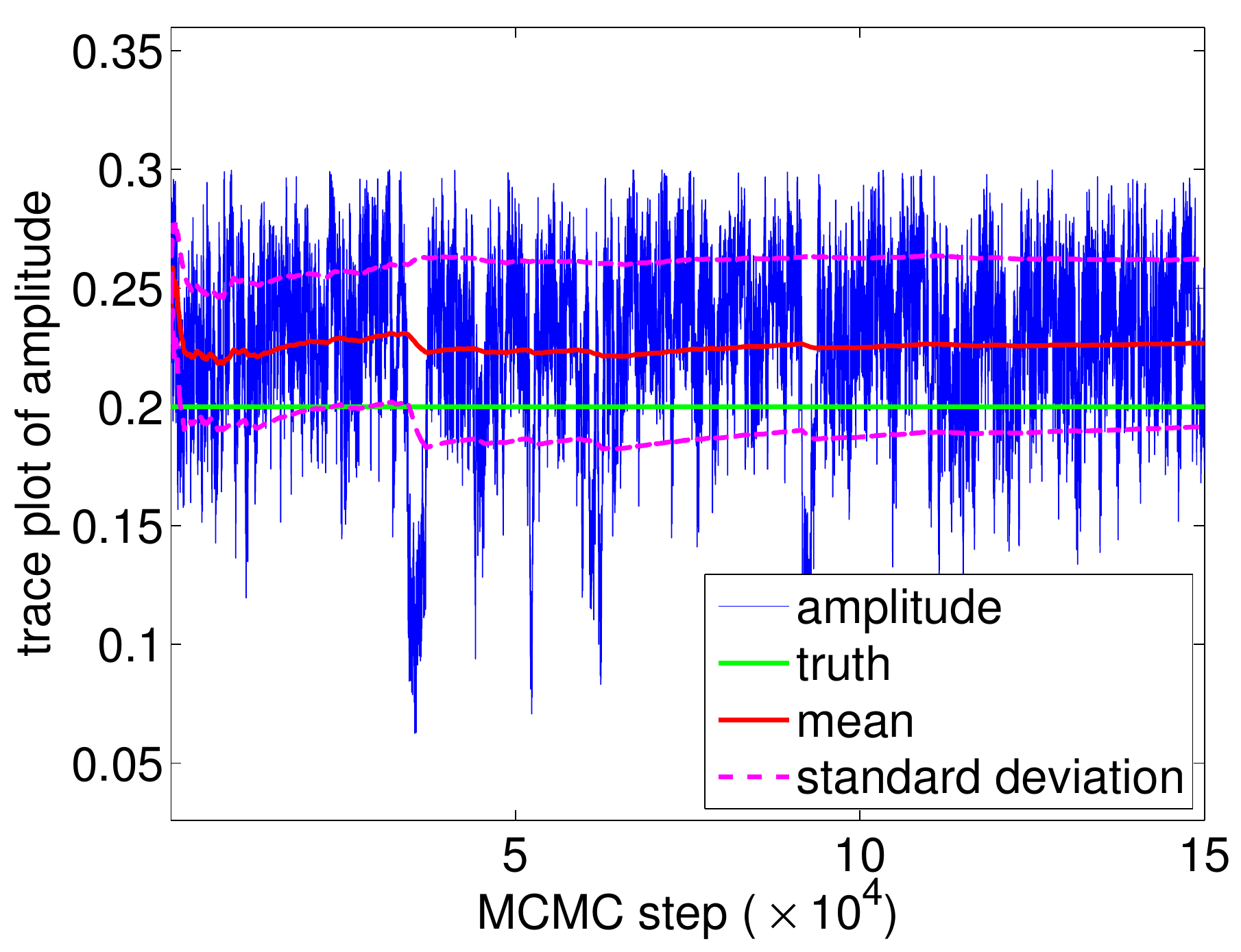}
\includegraphics[scale=0.2825]{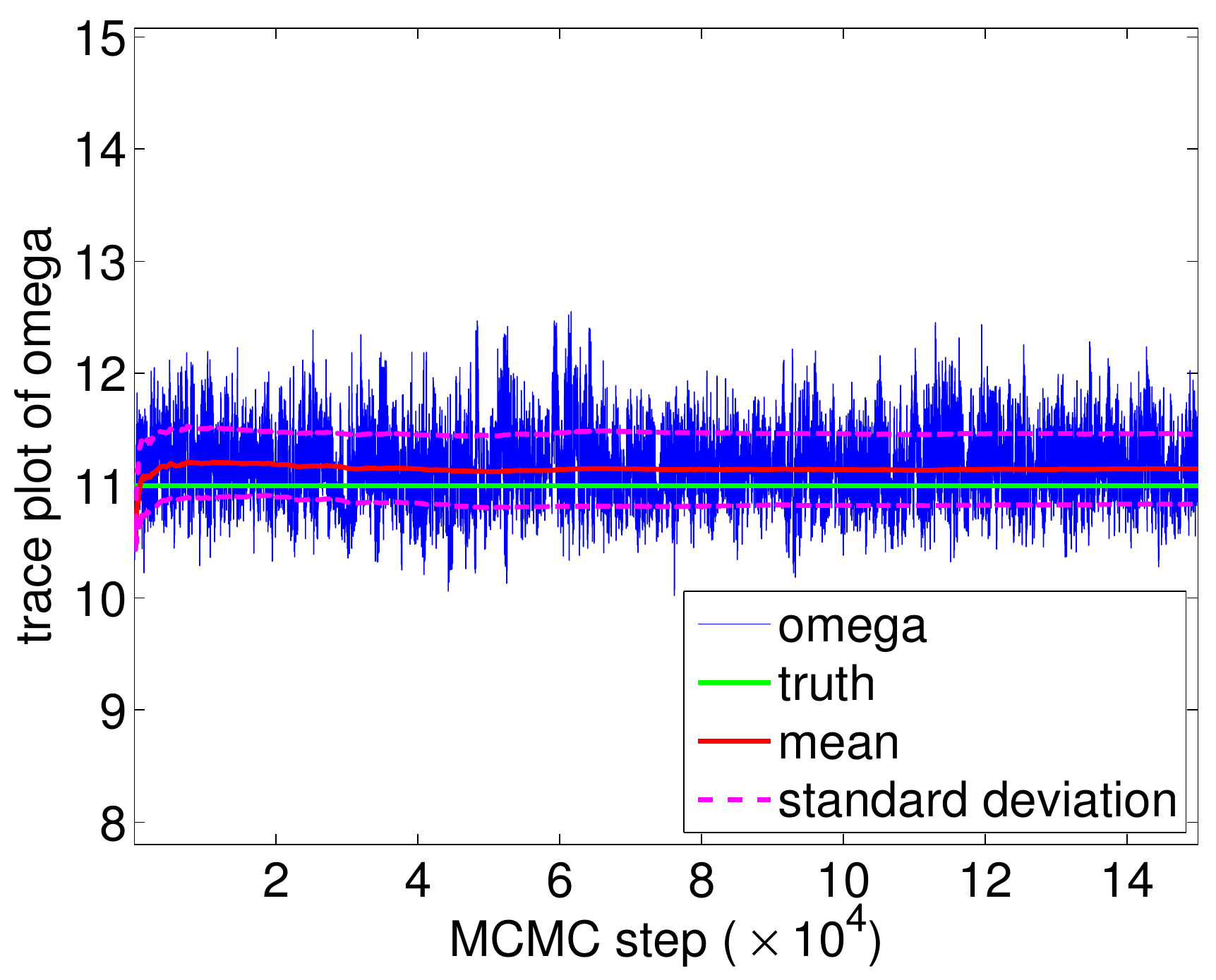}
\includegraphics[scale=0.2825]{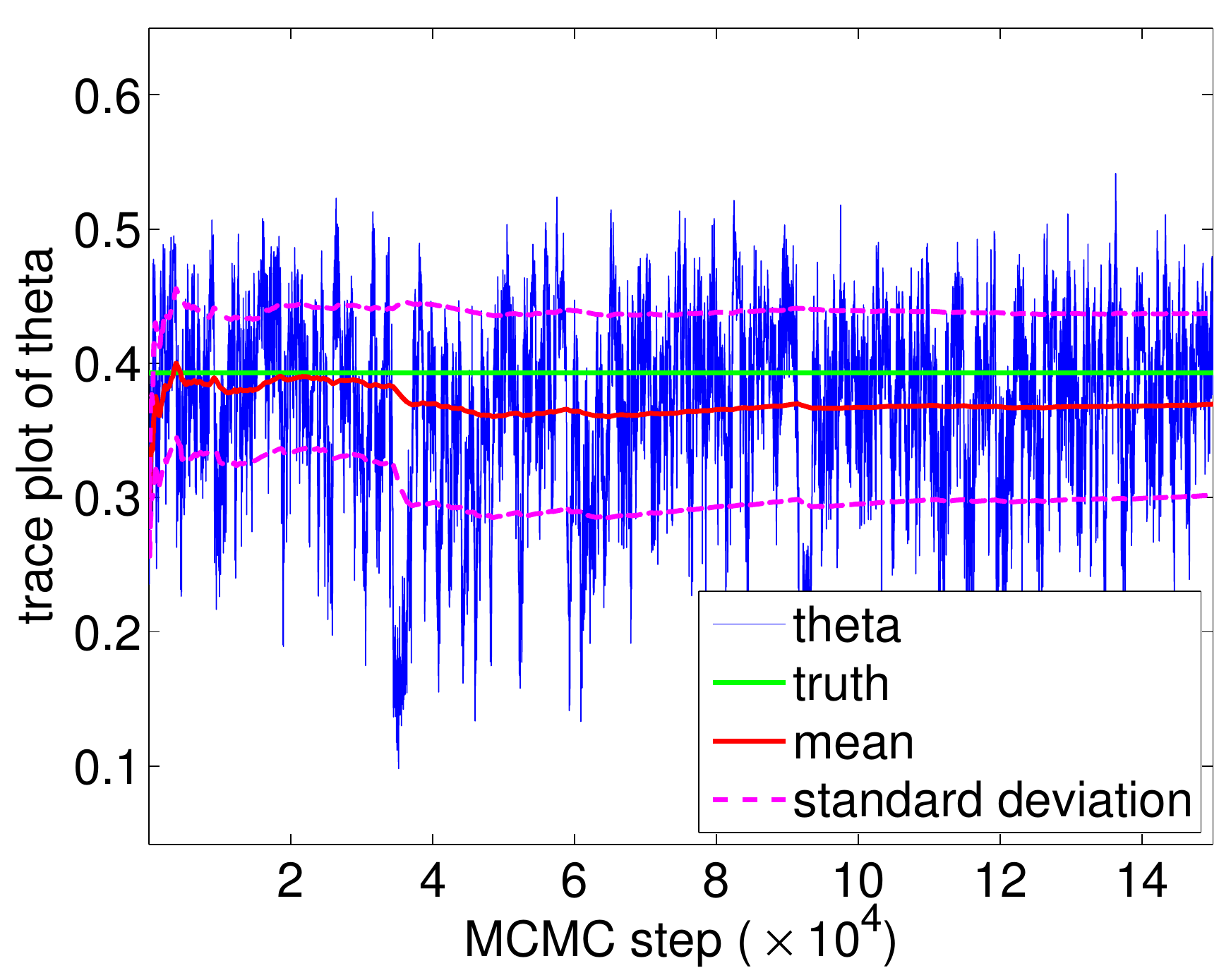}\\
\includegraphics[scale=0.2825]{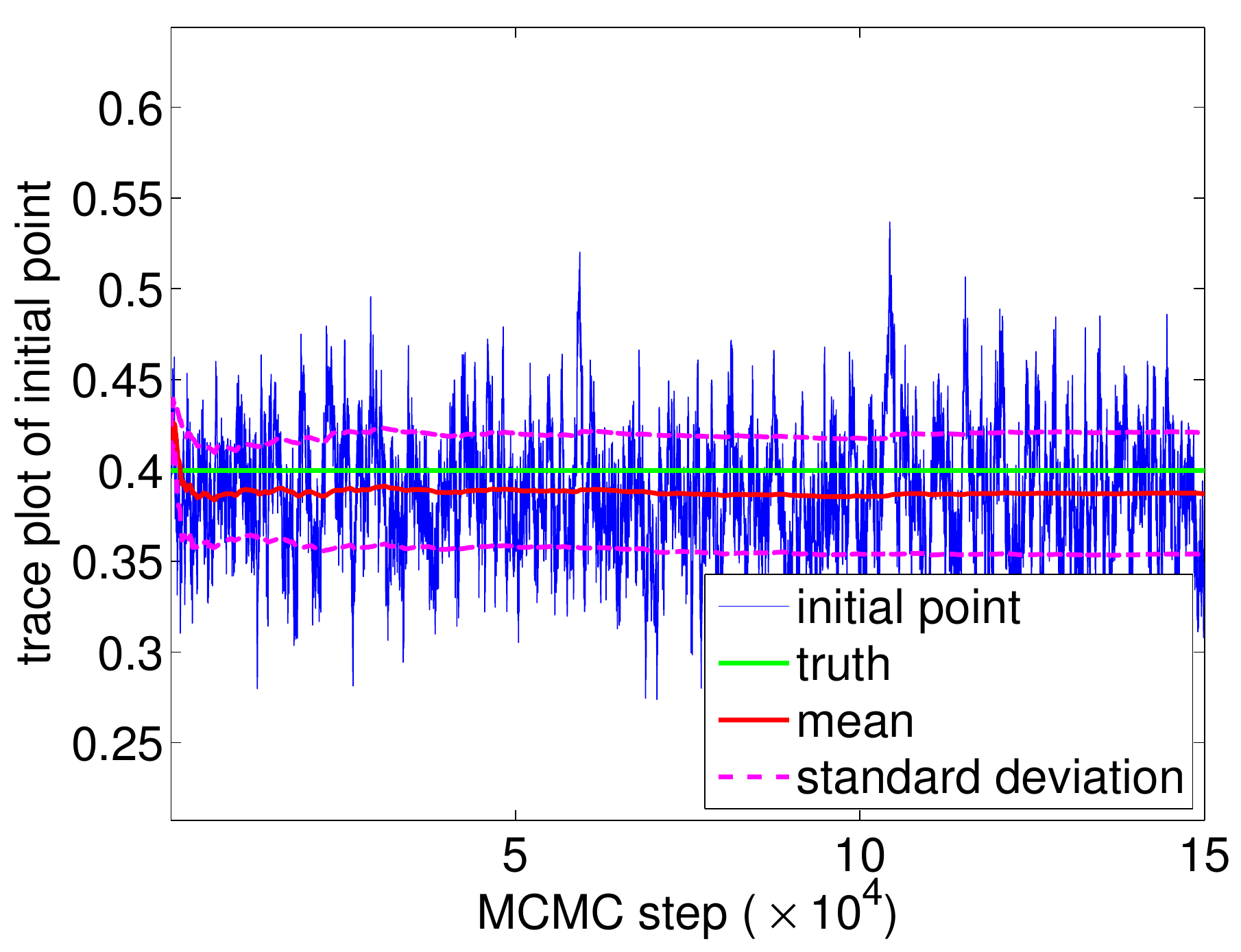}
\includegraphics[scale=0.2825]{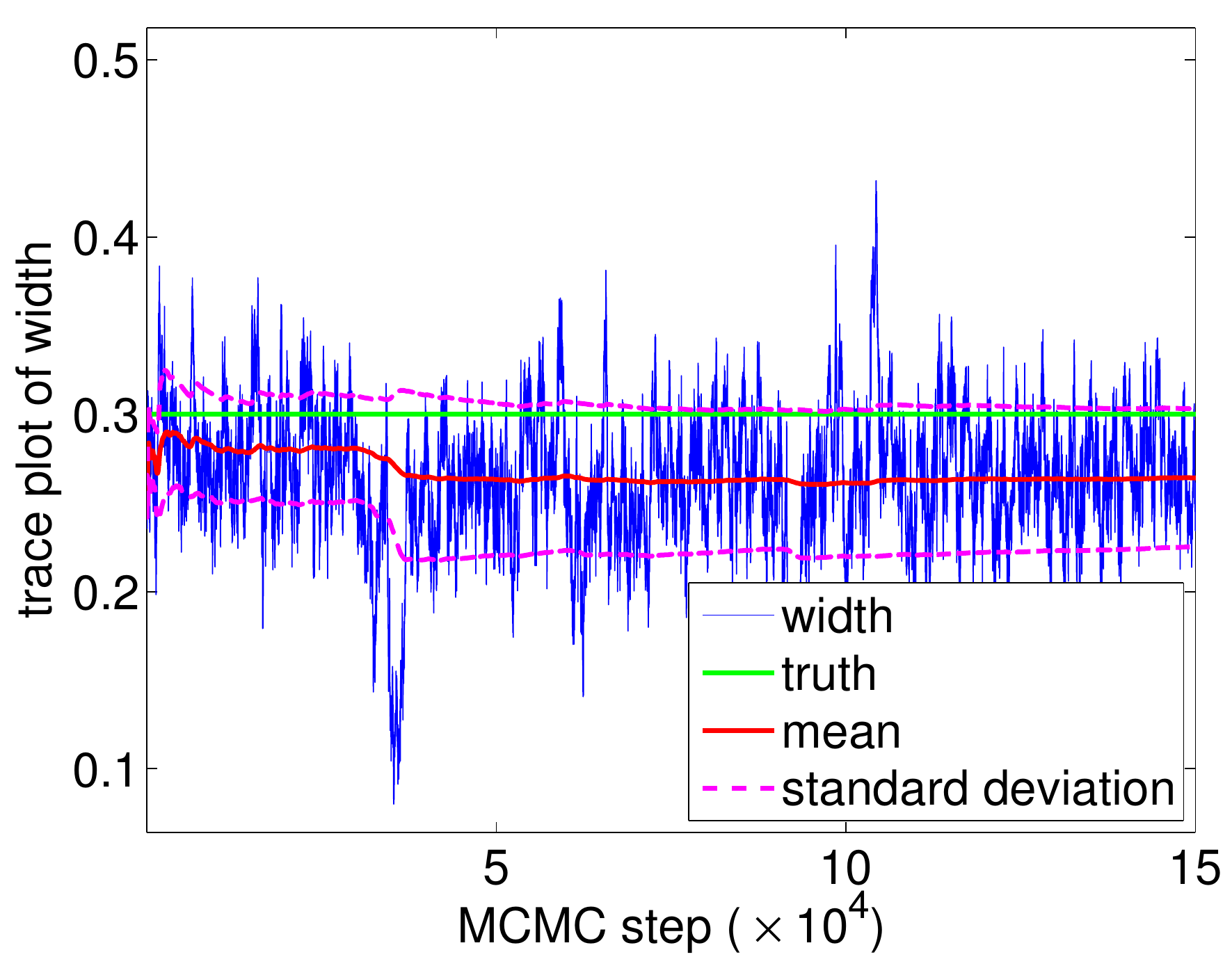}
\includegraphics[scale=0.2825]{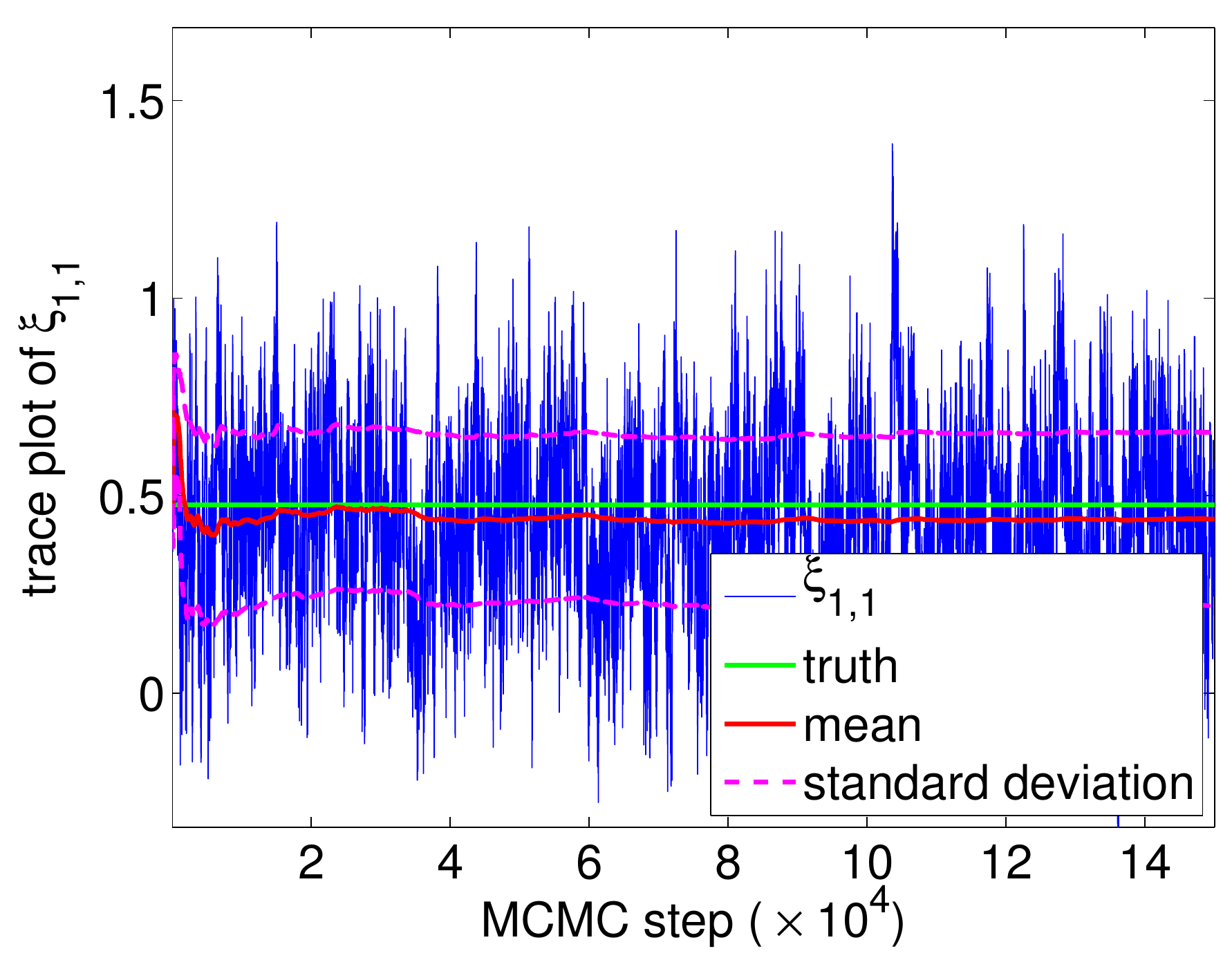}\\
\includegraphics[scale=0.2825]{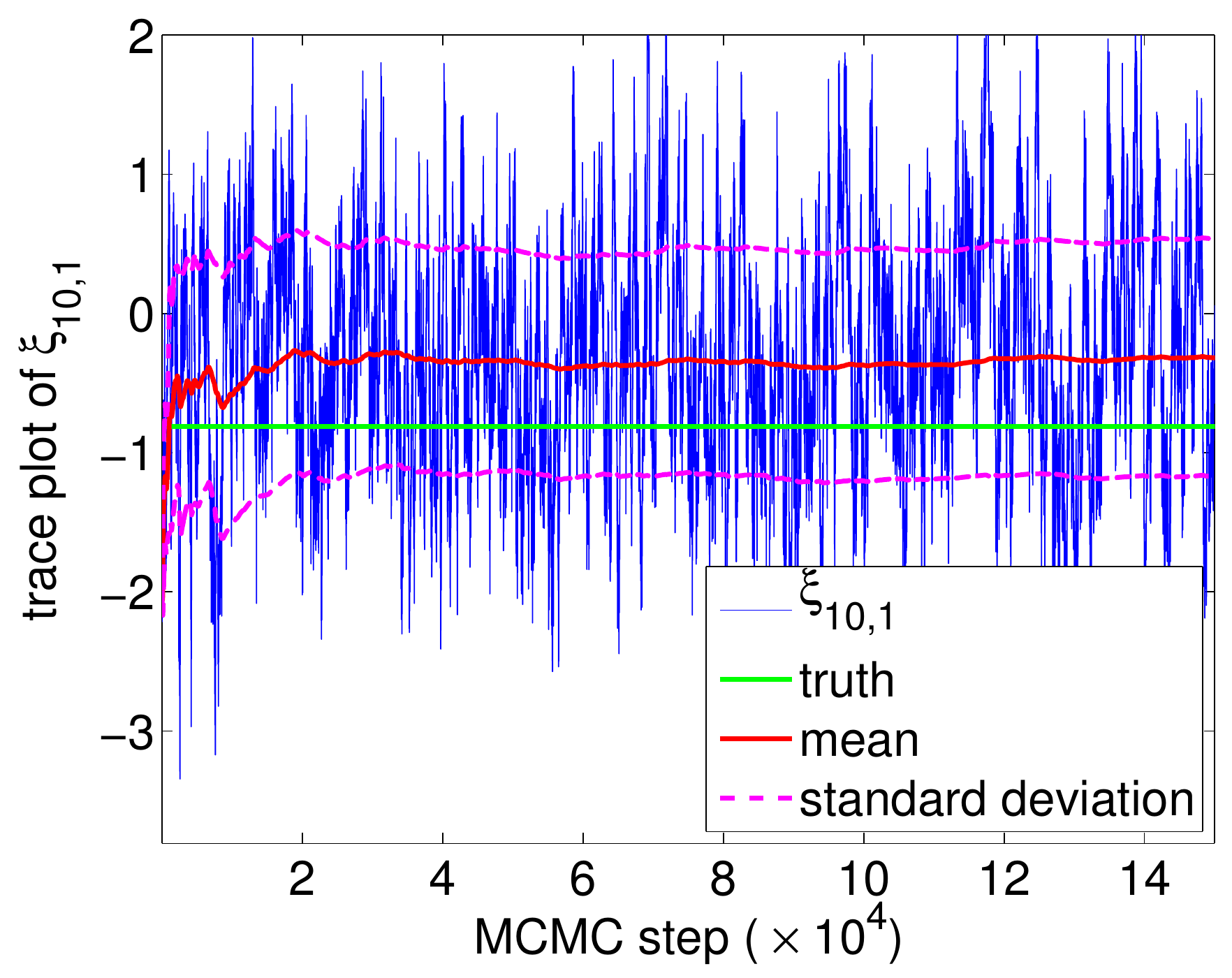}
\includegraphics[scale=0.2825]{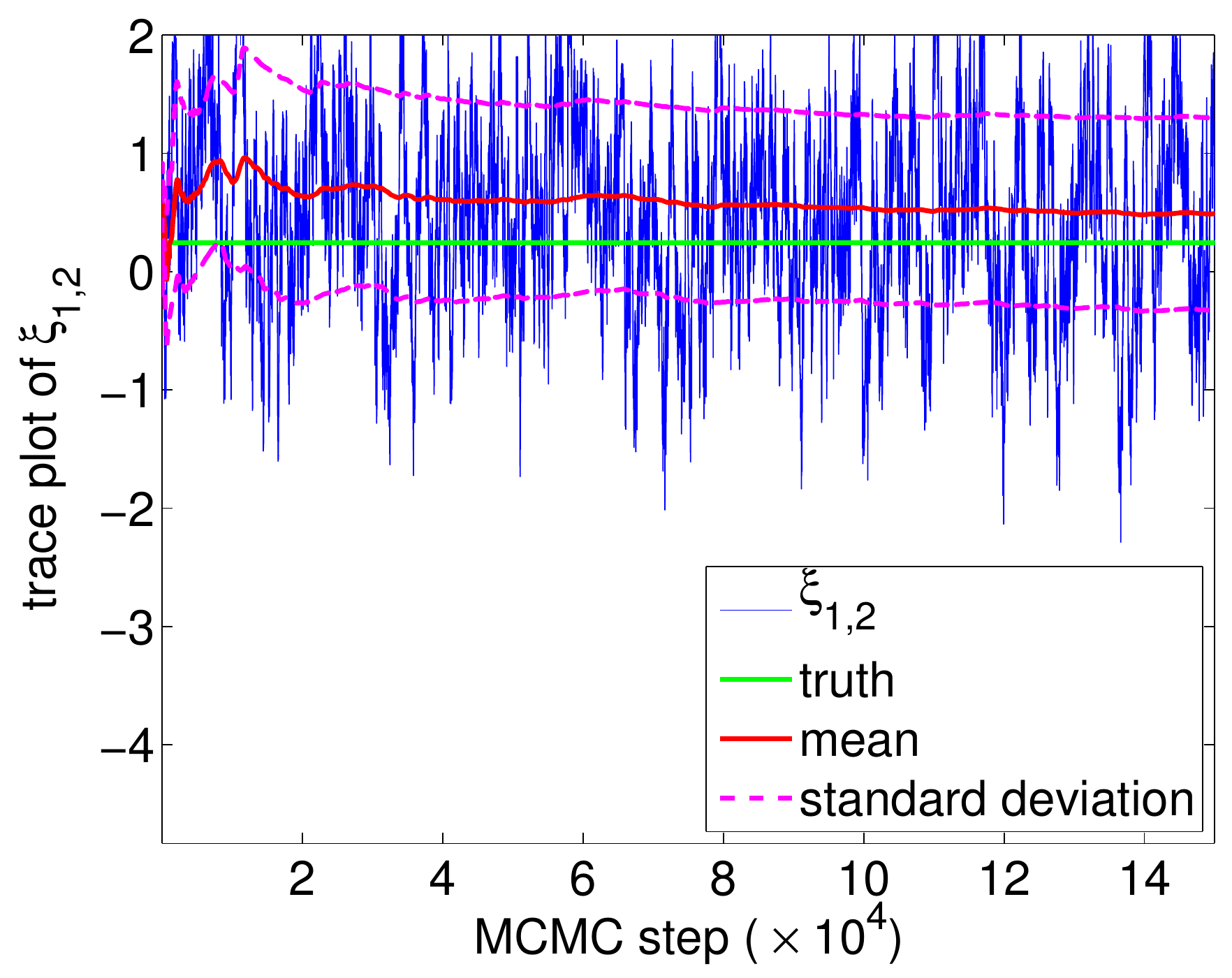}
\includegraphics[scale=0.2825]{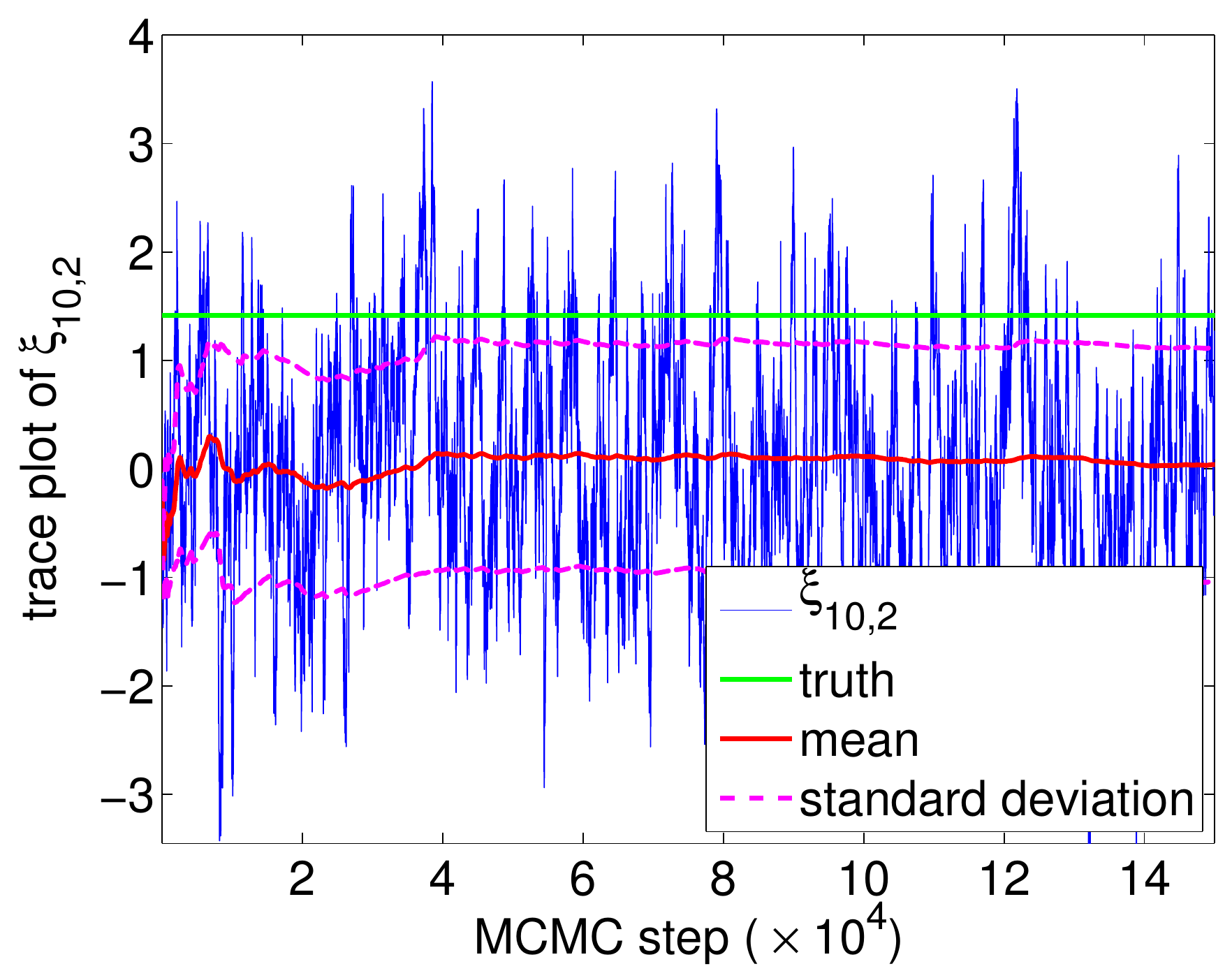}
 \caption{Trace plots from one MCMC chain.}
    \label{Figure19}
\end{center}
\end{figure}

\begin{figure}[htbp]
\begin{center}
\includegraphics[scale=0.2]{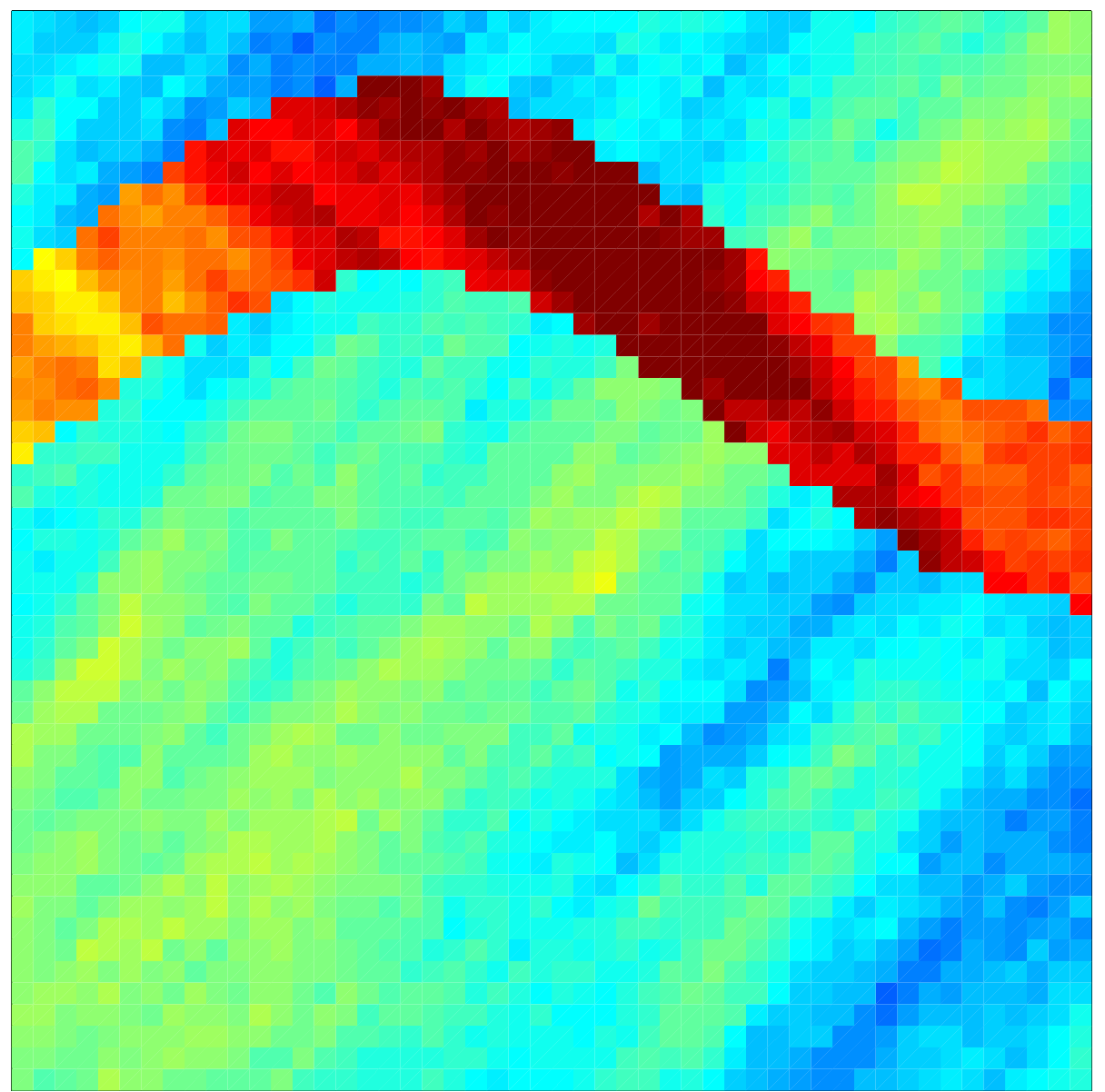}
\includegraphics[scale=0.2]{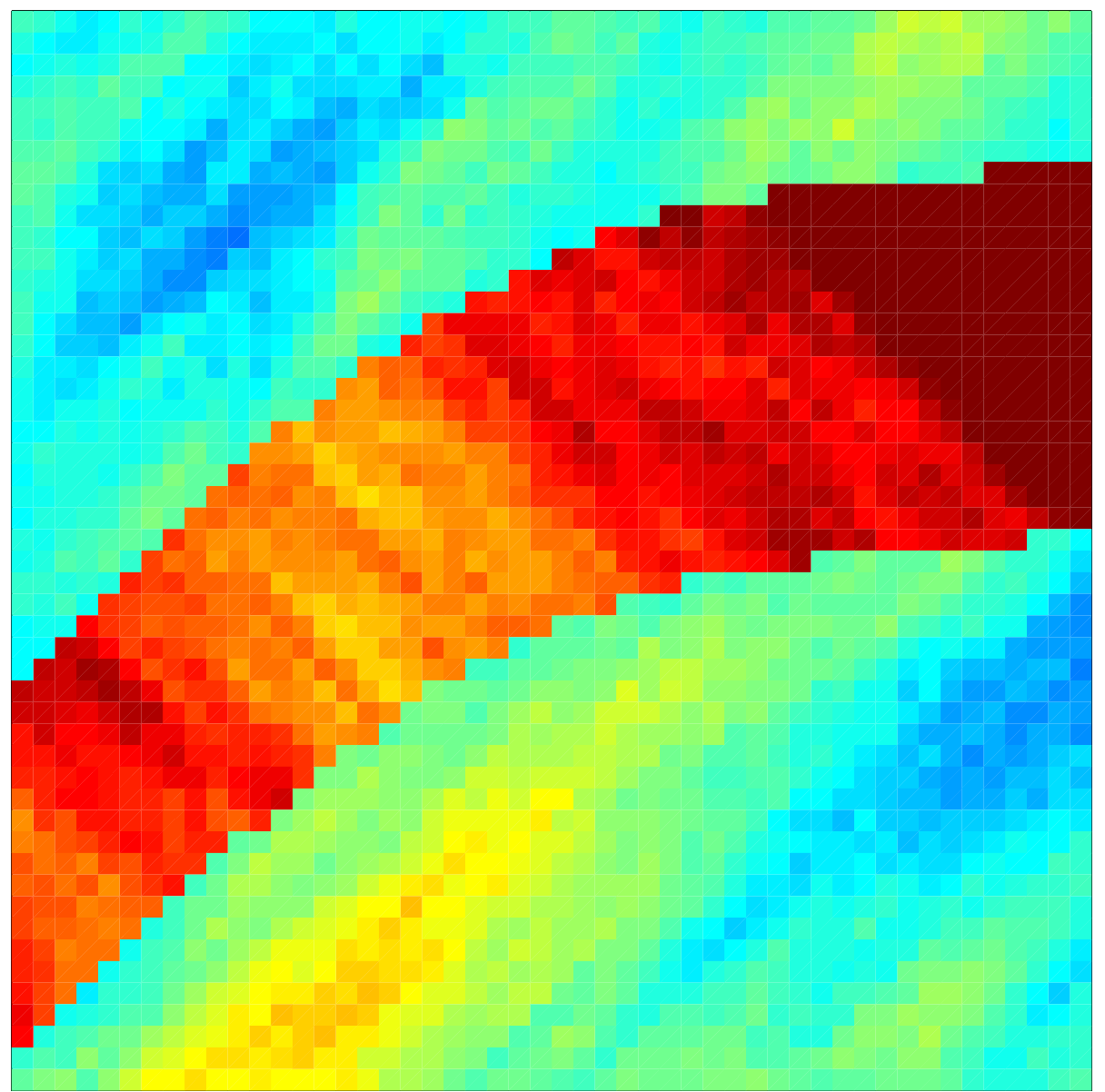}
\includegraphics[scale=0.2]{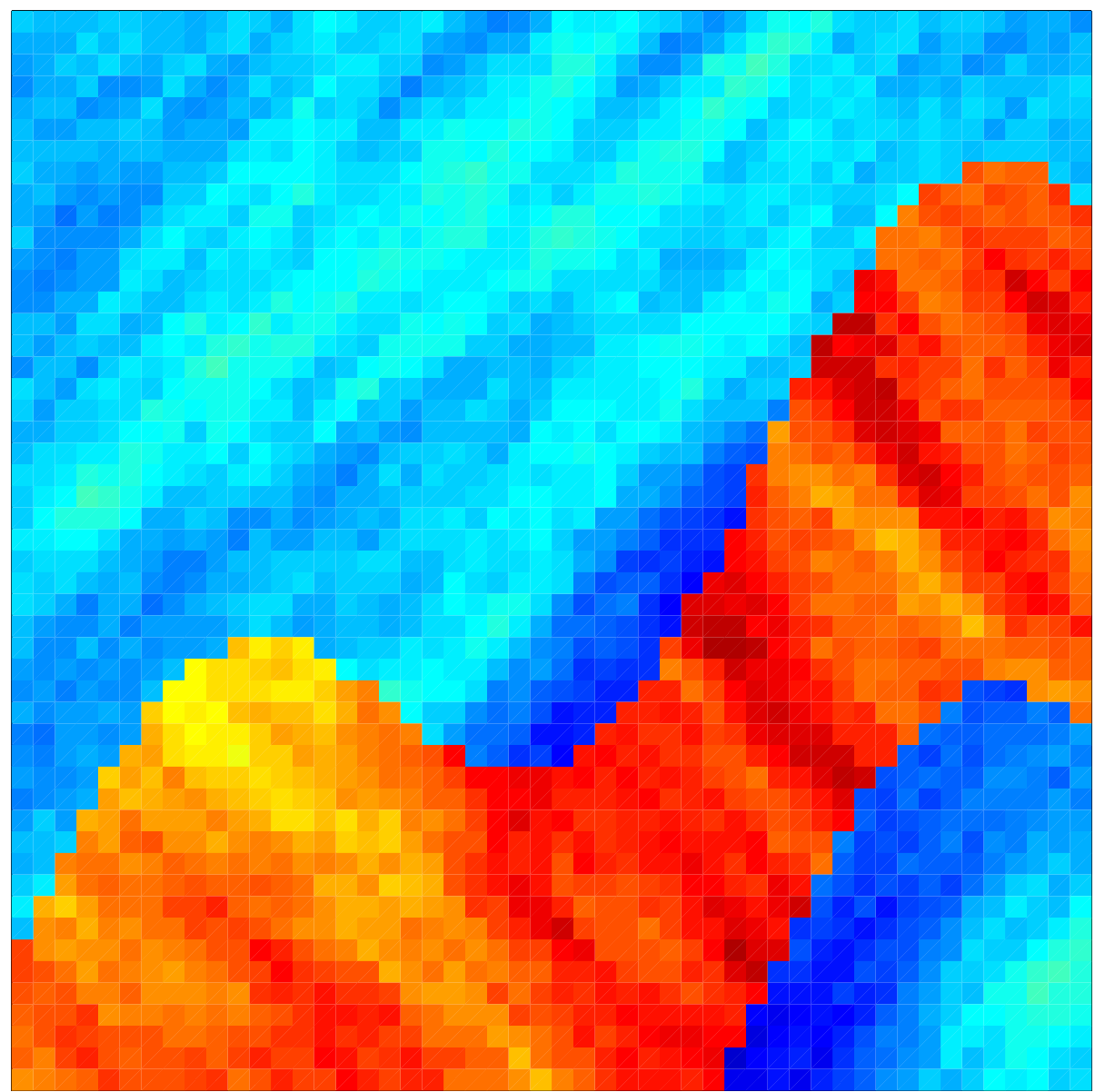}
\includegraphics[scale=0.2]{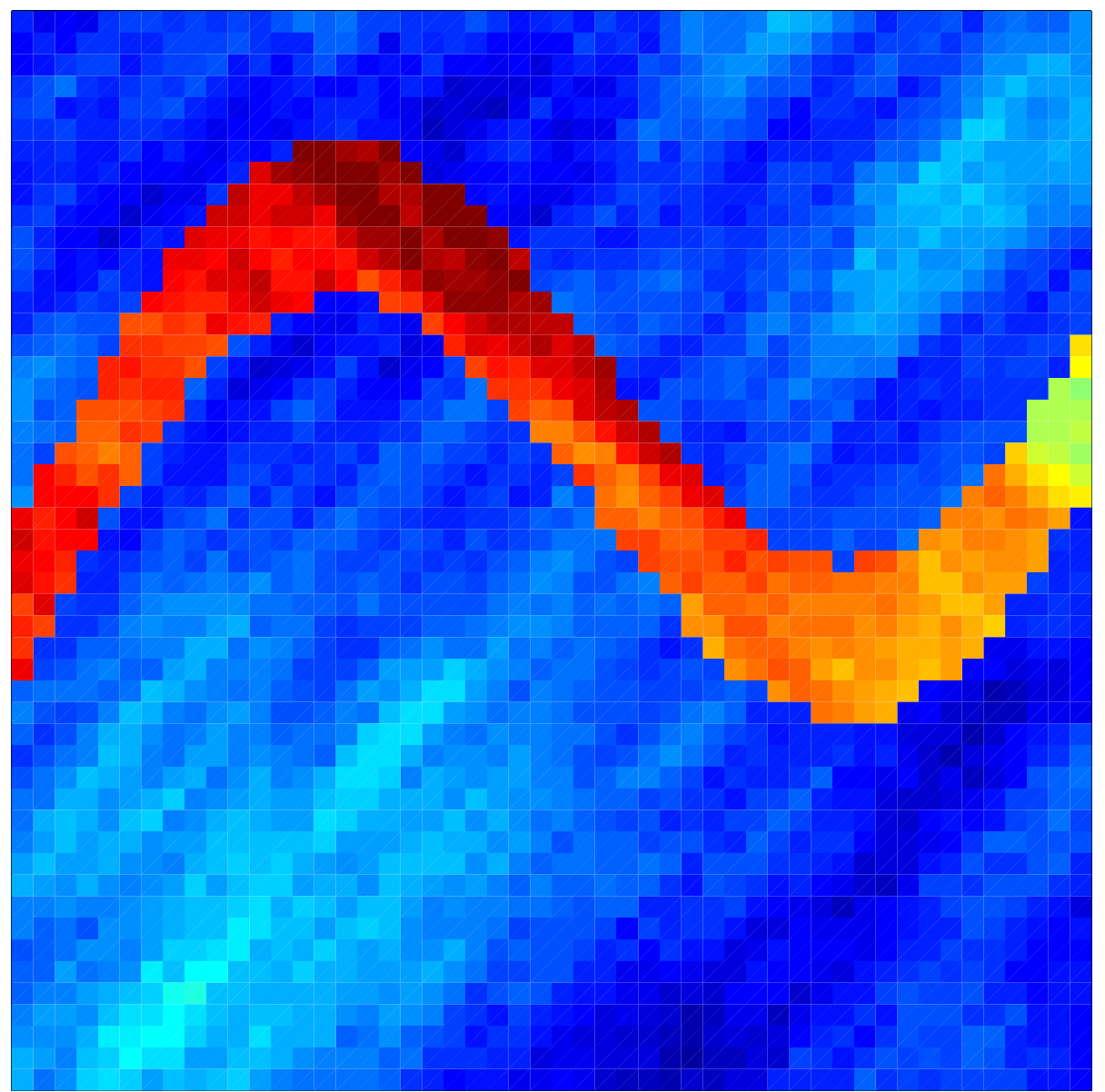}
\includegraphics[scale=0.2]{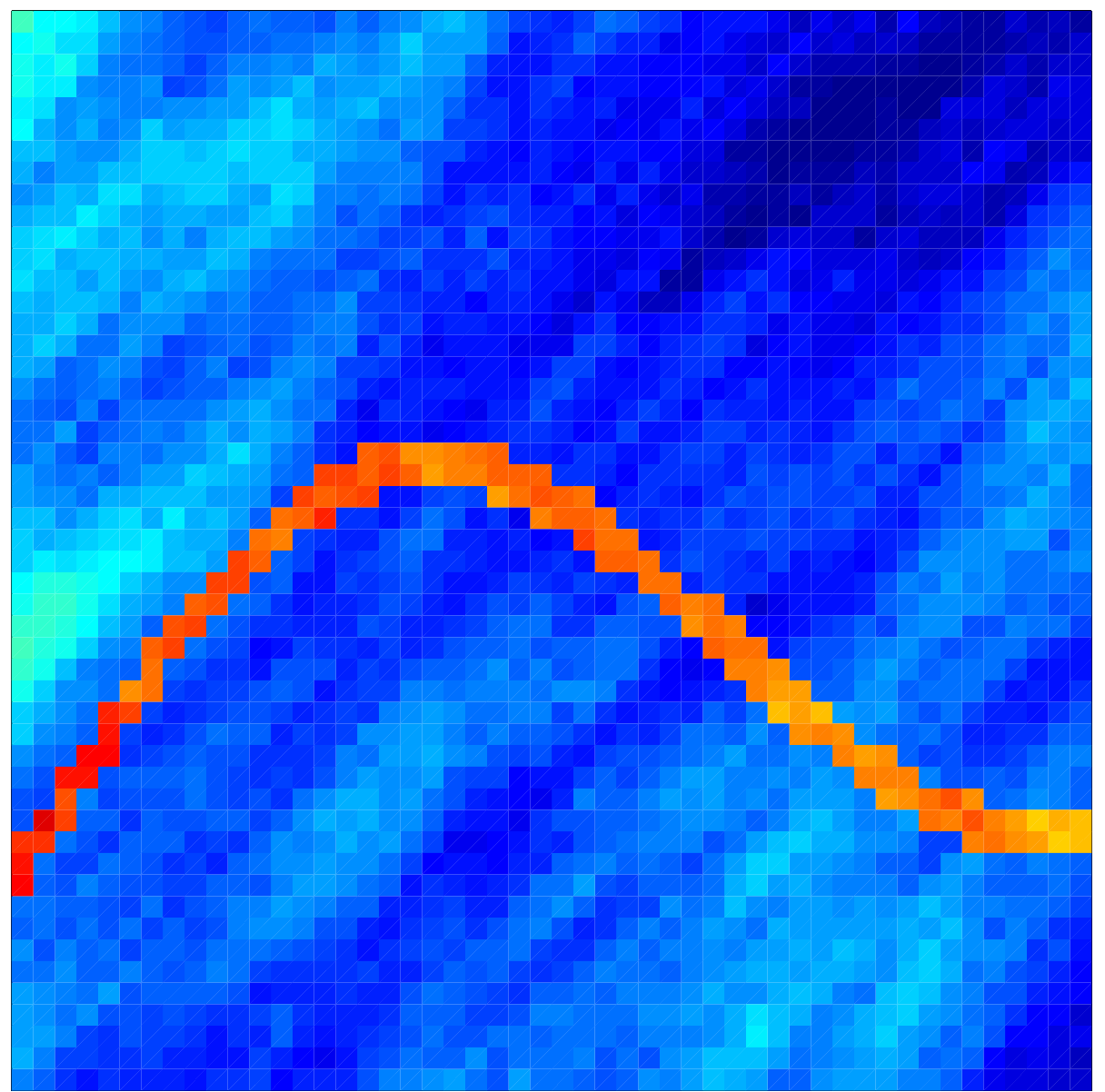}
\includegraphics[scale=0.2]{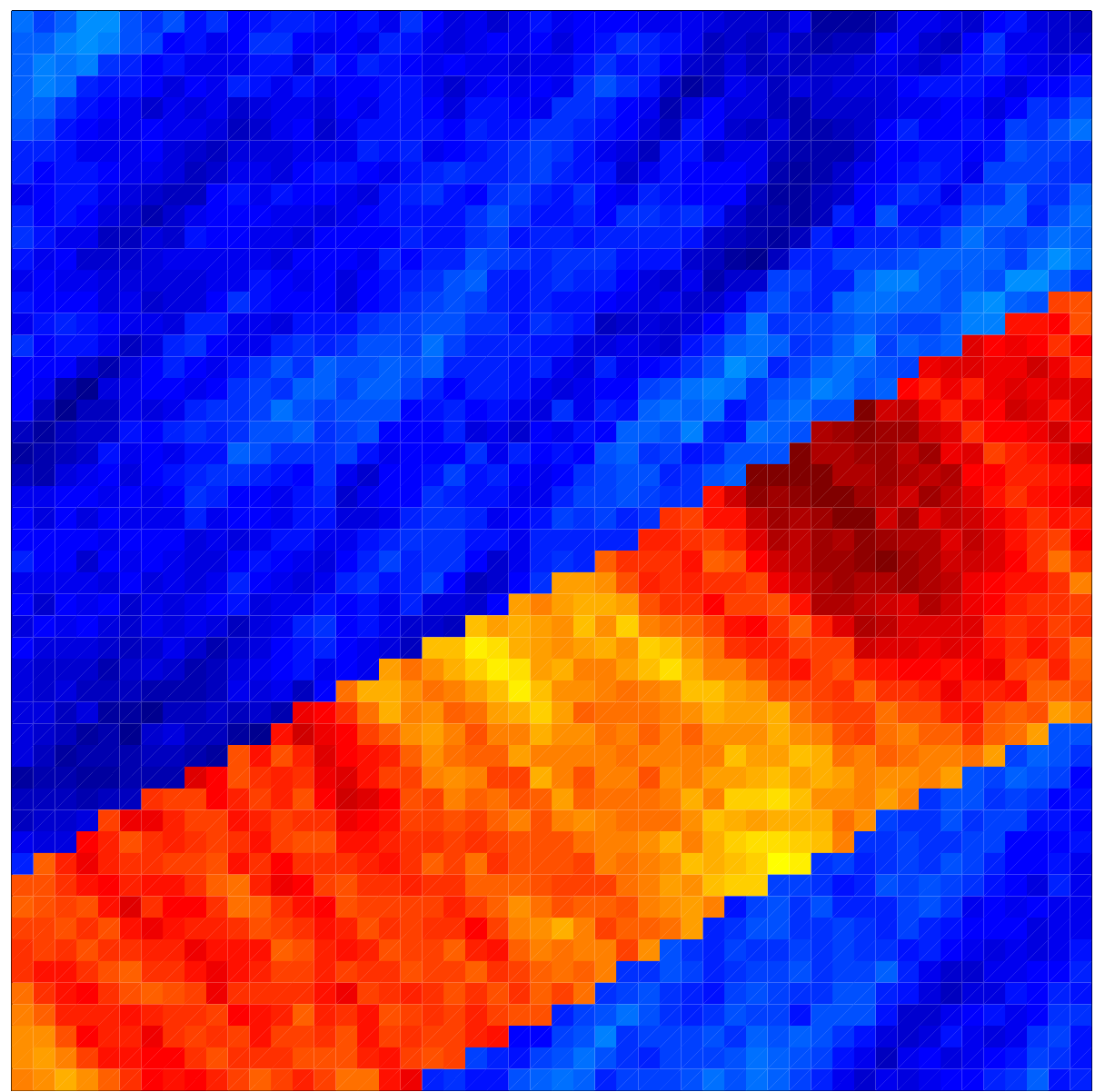}\\
\includegraphics[scale=0.2]{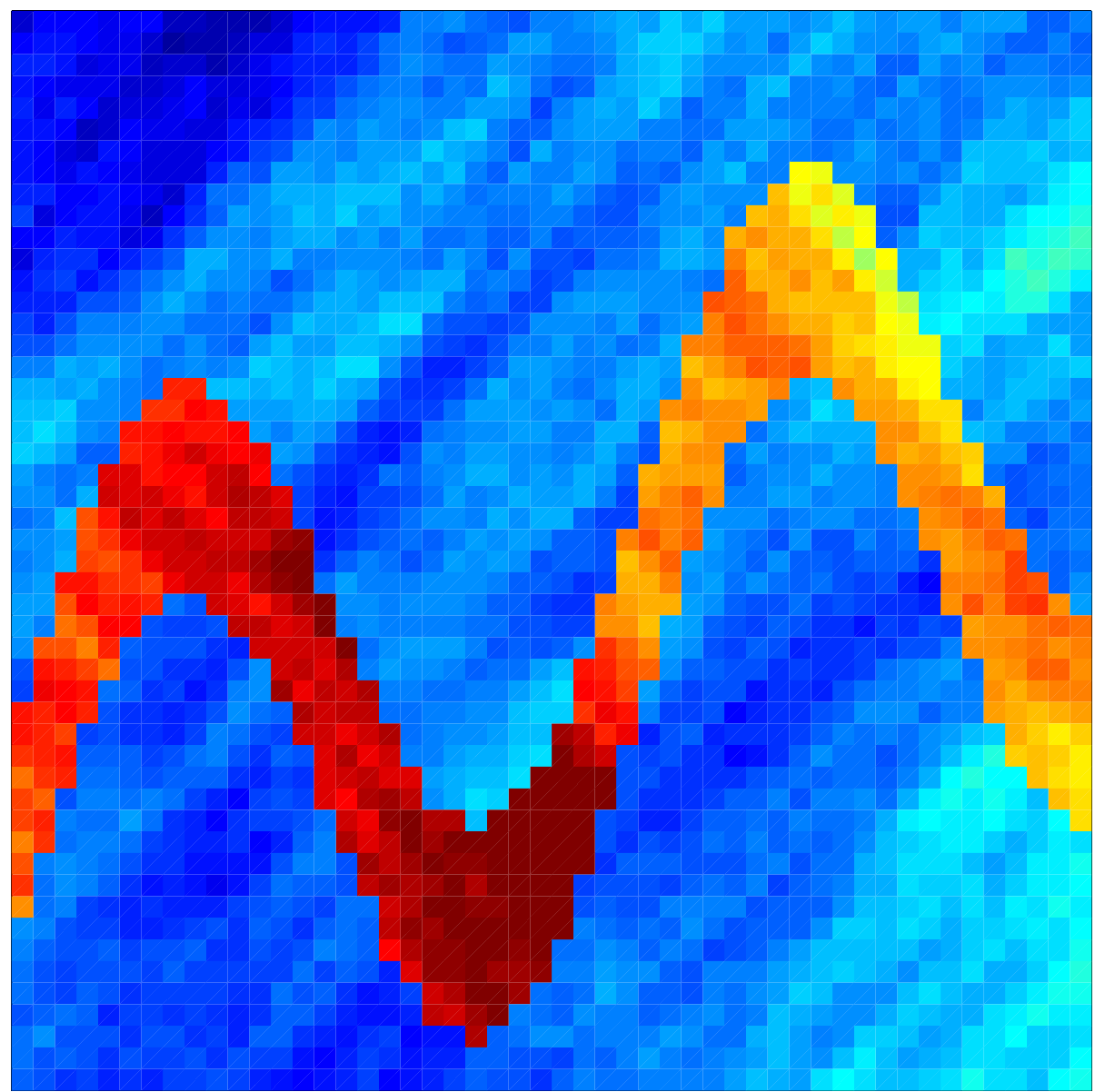}
\includegraphics[scale=0.2]{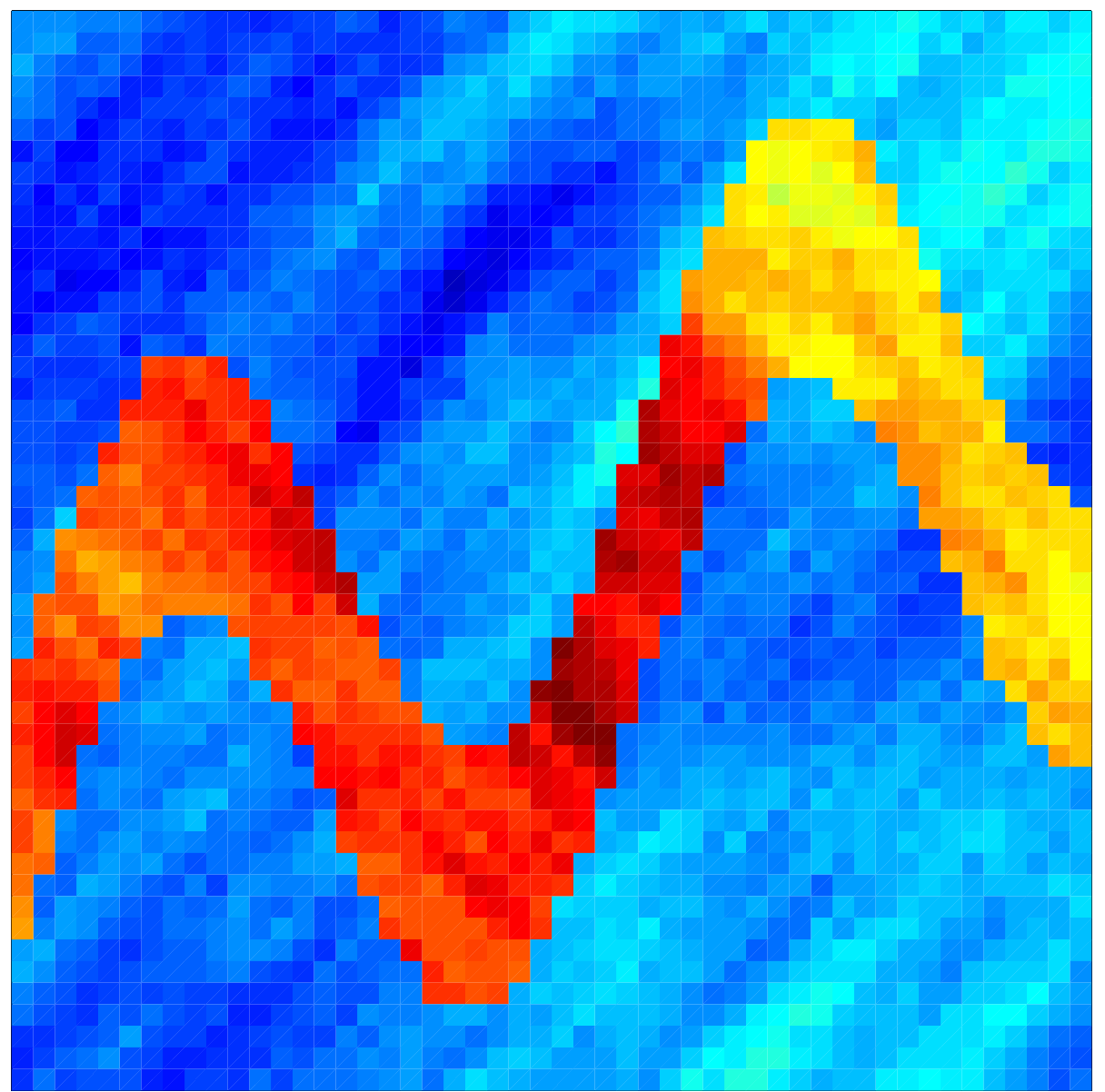}
\includegraphics[scale=0.2]{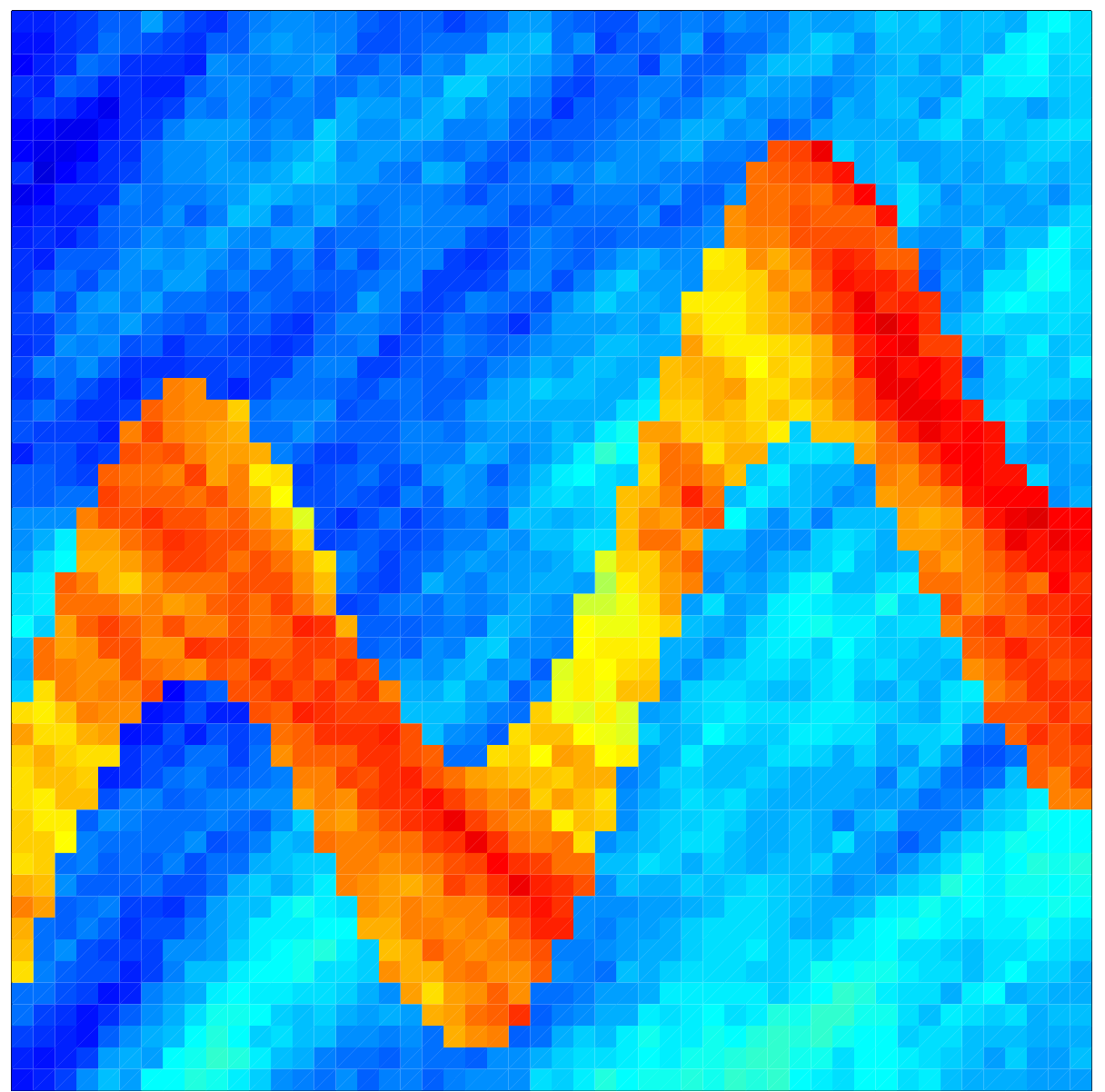}
\includegraphics[scale=0.2]{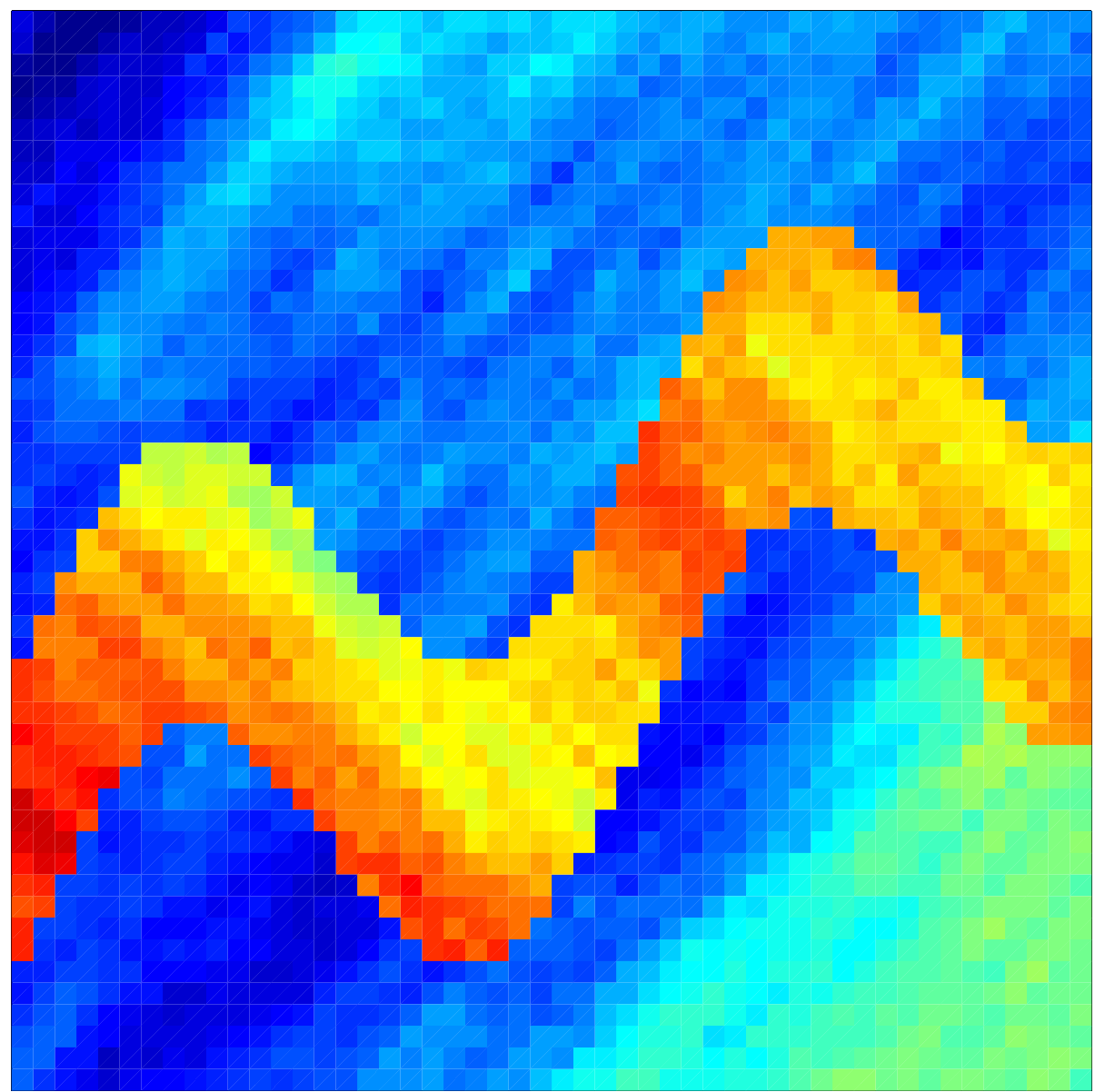}
\includegraphics[scale=0.2]{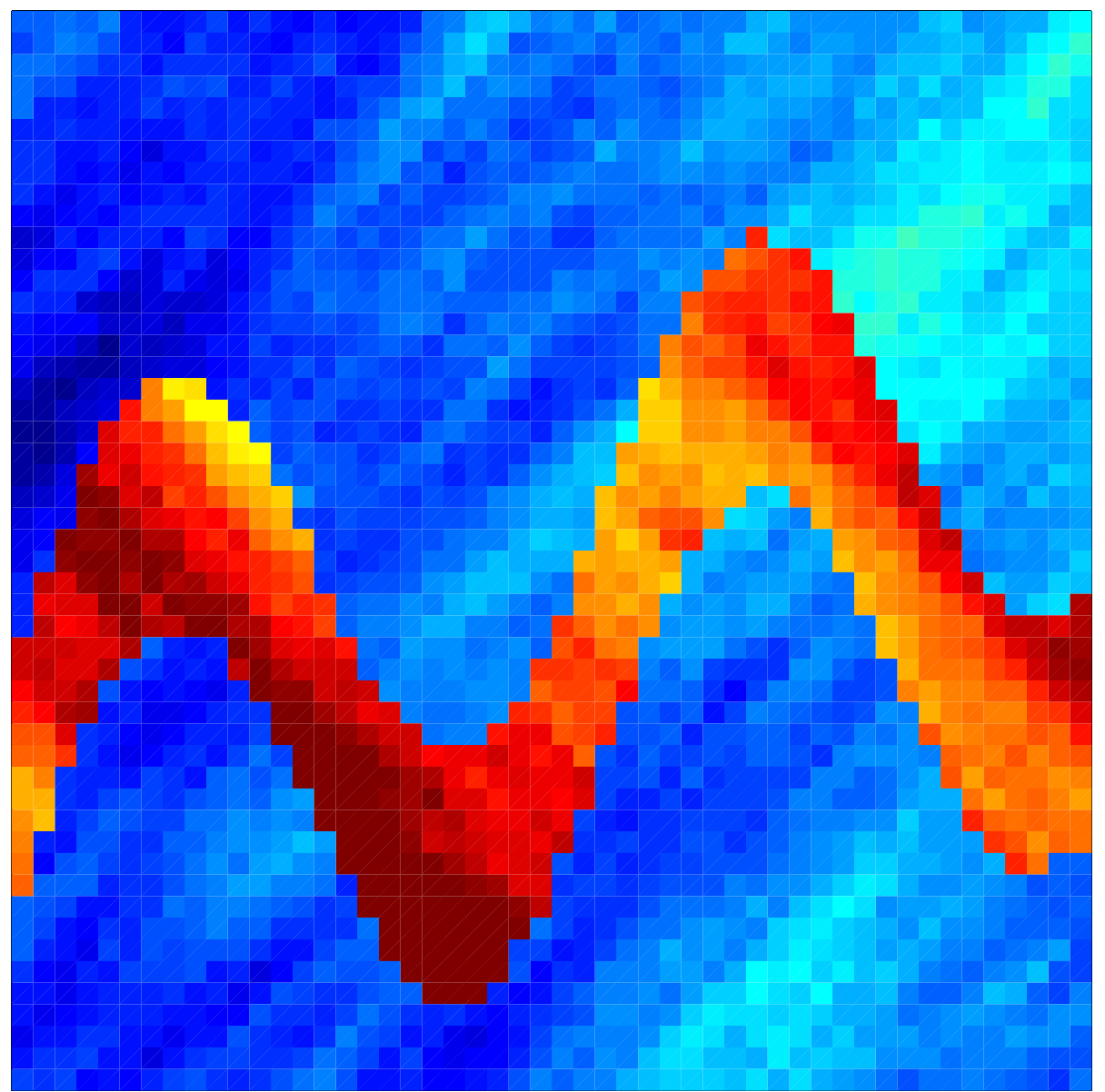}
\includegraphics[scale=0.2]{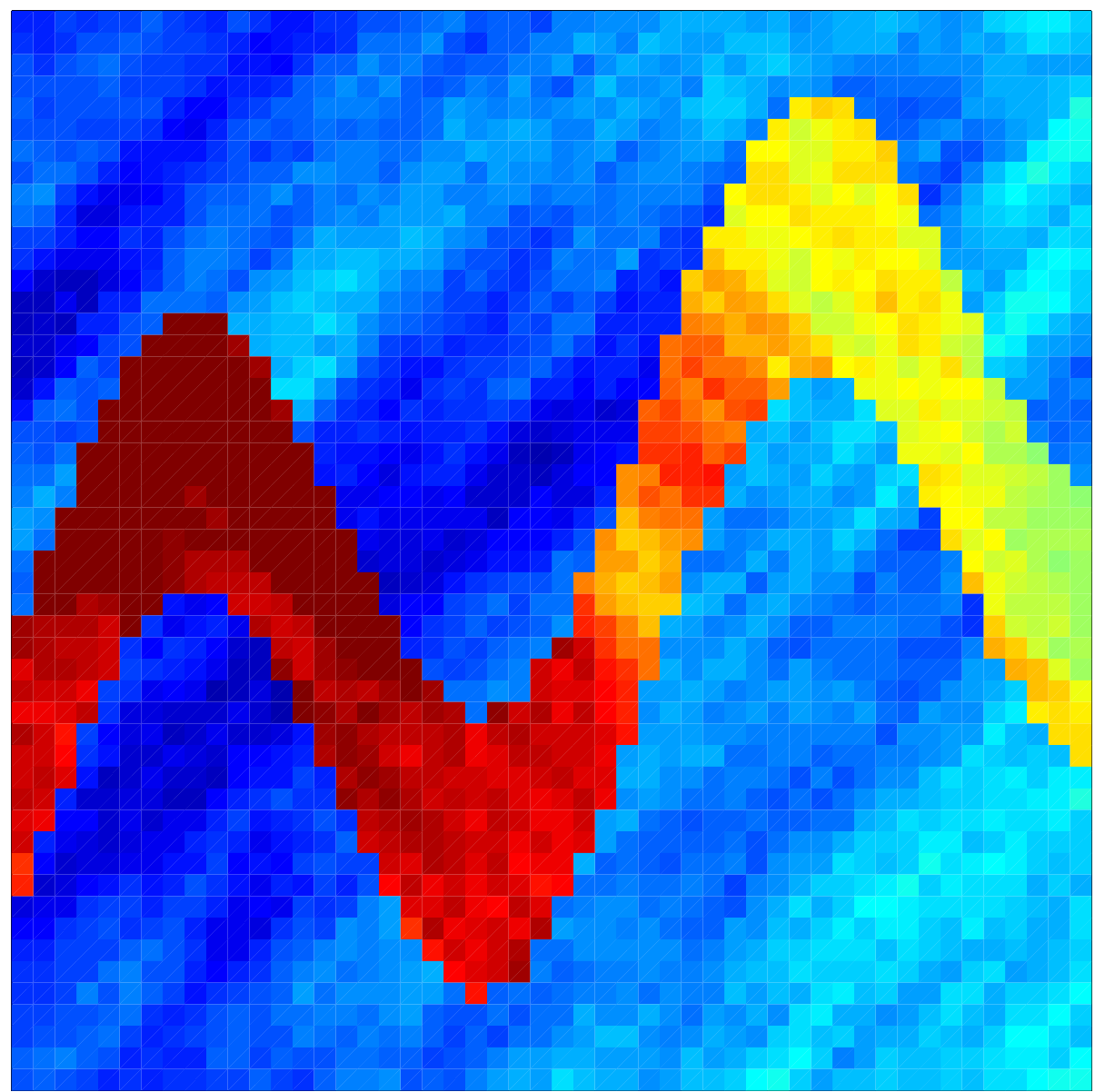}
 \caption{$\log \kappa$'s defined by (\ref{eq:num3}) (for the channelized geometry) from samples of the prior (top row) and the posterior(bottom-row)} \label{Figure18}
\end{center}
\end{figure}

\begin{figure}[htbp]
\begin{center}
\includegraphics[scale=0.35]{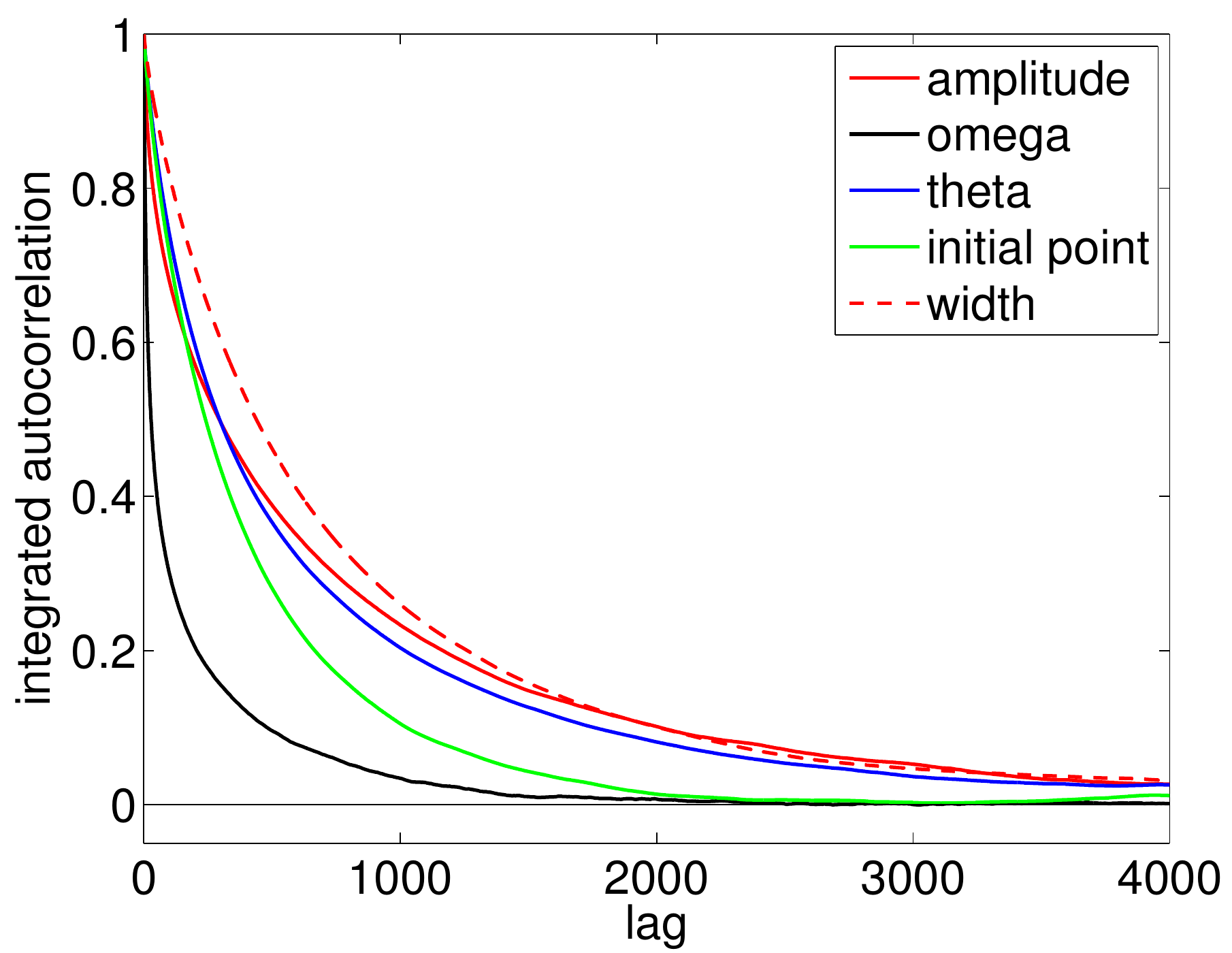}
\includegraphics[scale=0.35]{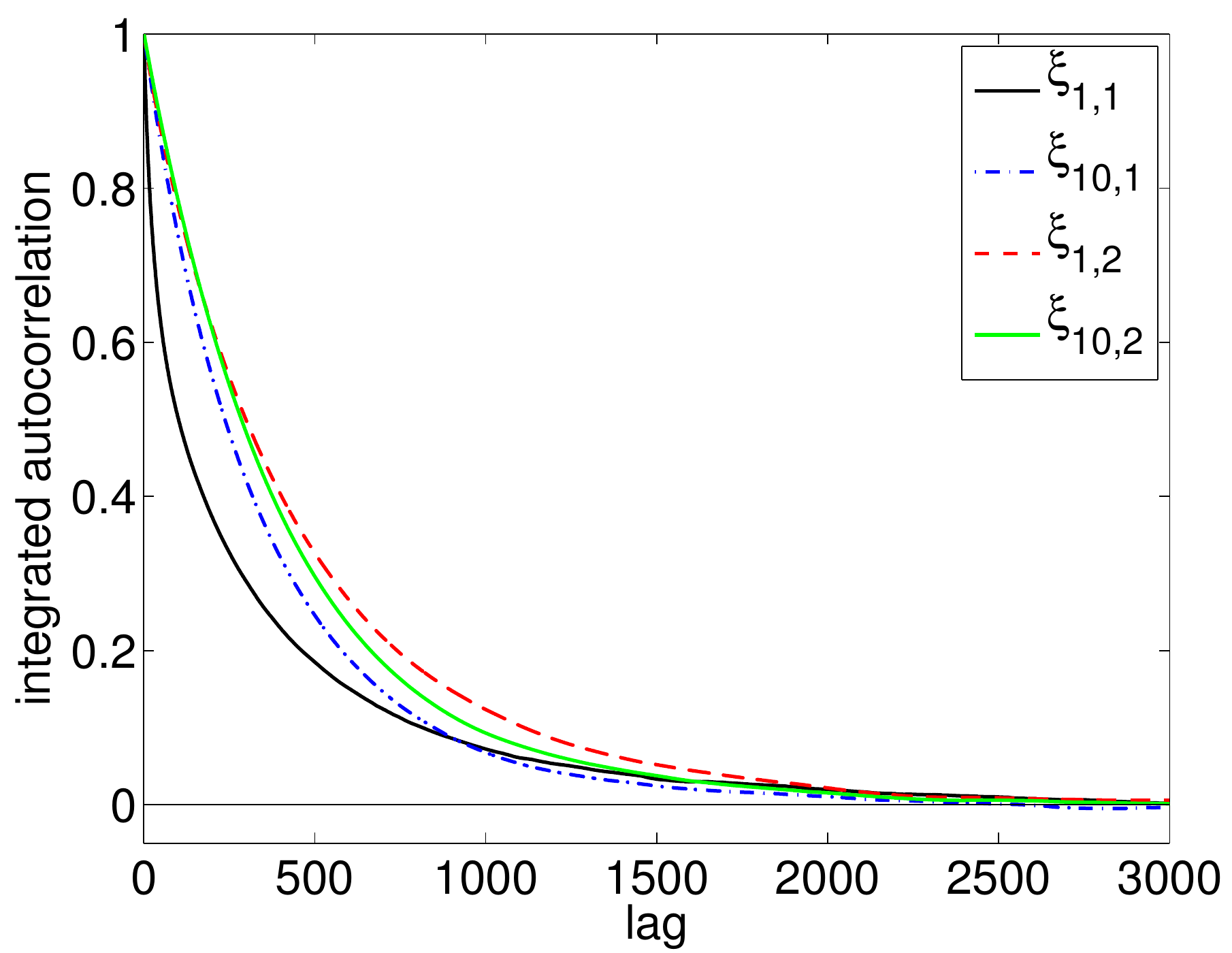}
 \caption{Autocorrelation from one MCMC chain. Left: geometric parameters. Right: Some KL model of  $\log \kappa_{1}$ and $\log \kappa_{2}$}
    \label{Figure20}
\end{center}
\end{figure}

\begin{figure}[htbp]
\begin{center}
\includegraphics[scale=0.25]{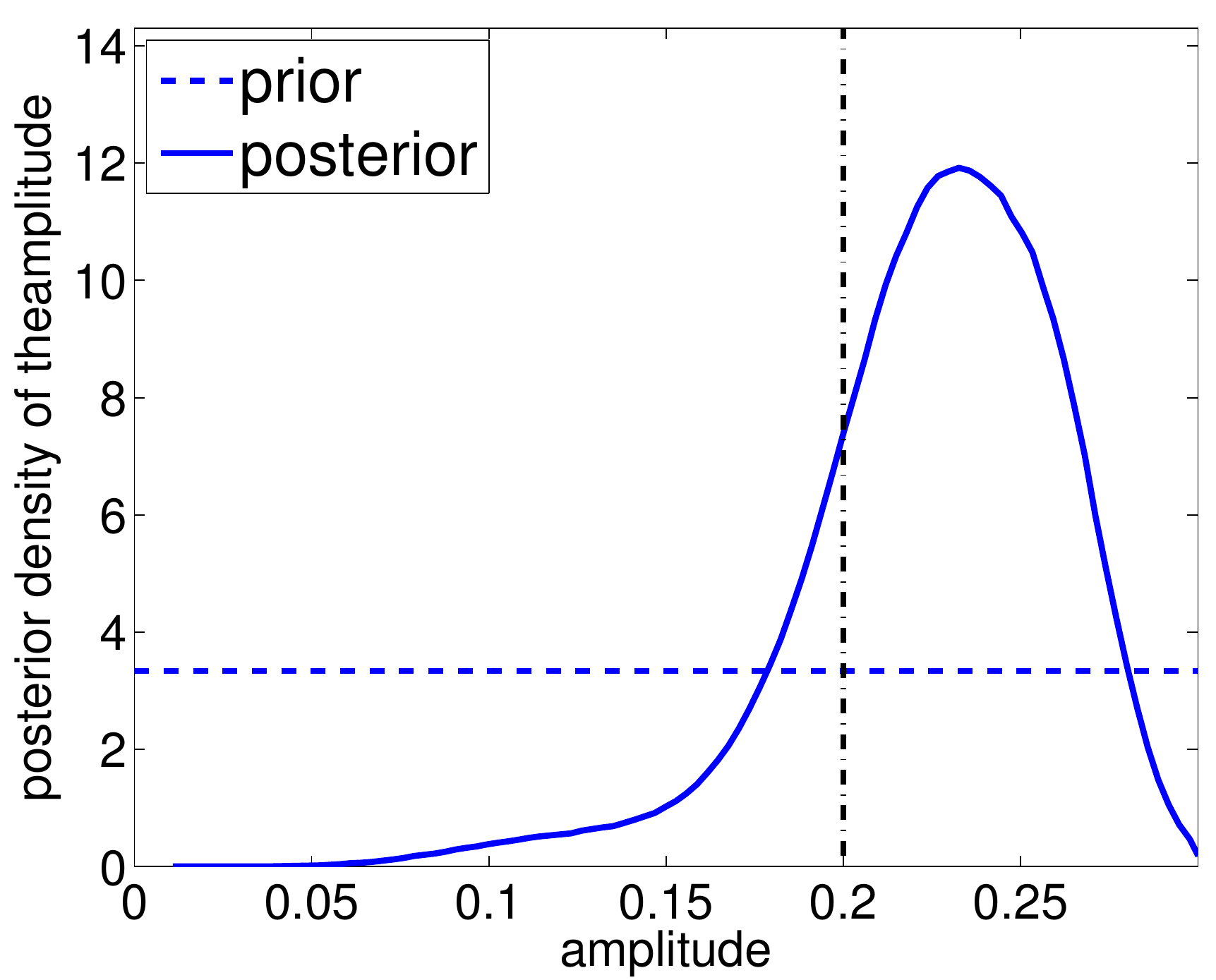}
\includegraphics[scale=0.25]{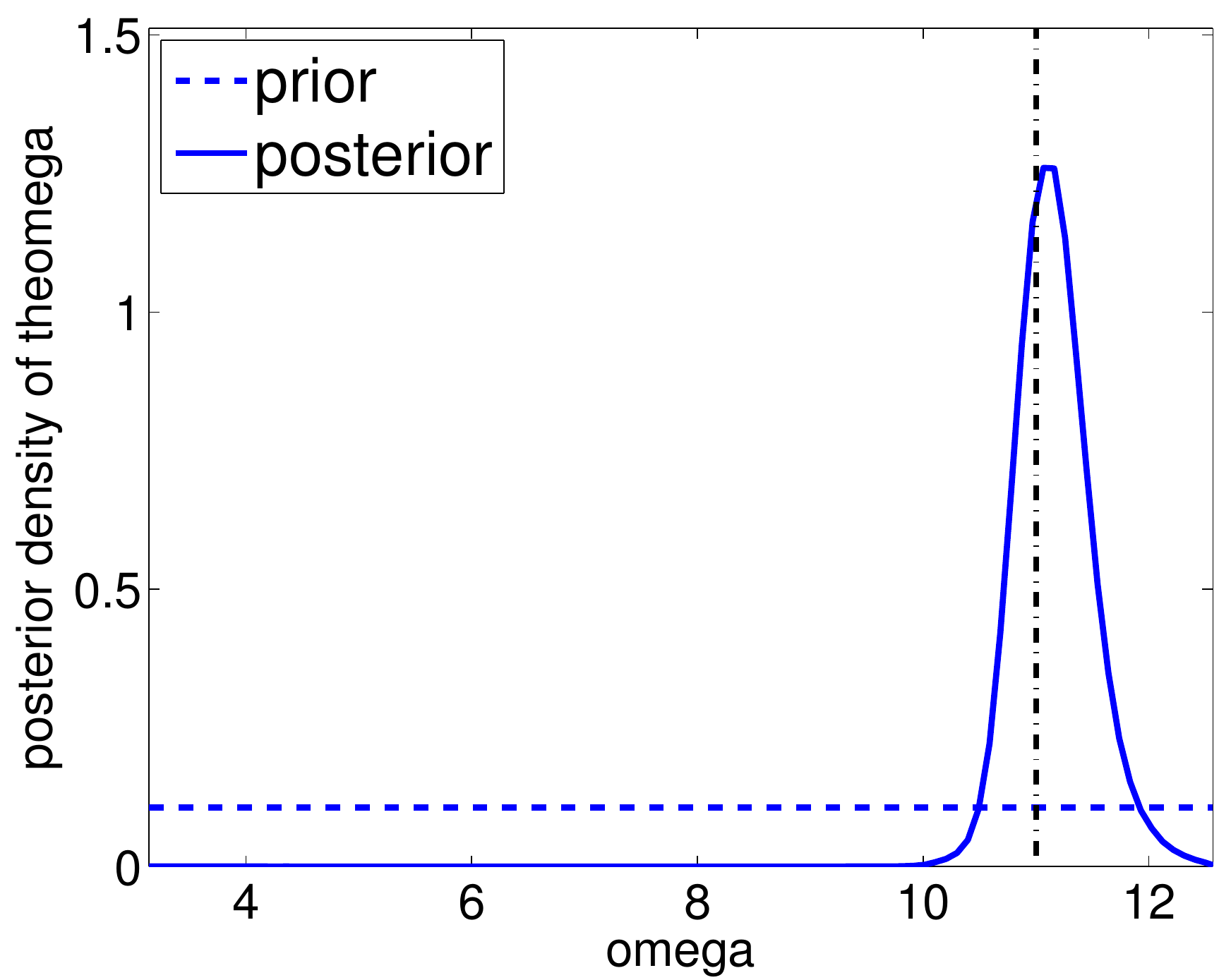}
\includegraphics[scale=0.25]{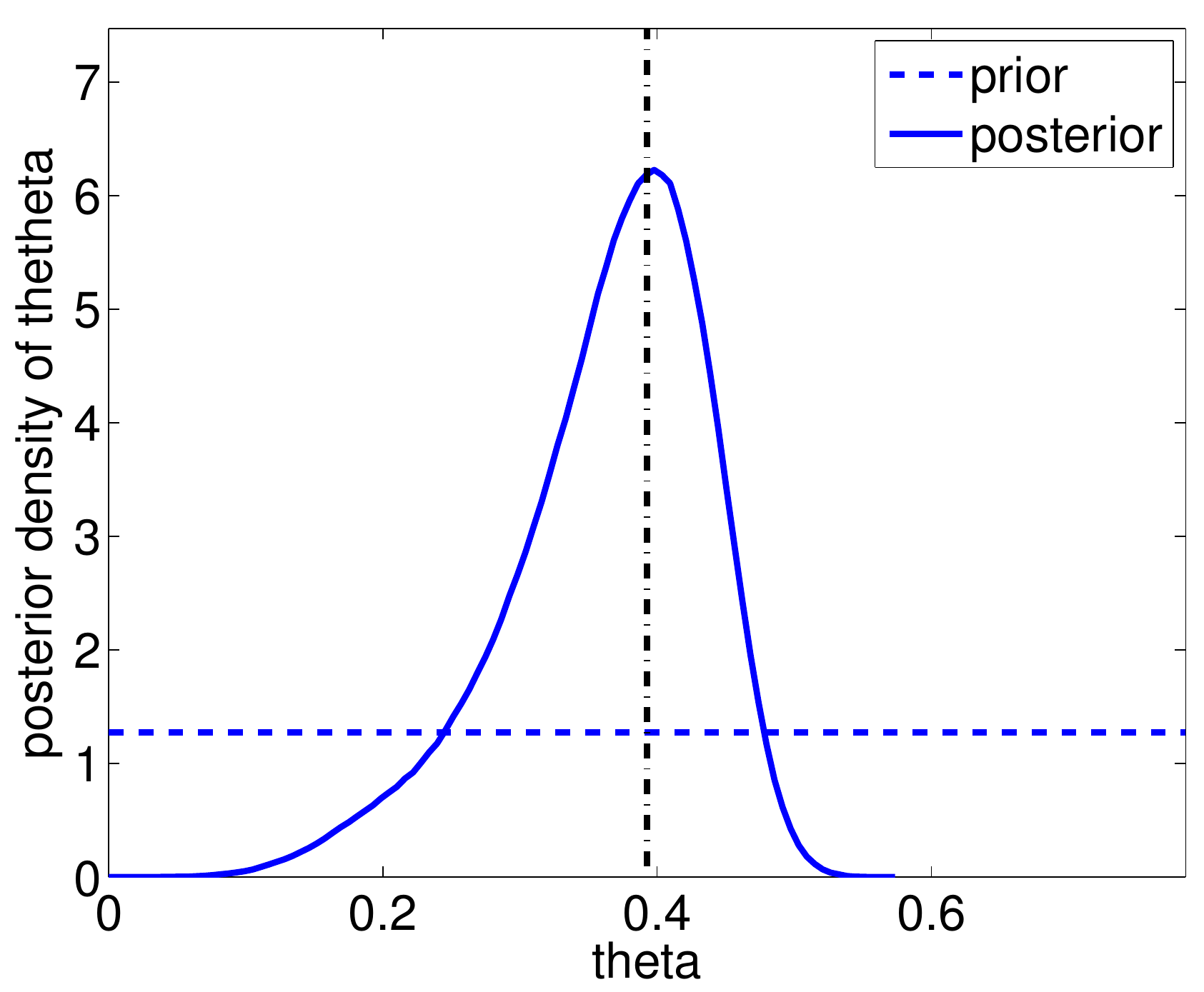}\\
\includegraphics[scale=0.25]{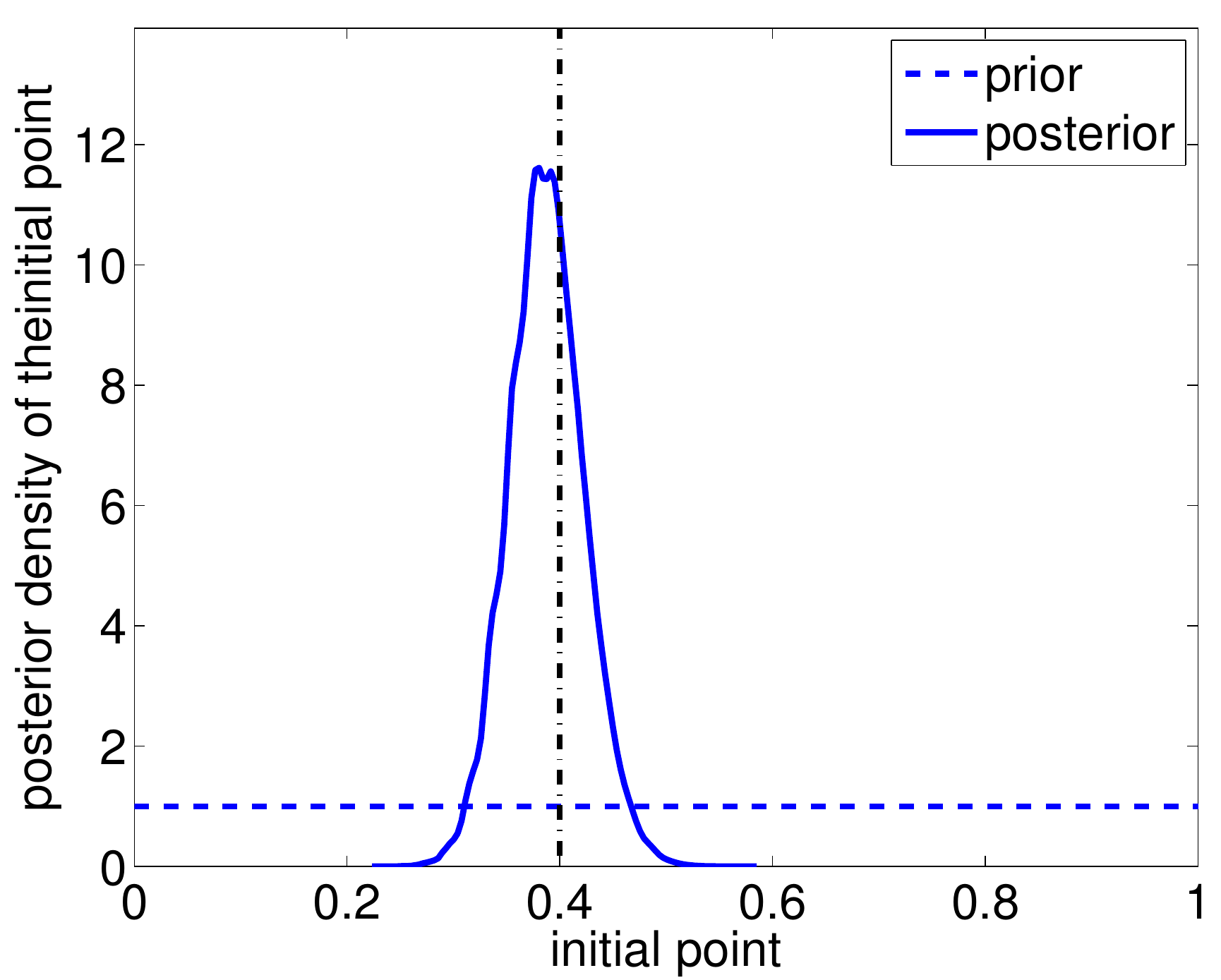}
\includegraphics[scale=0.25]{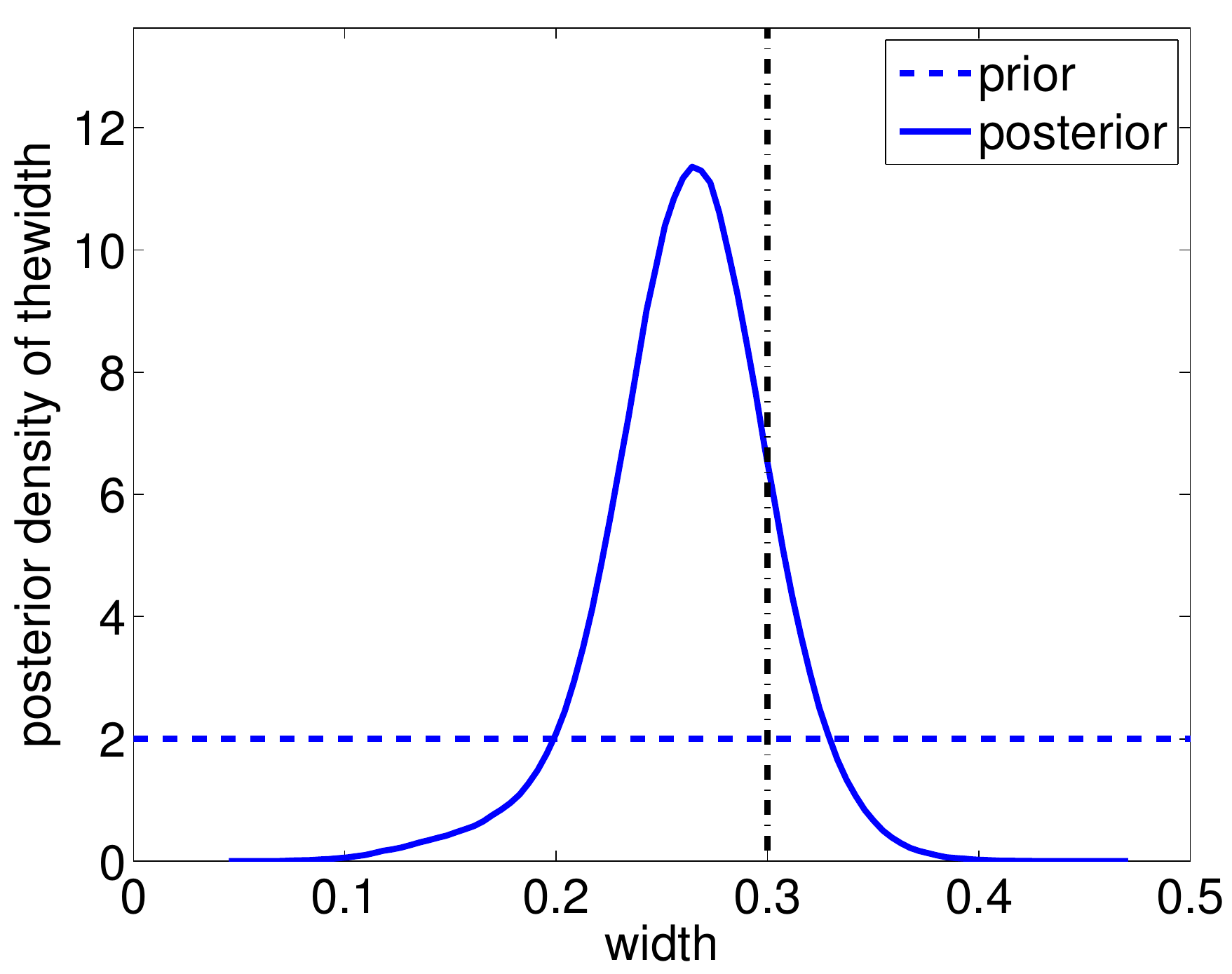}
\includegraphics[scale=0.25]{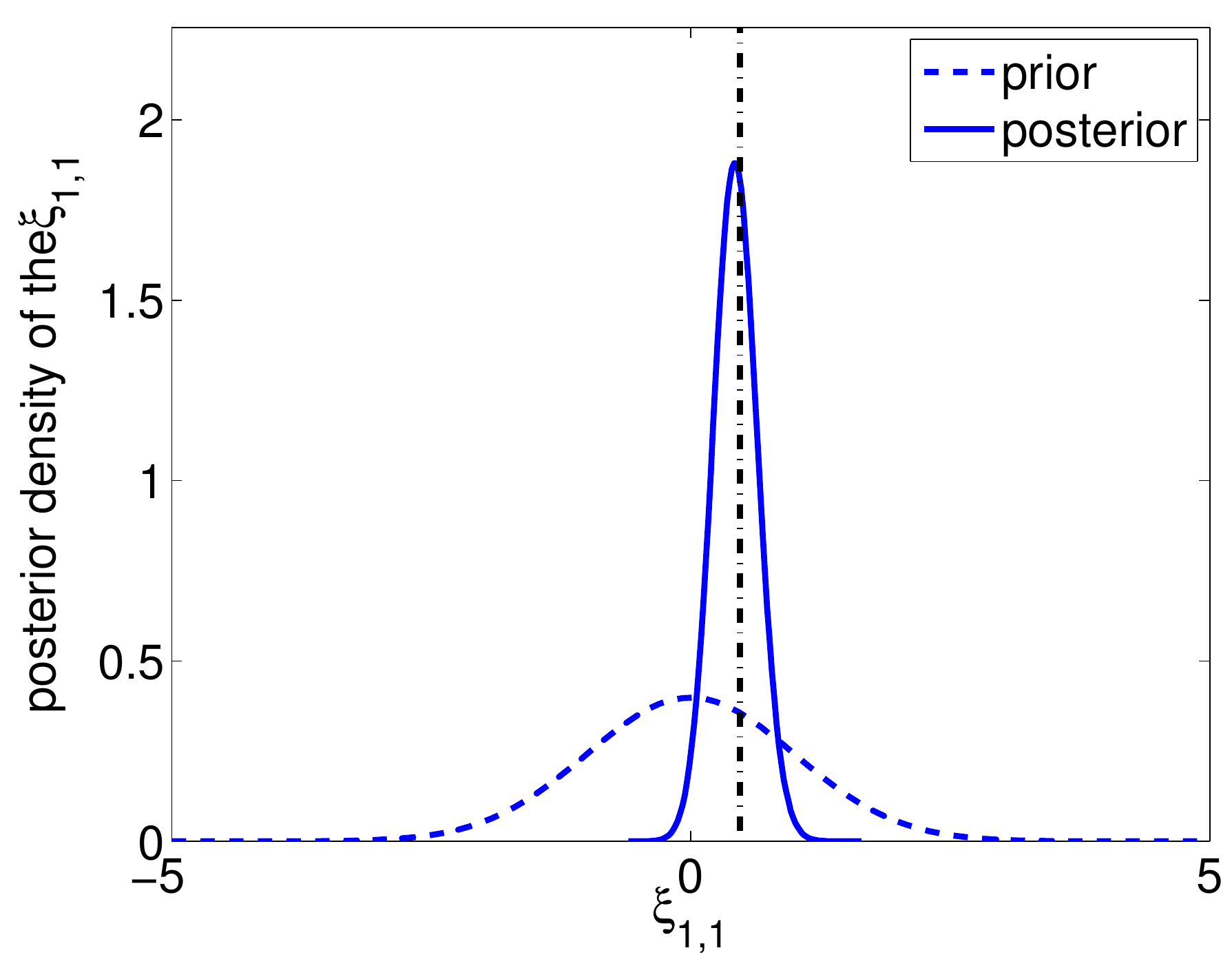}\\
\includegraphics[scale=0.25]{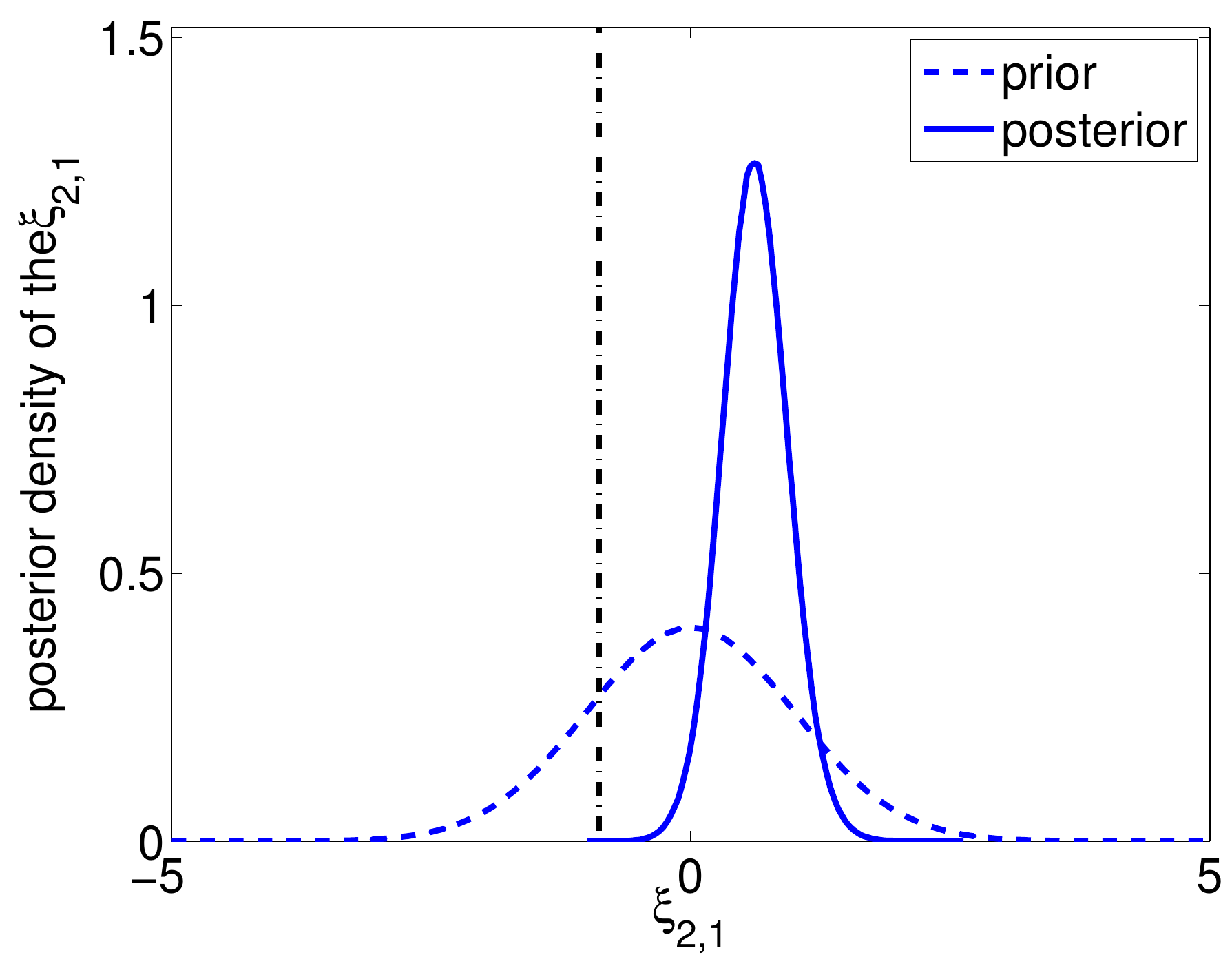}
\includegraphics[scale=0.25]{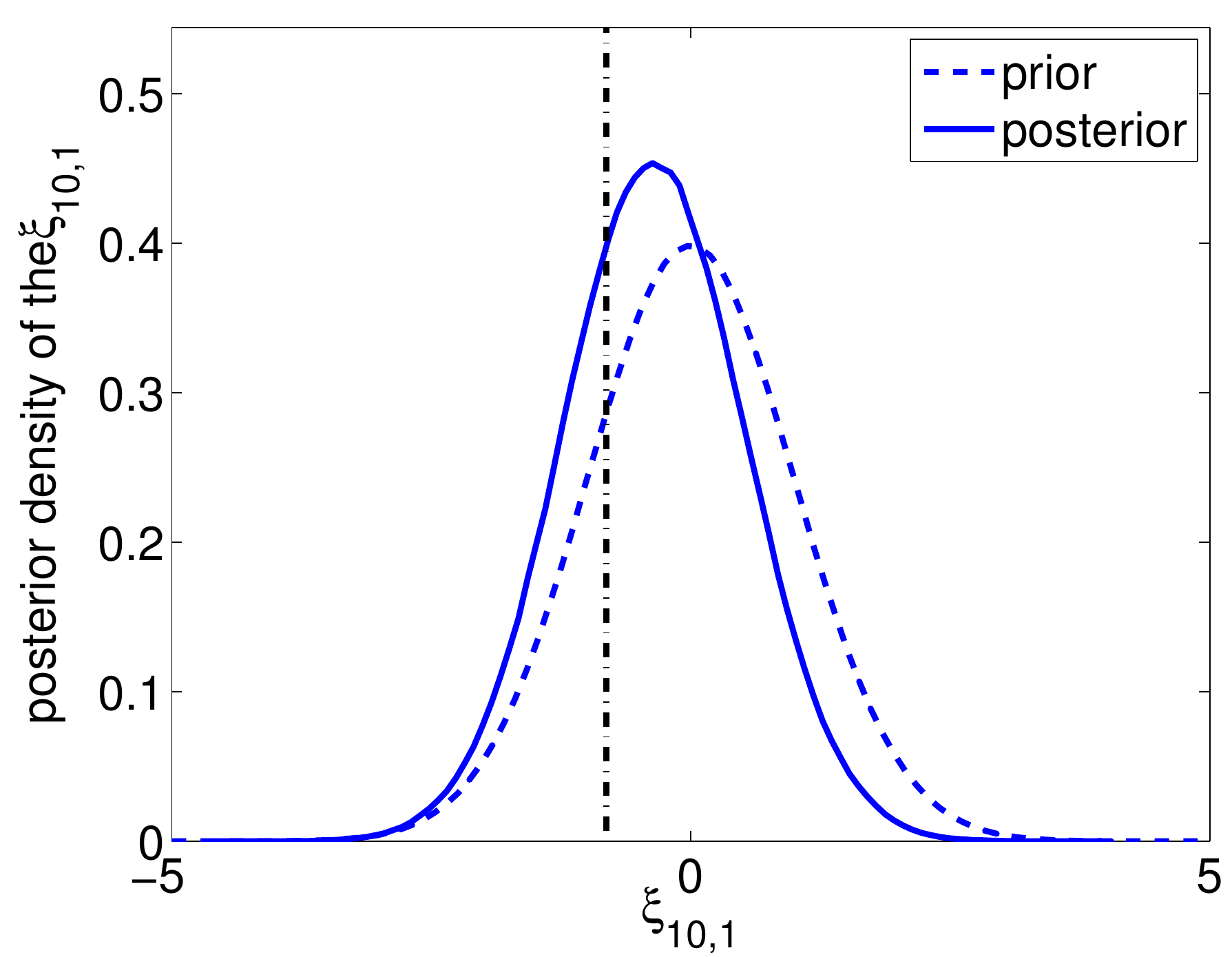}
\includegraphics[scale=0.25]{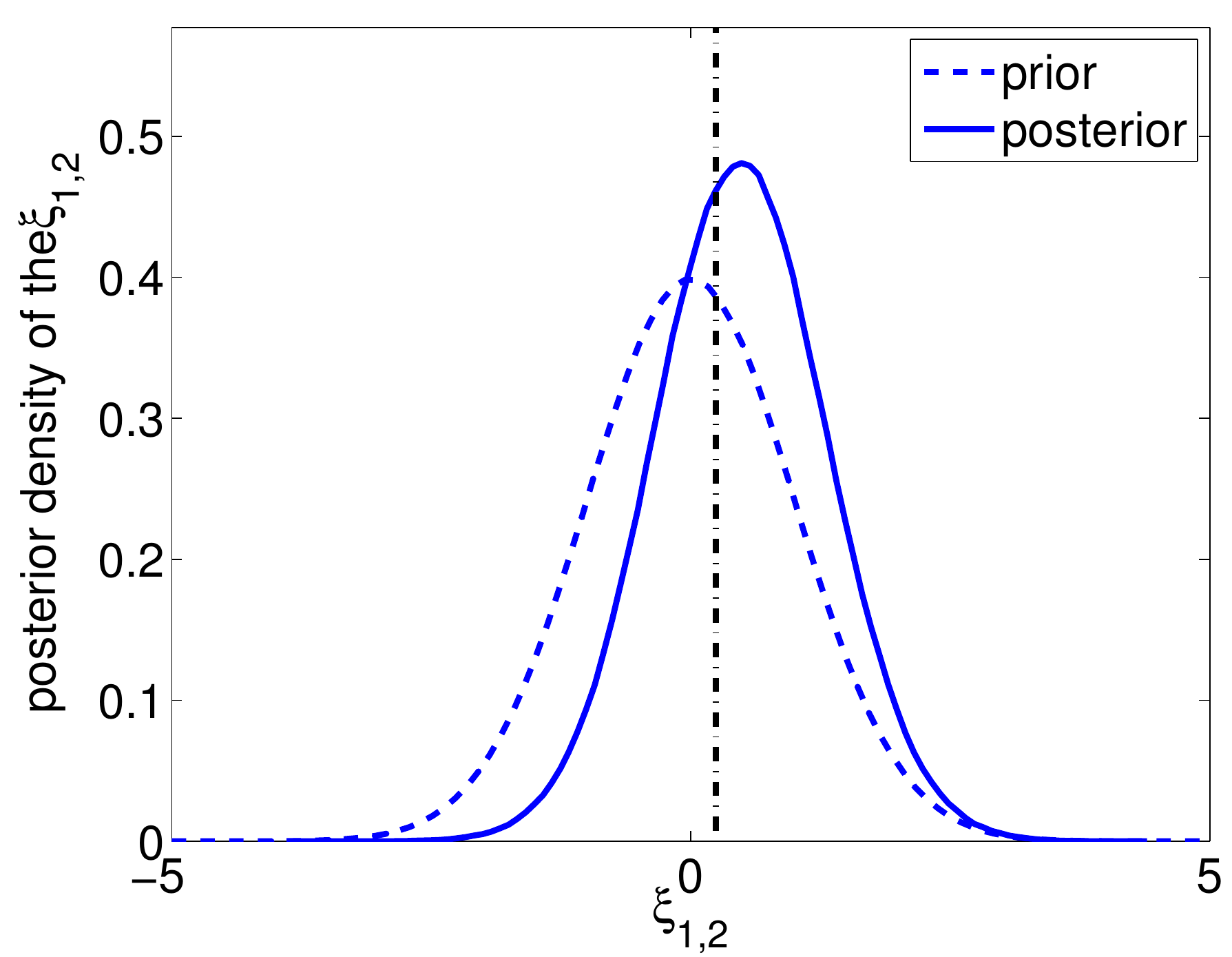}

 \caption{Prior and posterior densities of the unknown $u$}
    \label{Figure21}
\end{center}
\end{figure}

\clearpage
\section{Conclusion}

The Bayesian framework provides a rigorous quantification of the uncertainty in the solution to the inverse problem of estimating rock properties given data from the subsurface flow. A key aspect of the proposed Bayesian approach is to incorporate prior knowledge that honors geometric features relevant to the characterization of complex geologic properties of the subsurface. Although other authors have considered geometrically defined priors,
see for example \cite{Efendiev1,Mondal,Liu2005147,EnKF_level,EfendievL}, this is the first paper to give a rigorous
function-space based Bayesian formulation for such problems \cite{AMS10}.
Such formulations lead to the development of improved algorithms and allow
for rigorous estimation of the various approximation errors that necessarily
enter into the Bayesian approximation of inverse problems. In the present
work we establish the existence and well-posedness of the Bayesian posterior
that arises from determination of permeability within a Darcy flow model,
define function-space MCMC methods, which therefore have convergence rates
independent of the level of mesh-refinement used, and demonstrate the efficacy
of these methods on a variety of test problems exhibiting faults, channels
and spatial variability within different parts of the rock formations.
Particular highlights of the work include: (i) the introduction of a novel
Metropolis-within-Gibbs methodology which separates the effect of
parameters describing the geometry from those describing spatial
variability to accelerate convergence and does not require adjoint
solves, only forward model runs; (ii) demonstration of choices of prior
on the permeability values which lead to H\"older, but not Lipshitz,
continuity of the posterior distribution with respect to perturbations
in the data.

Our results indicate that the mean of the posterior often produce parameters whose permeabilities resemble the truth. However, substantial uncertainty in the inverse problem arises from the observational noise and the lack of observations. Increasing the accuracy in the data or increasing the measurement locations resulted in a significant decrease in the uncertainty in the inverse problem. In other words, we obtained posterior densities concentrated around the truth. In contrast to deterministic inverse problems where a variational optimization method is implemented to recover the truth, the proposed application of the Bayesian framework provides a derivative free method that produces a reasonable estimate of the truth alongside with an accurate estimate of its uncertainty. The present study indicates that the Bayesian framework herein, and resulting algorithmic approaches, have the potential to be applied to more complex
flow models and geometries arising in subsurface applications where
uncertainty quantification is required.

There are a number of natural directions in which this work might be
extended. As mentioned earlier in the text, the study of model error
is potentially quite fruitful: there is substantial gain in computational
expediency stemming from imposing simple models of the geometry; determining
how this is balanced by loss of accuracy when the actual data contains more
subtle geometric effects, not captured our models, is of interest. It is
also of interest to consider implementation of reversible jump type
algorithms \cite{green1995reversible}, in cases where the number of geometric
parameters (e.g. the number of layers) is not known. And finally it will be
of interest to construct rigorous Bayesian formulation of geometric inverse
problems where the interfaces are functions, and require an infinite set of
parameters to define them. We also highlight the fact that although we
emphasize the importance of MCMC methods which are mesh-independent this does
not mean that we have identified the definitive version of such methods; indeed
it would be very interesting to combine our mesh-independent approach with
other state-of-the art ideas in MCMC sampling such as adaptivity and delayed
acceptance, which are used in the context of geophysical applications in
\cite{cui2011bayesian}, and Riemannian manifold
methods \cite{girolami2011riemann}, which also give rise to a natural
adaptivity.

\section{Appendix}
\begin{proof}[Proof of Theorem \ref{t:c}] Since $L$ comprises a finite number of linear functionals it suffices
to prove continuity of the solution $p$ to (\ref{eq:elliptic}) with
respect to changes in the parameters $u$ which define $\kappa.$
Assuming $p$ (resp. $p^{\ve} $) is the solution of (\ref{eq:elliptic}) with  $f\in V^*$ and diffusion coefficient $\kappa\in X$ (resp. $\kappa^{\ve}\in X$) we first show that
\begin{equation}
\label{eq:estimate1}
\|p^{\ve}-p\|_V\leq\frac{1}{\kappa^{\ve}_{min}}\left(\int_D |\kappa^{\ve}-\kappa|^2|\nabla p|^2 dx\right)^{\frac{1}{2}}.
\end{equation}
Define $e=p^{\ve}-p$, $d=\kappa^{\ve}-\kappa$. Then by (\ref{eq:elliptic})
we have that $e\in H_0^1(D)$, $d\in L^{\infty}(D)$ satisfy
\begin{eqnarray*}
    -\nabla\cdot(\kappa^{\ve}\nabla e)& =\nabla\cdot(d\nabla p),\quad &x \in D\nonumber\\
    \phantom{-\nabla\cdot(\kappa^{\ve}(x}e & =0,\quad &x \in \partial D.
   % \label{eq:del_equ}
\end{eqnarray*}
Integration by parts gives
$$\int_D\kappa^{\ve}\nabla e\cdot \nabla e dx=-\int_D d\nabla p\cdot \nabla e dx.$$
Hence,
  \begin{eqnarray*}
      \kappa^{\ve}_{min}\|e\|^2_V  \leq \int_D|d\nabla p\cdot \nabla e|dx
                            \leq \int_D|d||\nabla p| |\nabla e| dx
                                        \leq \left(\int_D d^2|\nabla p|^2 dx\right)^{\frac{1}{2}}\|e\|_V,
\end{eqnarray*}
therefore (\ref{eq:estimate1}) holds. Now, by (\ref{eq:k_piecewise}) we have
\begin{equation}
\kappa^{\ve}-\kappa=\sum_{i=1}^n\sum_{j=1}^n (\kappa_j^{\ve}(x)-\kappa_i(x)) \chi_{D_j^{\ve}\cap D_i}(x).
\label{eq:del_k_function}
\end{equation}
Substituting (\ref{eq:del_k_function}) into (\ref{eq:estimate1}), we obtain that

\begin{eqnarray*}
\fl  \| p^{\ve}-p\|_V & \leq \frac{1}{\kappa^{\ve}_{min}}\left(\sum_{i=1}^n\sum_{j=1}^n\int_{D_j^{\ve}\cap D_i}|\kappa_j^{\ve}(x)-\kappa_i(x)|^2|\nabla p|^2dx\right)^{\frac{1}{2}}\nonumber\\
  & \leq \frac{1}{\kappa^{\ve}_{min}}\Big(\sum_{i=1}^n\|\kappa_i^{\ve}-\kappa_i\|_{X}^2\int_{D_i^{\ve}\cap D_i}|\nabla p|^2dx  +\sum_{i\neq j}\|\kappa_j^{\ve}-\kappa_i\|_{X}^2\int_{D_j^{\ve}\cap D_i}|\nabla p|^2dx \Big)^{\frac{1}{2}}\\
  \fl & \leq \frac{1}{\kappa^{\ve}_{min}}\Big(\sum_{i=1}^n\|\kappa_i^{\ve}-\kappa_i\|_{X}^2\int_{D_i^{\ve}\cap D_i}|\nabla p|^2dx +\sum_{i\neq j}2\bigl(\|\kappa_j^{\ve}\|_X^2+\|\kappa_i\|_{X}^2\bigr)\int_{D_j^{\ve}\cap D_i}|\nabla p|^2dx \Big)^{\frac{1}{2}}.
\end{eqnarray*}

The first term on the right hand side goes to zero as $\ve \to 0$ because
$p \in V$ and small changes in $u \in U \subset \cB$, measured with respect
to the norm $\|\cdot\rb$, lead to small
changes in the $\kappa_i$ in $X$ for both the constant and continuous function models.
Since Hypothesis \ref{hy:geo} holds for all these geometric models of the domains $D_i$, and because $p \in V$,
the dominated convergence
theorem shows that
$$\int_{D_j^{\ve}\cap D_i}|\nabla p|^2dx=\int_{D}\chi_{D_j^{\ve}\cap D_i}|\nabla p|^2dx
\to 0$$
as $\ve \to 0$.
Thus the second term goes to zero since $\kappa_i$ and $\kappa_j^{\ve}$ are bounded on bounded subsets of $X$.
\end{proof}

\begin{lemma}
Assume that $\kappa(x)$ is a piecewise constant function corresponding to parameter $u \in
U \subset \cB$.
Let $\Phi(u;y)$ be the model-data misfit function in (\ref{eq:potential}) and
assume that $y, y' \in Y$ with $\max\{|y|_{\Gamma},|y'|_{\Gamma}\}<r$.
If the prior $\mu_0$ on $u$ is constructed from the exponential distribution
on the values of the $\kappa_i$, then for any $\iota\in(0,1)$, $\int_U M^\iota(r,u) d\mu_0<\infty,$ and
     \begin{equation}
  \label{eq:phi_integral}
  \int_U |\Phi(u;y)-\Phi(u,y')|^{\iota}d\mu_0\leq C~|y-y'|_{\Gamma}^{\iota},
  \end{equation}
  where $C=\int_U M^\iota(r,u) d\mu_0$.
  \end{lemma}
  \begin{proof}
    By the Lax-Milgram Theorem, we have
     \begin{eqnarray*}
    |\cG(u)|_{\Gamma}& \leq \frac{C}{\min_i\{\kappa_i\}}= C\max_i\{\kappa_i^{-1}\},
    \end{eqnarray*}
    which yields
    \begin{eqnarray*}
\fl     \int_U M^{\iota}(r,u)d\mu_0 \leq \int_U \left(r+C\max_i\{\kappa_i^{-1}\}\right)^{\iota}d\mu_0 \nonumber\\
 \fl = \int_{\{r>C\max_i\{\kappa_i^{-1}\}\}} \left(r+C\max_i\{\kappa_i^{-1}\}\right)^{\iota}d\mu_0
 + \int_{\{r\leq C\max_i\{\kappa_i^{-1}\}\}}  \left(r+C\max_i\{\kappa_i^{-1}\}\right)^{\iota}d\mu_0\nonumber\\
\leq(2r)^{\iota}\mu_0\left({\{r>C\max_i\{\kappa_i^{-1}\}\}}\right)+(2C)^{\iota}\int_U\max_i\{\kappa_i^{-\iota}\}d\mu_0 < \infty.
    \end{eqnarray*}
    The second term of the right hand side is finite, since when $\iota \in (0,1)$,
    \begin{eqnarray*}
\fl     \int_U\max_i\{\kappa_i^{-\iota}\}d\mu_0 \leq  \int_U\sum_{i=1}^{n}\frac{1}{\kappa_i^{\iota}}d\mu_0
 =  \sum_{i=1}^n\int_0^{+\infty}\lambda_i\frac{1}{\kappa_i^{\iota}}\exp(-\lambda_i\kappa_i)d\kappa_i\\
 \leq  \sum_{i=1}^n\lambda_i\left(\int_0^{1}\frac{1}{\kappa_i^{\iota}}d\kappa_i+\int_1^{+\infty}\exp(-\lambda_i\kappa_i)d\kappa_i\right) < \infty.
    \end{eqnarray*}
Therefore, by (\ref{eq:diff_phi}) we obtain
  \begin{equation*}
  \int_U |\Phi(u;y)-\Phi(u,y')|^{\iota}d\mu_0\leq \int_U M^{\iota}(r,u)d\mu_0~|y-y'|_{\Gamma}^{\iota}.
  \end{equation*}
  \end{proof}
  \begin{lemma}
Assume that $\kappa(x)$ is a piecewise constant function corresponding to parameter $u \in
U \subset \cB$.
Let $\mu_0$ be the prior distribution on $U$
and such that the values of permeability $\kappa_i$ are drawn from the
exponential distribution. Then for any $\iota \in (0,1)$,
 $$\mu_0\left(|\Phi(u;y)-\Phi(u,y')|>1\right)\leq\int_U M^{\iota}(r,u)d\mu_0~|y-y'|_{\Gamma}^{\iota}.$$
  \label{lem:markov}
  \end{lemma}
\begin{proof}
By the Markov inequality we have
\begin{eqnarray*}
\fl  \mu_0 \left( | \Phi(u,y) - \Phi(u,y') | > 1 \right) =  \mu_0 \left( \frac{1}{2} \left| \langle y-y' , y+y'-2\mathcal{G}(u) \rangle_{\Gamma} \right| > 1 \right) \\[0.7em]
\fl \leq  \mu_0 \left( \left| y-y' \right|_{\Gamma} \left| y+y'-2\mathcal{G}(u) \right|_{\Gamma} > 2 \right)
=  \mu_0 \left( \left| y+y'-2\mathcal{G}(u) \right|_{\Gamma} > \frac{2}{\left| y-y' \right|_{\Gamma}} \right) \\
\fl \leq  \frac{\left| y-y' \right|_{\Gamma}^{\iota}}{2^{\iota}} \int_{U} \left| y+y'-2\mathcal{G}(u) \right|_{\Gamma}^{\iota} ~\mathrm{d}\mu_0 \leq  \frac{\left| y-y' \right|_{\Gamma}^{\iota}}{2^{\iota}} \int_{U} 2^{\iota} \left( r+|\mathcal{G}(u)|_{\Gamma} \right)^{\iota} ~\mathrm{d}\mu_0 \\
\leq  \left| y-y' \right|_{\Gamma}^{\iota} \int_{U} M^{\iota}(r,u) ~\mathrm{d}\mu_0
\end{eqnarray*}
\end{proof}

\begin{proof}[Proof of Proposition \ref{thm:mh}]
In order to construct a prior-reversible proposal for density $\pi_0(u)$ on $U$
we first extend this density to a target on the whole of $\bbR^d$ by setting the density
to zero outside $U$. We then draw $w$ from $p(u,w)$ and define $v$ according
to
\begin{equation}
 v= \left\{ \begin{array}{cc} w & \mathrm{with~probability}~a(u,w) \\ u & \mathrm{otherwise}\end{array} \right. ,
 \label{eq:mh4v}
\end{equation}
where
\begin{equation*}
a(u,w) = \frac{\pi_0(w) p(w,u)}{\pi_0(u) p(u,w)} \wedge 1.
\end{equation*}
Together these steps define a Markov kernel $q(u,v)$ on $U \times U$ which
is $\pi_0$ reversible. Furthermore, as we now show, the method reduces
to that given in Proposition \ref{thm:mh}. Because of the property (\ref{eq:rev_p}) of $p$ we have that, if $u \in U$, $w \in U$ then
\begin{equation*}
a(u,w) = \frac{\pi_0(w) p(w,u)}{\pi_0(u) p(u,w)} \wedge 1 = 1.
\end{equation*}

If $u \in U$, $w \notin U$ then
\begin{equation*}
a(u,w) = \frac{\pi_0(w) p(w,u)}{\pi_0(u) p(u,w)} \wedge 1 = 0 \wedge 1 = 0.
\end{equation*}

Therefore, if $u \in U$, then
\begin{equation*}
a(u,w) = \left\{ \begin{array}{ll} 1 & w \in U \\ 0 & w \notin U \end{array} \right. ,
\end{equation*}
and the Metropolis-Hastings algorithm (\ref{eq:mh4v}) simplifies to give the process
where $w$ is drawn from $p(u,w)$ and we set
\begin{equation*}
v = \left\{ \begin{array}{ll} w & w \in U \\ u & w \notin U \end{array} \right. .
\end{equation*}
This coincides with what we construct for $q(u,v)$ in (\ref{eq:gen_v}).
Hence, the process in (\ref{eq:gen_v}) is actually generated by a
Metropolis-Hastings algorithm for target distribution $\pi_0(u)$.
By the theory of Metropolis-Hastings algorithms,
we know that $q(u,v)$ will be reversible with respect to the equilibrium distribution $\pi_0$.
\end{proof}
\section*{References}
\bibliography{geo_elliptic}

\begin{thebibliography}{10}

\bibitem{beskos2008mcmc}
A.~Beskos, G.~Roberts, A.~M. Stuart, and J.~Voss.
\newblock {MCMC} methods for diffusion bridges.
\newblock {\em Stochastics and Dynamics}, 8(03):319--350, 2008.

\bibitem{Gelman}
S.~P. Brooks and A.~Gelman.
\newblock General methods for monitoring convergence of iterative simulations.
\newblock {\em Journal of Computational and Graphical Statistics}, 7:434--455,
  December 1998.

\bibitem{CW}
J.~N. Carter and D.~A. White.
\newblock History matching on the {Imperial College} fault model using parallel
  tempering.
\newblock {\em Computational Geosciences}, 17:43--65, 2013.

\bibitem{EnKF_level}
H.~Chang, D.~Zhang, and Z.~Lu.
\newblock History matching of facies distribution with the enkf and level set
  parameterization.
\newblock {\em J. Comput. Phys.}, 229(20):8011--8030, October 2010.

\bibitem{CDRS08}
S.~L. Cotter, M.~Dashti, J.~C. Robinson, and A.~M. Stuart.
\newblock {B}ayesian inverse problems for functions and applications to fluid
  mechanics.
\newblock {\em Inverse Problems}, 25:doi:10.1088/0266--5611/25/11/115008, 2009.

\bibitem{CDS09}
S.~L. Cotter, M.~Dashti, and A.~M. Stuart.
\newblock Approximation of {B}ayesian inverse problems for {PDEs}.
\newblock {\em SIAM Journal on Numerical Analysis}, 48(1):322--345, 2010.

\bibitem{CDRS09}
S.~L. Cotter, M.~Dashti, and A.~M. Stuart.
\newblock Variational data assimilation using targetted random walks.
\newblock {\em Int. J. Num. Meth. Fluids}, 68(4):403--421, 2012.

\bibitem{David}
S.~L. Cotter, G.~O. Roberts, A.~M. Stuart, and D.~White.
\newblock {MCMC} methods for functions: Modifying old algorithms to make them
  faster.
\newblock {\em Statistical Science}, 28(3):424--446, 08 2013.

\bibitem{cui2011bayesian}
T~Cui, C~Fox, and MJ~O'Sullivan.
\newblock Bayesian calibration of a large-scale geothermal reservoir model by a
  new adaptive delayed acceptance metropolis hastings algorithm.
\newblock {\em Water Resources Research}, 47(10), 2011.

\bibitem{DHS12}
M.~Dashti, S.~Harris, and A.~M. Stuart.
\newblock Besov priors for {B}ayesian inverse problems.
\newblock {\em Inverse Problems and Imaging}, 6:183--200, 2012.

\bibitem{DS11}
M.~Dashti and A.~M. Stuart.
\newblock Uncertainty quantification and weak approximation of an elliptic
  inverse problem.
\newblock {\em SIAM J. Num. Anal.}, 49:2524--2542, 2011.

\bibitem{0266-5611-25-12-125001}
O.~Dorn and D.~Lesselier.
\newblock Level set methods for inverse scattering-some recent developments.
\newblock {\em Inverse Problems}, 25(12):125001, 2009.

\bibitem{Dorn}
O.~Dorn and R.~Villegas.
\newblock History matching of petroleum reservoirs using a level set technique.
\newblock {\em Inverse Problems}, 24(3):035015, 2008.

\bibitem{Efendiev1}
Y.~Efendiev, A.~Datta-Gupta, X.~Ma, and B.~Mallick.
\newblock {Efficient sampling techniques for uncertainty quantification in
  history matching using nonlinear error models and ensemble level upscaling
  techniques}.
\newblock {\em Water Resources Research}, 45(W11414):11pp, 2009.

\bibitem{EmeRey}
A.~A. Emerick and A.~C. Reynolds.
\newblock Investigation of the sampling performance of ensemble-based methods
  with a simple reservoir model.
\newblock {\em Computational Geosciences}, 17(2):325--350, 2013.

\bibitem{Fra70}
J.~N. Franklin.
\newblock Well-posed stochastic extensions of ill-posed linear problems.
\newblock {\em Journal of Mathematical Analysis and Applications},
  31(3):682--716, 1970.

\bibitem{girolami2011riemann}
M.~Girolami and B.~Calderhead.
\newblock Riemann manifold langevin and hamiltonian monte carlo methods.
\newblock {\em Journal of the Royal Statistical Society: Series B (Statistical
  Methodology)}, 73(2):123--214, 2011.

\bibitem{green1995reversible}
P.~J. Green.
\newblock Reversible jump markov chain monte carlo computation and bayesian
  model determination.
\newblock {\em Biometrika}, 82(4):711--732, 1995.

\bibitem{HSV12}
M.~Hairer, A.~M. Stuart, and S.~J. Vollmer.
\newblock Spectral gaps for {M}etropolis-{H}astings algorithms in infinite
  dimensions.
\newblock Ann. Appl. Prob. To Appear.

\bibitem{Us}
M.~A. Iglesias, K.~J.~H. Law, and A.~M. Stuart.
\newblock Evaluation of {G}aussian approximations for data assimilation in
  reservoir models.
\newblock {\em Computational Geosciences}, pages 1--35, 2013.

\bibitem{Iglesias4}
M.~A. Iglesias and D.~McLaughlin.
\newblock Level-set techniques for facies identification in reservoir modeling.
\newblock {\em Inverse Problems}, 27:035008, 2011.

\bibitem{kaipio1999inverse}
J.~P. Kaipio, V.~Kolehmainen, M.~Vauhkonen, and E.~Somersalo.
\newblock Inverse problems with structural prior information.
\newblock {\em Inverse problems}, 15(3):713, 1999.

\bibitem{KS07b}
J.~P. Kaipio and E.~Somersalo.
\newblock Statistical inverse problems: discretization, model reduction and
  inverse crimes.
\newblock {\em J. Comp. Appl. Math.}, 198:493--504, 2007.

\bibitem{Landa}
J.~Landa and R.~Horne.
\newblock A procedure to integrate well test data, reservoir performance
  history and 4-d seismic information into a reservoir description.
\newblock {\em in Proccedings of the SPE Annual Technical Conference and
  Exhibition, 5-8 October 1997, San Antonio, Texas, U.S.A.}, {S}{P}{E} 38653,
  1997.

\bibitem{La02}
S.~Lasanen.
\newblock Discretizations of generalized random variables with applications to
  inverse problems.
\newblock {\em Ann. Acad. Sci. Fenn. Math. Diss., University of Oulu}, 130,
  2002.

\bibitem{las12}
S.~Lasanen.
\newblock {Non-Gaussian} statistical inverse problems. {Part I}: Posterior
  distributions.
\newblock {\em Inverse Problems and Imaging}, 6(2):215--266, 2012.

\bibitem{las12b}
S.~Lasanen.
\newblock {Non-Gaussian} statistical inverse problems. {Part II}: Posterior
  convergence for approximated unknowns.
\newblock {\em Inverse Problems and Imaging}, 6(2):267--287, 2012.

\bibitem{las09}
M.~Lassas, E.~Saksman, and S.~Siltanen.
\newblock Discretization-invariant {B}ayesian inversion and {B}esov space
  priors.
\newblock {\em Inverse Problems and Imaging}, 3:87--122, 2009.

\bibitem{LPS89}
M.~S. Lehtinen, L.~P{{\"a}}iv{{\"a}}rinta, and E.~Somersalo.
\newblock Linear inverse problems for generalised random variables.
\newblock {\em Inverse Problems}, 5(4):599--612, 1989.

\bibitem{Oliver2}
N.~Liu and D.~S. Oliver.
\newblock {Evaluation of {Monte Carlo} methods for assessing uncertainty}.
\newblock {\em SPE J.}, 8:188--195, 2003.

\bibitem{Liu2005147}
N.~Liu and D.~S. Oliver.
\newblock {Ensemble Kalman} filter for automatic history matching of geologic
  facies.
\newblock {\em Journal of Petroleum Science and Engineering}, 47(3‚Äì4):147
  -- 161, 2005.

\bibitem{Man84}
A.~Mandelbaum.
\newblock Linear estimators and measurable linear transformations on a
  {H}ilbert space.
\newblock {\em Z. Wahrsch. Verw. Gebiete}, 65(3):385--397, 1984.

\bibitem{marzouk2009stochastic}
Y.~Marzouk and D.~Xiu.
\newblock A stochastic collocation approach to {B}ayesian inference in inverse
  problems.
\newblock {\em Communications in Computational Physics}, 6:826--847, 2009.

\bibitem{Mondal}
A.~Mondal, Y.~Efendiev, B.~Mallick, and A.~Datta-Gupta.
\newblock {B}ayesian uncertainty quantification for flows in heterogeneous
  porous media using reversible jump {Markov} chain {Monte Carlo} methods.
\newblock {\em Advances in Water Resources}, 33(3):241 -- 256, 2010.

\bibitem{neal}
R.~Neal.
\newblock Regression and classification using {G}aussian process priors.
\newblock In {\em {B}ayesian Statistics}, pages 475--501, 2010.

\bibitem{OliverReview}
D.~Oliver and Y.~Chen.
\newblock Recent progress on reservoir history matching: a review.
\newblock {\em Computational Geosciences}, 15:185--221, 2011.
\newblock 10.1007/s10596-010-9194-2.

\bibitem{Oliver}
D.~S. Oliver, A.~C. Reynolds, and N.~Liu.
\newblock {\em Inverse Theory for Petroleum Reservoir Characterization and
  History Matching}.
\newblock Cambridge University Press, ISBN: 9780521881517, 1st edition, 2008.

\bibitem{AMS}
A.~M. Stuart.
\newblock {\em The {B}ayesian Approach to Inverse Problems}.
\newblock Lecture Notes, arXiv:1302.6989.

\bibitem{AMS10}
A.~M. Stuart.
\newblock Inverse problems: a {B}ayesian perspective.
\newblock {\em Acta Numerica}, 19(1):451--559, 2010.

\bibitem{Vollmer}
S.~Vollmer.
\newblock Dimension-independent {MCMC} sampling for elliptic inverse problems
  with {non-Gaussian} priors.
\newblock 2013.
\newblock arXiv preprint, arXiv:1302.2213.

\bibitem{EfendievL}
J.~Xie, Y.~Efendiev, and A.~Datta-Gupta.
\newblock Uncertainty quantification in history matching of channelized
  reservoirs using {Markov} chain level set approaches.
\newblock {\em in Proccedings of the SPE Reservoir Simulation Symposium, 21-23
  February 2011, The Woodlands, Texas, USA}, {S}{P}{E} 141811, 2011.

\bibitem{Bi}
B.~Zhuoxin, D.~S. Oliver, and A.~C. Reynolds.
\newblock Conditioning 3d stochastic channels to pressure data.
\newblock {\em SPE Journal}, 5(4):474--484, 2000.

\end{thebibliography}
%\bibliographystyle{unsrt}
%\bibliography{}
\bibliographystyle{plain}
\nocite{*}

\end{document}